\documentclass[11pt,twoside]{preprint}

\usepackage{amssymb}
\usepackage{mathtools}
\usepackage{hyperref}
\usepackage{breakurl}
\usepackage{mhenvs}
\usepackage{mhequ}
\usepackage{mhsymb}
\usepackage{booktabs}
\usepackage{tikz}
\usepackage{mathrsfs}
\usepackage{times}
\usepackage{microtype}
\usepackage{comment}
\usepackage{wasysym}
\usepackage{centernot}
\usepackage{enumitem}
\usepackage{bm}
\usepackage{stackrel}

\DeclareMathAlphabet{\mathbbm}{U}{bbm}{m}{n}

\DeclareFontFamily{U}{BOONDOX-calo}{\skewchar\font=45 }
\DeclareFontShape{U}{BOONDOX-calo}{m}{n}{
  <-> s*[1.05] BOONDOX-r-calo}{}
\DeclareFontShape{U}{BOONDOX-calo}{b}{n}{
  <-> s*[1.05] BOONDOX-b-calo}{}
\DeclareMathAlphabet{\mcb}{U}{BOONDOX-calo}{m}{n}
\SetMathAlphabet{\mcb}{bold}{U}{BOONDOX-calo}{b}{n}

\setlist{noitemsep,topsep=4pt}

\makeatletter
\def\DeclareSymbol#1#2#3{\expandafter\gdef\csname MH@symb@#1\endcsname{\tikzsetnextfilename{symbol#1}%
\tikz[baseline=#2,scale=0.15,line join=round]{#3}}\expandafter\gdef\csname MH@symb@#1s\endcsname{\scalebox{0.7}{\tikzsetnextfilename{symbol#1}%
\tikz[baseline=#2,scale=0.15,line join=round]{#3}}}}
\def\<#1>{\csname MH@symb@#1\endcsname}
\makeatother

\DeclareSymbol{0}{-2.4}{\node[dot] {};}
\DeclareSymbol{1}{0}{\draw[white] (-.5,0) -- (.5,0); \draw (0,0)  -- (0,1.5) node[sdot] {};}
\DeclareSymbol{11}{0}{\draw (-0.5,1.2) node[sdot] {} -- (0,0) -- (0.5,1.2) node[sdot] {};}
\DeclareSymbol{111}{0}{\draw (0,0) -- (0,1.2) node[sdot] {}; \draw (-.7,1) node[sdot] {} -- (0,0) -- (.7,1) node[sdot] {};}
\DeclareSymbol{131}{-3}{\draw (0,0) -- (0,-1) -- (1,0) node[sdot] {}; \draw (0,0) -- (0,-1) -- (-1,0) node[sdot] {}; \draw (0,0) -- (0,1.2) node[sdot] {}; \draw (-.7,1) node[sdot] {} -- (0,0) -- (.7,1) node[sdot] {};}
\DeclareSymbol{11}{0}{\draw (-0.5,1.2) node[sdot] {} -- (0,0) -- (0.5,1.2) node[sdot] {};}
\DeclareSymbol{12}{-3}{\draw (-.8,1) node[sdot] {} -- (0,0) -- (0,-1); \draw (1,0) node[sdot] {} -- (0,-1) -- (-1,0) node[sdot] {};}
\DeclareSymbol{30}{-3}{\draw (0,0) -- (0,-1); \draw (0,0) -- (0,1.2) node[sdot] {}; \draw (-.7,1) node[sdot] {} -- (0,0) -- (.7,1) node[sdot] {};}
\DeclareSymbol{10}{-3}{\draw (-.8,1) node[sdot] {} -- (0,0) -- (0,-1);}
\DeclareSymbol{22}{-3}{\draw (0,0.3) -- (0,-1) -- (1,0) node[sdot] {}; \draw (0,0.3) -- (0,-1) -- (-1,0) node[sdot] {};\draw (-.7,1) node[sdot] {} -- (0,0.3) -- (.7,1) node[sdot] {};}
\DeclareSymbol{211}{0}{\draw (-0.5,2.4) node[sdot] {} -- (-1,1.6);\draw (-1.5,2.4) node[sdot] {} -- (-0.5,0.8); \draw (0,1.6) node[sdot] {} -- (-0.5,0.8) -- (0,0) -- (0.5,0.8) node[sdot] {};}

\newcommand{\cut}{\mathfrak{C}}

\newcommand{\for}{\mathrm{for}}

\def\wnorm#1{\lfloor \hspace{-0.29em} \rceil #1 \lfloor \hspace{-0.29em} \rceil}

\def\mainroot{\rho_{\ast}}

\def\CE{\mathcal{E}}

\newcommand{\mcM}{\mathcal{M}}
\newcommand{\mcE}{\mathcal{E}}

\newcommand{\mcA}{\mathcal{A}}
\newcommand{\mcV}{\mathcal{V}}

\newcommand{\mcC}{\mathcal{C}}
\newcommand{\mcB}{\mathcal{B}}

\newcommand{\mcS}{\mathcal{S}}
\newcommand{\mcU}{\mathcal{U}}

\newcommand{\mcT}{\mathcal{T}}

\newcommand{\mcF}{\mathcal{F}}

\newcommand{\mcN}{\mathcal{N}}

\newcommand{\mcD}{\mathcal{D}}
\newcommand{\mcW}{\mathcal{W}}
\newcommand{\mcP}{\mathcal{P}}
\newcommand{\mcG}{\mathcal{G}}
\newcommand{\mcX}{\mathcal{X}}
\newcommand{\mcQ}{\mathcal{Q}}
\newcommand{\mcZ}{\mathcal{Z}}

\newcommand{\cu}{\mathrm{cu}}

\newcommand{\mbbF}{\mathbb{F}}

\newcommand{\mbbM}{\mathbb{M}}

\newcommand{\mbbD}{\mathbb{D}}
\newcommand{\mbbG}{\mathbb{G}}
\newcommand{\mbbA}{\mathbb{A}}

\newcommand{\mbn}{\mathbf{n}}
\newcommand{\mbm}{\mathbf{m}}
\newcommand{\mbj}{\mathbf{j}}
\newcommand{\mbk}{\mathbf{k}}

\newcommand{\T}{\mathbf{T}}

\newcommand{\mfu}{\mathfrak{u}}

\newcommand{\mfF}{\mathfrak{F}}

\newcommand{\mfT}{\mathfrak{T}}
\newcommand{\mfn}{\mathfrak{n}}
\newcommand{\mfo}{\mathfrak{o}}
\newcommand{\mfe}{\mathfrak{e}}

\newcommand{\mfL}{\mathfrak{L}}

\newcommand{\mfC}{\mathfrak{C}}
\newcommand{\mfR}{\mathfrak{R}}

\newcommand{\mft}{\mathfrak{t}}
\newcommand{\mfm}{\mathfrak{m}}
\newcommand{\mfM}{\mathfrak{M}}

\newcommand{\mfp}{\mathfrak{p}}

\newcommand{\mfl}{\mathfrak{l}}
\newcommand{\mfh}{\mathfrak{h}}
\newcommand{\mfq}{\mathfrak{q}}

\newcommand{\mff}{\mathfrak{f}}

\def\wwnorm#1{|\!|\!| #1 |\!|\!|}

\def\cC{\mathscr{C}}
\def\cA{\mathscr{A}}
\def\cG{\mathscr{G}}

\def\cS{\mathscr{S}}
\def\cD{\mathscr{D}}
\def\cE{\mathscr{E}}
\def\cB{\mathscr{B}}

\def\enS{\mathscr{S}} 
\def\Cum{\mathbf{E}^c}
\def\CCum{\textnormal{\tiny{Cum}}}
\def\all{\textnormal{\tiny{all}}}

\def\can{\textnormal{\tiny{can}}}
\def\BPHZ{\textnormal{\tiny \textsc{bphz}}}
\def\rand{\textnormal{\tiny rand}}
\def\Div{\mathrm{Div}}
\def\Dom{\mathrm{Dom}}
\def\DDom{\mathcal{D}}

\def\reduce{\mcb{R}}

\newcommand{\ch}{\mathrm{c}}
\newcommand{\p}{\mathrm{p}}
\newcommand{\inte}{\mathrm{int}}
\newcommand{\exte}{\mathrm{ext}}
\newcommand{\ex}{\mathrm{ex}}
\newcommand{\loc}{\mathrm{loc}}
\newcommand{\Coll}{\mathrm{Coll}}

\newcommand{\go}[1]{\mathbf{#1}}
\newcommand{\bo}[1]{\mathbf{#1}}
\newcommand{\mvert}{\mathcal{V}}
\newcommand{\mtree}{\mathcal{U}}

\newcommand{\set}[1]{\{#1\}} 

\newcommand{\clearrootn}{\tilde{P}}

\def\projplusshape{\mathbf{P}}
\def\projminusshape{\mathbf{P}}

\def\combplus[#1,#2,#3,#4]{\binom{#1\ {\scriptstyle #4} }{#2\ #3}}

\newcommand{\leavesleft}[2]{L(#1,#2)}
\newcommand{\kernelsleft}[2]{K(#1,#2)}
\newcommand{\nodesleft}[2]{\tilde{N}(#1,#2)}

\newcommand{\genvert}[1]{H_{#1}}
\newcommand{\singleslicegenvert}[2]{H_{#1}^{#2}}
\def\singlescalegenvert[#1,#2]{\hat{H}^{#2}_{#1}}
\def\multiscalegenvert[#1,#2]{H^{#2}_{#1}}
\def\Moll{\mathrm{Moll}}

\newcommand{\mmax}[1]{\overline{#1}}
\newcommand{\mmin}[1]{\underline{#1}}

\def\For{\mathfrak{F}}

\def\tDelta{\hat{\Delta}}

\def\nr[#1]{\tilde{N}[#1]} 
\def\inn[#1]{\mathring{N}[#1]}
\def\nrinn[#1]{\hat{N}_{#1}} 
\def\nrmod[#1,#2]{\tilde{N}_{#1}(#2)}
\def\nrinnmod[#1,#2]{\hat{N}_{#1}(#2)}

\def\ident[#1]{\underline{#1}}

\def\mylink#1#2{\mathrel{\vbox{\offinterlineskip\ialign{%
    \hfil##\hfil\cr
    $\scriptscriptstyle#1$\cr
    \noalign{\kern0.1ex}
    $#2$\cr
}}}}
\def\mysublink[#1]#2#3{\mathrel{\vbox{\offinterlineskip\ialign{%
    \hfil##\hfil\cr
    $\scriptscriptstyle#2$\cr
    \noalign{\kern0.1ex}
    $#3$\cr
    \noalign{\kern-0.2ex}
    \smash{\raisebox{-\height}{\hbox{$\scriptscriptstyle #1$}}}\cr
    \noalign{\kern0.2ex}
}}}}
\newcommand{\link}[3]{#1\mylink{#3}{\longleftrightarrow}#2}

\newcommand{\ulink}[3]{#1\mylink{#3}{\leftrightsquigarrow}#2}

\def\allf{\mcb{C}_{\ast}}
\def\smooth{\mcb{C}}
\def\fon[#1]{\cC_{#1}}

\newcommand{\allpon}[1]{\mcP[#1]}
\newcommand{\fillingpon}[1]{\bar{\mcP}[#1]}

\newcommand{\tpode}{\hat{\mcb{A}}}
\newcommand{\ttpode}{\tilde{\mcb{A}}}

\def\mincompproj[#1]{\mfp_{#1}}

\def\Proj_#1{\mathop{\mathrm{Proj}_{#1}}}

\def\negrenorm[#1]{\mfR_{#1}}
\def\topnegrenorm[#1]{\overline{\mfR}_{#1}}
\newcommand{\fullf}{\mathbb{F}_{\pi}}
\newcommand{\fmod}[1]{\mathbb{F}_{#1,\pi}}

\def\quotedge[#1]{E^{q}_{#1}}
\def\cuts{\mathbb{C}_{+}}
\def\posrenorm[#1]{\mcC_{#1}}
\def\topposrenorm[#1]{\overline{\mcC_{#1}}}
\def\cutsmod[#1]{\mathbb{C}_{+,#1}}
\def\fullcuts{\cut}
\def\fullcutsmod[#1]{\cut_{#1}}

\def\fict{\mcb{f}}

\def\emptyset{{\centernot\ocircle}}

\colorlet{symbols}{blue!90!black}
\colorlet{testcolor}{green!60!black}
\colorlet{darkblue}{blue!60!black}
\colorlet{darkgreen}{green!60!black}

\colorlet{redkernel}{red!80}

\def\symbol#1{\textcolor{symbols}{#1}}
\def\1{\mathbf{\symbol{1}}}

\usetikzlibrary{shapes.misc}
\usetikzlibrary{shapes.symbols}
\usetikzlibrary{shapes.geometric}
\usetikzlibrary{snakes}
\usetikzlibrary{decorations}
\usetikzlibrary{decorations.markings}

\usetikzlibrary{calc}
\usetikzlibrary{external}

\newcommand*{\mathcolor}{}
\def\mathcolor#1#{\mathcoloraux{#1}}
\newcommand*{\mathcoloraux}[3]{%
  \protect\leavevmode
  \begingroup
    \color#1{#2}#3%
  \endgroup
}

\def\drawx{\draw[-,solid] (-3pt,-3pt) -- (3pt,3pt);\draw[-,solid] (-3pt,3pt) -- (3pt,-3pt);}
\tikzset{
  coalline/.style={semithick,draw=black},
  coalnode/.style={circle,fill=black!10,draw=black,inner sep=0pt, minimum size=2mm},
  coalnode2/.style={circle,fill=black!60,draw=black,inner sep=0pt, minimum size=2mm},
  shadednode/.style={circle,fill=blue!40,inner sep=0pt,minimum size=9pt},
  shadededge/.style={line width=6pt,blue!40, shorten >= -3pt,shorten <=-3pt},
  subtreenode/.style={circle,fill=gray!40,inner sep=0pt,minimum size=12pt},
  smallsubtreenode/.style={circle,fill=gray!40,inner sep=0pt,minimum size=9pt},
  subtreeedge/.style={line width=10pt,gray!40, shorten >= -3pt,shorten <=-3pt},
  smallsubtreeedge/.style={line width=6pt,gray!40, shorten >= -3pt,shorten <=-3pt},
  root/.style={circle,fill=testcolor,inner sep=0pt, minimum size=2mm},
  logof/.style={circle,fill=darkblue,inner sep=0pt, minimum size=2mm},
  dot/.style={circle,fill=black,inner sep=0pt,minimum size=1mm},
  sdot/.style={circle,fill=black,inner sep=0pt,minimum size=.5mm},
  int/.style={circle,fill=black,draw=black,inner sep=0pt,minimum size=0.7mm},
  circ/.style={circle,draw=black,inner sep=0pt, minimum size=1mm},
  var/.style={circle,fill=black!10,draw=black,inner sep=0pt, minimum size=2mm},
  dotred/.style={circle,fill=black!50,inner sep=0pt, minimum size=2mm},
  generic/.style={semithick,shorten >=1pt,shorten <=1pt},
  oddfunc/.style={generic, dotted},
  dist/.style={ultra thick,draw=testcolor,shorten >=1pt,shorten <=1pt},
  testfcn/.style={ultra thick,testcolor,shorten >=1pt,shorten <=1pt,<-},
  testfunction/.style={ultra thick,testcolor,shorten >=1pt,shorten <=1pt},
  testfcnx/.style={ultra thick,testcolor,shorten >=1pt,shorten <=1pt,<-,
    postaction={decorate,decoration={markings,mark=at position 0.6 with {\drawx}}}},
  kprime/.style={semithick,shorten >=1pt,shorten <=1pt,densely dashed,->},
  kprimex/.style={semithick,shorten >=1pt,shorten <=1pt,densely dashed,->,
    postaction={decorate,decoration={markings,mark=at position 0.4 with {\drawx}}}},
  kernel/.style={semithick,shorten >=1pt,shorten <=1pt,->,draw=black},
  leaf/.style={decorate, decoration = {zigzag, amplitude=1.5pt, segment length=3pt}, semithick,shorten >=1pt,shorten <=1pt,draw=black},
  keps/.style={semithick,densely dashed,shorten >=1pt,shorten <=1pt,->},
  Pkernel/.style={ultra thick,shorten >=1pt,shorten <=1pt,->,draw=blue},
  PkernelBig/.style={very thick,shorten >=1pt,shorten <=1pt,decorate, draw=blue, decoration={zigzag,amplitude=1.5pt,segment length = 3pt,pre length=2pt,post length=2pt}},
  multx/.style={shorten >=1pt,shorten <=1pt,
    postaction={decorate,decoration={markings,mark=at position 0.5 with {\drawx}}}},
  kernelx/.style={semithick,shorten >=1pt,shorten <=1pt,->,
    postaction={decorate,decoration={markings,mark=at position 0.4 with {\drawx}}}},
  kernel1/.style={->,semithick,shorten >=1pt,shorten <=1pt,postaction={decorate,decoration={markings,mark=at position 0.5 with {\draw[-] (0,-0.2) -- (0,0.2);}}}},
  kernel2/.style={->,semithick,shorten >=1pt,shorten <=1pt,postaction={decorate,decoration={markings,mark=at position 0.45 with {\draw[-] (0.05,-0.2) -- (0.05,0.2);\draw[-] (-0.05,-0.2) -- (-0.05,0.2);}}}},
  kernel3/.style={->,semithick,shorten >=1pt,shorten <=1pt,postaction={decorate,decoration={markings,mark=at position 0.45 with {
    \draw[-] (0.075,-0.2) -- (0.075,0.2);
    \draw[-] (-0.075,-0.2) -- (-0.075,0.2);
    \draw[-] (0,-0.2) -- (0,0.2);}}}},
  kernel4/.style={->,semithick,shorten >=1pt,shorten <=1pt,postaction={decorate,decoration={markings,mark=at position 0.45 with {
    \draw[-] (0.15,-0.2) -- (0.15,0.2);
    \draw[-] (0.05,-0.2) -- (0.05,0.2);
    \draw[-] (-0.05,-0.2) -- (-0.05,0.2);
    \draw[-] (-.15,-0.2) -- (-.15,0.2);}}}},  
  kernelBig/.style={semithick,shorten >=1pt,shorten <=1pt,decorate, decoration={zigzag,amplitude=1.5pt,segment length = 3pt,pre length=2pt,post length=2pt}},
  kernelBigg/.style={thick,shorten >=1pt,shorten <=1pt,decorate, decoration={zigzag,amplitude=3.5pt,segment length = 7pt,pre length=2pt,post length=2pt}},
  kernelBigg1/.style={thick,shorten >=1pt,shorten <=1pt,decorate, decoration={saw,amplitude=3.5pt,segment length = 7pt,pre length=2pt,post length=2pt}},
  kernelBigg2/.style={thick,shorten >=1pt,shorten <=1pt,decorate, decoration={bumps,amplitude=3.5pt,segment length = 7pt,pre length=2pt,post length=2pt}},
  rho/.style={dotted,semithick,shorten >=1pt,shorten <=1pt},
  renorm/.style={shape=circle,fill=white,inner sep=1pt},
  labl/.style={shape=rectangle,fill=white,inner sep=1pt},
cumu2n/.style={inner sep=3pt},
cumu2/.style={draw=red!80,fill=red!40},
cumu3/.style={regular polygon, regular polygon sides=3,draw=red!80,rounded corners=3pt,fill=red!40,minimum size=5mm},
cumu4/.style={regular polygon, regular polygon sides=4,draw=red!80,rounded corners=3pt,fill=red!40,minimum size=7mm},
cumu5/.style={regular polygon, regular polygon sides=5,draw=red!80,rounded corners=3pt,fill=red!40,minimum size=7mm},
bcumu2n/.style={inner sep=3pt},
bcumu2/.style={draw=blue!80,fill=blue!40},
bcumu3/.style={regular polygon, regular polygon sides=3,draw=blue!80,rounded corners=3pt,fill=blue!40,minimum size=5mm},
bcumu4/.style={regular polygon, regular polygon sides=4,draw=blue!80,rounded corners=3pt,fill=blue!40,minimum size=7mm},
bcumu5/.style={regular polygon, regular polygon sides=5,draw=blue!80,rounded corners=3pt,fill=blue!40,minimum size=7mm},
  xi/.style={circle,fill=symbols!10,draw=symbols,inner sep=0pt,minimum size=1.2mm},
  xib/.style={circle,fill=symbols!10,draw=symbols,inner sep=0pt,minimum size=1.6mm},
  not/.style={circle,fill=symbols,draw=symbols,inner sep=0pt,minimum size=0.5mm},
  >=stealth,
  }

\newtheorem{assumption}[lemma]{Assumption}

\colorlet{darkblue}{blue!90!black}
\let\comm\comment
\def\martin#1{}
\def\ajay#1{}

\def\s{\mathfrak{s}}
\def\c{\mathfrak{c}}
\newcommand{\e}{\varepsilon}

\def\${|\!|\!|}

\def\Ker{\mathrm{Ker}}
\def\RKer{\mathrm{RKer}}
\def\KerHat{\widehat{\RKer}}
\def\KerTilde{\widetilde{\RKer}}

\def\?{{\color{red}?}}

\def\id{\mathrm{id}}
\def\Max{\mathop{\mathrm{Max}}}
\def\Min{\mathop{\mathrm{Min}}}

\def\restr{\mathord{\upharpoonright}}

\def\Id{\mathrm{Id}}

\def\PPi{\boldsymbol{\Pi}}

\def\id{\mathrm{id}}

\def\Lab{\mathfrak{L}}

\def\Ke{\mathfrak{L}_{+}}
\def\Le{\mathfrak{L}_{-}}

\def\gv{\mathrm{gv}}
\def\Con{\mathrm{Conn}}
\def\Tr{\mathfrak{T}}

\def\sT{{\overline{T}}} 
\def\sn{\overline{\mfn}} 
\def\se{\overline{\mfe}} 
\def\allnodes{N^{\ast}}

\newcommand{\pow}[2]{\mathrm{X}^{#1}_{#2}}
\newcommand{\powroot}[3]{\mathrm{X}^{#1}_{#2,#3}}
\newcommand{\fpowquot}[3]{\overline{\mathrm{X}}^{#1,#3}_{#2}}
\newcommand{\powquot}[3]{\mathrm{X}^{#1,#3}_{#2}}
\newcommand{\powrootquot}[4]{\mathrm{X}^{#1,#4}_{#2,#3}}

\newcommand{\ke}[2]{\Ker^{#1}_{#2}}
\newcommand{\kerec}[3]{\Ker^{#1,#3}_{#2}}

\newcommand{\wwick}[1]{\mathrm{Wick}\left(#1\right)}

\newcommand{\wickleaves}{\tilde{L}}
\newcommand{\genwickleaves}{\mcb{l}}
\newcommand{\allwickleaves}{\mcb{l}^{\all}}

\def\one{\mathbbm{1}}

\let\basepoint\logof
\def\logof{\mathord{{\basepoint}}} 
\def\back{\!\!\!}

\begin{document}

\title{An analytic BPHZ theorem for Regularity Structures}

\author{Ajay Chandra and Martin Hairer}
\institute{University of Warwick, UK\\ \email{A.Chandra@Warwick.ac.uk, M.Hairer@Warwick.ac.uk}}
\date{\today}
\maketitle

\begin{abstract}
We prove a general theorem on the stochastic convergence of appropriately renormalized models arising from nonlinear stochastic PDEs. 
The theory of regularity structures gives a fairly automated framework for studying these problems 
but previous works had to expend significant effort to obtain these stochastic estimates in an ad-hoc manner.

In contrast, the main result of this article operates as a black box which automatically 
produces these estimates for nearly all of
the equations that fit within the scope of the theory of regularity structures. 
Our approach leverages multi-scale analysis strongly reminiscent to that used in constructive field theory, but with several 
significant twists. These come in particular from the presence of ``positive renormalizations'' caused
by the recentering procedure proper to the theory of regularity structure, 
from the difference in the action of the group of possible
renormalization operations, as well as from the fact that we allow for non-Gaussian driving fields. 

One rather surprising fact is that although the ``canonical lift''
is not continuous on any H\"older-type space containing the noise (which is why
renormalization is required), we show that the ``BPHZ lift'' where the 
renormalization constants are computed using the formula given in \cite{BHZalg}, \textit{is} continuous
in law when restricted to a class of stationary random fields with sufficiently many moments.
\end{abstract}

\setcounter{tocdepth}{2}
\tableofcontents
\section{Introduction}
The theory of regularity structures \cite{Regularity} presents a systematic and robust methodology for 
interpreting / solving a wide class of parabolic stochastic partial differential equations (SPDEs).
One natural subclass of such SPDEs are equations of the type 
\begin{equ}\label{semilinear SHE}
(
\partial_{t} - \CL
)
\varphi 
= 
F(\varphi,\zeta)\;,
\end{equ}
where $\CL$ is an elliptic differential operator
(think of the standard Laplacian on $d$-dimensional space $\R^{d}$ or the $d$-dimensional
torus $\T^d$) and $F(\cdot,\cdot)$ is a local non-linearity that depends on less derivatives of $\phi$ than $\CL$.
The solution $\varphi(t,x)$ can be formally thought of as a random space-time field $\phi:\R^{d+1} \rightarrow \R$, 
and $\zeta$ is a random space time field, the prototypical example being given by space-time white 
noise, the centred Gaussian generalised random field satisfying $\E[\zeta(t,x) \zeta(s,y)] = \delta(t-s)\delta(x-y)$.  

Equations like \eqref{semilinear SHE} can formally be obtained as limits of the dynamics of classical 
models of statistical mechanics in suitable regimes. Two intensely studied examples are the 
KPZ equation appearing from the evolution of random interfaces, where $F(\varphi, \zeta) = |\nabla \varphi|^{2} + \zeta$, and the $\Phi^{4}_{d}$ equation appearing from limits of dynamical Ising models, where $F(\varphi,\zeta) = -\varphi^{3} + \zeta$.

An important restriction on the class of equations that \cite{Regularity} is able to treat is \emph{local subcriticality} -- a precise definition of this condition is fairly technical\footnote{See \cite[Sec.~5]{BHZalg}
for a formal definition with a wider scope (it also incorporates systems of equations)
than the somewhat informal definition given in \cite{Regularity}.}  
but heuristically it states that all non-linear terms appearing on the right hand side 
of \eqref{semilinear SHE} behave like perturbations of the noise term as one rescales the equation towards smaller scales and that these perturbations can be written as multi-linear functionals of the driving noise $\zeta$. 

An immediate obstacle to rigorously formulating these SPDEs is that the roughness of the driving 
noise $\zeta$ forces the solution $\varphi$ to live in a space of functions / distributions with insufficient 
regularity for $F(\varphi,\zeta)$ to have a canonical meaning.  
As with any other integration theory, the program of giving a rigorous solution theory to such SPDEs 
entails understanding how to define these solutions as limits of well-defined approximate solutions.

To obtain such well-defined approximations we can convolve the driving noise with some smooth approximate identity $\rho_{\eps}$ with $\eps > 0$ and $\lim_{\eps \downarrow 0} \rho_{\eps}(\cdot) = \delta(\cdot)$. It is then 
typically easy to obtain, for every realization of $\zeta$, a local existence theorem for 
\begin{equ}\label{mollified semilinear SHE}
(
\partial_{t} - \CL
)
\varphi_{\eps}
= 
F(\varphi_{\eps},
\zeta_{\eps})\;,
\end{equ}
where $\zeta_{\eps} \eqdef \zeta \ast \rho_{\eps}$. 
Naively one would hope to define the solution of \eqref{semilinear SHE} via the limit $\varphi \eqdef \lim_{\eps \downarrow 0} \varphi_{\eps}$ -- unfortunately it is generic for this limit to fail to exist and this signals the need to \emph{renormalize}.
The process of renormalization entails prescribing a scheme of modifying the nonlinearity $F$ at each value of $\eps$, giving rise to a family of non-linearities $F_{\eps}(\cdot,\cdot)$, such that if one considers the modified equations
\begin{equ}\label{renormalized mollified semilinear SHE}
(
\partial_{t} - \CL
)
\hat{\varphi}_{\eps}
= 
F_\eps(\hat{\varphi}_{\eps},\zeta_{\eps})\;,
\end{equ}
then the space-time functions $\hat{\varphi}_{\eps}$ will converge in probability, as $\eps \downarrow 0$, to a non-trivial limiting space-time distribution $\varphi$.
The need for renormalization can be seen very explicitly when one uses Picard iteration to replace the non-linearity $F$ with a formal expansion in terms of multilinear functionals of the linear solution. 

We now give a detailed presentation of this for the specific case where $F(\phi,\zeta) = \phi^{3} + \zeta$
and we work on the $3$-dimensional torus. 
We write $G$ for the space-time Green's function of the differential operator $(\partial_{t} - \CL)$, recentered
so that it averages to $0$.
Then, using Picard iteration, one obtains a formal expansion of the RHS of \eqref{mollified semilinear SHE}
\begin{equs}
(\partial_{t} - \CL) 
\varphi_{\eps}
=&
\zeta_{\eps} + 
(G \ast \zeta_{\eps})^{3}
+
3[G\ast(G \ast \zeta_{\eps})^{3}](G \ast \zeta_{\eps})^{2}
\\
&
+
3[G\ast(G \ast \zeta_{\eps})^{3}]^{2}(G \ast \zeta_{\eps})
+
[G\ast(G \ast \zeta_{\eps})^{3}]^{3}
+
\cdots
\end{equs}
where $\ast$ denotes space-time convolution. 
Observe that $G \ast \zeta_{\eps}$ is nothing but the solution to the linear equation underlying \eqref{mollified semilinear SHE}. 
Some simple graphical notation borrowed from \cite{Regularity}\footnote{This graphical notation is more than a minor convenience -- developing a formalism for seeing all these multilinear functionals of $\zeta$ as combinatorial objects allows us to describe various complicated operations with full precision in great generality.} clarifies this expansion -- using this for the first three terms of the RHS gives
\begin{equ}\label{expansion}
(\partial_{t} - \CL) 
\varphi_{\eps}
=
\<0>_{\eps}
+
\<111>_{\eps}
+
3
\<131>_{\eps}
+
\cdots\;.
\end{equ}
As $\eps \downarrow 0$ the second two objects on the RHS (and infinitely many others hiding in the ``$\cdots$'') will not converge for generic realizations of $\zeta$. 
However, one can switch to a stochastic point of view and see that the $\eps \downarrow 0$ limits of the space-time functions
\begin{equs}\label{renormalization of trees}
\widehat{\<111>}_{\eps}
&\eqdef
\<111>_{\eps} - 3\,C^\eps_{\<11s>}\,\<1>_{\eps}\\
\widehat{\<131>}_{\eps}
&\eqdef
\<131>_{\eps} + 3\, (C^\eps_{\<11s>})^2\,\<10>_{\eps}
-
C^\eps_{\<11s>}\,\<30>_{\eps}
-3C^\eps_{\<11s>}\,\<12>_{\eps}
- 3C^\eps_{\<22s>}\,\<1>_{\eps}
\end{equs}
exist for \emph{almost every} realization of $\zeta$ in suitable spaces of space-time distributions,
provided that we set
\begin{equ}
C^\eps_{\<11s>} \eqdef \E[\<11>_{\eps}(0)]
\quad
\textnormal{and}
\quad
C^\eps_{\<22s>}
\eqdef
\E[\<22>_{\eps}(0)]
\;.
\end{equ}
We also extended our graphical notation by setting
\begin{equ}
\<11>_{\eps} \eqdef (G \ast \zeta_{\eps})^{2}\;,\quad
\<10>_{\eps} \eqdef G\ast(G \ast \zeta_{\eps})\;,\quad
\<12>_{\eps} \eqdef [G\ast(G \ast \zeta_{\eps})](G \ast \zeta_{\eps})^2\;,
\end{equ}
%
etc.\ and used the notation $\tau_{\eps}(0)$ to indicate we are evaluating the space-time function $\tau_{\eps}$ at $z = 0$.

This $\eps$-dependent subtraction of ``counterterms'' on individual terms of our formal expansion \eqref{expansion} to make each term convergent as $\eps \downarrow 0$ is reminiscent of the \emph{perturbative renormalization} developed for Quantum Field Theory (QFT)\footnote{One important difference is that the perturbative renormalization in QFT is carried out on expectation values, or ``amplitudes'', not stochastic objects.}.

Making these individual terms convergent is only part of the procedure of perturbative renormalization, the other is rewriting the nonlinearity $F$ so that the needed counterterms are automatically generated in the Picard iteration. 
If, in \eqref{renormalized mollified semilinear SHE}, we set $F_{\eps}(\varphi_{\eps},\zeta_\eps) \eqdef  \varphi_{\eps}^{3} - c_{\eps} \varphi_{\eps} + \zeta_\eps$, where $c_{\eps} \eqdef 3C^\eps_{\<11s>} + 9C^\eps_{\<22s>}$, then, generating an expansion like \eqref{expansion} for $\hat{\varphi}_{\eps}$ and collecting terms appropriately, one will see
\begin{equ}\label{renormalized expansion}
(\partial_{t} - \CL) 
\hat{\varphi}_{\eps}
=
\<0>_{\eps}
+
\widehat{\<111>}_{\eps}
+
3
\widehat{\<131>}_{\eps}
+
\cdots\;,
\end{equ}
where all the divergent terms we did not explicitly write down are also organized with counterterms so that \eqref{renormalized expansion} is an infinite expansion where each individual term, for almost every realization of $\zeta$, is convergent as $\eps \downarrow 0$. 

This analysis of formal expansions is an old story -- a systematic procedure for choosing these counterterms 
is the BPHZ formalism \cite{BP,Hepp,Zimmermann} which can be adapted from perturbative QFT to ``perturbative SPDE''. 
Unfortunately, this analysis at the level of formal series falls far short of showing that once suitably 
renormalized, the actual solutions $\hat{\varphi}_{\eps}$ converge to a limit $\hat{\varphi}$. 
This proved to be a difficult problem until the publication of \cite{KPZ,Paracontrol,Regularity} which was followed by a
substantial amount of progress in the field.

We now summarize the approach of the theory of regularity structures. 
To motivate this approach we first observe that the $\eps \downarrow 0$ limits of the left hand sides of \eqref{renormalization of trees}, which we denote by $\widehat{\<111>}$ and $\widehat{\<131>}$, are probabilistic limits which are measurable functions of $\zeta$. 
Adopting a deterministic approach, they are not continuous functions of $\zeta$ in any reasonable topology.
Following this thread, if we denote by $\zeta \mapsto \varphi(\zeta)$ the ``solution map'' that takes a realization of the driving noise to the corresponding solution to the SPDE \eqref{semilinear SHE} then it seems natural to expect that $\varphi(\cdot)$ will not be continuous either -- this is discouraging since we would want to define $\varphi(\zeta)$ for typical realizations of $\zeta$ by setting $\varphi(\zeta) \eqdef \lim_{\eps \downarrow 0} \varphi(\zeta_{\eps})$. 

The theory of regularity structures does not view $\varphi$ as the primary solution map and rather works with a map defined on a much richer space. 
It replaces the (linear) space of realizations of $\zeta$ with a non-linear metric space $\mcb{M}_{0}$ of objects called \emph{models} which encode, loosely speaking, the realization of $\zeta$ \emph{and} sufficiently many multilinear functionals of $G \ast \zeta$ so that one can approximate $F$ sufficiently well. 

In order to describe the elements of the space $\mcb{M}_{0}$, one must go further then just specifying a realization of each such multi-linear functional. Additionally, one must also specify, for every space-time point $z$, a suitable
method of recentering this functional around $z$ so that the recentered functional satisfies certain analytic bounds close to that point $z$. 
This recentering is performed via the subtraction of $z$-dependent counterterms\footnote{These counterterms have no obvious counterpart in perturbative renormalization for QFT.} and can be interpreted as
a type of ``positive renormalization'', see \cite{BHZalg}. 
Moreover, the definition of the space $\mcb{M}_{0}$ enforces analytic and highly non-trivial algebraic constraints on how these counterterms depend on $z$ -- the latter constraints are responsible for the non-linearity of the space $\mcb{M}_{0}$.
The reward for this increase in complexity is that the map from a model $Z \in \mcb{M}_{0}$ to the solution $\Phi(Z)$ of the equation ``driven'' by $Z$ is continuous.  

At first sight, it is not clear why this would be of any help. Indeed, there is a natural ``naive'' 
way of defining a family of models $(Z^{\zeta_\eps}_{\can})_{\eps,\zeta}$ so that the solution 
$\varphi(\zeta_{\eps})$ to $\eqref{mollified semilinear SHE}$ is given by $\Phi(Z^{\zeta_\eps}_{\can})$. 
Viewing $\zeta$ as random, $(Z^{\zeta_\eps}_{\can})_{\eps,\zeta}$ is then a sequence of random models but,
for essentially the same reasons as above, this sequence typically 
fails to converge in $\mcb{M}_{0}$ as $\eps \downarrow 0$. 

The reason why this perspective is still useful is that there is more than one way
of associating a model to a given realisation of the driving noise.
In particular, it turns out that the renormalisation procedure which was required then in order to
make the terms in the formal series expansion converge can be encoded into a suitable definition of
what we mean by ``the model $Z_\eps$ associated to $\zeta_\eps$''.
Indeed, it was shown in many different special cases \cite{Regularity,Paracontrol,HaoSG,ZhuZhu,Hoshino2016} 
that there is a 
probabilistically convergent family of models $(Z_{\BPHZ}^{\zeta_{\eps}})_{\eps,\zeta}$ such that, for each fixed $\zeta$, the solution $\hat{\varphi}_\eps(\zeta_{\eps})$ to \eqref{renormalized mollified semilinear SHE}, with a suitable choice of $F_{\eps}$, is given by $\Phi(Z_{\BPHZ}^{\zeta_{\eps}})$. 
Thanks to the deterministic continuity of the map $Z \mapsto \Phi(Z)$, it then follows that the probabilistic convergence of the models $Z_{\BPHZ}^{\zeta_{\eps}}$ is inherited by $\hat{\varphi}_\eps(\zeta_{\eps})$. 
Writing $Z_{\BPHZ}^{\zeta} \eqdef \lim_{\eps \downarrow 0} Z_{\BPHZ}^{\zeta_{\eps}}$, one can view $\Phi(Z_{\BPHZ}^{\zeta})$ as ``a solution'' to the SPDE at hand\footnote{We do not say ``the solution'' because there will actually be an infinite family of solutions. This degeneracy can be seen in our example -- one would expect $\hat{\varphi}_{\eps}$ to converge for any choice of the form $F_{\eps}(\varphi,\zeta) = \varphi^{3} - (c_{\eps} + c_{1})\varphi + c_{2} + \zeta$ where $c_{1}$ and $c_{2}$ are finite constants. The full space of solutions to a given locally subcritical SPDE will always be parameterized by a finite number of real parameters.}

While the machinery of \cite{Regularity} was formulated systematically and covered a wide class of SPDEs, 
two tasks in this procedure were left to be dealt with in a somewhat ad-hoc way: 
\begin{enumerate}
\item Create a concrete space of models $\mcb{M}_{0}$ tailor-made so that the type of SPDEs at hand 
can be realized as one driven by a model in $\mcb{M}_{0}$. This space should have enough flexibility 
to allow for the renormalization subtractions needed to yield probabilistically convergent families of models.
\item Find a probabilistically convergent family of  ``renormalised'' random models $Z^{\zeta}_{\eps}$.
\end{enumerate}
A systematic approach for the first task has recently been presented in \cite{BHZalg}. 
In particular, given a random, smooth, driving noise $\zeta$ the space of models described by \cite{BHZalg} is rich enough to include a specific random model $Z^{\zeta}_{\BPHZ}$ henceforth refereed to as the 
``BPHZ lift of $\zeta$'' defined in \cite[Eq.~6.22]{BHZalg}, see also \cite[Eq.~3.1]{Proceedings}. 

The goal of the present article is to prove the following implication: if a family of random smooth driving 
noises $(\zeta_{n})_{n \in \N}$ is uniformly bounded in an appropriate sense and converges, as $n \rightarrow \infty$, in probability to a random driving noise $\xi$, then $\xi$ admits a BPHZ lift $Z^{\xi}_{\BPHZ}$.
Furthermore, the random model $Z^{\xi}_{\BPHZ}$ is obtained as the limit in probability of the 
random models $(Z^{\zeta_{n}}_{\BPHZ})_{n \in \N}$. In particular, the BPHZ lift is canonical
and stable, i.e.\ independent of any specific choice of approximation procedure.
Our Theorem~\ref{thm: Gaussian case} states this result in the Gaussian case and covers in particular 
all examples of  locally sub-critical SPDEs driven by space-time white noise. 
We state our main result in full generality in Theorem~\ref{main thm: continuity}.
The trio of papers \cite{Regularity}, \cite{BHZalg}, and the present article then give a completely 
``automatic'' and self-contained black box for obtaining local existence theorems for a wide class of SPDEs.

We now describe the approach and outline of this paper. 
As already mentioned, there are two types of renormalization that appear here -- negative renormalizations which compensate for divergences and base-point dependent positive renormalizations which appear in our recentering procedure. 
The definition of the BPHZ model found in \cite{BHZalg,Proceedings} invokes two maps called the positive and negative twisted antipodes.
 Each of these maps is defined through a recursive formula. A very similar (in spirit) definition is given in Section~\ref{Sec: def of BPHZ model} which is later to shown to coincides with the definition given in \cite{BHZalg,Proceedings}.

Our first step is to obtain a more explicit integral formula for the 
BPHZ renormalized model in terms of a sum over forests and cuts, and this is the content of 
Section~\ref{sec: derivation of formula}. 
Here the forests come from negative renormalization and are really the same as those of Zimmerman's forest formula \cite{Zimmermann} -- they are families of divergent structures where each pair of structures in each family must be either nested or disjoint.
The cuts correspond to locations of positive renormalizations. 

The most fundamental obstacle that appears when estimating this explicit integral formula is that renormalizations can obstruct each other. This can happen in two ways: (i) two negative renormalizations can obstruct each other if the two divergent structures are neither nested nor disjoint (ii) two renormalizations of opposite sign can obstruct each other if they overlap, that is a positive renormalization might occur within a divergent structure. 
As a result, one cannot in general expect to be able to ``harvest'' all cancellations created by the
two renormalization procedures.

The first problem is the famous problem of overlapping divergences of perturbative quantum field theory. 
Its solution comes from using scale decompositions so that, for any given scale assignment, 
one can identify which one of a pair of 
overlapping divergent structures to renormalize. 
The original approach to this problem was to represent the explicit formula in terms of Fourier 
variables and splitting up the Fourier integration domain into ``Hepp sectors''. 
The resulting estimates can be summarised by the mantra that ``overlapping divergences don't overlap in phase-space''. 

The use of momentum space techniques is too limiting for our purposes. 
However an alternative approach in position space was described in \cite{FMRS85}, where 
one expands the kernels associated to the lines of a Feynman diagram into an infinite sum 
of slices, each of which is localized at a particular length scale. 
A full scale assignment is an assignment of scale to every single line of a Feynman diagram.
For each full scale assignment on a Feynman diagram, the authors of \cite{FMRS85} define an operator on the set of Zimmerman forests for that diagram called the  ``projection onto safe forests''.
This allows for the sum over Zimmerman forests to be re-organized, as a function of the scale assignment, into sums over subsets of forests we call ``intervals''. 
With this reorganization one is guaranteed to harvest the right negative renormalizations for each given scale assignment.

The problem of overlaps between positive and negative renormalizations does not appear in 
perturbative QFT, but it turns out that it also does not cause a problem.
One can show that if one chooses appropriately which cuts to harvest, as a function of scale assignments, then the larger sum over forests and cuts can also be reorganized appropriately for each scale assignment. 
The mantra here is that ``positive and negative renormalizations don't overlap in phase-space''. 

The set-theoretic aspects and notation of forest projections and intervals, how they can be used to reorganize sums over forests, and sufficient conditions (which we call ``compatibility'') for the sum over forests and sum over cuts to interact well are all described in Section~\ref{Sec: reorganizing sums}. 
This is all done abstractly, then in Section~\ref{Sec: multiscale expansion} we present our multiscale expansion procedure.
In Section~\ref{sec: Safe forest projection} we make the notions of Section~\ref{Sec: reorganizing sums} concrete: we present, as a function of each scale assignment, the \cite{FMRS85} forest projection, our algorithm for choosing which cuts to harvest, and the proof that these two notions are compatible.

In Section~\ref{Sec: summing over scales} we describe how one inductively controls the nested renormalization cancellations that might appear and we obtain full control over one copy of the BPHZ renormalized model. 
When estimating these nested renormalizations we use seminorms similar to ones introduced in \cite{FMRS85} but simplified.  
In Section~\ref{sec: est on moments of trees} we state and verify conditions which take as input the control over a single copy of the model that we obtained in Section~\ref{Sec: summing over scales} and then gives as output estimates on higher moments of that model. 

One key generalization in this work is that we do not assume Gaussianity of the driving noise distribution. 
This was already done earlier in special cases in \cite{HS15,CS16,SX16}, but the novelty here 
is that we are able to also 
accommodate very weakly mixing space-time random fields and genuinely non-Gaussian situations\footnote{As an example, we can accommodate a driving noise which is given by a Wick power of the linear solution to the stochastic heat equation.}.

We are also able to deal with driving noises that are more singular than space-time white noise, but we 
are still restricted to 
regimes where the multilinear functionals of the noise 
appearing in our analysis have ``regularity'' (in some power-counting sense to be made precise below) 
strictly better than that of white noise.
Instances of this limitation were already observed in the 
works \cite{MR1883719,MR2667703,Hoshino2016,HuNualart}. The reason for
this is that the objects indexed by trees that are being renormalised in the theory of regularity structures should
really be thought of as ``half Feynman diagrams'', with the full diagram arising in the moment
estimates by gluing together several trees by their leaves, see for example \cite{KPZ}. 
The analogue of the BPHZ renormalisation procedure arising in our context however takes place at the
level of the trees, not at the level of the resulting Feynman diagrams, so that divergent structures 
straddling more than one of the constituent trees cannot be taken into account, thus leading to
divergencies. This is not a mere technical problem but points to a genuine change in behaviour of the
corresponding stochastic PDEs, as already seen in \cite{MR1883719} in the context of rough paths,
where SDEs driven by fractional Brownian motion can only be made sense of for Hurst parameters $H > {1\over 4}$,
even though the theory of rough paths should in principle ``work'' all the way down to $H > 0$.
The works \cite{MR2823213,MR2889659} suggest that even if there is a ``canonical'' rough path corresponding to fractional
Brownian motion with $H \le {1\over 4}$, one would expect it to live over a larger 
probability space, so that the L\'evy area is no longer a measurable function of the underlying path.
In the genuinely non-Gaussian case, our analysis points to a similar problem already arising
at $H = {1\over 3}$ in the case of non-vanishing third cumulants.

\subsection*{Acknowledgements}

{\small
We are very grateful to Abdelmalek Abdesselam, Yvain Bruned, Kurusch Ebrahimi-Fard, and Lorenzo Zambotti for many 
interesting discussions on the subject of this article.
AC would like to thank Henri Elad Altman, Joscha Diehl, and Philipp Schoenbauer for their close reading which helped in the proofreading of this paper. 
MH gratefully acknowledges support by the Leverhulme Trust and by an ERC consolidator grant, project 615897.
}

\section{Setting, preliminary definitions, and the main theorem}

The most primitive object of the theory of \cite{Regularity} is the \emph{regularity structure}. 
In \cite{BHZalg} a very general formulation for the construction of regularity structures is given by defining various algebraic spaces of trees and forests and taking progressive quotients/projections. 
In this section, we introduce some of the objects and notations from \cite{BHZalg}.

\subsection{Space-time scalings, types, and homogeneities}
For the rest of the paper we fix a dimension $d \ge 1$ of space-time. 
For a multi-index $k = (k_{i})_{i=1}^{d} \in \N^{d}$ we define the degree of $k$ to given by
\begin{equ}\label{def: multi-index degree}
|k| \eqdef \sum_{i=1}^{d}|k_{i}|\;.
\end{equ}
Most interesting SPDE falling under our framework have underlying linear equations which scale anistropically with respect to the various components of space-time -- for example the stochastic heat equation underlying \eqref{semilinear SHE} scales \emph{parabolically}. 

A $d$-dimensional space-time ``scaling'' is a multi-index $\s = (\s_{i})_{i=1}^{d} \in \N^{d}$ with strictly positive entries. From here on we treat the space-time scaling $\s$ as fixed. 

For $z = (z_{1},\dots,z_{d})\in \R^{d}$ we do not define $|z|$ to be given by the standard Euclidean norm but instead set
\begin{equ}\label{def: scaled space-time distance}
|z| \eqdef \sum_{i=1}^{d} |z_{i}|^{1 \over \s_{i}}\;.
\end{equ}

There appears to be a bad overload of notation between \eqref{def: scaled space-time distance} and \eqref{def: multi-index degree} but it will always be clear from context whether the argument of $|\cdot|$ is a multi-index or a space-time vector. 

While $|\cdot|$ on $\R^{d}$ is not a norm it certainly defines a metric. 
With this metric the space $\R^{d}$ has Hausdorff dimension $|\s|$ -- in particular the function $|z|^{\alpha}$ on $\R^{d}$ is locally integrable at $z=0$ if and only if $\alpha > -|\s|$. 

We introduce a notion of $\s$-degree for multi-indices $k = (k_{i})_{i=1}^{d} \in \Z^{d}$ by setting
\[
|k|_{\s} 
\eqdef \sum_{i=1}^{d} k_{i}\s_{i}\;.
\]  
This induces a notion of $\s$-degree for polynomials on $\R^{d}$: for $k \in \N^{d}$ the $\s$-degree of a monomial $\R^{d} \ni  z \mapsto z^{k}$ is $|k|_{\s}$ and the $\s$-degree of a polynomial  is the maximum among the $\s$-degrees of its monomials (with the $\s$-degree of the $0$ polynomial set to $-\infty$). 
\subsubsection{Types and homogeneities}
The combinatorial trees we introduce should be thought of as arising, in spirit, from Picard iteration of the mild formulation system of SPDEs one is trying to solve. As we saw in the introduction, for \eqref{semilinear SHE} one obtains trees where the edges correspond to convolution with the kernel $G$ and the leaves correspond to the driving noise. 

To accommodate systems of equations we work with a collection of different convolution kernels and driving noises.
To formalize this we fix, for the rest of the article, a finite set of \emph{types} $\label{def: set of types}\Lab$ which comes with a partition $\Lab = \Ke \sqcup \Le$. The elements of $\Ke$ and $\Le$ are respectively called kernel types and noise types.
\begin{definition}\label{def:hom} 
A homogeneity assignment is a map $|\cdot|_{\s}: \Lab \rightarrow \R$ 
such that 
\[\label{def: kernel and leaf types}
\Ke = \{ \mft \in \Lab:\ |\mft|_{\s} > 0\}
\enskip
\textnormal{and}
\enskip
\Le = \{\mft \in \Lab:\ |\mft|_{\s} < 0\}\;.
\]
\end{definition}
A choice of homogeneity assignment should be considered fixed except in a few specific arguments where we explicitly state this is not case. 

For each $\mft \in \Ke$, $|\mft|_{\s}$ quantifies the regularizing effect of the corresponding kernel. On the other hand, for $\mft \in \Le$, $|\mft|_{\s}$ encodes the regularity (or rather lack thereof) 
of the driving noise corresponding to $\mft$.
For any set $A$ and map $\mft:A \rightarrow \Lab$, we write
$\label{def: homogeneity of a set}|\mft(A)|_{\s} \eqdef \sum_{a \in A} |\mft(a)|_{\s}$.
Given a homogeneity assignment $|\cdot|_{\s}$ we extend it to $\Z^{d} \oplus \Z(\Lab)$ by setting
\[
a = k + \sum_{\mft \in \Lab}c_{\mft} \mft \in \Z^{d} \oplus \Z(\Lab)\qquad \Rightarrow\qquad
|a|_{\s} \eqdef |k|_{\s} + \sum_{\mft \in \Lab}c_{\mft} |\mft|_{\s}\;.
\]
\subsection{Trees, forests, and decorations} 
\subsubsection{Trees}\label{subsec: trees}
A rooted tree $T$ is an acyclic finite connected simple graph with a distinguished 
vertex $\rho_{T}$ which we call \emph{the root}. 
This graph can be presented as a set of nodes $N_{T}$ and edges $E_{T} \subset N_{T}^{2}$. Since we force our trees to have roots it follows that  $\rho_{T} \in N_{T} \not = \emptyset$. 
On the other hand we do allow $E_{T} = \emptyset$ which forces $|N_{T}| = 1$; such a tree 
will be called \emph{trivial} and is also denoted by $\bullet$.

We view the node set $N_{T}$ of a rooted tree $T$ as being equipped with a partial order $\le$ where $u \le v$ 
if and only if $u$ lies on the unique path connecting $v$ to the root. 
We then view $T$ as being directed, with every edge being of the form $(u,v) \in E_{T}$ with $u < v$. 
With this convention, the root of $T$ is the unique vertex with no incoming edges.
We define maps $e_{\ch}, e_{\p}: E_{T} \rightarrow N_{T}$ with the property that for any $e \in E_{T}$ one has $e = (e_{\p}(e),e_{\ch}(e))$ -- here $\ch$ stands for ``child'' and $\p$ stands for parent. 
When $e$ has been fixed for an expression we may write $e_{\p}$ and $e_{\ch}$, suppressing the argument.
Note that we inherit a poset structure on edges by setting $e \le \tilde{e} \Leftrightarrow e_{\p} \le \tilde{e}_{\p}$. 

An morphism of rooted trees $T$ and $\bar{T}$ is a map $f: N_{T} \sqcup E_{T} \rightarrow N_{\bar{T}} \sqcup E_{\bar{T}}$ with the property that for any $e = (x,y) \in E_{T}$ one has $f(e) = (f(x),f(y)) \in E_{\bar{T}}$. 

A typed rooted tree is a pair $(T,\mft)$ where $T$ is a rooted tree and $\mft: E_{T} \rightarrow \Lab$ with the properties that (i) if for some $e \in E_{T}$ one has $\mft(e) \in \Le$ then $e$ is maximal in the partial order on $E_{T}$ and (ii) For distinct $e,\ e' \in E_{T}$ with $\mft(e),\mft(e') \in \Le$ one has $e_{\p} \neq e_{\p}'$.

A morphism of typed rooted trees is morphism of rooted trees which preserves edge types. 
The type map $\mft$ is often suppressed from from notation. 
Henceforth the term ``tree'' denotes a typed rooted tree. 
\subsubsection{Forests}\label{subsec: forests}
A typed rooted forest $F$ is a typed directed graph (consisting of a node set $N_{F}$, an edge set $E_{F} \subset N_{F}^{2}$, and a type map $\mft: E_{F} \rightarrow \Lab$) with the property that every connected component of $F$ is a typed rooted tree (we often call each of these components a ``tree'' of $F$). 
Henceforth the term ``forest'' will always refer to a typed rooted forest. 

Clearly any typed rooted tree is also a typed rooted forest.
We do allow for $N_{F} = \emptyset$ in the definition of a forest and this empty forest will be denoted by $\mathbf{1}$.
Note that $\mathbf{1}$ is \emph{not} a tree. 
We also write $\rho(F) \subset N_{F}$ for the collection of all nodes of $F$ having no incoming edges.
Additionally, we view $N_{F}$ and $E_{F}$ as posets by inheriting the poset structures from their connected components 
of $F$ and postulating that pairs of edges / vertices belonging to different connected components are not comparable.
We say $f\colon N_{F} \sqcup E_{F} \rightarrow N_{F'} \sqcup E_{F'}$ is a typed rooted forest morphism from $F$ to $F'$ if it restricts to a morphism of rooted trees on each connected component.

One has the decomposition $E_{F} = K(F) \sqcup \mathbb{L}(F)$ where
\[
\label{e:leaf and kernel edges}
\mathbb{L}(F) \eqdef 
\{ e \in E_{F}:\ \mft(e) \in \Le \}
\enskip
\textnormal{and} 
\enskip
K(F) \eqdef 
\{ e \in E_{F}:\ \mft(e) \in \Ke \}\;.
\]
We call the elements of these two sets \textit{noise edges} and \textit{kernel edges}, respectively.
While viewing noises as edges was the point of view taken in \cite{BHZalg}, where it
leads to a cleaner description of the algebraic properties of these objects, it will be much
more natural in the present setting to view each noise as a node.
We will therefore always restrict ourselves to situations in which, for any given node $v \in N(F)$,
there can be at most one kernel edge $e \in \mathbb{L}(F)$ with $e_\p = v$.
As a consequence, we will never work with $\mathbb{L}(F)$, but instead with
the set $\label{def: leaf nodes} L(F) \eqdef e_{\p}[\mathbb{L}(F)]$
Observe that, thanks to the above-mentioned restriction, $L(F)$ naturally
inherits a type map $\mft: L(F) \rightarrow \Le$.
We also define a subset of nodes $N(F) \eqdef N_{F} \setminus e_{\ch}(\mathbb{L}(F))$  which should 
be thought of as ``true'' nodes (they are the ones that will later on correspond to integration variables) 
and consider the elements of $e_{\ch}(\mathbb{L}(F))$ as ``fictitious'' nodes that will not play any role
in the sequel. 
We also use the shorthand $\tilde{N}(F) \eqdef N(F) \setminus \rho(F)$.
The reason why this shorthand will often be used is that, when renormalizing 
a subtree $S$ of a larger tree $T$, one should think of ``contracting'' this subtree to a
single node, see \cite{Regularity,BHZalg}. We will then always identify the integration variable
corresponding to the contraction of $S$ with the root of $S$ in $T$, so that $\tilde{N}(S)$
precisely corresponds to the variables that get integrated out by the renormalization procedure. 

We often include figures representing trees and forests. To keep things simple our pictures will 
only distinguish between kernel type and noise type:
edges of kernel type will be depicted by straight arrows and edges of noise type will be 
depicted by a zigzag line, as in the following figure.  
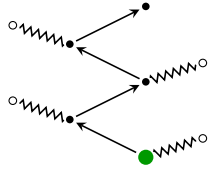
\[
\begin{tikzpicture}
    \node[root] (root) at (0.5, 0) {};
    \node[dot](first) at (-0.5, 0.5) {};
    \node[dot] (second) at (0.5, 1) {};
    \node[dot] (third) at (-0.5, 1.5) {};
    \node[circ] (root anoise) at (1.25, 0.25) {};
    \node[circ] (second anoise) at (1.25, 1.25) {};
    \node[circ] (first anoise) at (-1.25, 0.75) {};
    \node[circ] (third anoise) at (-1.25, 1.75) {};
    \node[dot] (fourth) at (0.5, 2) {};

    \draw[kernel] (root) to (first);
    \draw[kernel] (first) to (second);
    \draw[kernel] (second) to (third);
    \draw[kernel] (third) to (fourth);

    \draw[leaf] (root) to (root anoise);
    \draw[leaf] (first) to (first anoise);
    \draw[leaf] (second) to (second anoise);
    \draw[leaf] (third) to (third anoise);
\end{tikzpicture}
\]
The arrow for an edge $e$ always travels from $e_{\p}$ to $e_{\ch}$. 
Above we have marked the true vertices with solid circles and the fictitious nodes with open circles,
with the root being drawn as a larger green dot. 

We define the \emph{forest product} $F_{1} \cdot F_{2}$ between two forests $F_{1}$ and $F_{2}$ as their disjoint union. 
Since there is no ordering on the set of components of a forest, the forest product is commutative and associative 
with the empty forest $\mathbf{1}$ serving as unit.
Given a forest $F$ we say that a forest $A$ is a subforest of $F$ if $N_{A} \subset N_{F}$, $E_{A} \subset E_{F}$ and the natural inclusion map $\iota: N_{A} \sqcup E_{A} \rightarrow N_{F} \sqcup E_{F}$ is a forest morphism. 
In particular, the empty forest $\mathbf{1}$ is a subforest of every forest.
We define $\Con(F)$ to be the set of connected components of $F$, each viewed as a subtree of $F$ (thus $\Con(F)$ may contain distinct elements which are isomorphic as trees). 

Two subforests $A$ and $B$ of a forest $F$ are said to be disjoint if $N_{A} \cap N_{B} = \emptyset$. 
We also define the union $A \cup B$ and intersection $A \cap B$ of two subforests $A,B$ of $F$ to be the subforests with node and edge sets $N_{A} \cup N_{B}, E_{A} \cup E_{B}$, and $N_{A} \cap N_{B}$, $E_{A} \cap E_{B}$, respectively. 
If $B$ is a subforest of $A$ such that $B \cap A$ consists of a union of connected components of $A$ then $A \setminus B$ is also a subforest.
\begin{remark}\label{disappearing noises}
Caution: one may have $G$ a subforest of a forest $F$, $u \in L(F) \cap N(G)$, but $u \not \in L(G)$ - this would happen if the unique $e \in \mathbb{L}(F)$ with $e_{p} = u$ did not belong to $\mathbb{L}(G)$.
\end{remark}
A subforest that happens to be a tree is called a subtree.  
Both the intersection and the union of a non-disjoint pair of subtrees is again a subtree.
\subsubsection{Decorations}
A semi-decorated tree consists of a tree $T$ and two maps: a node label $\mfn: N(T) \rightarrow \N^{d}$ and an edge label $\mfe: E_{T} \rightarrow \N^{d}$.
Given a forest $T$, node decoration $\mfn$ and edge decoration $\mfe$ on $T$ we denote the corresponding semi-decorated tree by $T^{\mfn}_{\mfe}$.

A choice of homogeneity assignment $|\cdot|_{s}$ induces a notion of a homogeneity for semi-decorated trees by setting 
\begin{equ}\label{Regularity structure homogeneity}
|T^{\mfn}_{\mfe}|_{\s}
\eqdef 
\sum_{
e \in E_{T}
}
\left(
|\mft(e)|_{\s}
-
|\mfe(e)|_{\s}
\right)
+
\sum_{u \in N(T)}
|\mfn(u)|_{\s}\;.
\end{equ}
\subsection{Regularity structures of semi-decorated trees}

\subsubsection{The reduced regularity structure}

We first recall the definition of a regularity structure from \cite[Def.~2.1]{Regularity}. 
\begin{definition}\label{def: reg struct}
A \emph{regularity structure } $\mathscr{T} = (A,\mcb{T},G)$ consists of
\begin{itemize}
\item An index set $A \subset \R$ which is locally finite and bounded from below.
\item A model space $T$, which is a graded vector space $\mcb{T} = \bigoplus_{\alpha \in A} \mcb{T}_{\alpha}$, with each $\mcb{T}_{\alpha}$ a Banach space.
\item A \emph{structure group} $G$ of linear operators acting on $\mcb{T}$ such that, for every $\Gamma \in G$, every $\alpha \in A$, and every $\tau \in \mcb{T}_{\alpha}$, one has
\[
\Gamma \tau - \tau 
\in 
\bigoplus_{\beta < \alpha} \mcb{T}_{\beta}\;.
\]
\end{itemize}
\end{definition}
When formulating our main theorem in the context of the theory of regularity structures the relevant regularity structure will be the \emph{reduced regularity structure} $\mathscr{T}$ described by the procedure of \cite[Sec.~6.4]{BHZalg}\footnote{This is a generalisation of the more informal prescription of \cite[Sec.~8]{Regularity}.}.  
This procedure takes as input a choice of space-time dimension $d \ge 1$, a $d$-dimensional space-time scaling $\s$, sets of leaf and kernel types $\Ke$ and $\Le$, a homogeneity assignment $|\cdot|_{\s}$ on these types, and finally a \emph{rule} $R$.
We have introduced all these notions except the last one. 

The vector space $\mcb{T}$ of the reduced regularity structure is the free vector space generated by a 
collection $B_{\circ}$ of semi-decorated trees prescribed by the rule $R$. 
Then one has $A \eqdef \{ |T^{\mfn}_{\mfe}|_{\s}: T^{\mfn}_{\mfe} \in B_{\circ} \}$ and for each $\alpha \in A$ one sets $\mcb{T}_{\alpha}$ to be the free vector space generated by all semi-decorated trees $T^{\mfn}_{\mfe} \in B_{\circ}$ 
with $|T^{\mfn}_{\mfe}|_{\s} = \alpha$. 

\begin{remark}
The collection $B_\circ$ introduced in \cite{BHZalg} is slightly different than the one introduced here
since it furthermore allows a tree to be endowed with an additional colored subforest, as well
as an additional label $\mfo$ which keeps track of some information coming from contracted
subtrees. The collection $B_\circ$ we are using here consists of those elements for which $\mfo = 0$ and
the colored subforest is empty. This is because we will only ever consider the \textit{reduced regularity structure}
introduced in \cite[Sec.~6.4]{BHZalg}.
\end{remark}

This collection $B_{\circ}$ has to be rich enough to describe the solution to the system of SPDEs in question and be closed under positive and negative renormalization operations. 
A natural method to build $B_{\circ}$ so that this is the case is to build it up inductively by starting from some primitive objects and then applying operations like integration and point-wise products. 
We describe this now.

For every $\mft \in \Ke$ and $k \in \N^{d}$ we define an operator $\mcb{I}_{\mft}^{k}[\cdot]$ on semi-decorated trees as follow: for an arbitrary semi-decorated tree $\tau$ one obtains $\mcb{I}_{\mft}^{k}[\tau]$ by adding a new vertex to the set of vertices of $\tau$ which becomes the new root and has a vanishing node label -- we then attach the new root to the old one with an edge of type $\mft$ and edge decoration $k$. 

Additionally, for any $\mfl \in \Le$ we write $\Xi_{\mfl}$ for the semi-decorated tree which consists of a single edge of type $\mfl$  with vanishing edge and node labels.
We also introduce a commutative and associative binary operation on semi-decorated trees called the tree product: for two semi-decorated trees $\tau$ and $\bar{\tau}$, the tree product $\tau \star \bar{\tau}$ is obtained by identifying the roots of $\tau$ and $\bar{\tau}$, which becomes the root of $\tau \star \bar{\tau}$, and then setting the new node label of this root to be the sum of the node label of the two previous roots. 

The set of primitive semi-decorated trees we start with is $\{\bullet^{n}\}_{n \in \N^{d}} \sqcup \{ \Xi_{\mfl} \}_{\mfl \in \Le}$ -- the first set is called the set of ``abstract polynomials'' and the second is called the set of ``abstract noises''. 
Then, one builds up $B_{\circ}$ by applying tree products and applying the operators $\mcb{I}_{\mft}^{k}$. 
However, this produces a collection $B_{\circ}$ which is too large -- certainly $A$ will fail to be bounded below. 
A \textit{rule} is then an algorithm which explicitly specifies when one is allowed to take tree products and apply the operators $\mcb{I}_{\mft}^{k}$ as one builds $B_{\circ}$.
For a precise definition of what a rule is we point the reader to \cite[Sec.~5.2]{BHZalg}. 
For the present article, this precise definition is irrelevant, one should think of a rule as a way of formalising
what one means by a ``class of semilinear stochastic PDEs'' and $B_{\circ}$ as the set of trees that are
required when trying to formulate elements from that class within the theory of regularity structures.
A rule is then ``subcritical'' if all elements in the corresponding class of SPDEs are subcritical in the
sense of \cite{Regularity}. It is ``complete'' if the class is sufficiently large so that it is closed
under the natural renormalisation operations described in \cite{BHZalg}.
We now fix a for the rest of the paper a complete subcritical rule $R$ and a 
corresponding reduced regularity structure $\mcT$.

\begin{remark}
For example, if we wish to consider the SPDE formally given by
\begin{equ}
\d_t \Phi = \Delta \Phi - \Phi^3 + \xi\;,
\end{equ}
in dimension $d < 4$ and for $\xi$ space-time white noise, then the smallest complete rule 
allowing to describe this equation contains all equations of the type
\begin{equ}
\d_t \Phi = \Delta \Phi - P(\Phi) + \xi\;,
\end{equ}
where $P$ is a polynomial of degree $3$.
\end{remark}

\subsubsection{Admissible models}

Next we recall the definition of a smooth model on a regularity structure $\mathscr{T}$. 
In what follows we denote by $\label{def: smooth functions 1}\mcb{C}(\R^{d})$ the space of all smooth real valued functions on $\R^{d}$.  
\begin{definition}\label{def: smooth model}
Given a regularity structure $\mathscr{T} = (A,\mcb{T},G)$, a smooth model $Z = (\Pi,\Gamma)$ on $\mathscr{T}$ on $\R^{d}$ with scaling $\s$ consists of
\begin{itemize}
\item A map $\Gamma: \R^{d} \times \R^{d} \rightarrow G$ such that $\Gamma_{xx} = \Id$, and such $\Gamma_{xy}\Gamma_{yz} = \Gamma_{xz}$ for all $x,y,z \in \R^{d}$.
\item A collection $\Pi = (\Pi_{x})_{x \in \R^{d}}$ of continuous linear maps $\Pi_{x}: \mcb{T} \rightarrow \mcb{C}^{\infty}(\R^{d})$ such that $\Pi_{y} = \Pi_{x} \Gamma_{xy}$ for all $x,y \in \R^{d}$.
\end{itemize}
Furthermore, for every $\alpha \in A$ and every compact set $\mathfrak{K} \subset \R^{d}$, we assume the existence of a finite constant $C_{\alpha,\mathfrak{K}}$ such that the bounds
\begin{equ}\label{analytic requirements}
|(\Pi_{x}\tau)(y)| \le C_{\alpha,\mathfrak{K}}
\| \tau \|_{\alpha}
|x-y|^{\alpha}
\enskip
\textnormal{and}
\enskip
\|\Gamma_{xy} \tau 
\|_{\beta}
\le
C_{\alpha,\mathfrak{K}}
\|\tau \|_{\alpha}
|x-y|^{\beta - \alpha}
\end{equ}  
hold uniformly in $x,y \in \mathfrak{K}$, $\beta \in A$ with $\beta < \alpha$, and $\tau \in \mcb{T}_{\alpha}$. Here for any $\beta \in A$, $\|\cdot\|_{\beta}$ denotes the norm on the Banach space $\mcb{T}_{\alpha}$.
\end{definition}
Above, there is no constraint how a model interacts with the various kernel types, noise types, and edge labels -- to do this one introduces an ``admissibility'' condition. 
This requires associating to each abstract kernel type a concrete convolution kernel.

First, following \cite[Sec.~10.3]{Regularity}, for any $\zeta \in [0,|\s|)$ and $m \in \N$ we introduce norms $\|\cdot\|_{\zeta,m}$ on smooth functions $K: \R^{d} \setminus \{ 0\} \rightarrow \R$ by setting
\begin{equ}\label{def: singular kernel norm}
\|K\|_{\zeta,m}
\eqdef
\max_{
\substack{
k \in \N^{d}\\
|k|_{\s} \le m
}
}
\sup_{x \in \R^{d}\setminus \{0\}}
|x|^{|k|_{\s} + |\s| - \zeta}
|D^k K(x)|\;.
\end{equ}
\begin{definition}\label{def: kernel type map}
A \emph{kernel type map} is a tuple $K = (K_{\mft})_{\mft \in \Ke}$ where for every $\mft \in \Ke$ 
\begin{itemize}
\item $K_{\mft}:\R^{d} \setminus \{0\} \rightarrow \R$ is smooth function. 
\item $K_{\mft}(x)$ vanishes if $|x| > 1$.
\item For every $m \in \N$ one has $\|K_{\mft}\|_{|\mft|_{\s},m} < \infty$.
\end{itemize}
\end{definition}
\begin{remark} 
Within the scope of regularity structures it is common to also require that for some value of $r > 0$ one has $\int_{\R^{d}}K_{\mft}(x)P(x)\ dx = 0$ for any polynomial $P$ of $\s$-degree strictly less than $r$. 
Since there is no need for this in any of the theorems stated in this paper we do not enforce any such requirement.
\end{remark}
For the rest of the paper we take our kernel type map $K$ as fixed.
With all these notions in hand we can now give a definition of an admissible smooth model.
This definition only makes sense in the context of the regularity structure built
from a set of types and a rule $R$ as described above, so we assume that we are in this setting from now on.

\begin{definition}\label{def: admissibility}
A smooth model $Z = (\Pi,\Gamma)$ is said to be admissible it satisfies all of the following properties.
\begin{itemize}
\item For every $z,\bar{z} \in \R^{d}$, $\Pi_{z} [\bullet^{0}](\bar{z}) = 1$. 
\item For any $\mft \in \Ke$, $k \in \N^{d}$ and $\tau \in B_{\circ}$ such that $\mcb{I}^{k}_{\mft}(\tau) \in B_{\circ}$ one has
\[
\big(
\Pi_{z} 
\mcb{I}^{k}_{\mft}(\tau)
\big)
(\bar{z})
=
\big(
D^{k}K_{\mft}
\ast
\Pi_{z}\tau
\big)(\bar z)
-
\sum_{
|j|_{\s} < 
|\mcb{I}^{k}_{\mft}(\tau)|_{\s}
}
\frac{
(\bar{z} - z)^j
}
{j!}
(D^{k+j}K_{\mft} \ast \Pi_{z}\tau)(z)\;,
\]
for all $z,\bar{z} \in \R^{d}$. Here $\ast$ denotes convolution in $\R^{d}$ and $j$ denotes
a multiindex in $\N^d$. 
\item For any $n \in \N^{d}$, $z,\bar{z} \in \R^{d}$, and $\tau \in B_{\circ}$ such that $ \bullet^{n} \star \tau \in B_{\circ}$, one has 
\[
\big(\Pi_{z} [\bullet^{n} \star \tau] \big)(\bar{z})
=
(\bar z- z)^{n}
\big(\Pi_{z} \tau \big)(\bar{z})\;.
\]
\end{itemize}
\end{definition}
We denote by $\mathscr{M}_{\infty}(\mathscr{T})$ the collection of all smooth admissible models on $\mathscr{T}$. 
The notion of a model is formulated at a deterministic level. 
When formulating a system of SPDEs as a system ``driven'' by a model, each realization of the 
noise yields a different model.

If we turn to solving an SPDE like \eqref{mollified semilinear SHE} then there is a 
canonical way to lift each smoothened realization of the driving noise to a corresponding 
model, which we describe now.
\begin{definition}\label{def: noise type map}
A \emph{smooth noise} is a tuple $\xi = (\xi_{\mft})_{\mft \in \Le}$ with $\xi_{\mft} \in \mcC^{\infty}(\R^{d})$ for every $\mft \in \Le$. We denote by $\Omega_{\infty}$ the set of all noises for our fixed set of types $\Le$. 
\end{definition}
Given a smooth noise, there is a canonical way of associating a model to it.
\begin{definition}
Given $\xi \in \Omega_{\infty}$, the \emph{canonical lift} $Z_{\can}^{\xi} = (\Pi,\Gamma) \in \mathscr{M}_{\infty}(\mathscr{T})$  of $\xi$ is the unique admissible model such that
\begin{itemize}
\item for every $\mfl \in \Le$ and $z \in \R^{d}$ one has
$\big(\Pi_{z}\Xi_{\mfl}\big)(\cdot) = \xi_{\mfl}(\cdot)$,
\item for every $z,\bar{z} \in \R^{d}$ and $\tau, \bar{\tau} \in B_{\circ}$ with $\tau \star \bar{\tau} \in B_\circ$,
\[
\big(\Pi_{z}[\tau \star \bar{\tau}] \big)(\bar{z})
=
\big(\Pi_{z}\tau\big)(\bar{z})\cdot \big(\Pi_{z}\bar \tau\big)(\bar{z})\;.
\]
\end{itemize}
\end{definition}
Reiterating what was said in the introduction, for a fixed realization of $\zeta_{\eps}$, the equation \eqref{mollified semilinear SHE} will be viewed as driven by the model $Z^{\zeta_{\eps}}_{\can}$ and the solution $\phi(\zeta_{\eps})$ can be written in terms of $\Phi(Z^{\zeta_{\eps}}_{\can})$ where $Z \rightarrow \Phi(Z)$ is defined on $\mathscr{M}_{\infty}(\mathscr{T})$. 

Next, we describe the topology that makes $\Phi$ \emph{continuous}.
First, we introduce a family of test functions.
\begin{definition}\label{space of functions}
Given a homogeneity assignment $|\cdot|_{\s}$ we define the following notation.
For any smooth function $\psi$ on $\R^{d}$ we define
\[
\| 
\psi 
\|_{{\s}}
\eqdef
\max
\left\{
\sup_{x \in \R^{d}}
|D^{k}\psi(x)|:\ 
k \in \N^{d}, |k|_{\s} \le \Big\lceil - \min_{\mft \in \Le} |\mft|_{\s} 
\Big\rceil
\right\}\;.
\]
We define $\mcb{B}_{{\s}}$ to be the set of all smooth functions $\psi$ on $\R^{d}$ with $\|\psi\|_{{\s}} \le 1$ and $\psi(x) = 0$ for all $x \in \R^{d}$ with $|x| \ge 1$.
\end{definition}
Next we define a family of pseudometrics on $\mathscr{M}_{\infty}(\mathscr{T})$ as follows.
For $(\Pi,\Gamma),\ (\bar{\Pi},\bar{\Gamma}) \in \mathscr{M}_{\infty}(\mathscr{T})$, $\alpha \in A$, and compact $\mathfrak{K} \subset \R^{d}$, we define
\begin{equ}\label{def: pseudometric on smooth models}
\$
(\Pi,\Gamma); (\bar{\Pi},\bar{\Gamma})
\$_{\alpha, \mathfrak{K}}
\eqdef
\| \Pi - \bar{\Pi} \|_{\alpha; \mathfrak{K}}
+
\| \Gamma - \bar{\Gamma} \|_{\alpha; \mathfrak{K}}
\end{equ}
where
\begin{equs}
\| \Pi - \bar{\Pi} \|_{\alpha; \mathfrak{K}}
&\eqdef 
\sup
\left\{
\frac{
|
\langle 
(\Pi_{x} - \bar{\Pi}_{x})
\tau,
\phi_{x}^{\lambda}
\rangle|
}
{\|\tau\|_{\alpha} \lambda^{\alpha}}
:
\begin{split}
&x \in \mathfrak{K},\ 
\tau \in \mathscr{T}_{\alpha}\\
& \lambda \in (0,1],\ 
\phi \in \mcb{B}_{{\s}}.
\end{split}
\right\}\;,\\[1.5ex]
\| \Gamma - \bar{\Gamma} \|_{\alpha; \mathfrak{K}}
&
\eqdef
\sup
\left\{
\frac{
\|
(\Gamma_{xy} - \bar{\Gamma}_{xy})
\tau,
\phi_{x}^{\lambda}
\|_{\beta}
}
{\|\tau\|_{\alpha} 
\cdot
|x-y|^{\alpha - \beta}}:\ 
\begin{split}
&
x,y \in \mathfrak{K},\ 
x \not = y,\ \phi \in \mcb{B}_{{\s}}\\
&
\tau \in \mathscr{T}_{\alpha},
\beta \in A, \beta < \alpha.
\end{split}
\right\}\;.
\end{equs}
Endowing $\mathscr{M}_{\infty}(\mathscr{T})$ with the system of pseudometrics $(\$\bullet ; \bullet\$_{\alpha, \mathfrak{K}})_{\alpha, \mathfrak{K}}$, we can view $\mathscr{M}_{\infty}(\mathscr{T})$ as a metric space. 
The space $\mathscr{M}_{\infty}$ is certainly not complete -- we denote by $\mathscr{M}_{0}(\mathscr{T})$ the completion of $\mathscr{M}_{\infty}(\mathscr{T})$. Elements of $\mathscr{M}_{0}(\mathscr{T})$ are called admissible models -- they include models which are \emph{distributional} in character. 

Viewing the map $\Phi(\cdot)$'s target space as an appropriate H\"older-Besov space of distributions one can prove estimates that show $\Phi$ admits a unique extension to $\mathscr{M}_{0}(\mathscr{T})$. 
However at this point we have reached the limit of what one can do with a purely pathwise analysis. 
As we stated in the introduction, if one fixes a realization of $\zeta$ then the deterministic convergence of $\zeta_{\eps}$ to $\zeta$ as $\eps \downarrow 0$ cannot be turned into a similar statement about convergence of the canonical lifts $Z^{\zeta_{\eps}}_{\can}$ -- the latter will in general not converge in $\mathscr{M}_{0}(\mathscr{T})$. 
\subsection{Probabilistic assumptions and the main theorem}
Switching to a probabilistic approach, even when one views $\zeta$ as a random noise the model-valued random variables $Z^{\zeta_{\eps}}_{\can}$ will often not converge in a probabilistic sense -- one may still need to renormalize. 

Before stating our main theorem guaranteeing that this can be done we formalize the stochastic setting we are working in. 
\begin{definition}\label{def: smooth random variables}
We denote by $\mathcal{M}(\Omega_{\infty})$ the collection of all $\Omega_{\infty}$-valued random variables $\xi$ which satisfy all of the following properties.
\begin{enumerate}
\item For every $z \in \R^{d}$, $\mft \in \Le$ and $k \in \N^d$, the random variable $D^k \xi_{\mft}(z)$ has finite moments of all orders with its first moment vanishing.
\item The law of $\xi$ is invariant under the action of $\R^{d}$ on $\Omega_{\infty}$ given by translations.
\end{enumerate}
\end{definition}

These random variables are assumed to be defined on some underlying probability space $\Omega$ and, when
needed, we write $\omega$ for an element of $\Omega$. We will however mostly follow usual custom in
suppressing the dependence on $\omega$ from our notations.

Given $\xi \in \mathcal{M}(\Omega_{\infty})$ we will describe below in \eqref{def: BPHZ renormalized tree}
and \eqref{def of BPHZ model} how to construct what we call the 
``BPHZ renormalized lift'' of a realization $\xi(\omega)$, henceforth denoted by $Z^{\xi(\omega)}_{\BPHZ}$. 

Note that this is a slight abuse of notations since this object actually depends both 
on a specific realisation $\xi(\omega)$  and on the distribution of $\xi$.
We often write $Z^{\xi}_{\BPHZ}$ for the corresponding model-valued random variable.
We will also see in Lemma~\ref{lem: BPHZ model and tree} below that this construction is 
equivalent to that given in \cite[Eq.~3.1]{Proceedings} and \cite[Thm~6.17]{BHZalg}.
Additionally we define the space of stationary random models the BPHZ lift takes values in.
\begin{definition}
Given a regularity structure $\mathscr{T} = (\mcb{T},G)$, we denote by $\mathscr{M}_{\rand}(\mathscr{T})$ the space of all $\mathscr{M}_{0}(\mathscr{T})$-valued random variables $(\Pi,\Gamma)$ which are stationary in the sense that
there exists an action $\tau$ of $\R^d$ onto the underlying probability space by measure preserving maps 
such that, for every $x,y, h \in \R^d$, the identities
\begin{equ}
\Pi_{x+h}(\omega) = \Pi_{x}(\tau_h \omega)\;,\qquad
\Gamma_{x+h,y+h}(\omega) = \Gamma_{x,y}(\tau_h\omega)\;,
\end{equ}
hold almost surely.
\end{definition}

Then, after fixing a regularity structure $\mathscr{T}$, we can view the BPHZ lift as a map $Z^{\bullet}_{\BPHZ}:\mathcal{M}(\Omega_{\infty}) \rightarrow \mathscr{M}_{\rand}(\mathscr{T})$ corresponding to $\xi \mapsto Z^{\xi}_{\BPHZ}$. 

Note here that $Z_\BPHZ$ maps a space of random variables to another space of random variables
defined over the same probability space.
Before continuing we state a less technical and more specialized version of our main theorem. 
Henceforth, given $d \ge 1$, we call a random element of 
$\Omega_0 = \bigoplus_{\mft \in \Le} \mcD'(\R^{d})$ simply a \textit{noise}.
Given furthermore a scaling $\s$, a finite set of types $\Le$, and a homogeneity 
assignment $|\cdot|_{\s}$, we say that a Gaussian noise $\xi$ is \textit{compatible} with the homogeneity
assignment if, for every 
$\mft_{1}, \mft_{2} \in \Le$ there exists a distribution $C_{\mft_{1},\mft_{2}} \in \mcD'(\R^{d})$ with singular support contained in $\{0\}$ such that, for every $f \in \mcD'(\R^{d})$ one has
\[
\E[\xi_{\mft_{1}}(f) \xi_{\mft_{2}}(f)]
=
C_{\mft_{1},\mft_{2}}\Big( \int_{\R^{d}}dy f(y - \cdot)f(y) \Big)\;.
\]
Denoting by $x \mapsto C_{\mft_{1},\mft_{2}}(x)$ the smooth function representing $C_{\mft_{1},\mft_{2}}$ away from $0$,
we also require that for any $g \in \mcD'(\R^{d})$ with $D^{k}g(0) = 0$ for every $k \in \N^{d}$ with 
$|k|_{\s} < - |\mft_{1}|_{\s} - |\mft_{2}|_{\s} - |\s|$ one has
\[
C_{\mft_{1},\mft_{2}}(g) = 
\int_{\R^{d}}dx\ C_{\mft_{1},\mft_{2}}(x) g(x)\;.
\]
Finally, we impose that there exists $\kappa > 0$ such that for every $k \in \N^{d}$ with 
$|k|_{\s} \le 6 |\s|$ and $\mft_{1},\mft_{2} \in \Le$ one has
\begin{equ}[e:boundCompat]
\sup_{x \in \R^{d} \setminus \{0\}}
\big|
D^{k}
C_{\mft_{1},\mft_{2}}(x)
\big|
\cdot
|x|^{ - |\mft_{1}|_{\s} - |\mft_{2}|_{\s} + \delta + |k|_{\s}}
< \infty\;.
\end{equ}
Furthermore, we say that a family $\xi^{(n)}$ of compatible Gaussian noises is \textit{uniformly compatible}
if the quantity \eqref{e:boundCompat} is bounded for the corresponding covariance functions $C_{\mft_{1},\mft_{2}}^{(n)}$, 
uniformly over $n$.

\begin{theorem}\label{thm: Gaussian case}
Fix $d \ge 1$, a scaling $\s$, finite sets of types $\Le$ and $\mfL_{+}$, a homogeneity assignment $|\cdot|_{\s}$, 
a rule $R$ which is complete and subcritical with respect to $|\cdot|_{\s}$, and write
$\mathscr{T} = (\mcb{T},G)$ for the corresponding regularity structure. Assume that for every 
basis vector $T^{\mfn}_{\mfe} \in \mcb{T}$ and every subtree $S$ of $T$ with $|N(S)| \ge 2$, the following holds:
\begin{itemize}
\item For any non-empty leaf typed set $A$ with $\mft(A) \subset \mft(L(T))$ (in the sense of sets, not multi-sets) and $|A| + |L(S)|$ even, one has $|S^{0}_{\mfe}|_{\s} + |\mft(A)|_{\s} + |A| \cdot |\s|
> 0$. 
\item For every $u \in L(S)$ one has $|S^{0}_{\mfe}|_{\s} - |\mft(u)|_{\s} > 0$.
\item $|S^{0}_{\mfe}|_{\s} > -\frac{|\s|}{2}$.
\end{itemize}

Then, for every compatible Gaussian noise $\xi$, 
$Z^{\bullet}_{\BPHZ}$ extends continuously to $\xi$ in the following sense. 
There exists a unique element $Z^{\xi}_{\BPHZ} \in \mathscr{M}_{\rand}(\mathscr{T})$ such that, for any 
uniformly compatible sequence of Gaussian noises $(\xi_{n})_{n \in \N}$ with $\lim_{n \to \infty} \xi_{n} = \xi$ in probability in $\Omega_0$, 
one has $\lim_{n \to \infty} Z^{\xi_{n}}_{\BPHZ} = Z^{\xi}_{\BPHZ}$ in probability in $\mathscr{M}_{0}(\mathscr{T})$.  
\end{theorem}
\begin{proof}
This is a straightforward consequence of Theorem~\ref{main thm: continuity} below.
\end{proof}
The outline of this section is as follows. Mirroring our discussion of models, we will view $\mathcal{M}(\Omega_{\infty})$ as a subset of a larger space of random distributions $\mathcal{M}(\Omega_{0})$. 
We will then introduce some quantitative tools for looking at these random distributions which allow us to show that a notion of probabilistic ``convergence'' of a sequence $(\xi_{n})_{n \in \N} \subset \mathcal{M}(\Omega_{\infty})$ yields stochastic $L^{p}$ convergence of the corresponding BPHZ models $(Z^{\xi_{n}}_{\BPHZ})_{n \in \N}$.
 
\subsubsection{Quantitative estimates on cumulants}

Our quantitative assumptions on the driving noise all involve their \emph{cumulants}.
Throughout the paper, for any (multi)set $B$ of real-valued random variables 
with finite moments of all orders defined on some underlying probability space $\Omega$, 
we write $\Cum[B]$ for the joint cumulant of the elements of $B$.
Recall that, writing $\E[B]$ as a shorthand for the expectation of the product
of all elements of $B$, the cumulant is defined recursively by the identities
$\Cum[\emptyset] = 1$ and then
$\E[B] = \sum_{\pi \in \CP(B)}\prod_{A \in \pi} \Cum[A]$, where
$\CP(B)$ denotes the set of all partitions of $B$ into disjoint subsets, see for example
\cite[Sec.~13.5]{GlimmJaffe}.

We introduce the ``large diagonal'' of the set
$\R^{dN}$ of configurations of $N$ points in $\R^d$ by
\[
\mathrm{Diag}_{N}
\eqdef
\{ 
x \in \R^{dN}:\ 
\exists 1 \le i < j \le d \textnormal{ with } x_{i} = x_{j} 
\}\;.
\]
While we allow for our random driving noises to be quite singular we impose some strong, but natural, regularity assumptions on their cumulants -- this is encoded via the following notion. 
\begin{definition}
Suppose that, for $N \ge 2$, we are given a collection\footnote{This collection may be a multiset in that the same
 random variable may appear more than once.} $\{\zeta_{n}\}_{n=1}^{N}$ of $\mcD'(\R^{d})$-valued random variables that each have moments of all orders, invariant in joint law under the natural action of $\R^{d}$ on this collection given by simultaneous translation, and are all defined on the same probability space. 

We say that
such a collection \emph{admits pointwise cumulants} if there exists a distribution $F \in \mcD'(\R^{dN})$ with the singular support of $F$ contained in $\mathrm{Diag}_{N}$ (we denote by $F(x)$ the corresponding function, which is smooth away from $\mathrm{Diag}_{N}$) and satisfying, for every collection of test functions $f_{1},\dots,f_{N} \in \mcD(\R^{d})$ the identity
\begin{equ}\label{cumulant as dist}
\Cum[ \{\zeta_{n}(f_{n})\}_{n=1}^{N} ]
=
F(f_{1} \otimes \cdots \otimes f_{N})
\end{equ}
where $(f_{1} \otimes \cdots \otimes f_{N})(x) \eqdef \prod_{n=1}^{N} f_{n}(x_{n}) \in  \mcD(\R^{dN})$.

As part of this definition, we also enforce two regularity conditions, one on second cumulants and the other on higher cumulants.  
When $N=2$, translation invariance guarantees that there exists a unique $\reduce[F] \in \mcD'(\R^{d})$ such that for every $f,g \in \mcD'(\R^{d})$ one has
\begin{equ}\label{transinvar second cumulant}
F(f \otimes g) = \reduce[F]\big(f \ast \mathrm{Ref}(g)\big)
\end{equ}
where $\mathrm{Ref}(g)(x) \eqdef g(-x)$ and $\ast$ denotes convolution. 
Note that the singular support of $\reduce(F)$ is given by $\{0\}$. 
Given a distribution $G \in \mcD'(\R^{d})$ with singular support contained in $\{0\}$ and $r \in \R$ we say $G$ is of order $r$ at $\{0\}$ if, for $f \in \mcD(\R^{d})$ 
with $D^k f(0) = 0$ for every $k \in \N^{d}$ with $|k|_\s < r$, one has
\begin{equ}
G(f) = \int_{\R^{d}} f(y)G(y)\,dy\;,
\end{equ}
When $N \ge 3$, we enforce that $F(x)$ be a locally integrable function on $\R^{dN}$ and for the distribution $F$ to be given by integration against $F(x)$.  
\end{definition}\label{def: all random variables}
For a collection 
$\{\zeta_{n}\}_{n=1}^{N}$ admitting pointwise cumulants $F$ as above, we will write
$\Cum[ \{\zeta_{n}(x_{n})\}_{n=1}^{N} ]$ for the value $F(x_{1},\dots,x_{N})$
(provided that $x \not\in \mathrm{Diag}_{N}$) and 
$\Cum[ \{\zeta_{n}\}_{n=1}^{N}]$ for the distribution
$F$.
\begin{definition} 
We define $\mathcal{M}(\Omega_{0})$ to be the set of all $\bigoplus_{\mft \in \Le} \mcD'(\R^{d})$-valued random variables $\xi$ that satisfy the following properties.
\begin{itemize}
\item For every $\mft \in \Le$ and $f \in \mcD(\R^{d})$, $\xi_{\mft}(f)$ has moments of all orders.
\item The law of $\xi$ is stationary, i.e. invariant under the action of $\R^{d}$ on $\bigoplus_{\mft \in \Le} \mcD'(\R^{d})$ by translation.
\item For every $N \ge 2$ and any map $\mft: [N] \rightarrow \Le$ the collection of 
generalised random fields $\{\xi_{\mft(n)}\}_{n=1}^{N}$ admits pointwise cumulants. 
\end{itemize}
\end{definition}
Clearly, one has $\mathcal{M}(\Omega_{\infty}) \subset \mathcal{M}(\Omega_{0})$.
Also, for $\xi, \bar{\xi} \in \mathcal{M}(\Omega_{0})$ defined on the same probability space, we say 
that $\xi$ and $\bar{\xi}$ jointly admit pointwise cumulants if, for every $N\ge 2$, $1 \le M \le N$, and $\mft:[N] \rightarrow \Le$ the multiset $\{ \xi_{\mft(m)} \}_{m=1}^{M} \sqcup \{\bar{\xi}_{\mft(n)}\}_{n=M+1}^{N}$ admits pointwise cumulants.

In order to prove theorems that give near optimal results it is natural to take advantage of the fact that for 
certain random noise-types many cumulants vanish: the most obvious case is that of Gaussian noise where all 
cumulants of order greater than $2$ vanish, but another interesting case is that of noise belonging to a Wiener 
chaos of fixed order.
We define a set that indexed all possible cumulants by 
\[
\mfL_{\CCum}^{\all}
\eqdef
\{ (\mft,[M]):\ M \in \N, \textnormal{ and } \mft: [M] \rightarrow \Le \}\;.
\]
We will then consider noises such that all cumulants outside of
some subset $\mfL_{\CCum} \subset \mfL_{\CCum}^{\all}$ vanish.
We make the following natural assumptions on this subset. 
\begin{assumption}\label{assump: set of cumulants}
We assume that we have fixed a set $\mfL_{\CCum} \subset \mfL_{\CCum}^{\all}$ of allowable cumulants with the following properties:
\begin{enumerate} 
\item For every $(\mft,[M]) \in \mfL_{\CCum}$ one has $M \ge 2$.
\item $\mfL_{\CCum}$ is closed under permutations - that is if $(\mft,[M]) \in \mfL_{\CCum}$ then for every permutation $\sigma:[M] \rightarrow [M]$ one has $(\mft \circ \sigma,[M]) \in \mfL_{\CCum}$. 
It is natural to think of $\mfL_{\CCum}$ as a collection of ``leaf typed multisets'', thus it makes sense to ask, given a finite set $B$ and a type map $\mft: B \rightarrow \Le$, whether $(\mft,B) \in \mfL_{\CCum}$. 
\item If $(\mft,B) \in \mfL_{\CCum}$ then for any $A \subset B$ with $|A| \ge 2$ one has $(\mft,A) \in \mfL_{\CCum}$.
\item For every $\bar{\mft} \in \Le$ one has $(\mft,[2]) \in \mfL_{\CCum}$ where $\mft(1) = \mft(2) = \bar{\mft}$. 
\end{enumerate}
\end{assumption}
\begin{remark}
Our main theorem will enforce an assumption of ``super-regularity'' on the semi-decorated trees. This notion of super-regularity will depend on our choice of $\mfL_{\CCum}$, becoming more stringent as this set becomes richer. 
The second item in the above assumption on $\mfL_{\CCum}$ enforces a type of ``completeness'' that guarantees that the constraint of super-regularity is suitably strong. 
\end{remark}
\begin{remark}
Given $(\mft,[N]) \in \mfL^{\all}_{\CCum}$ with $\mft(i) = \bar{\mft}$ for $i \in [N]$ we sometimes write $(\bar{\mft},[N])$ instead of $(\mft,[N])$.
\end{remark}
We also define $\mcM(\Omega_{0},\mfL_{\CCum})$ and $\mcM(\Omega_{\infty},\mfL_{\CCum})$ to be the
subsets of $\mcM(\Omega_{0})$ and $\mcM(\Omega_{\infty})$, respectively, consisting of those elements $\xi$ with 
$\Cum[\{\xi_{\mft(n)}\}_{n=1}^{N}] = 0$ for every $(\mft,[M]) \in \mfL_{\CCum}^{\all} \setminus \mfL_{\CCum}$.
We now define the following family of ``norms'' on $\mathcal{M}(\Omega_{0},\mfL_{\CCum})$.
First, we fix, for the entire paper, a smooth non-negative function $u: \R^{d} \rightarrow u$ satisfying $|u(x)| \le 1$ for all $x$, $u(x) = 1$ for $|x| \le 1$, and $u(x) = 0$ for $|x| \ge 2$. 
We then set, for each $k \in \N^{d}$, $p_{k}(x) = u(x)x^{k}$. 
\begin{definition}\label{def: cumulantbound}
For $\xi \in \mcM(\Omega_{0},\mfL_{\CCum})$ and $N \ge 2$, we set $\|\xi\|_{N,|\cdot|_{\s}}$ to be equal to  
\[
\mathrm{Diag}(\xi)
\vee
\max_{
\substack{
(\mft, [M]) \in \mfL_{\CCum}\\
2 \le M \le N
}
}
\sup_{
\substack{
x \in \R^{dM}\\
x \not \in \mathrm{Diag}_{M}}
}
\left|\Cum[ \{\xi_{\mft(m)}(x_{m})\}_{m=1}^M ]\right|
\Bigl(\sup_{1 \le i < j \le M} |x_i - x_j|\Bigr)^{-|\mft([M])|_{\s}}\;,
\] 
where we use the notation $[M] = \{1,\ldots,M\}$ and $|\mft([M])|_{\s}$ is as in Definition~\ref{def:hom}.

The quantity $\mathrm{Diag}(\xi)$ is defined as follows. 
If there exists any $(\mft,[2]) \in \mfL_{\CCum}$ such that $\reduce[\Cum[\{\xi_{\mft(1)},\xi_{\mft(2)}\}]]$ is not of order $(- |\mft([2])|_{\s} - |\s|) \vee 0$ at $0$ then we set  $\mathrm{Diag}(\xi) \eqdef \infty$.
Otherwise we set
\begin{equ}
\label{def: control over the diagonal}
\mathrm{Diag}(\xi)
\eqdef
\max_{(\mft,[2]) \in \mfL_{\CCum} }
\max_{
\substack{
k \in \N^{d},\\
|k|_{\s} < -|\mft([2])|_{\s} - |\s| 
}
}
|\reduce[\Cum[\{\xi_{\mft(1)},\xi_{\mft(2)}\}]](p_{k})|\;.
\end{equ}
where $\reduce[\cdot]$ is as in \eqref{transinvar second cumulant}. 
\end{definition}
\begin{remark}
The point of the definition of $\mathrm{Diag}(\xi)$ is that in general we do not impose that the 
function $\Cum[\{\xi_{\mft(1)},\xi_{\mft(2)}\}]$ is locally integrable near the diagonal. 
In general,
there are then several inequivalent ways of extending it to a distribution, and these are precisely
parametrised by the values $\tilde F(p_{k})$. 
\end{remark}
It turns out that this ``norm'' also naturally defines a ``metric'' on 
elements of $\mathcal{M}(\Omega_{0},\mfL_{\CCum})$ which are defined on the same probability space. 
(We intentionally write ``metric'' in quotation marks since it does not satisfy the triangle inequality.)
\begin{definition}\label{def: metric on mollified measures}
For $\xi, \bar \xi \in \mcM(\Omega_{0},\mfL_{\CCum})$ defined on the same probability space and jointly admitting pointwise cumulants we set $\|\xi; \bar\xi\|_{N,|\cdot|_{\s}}$ to be given by
\[
\mathrm{Diag}(\xi ; \bar{\xi})
\vee
\max_{
\substack{
(\mft, [M]) \in \mfL_{\CCum}\\
2 \le \hat{M} \le M \le N
}
}
\sup_{
\substack{
x \in \R^{dM}\\
x \not \in \mathrm{Diag}_{M}}
}
\left|
\Cum
\left[
\{\hat{\xi}_{\mft(i),i}(x_{i})\}_{i=1}^{M}\\
\right]
\right|
\Bigl(\sup_{1 \le i < j \le M} |x_i - x_j|\Bigr)^{-|\mft([M])|_{\s}}
\]
where, for any $\hat M \le M$, the generalized random fields $\{\hat{\xi}_{\mft,i}\}_{i=1}^{M}$ are given by
\begin{equ}\label{def: telescoping field}
\hat{\xi}_{\mft,i}
\eqdef
\xi_{\mft} \one\{i \le \hat M\} - \bar\xi_{\mft} \one\{i \ge \hat M\}\;.
\end{equ}
The quantity $\mathrm{Diag}(\xi;\bar{\xi})$ is defined as follows. 
If there exists any $(\mft,[2]) \in \mfL_{\CCum}$ and $\hat{M} \in \{1,2\}$ such that $\reduce[\Cum[\{\hat{\xi}_{\mft(1)},\hat{\xi}_{\mft(2)}\}]]$ is not of order $|\mft([2])|_{\s} - |\s|$ at $0$ then we set  $\mathrm{Diag}(\xi;\bar{\xi}) \eqdef \infty$.
Otherwise we set
\begin{equ}\label{eq: control over diagonal 2}
\mathrm{Diag}(\xi ; \bar{\xi})
\eqdef
\max_{
\substack{
(\mft,[2]) \in \mfL_{\CCum}\\
1 \le \hat{M} \le 2 }
}
\max_{
\substack{
k \in \N^{d},\\
|k|_{\s} < -|\mft([2])|_{\s} - |\s| 
}
}
|\reduce[\Cum[\{\hat{\xi}_{\mft(1)},\hat{\xi}_{\mft(2)}\}]](p_{k})|\;.
\end{equ}
If $\xi$ and $\bar \xi$ do not admit joint pointwise cumulants we set $\|\xi; \bar\xi\|_{N} = \infty$ for every $N$. 
\end{definition}
\subsubsection{Super-regular semi-decorated trees}
The following assumption on the set of types $\mfL$ and on the set $\mfL_\CCum$ from Assumption~\ref{assump: set of cumulants}
will be in place throughout the entire paper. 
\begin{assumption}\label{assump - noise + kernel}
For every $\mft' \in \Le$ one has $|\mft'|_{\s} > -|\s|$ and for every $(\mft,[M]) \in \mfL_{\CCum}$ with $M \ge 3$ one has $|\mft([M])|_{\s} > (1-M)|\s|$.
\end{assumption} 
\begin{remark}
The above assumption guarantees that only cumulants which need to be renormalized are second cumulants. 
There is no fundamental obstruction in renormalizing higher cumulants in our approach but we refrain from doing so since this would make the presentation harder to follow. \tabularnewline
\end{remark}
Next we state a definition and an assumption on semi-decorated trees we will need in order to prove our stochastic bounds.
\begin{definition}\label{external cumulant homogeneity jump}
Given finite leaf typed sets $A$ and $B$ with $\mft(B) \subset \mft(A)$ (in the sense of sets, not multi-sets) we define $\mcb{j}_{A}(B)$ to be the minimum value of $|\mft(C)| + |C| \cdot |\s|$ where $C$ is a non-empty finite leaf typed set with $\mft(C) \subset \mft(A)$ (in the sense of sets, not multi-sets) and with the property that there exists a partition $\pi$ of $C \sqcup B$ with $(\mft,\tilde{B}) \in \mfL_{\CCum}$ and $\tilde{B} \cap B \not = \emptyset$ for every $\tilde{B} \in \pi$.
\end{definition}
\begin{remark}
The quantity $\mcb{j}_{A}(B)$ represents the worse case homogeneity change one sees when the noises of $B$ form cumulants with some other noises drawn from multiple copies of $A$. 
\end{remark}
\begin{definition}\label{def: strong subcriticality}
Given a homogeneity assignment $|\cdot|_{\s}$ and a set of cumulants $\mfL_{\CCum}$, a semi-decorated tree $T^{\mfn}_{\mfe}$ is said to be $(|\cdot|_{\s},\mfL_{\CCum})$-\emph{super-regular} if for every subtree $S$ of $T$ with $|N(S)| > 1$ and $L(S) \not = \emptyset$ and for any $u \in L(S)$ one has
\begin{equ}\label{super-reg 0}
|S^{0}_{\mfe}|_{\s} + \frac{|\s|}{2} \wedge ( - |\mft(u)|_{\s}) \wedge \mcb{j}_{L(T)}(L(S)) > 0\;.
\end{equ}
\end{definition}
\subsubsection{Main results}
We make use of translated and rescaled test functions,
namely we write
\[
\psi^{\lambda}_{z}(x_{1},\dots,x_{d}) \eqdef \lambda^{-|\s|}\psi
\left(
\lambda^{-\s_{1}}(x_{1}-z_{1}),\dots,\lambda^{-\s_{d}}(x_{d}-z_{d})
\right)\;.
\]

\begin{remark}
In what follows we often shrink the co-domain and view $Z^{\bullet}_{\BPHZ}$ as a map $Z^{\bullet}_{\BPHZ}: \mcM(\Omega_{\infty},\mfL_{\CCum}) \rightarrow \mathscr{M}_{\rand}(\mathscr{T})$ given by $\xi \mapsto Z^{\xi}_{\BPHZ}$. All of our main theorems concern bounds or continuity properties of this map and will take as assumptions conditions that depend on $\mfL_{\CCum}$. 
\end{remark}

\begin{remark}\label{remark: simplified cumulant homogeneity}
In the statement of the main theorem which immediately follows and subsequent proofs, we will use generalisations $\|\bullet\|_{N,\c}$ and $\|\bullet;\bullet\|_{N,\c}$ of Definitions~\ref{def: cumulantbound} and~\ref{def: metric on mollified measures} as well as a notion of ``$(\c,|\cdot|_{\s},\mfL_{\CCum})$-super-regularity'' which generalizes Definition~\ref{def: strong subcriticality}. 
Here $\c$ is a parameter which will be called a \emph{cumulant homogeneity} for $\mfL_{\CCum}$ consistent with $|\cdot|_{\s}$ -- this is explained in Appendix~\ref{cumulant homogeneities}. 
As this notion is a bit technical we prefer to state the main theorem without introducing it. For now, the reader
who doesn't wish to delve into these details
could just keep in mind that there is an acceptable choice of the parameter $\c$ for which $(\c,|\cdot|_{\s},\mfL_{\CCum})$-super-regularity reduces to the $|\cdot|_{\s}$ super-regularity of Definition~\ref{def: strong subcriticality} and for which
the quantities $\| \bullet \|_{N,\c}$ and $\| \bullet; \bullet \|_{N,\c}$ reduce to $\| \bullet \|_{N,|\cdot|_{\s}}$ and $\| \bullet; \bullet \|_{N,|\cdot|_{\s}}$, respectively, as noted in Remark~\ref{rem:backtrack}.
In particular, these notions agree in the Gaussian case.
\end{remark}
\begin{theorem}\label{upgraded thm - main theorem}
Fix $d \ge 1$, a scaling $\s$, finite sets of types $\Le$ and $\mfL_{+}$, a homogeneity assignment $|\cdot|_{\s}$, a set of cumulants $\mfL_{\CCum}$ satisfying Assumption~\ref{assump - noise + kernel}, and a cumulant homogeneity $\c$ for $\mfL_{\CCum}$ consistent with $|\cdot|_{\s}$. 
Let $R$ be a complete subcritical rule and $\mathscr{T}$ be the corresponding reduced regularity structure.

Then the map $Z^{\bullet}_{\BPHZ}:\mcM(\Omega_{\infty},\mfL_{\CCum}) \rightarrow \mathscr{M}_{\rand}(\mathscr{T})$ has the following properties.

For every compact $\mathfrak{K} \subset \R^{d}$, $\xi \in \CM(\Omega_\infty,\mfL_{\CCum})$,  $p \in \N$, and $\tau \eqdef T^{\mfn}_{\mfe} \in B_{\circ}$ which is $(\c,|\cdot|_{\s},\mfL_{\CCum})$-super regular, there exists $C_{\tau,p}$ such that, writing $Z^{\xi}_{\BPHZ} = (\hat{\Pi}^{\xi},\hat{\Gamma}^{\xi})$, one has
\begin{equ}\label{eq: uniform bound - main theorem}
\E \left( \hat{\Pi}^{\xi}_{z}[\tau](\psi^{\lambda}_{z}) \right)^{2p}
\le
C_{\tau,p}
\Big( 
\prod_{e \in K(T)} \|K_{\mft(e)}\|_{|\mft(e)|_{\s},m}
\Big)^{2p}
\|\xi\|_{N,\c}
\lambda^{2p|\tau|_{\s}}\;,
\end{equ}
where $m \eqdef 2|\s|\cdot |N(T)|$ and $N \eqdef 2p |L(T)|$,
uniformly over $\psi \in \mcb{B}_{{\s}}$, $\lambda \in (0,1]$ and $z \in \mathfrak{K}$.
Moreover, with the same notation and conventions as above, for any $\bar\xi \in \CM(\Omega_\infty)$ defined on the same probability space one has 
\begin{equs}\label{eq: convergence bound - main theorem}
&
\E \left[ 
\left(
\hat{\Pi}^{\xi}_{z}  
-
\hat{\Pi}^{\bar{\xi}}_{z} \right) [\tau]
(\psi^{\lambda}_{z})\right]^{2p}\\
&
\quad
\quad
\le
C_{\tau,p}
\Big( 
\prod_{e \in K(T)} \|K_{\mft(e)}\|_{|\mft(e)|_{\s},m}
\Big)^{2p}
(\|\xi\|_{N,\c} \vee \|\bar \xi\|_{N,\c})
\| \xi; \bar\xi\|_{N,\c}
\lambda^{2p|\tau|_{\s}}\;.
\end{equs}
\end{theorem} 
\ajay{Should add something about all homogeneity assignments staying away from integers...}
\begin{definition}
Given homogeneity assignments $|\cdot|_{\s}$ and $\wnorm{\cdot}_{\s}$ and a rule $R$ which is complete and subcritical with respect to $\wnorm{\cdot}_{\s}$, we say $|\cdot|_{\s}$ is $R$-consistent with $\wnorm{\cdot}_{\s}$ if all of the following conditions hold:
\begin{itemize}
\item $\wnorm{\cdot}_{\s}$ satisfies Assumptions~\ref{assump - noise + kernel}.
\item $R$ is also a complete and subcritical rule with respect to $\wnorm{\cdot}_{\s}$. 
\item The collection $B_{-}$ of semi-decorated trees generated by the rule $R$ under either $|\cdot|_{\s}$ and $\wnorm{\cdot}_{\s}$ coincide.
\item For any semi-decorated tree $T^{\mfn}_{\mfe} \in B_{-}$, one has 
$
\lfloor
|T^{\mfn}_{\mfe}|_{\s}
\rfloor
=
\big\lfloor
\wnorm{T^{\mfn}_{\mfe}}_{\s}
\big\rfloor
$.
\end{itemize}
\end{definition}
Combining Theorem~\ref{upgraded thm - main theorem} with the proof of \cite[Thm~10.7]{Regularity} immediately gives the following theorem. 
\begin{theorem}\label{reg struct - main theorem}
Fix $d \ge 1$, a scaling $\s$, finite sets of types $\Le$ and $\mfL_{+}$, a homogeneity assignment $\wnorm{\cdot}_{\s}$ and a set of cumulants $\mfL_{\CCum}$ satisfying Assumption~\ref{assump - noise + kernel}, and a rule $R$ which is complete and subcritical with respect to $\wnorm{\cdot}_{\s}$. 
Denote by $\mathscr{T} = (A,\mcb{T},G)$ be the corresponding reduced regularity structure with $\wnorm{\cdot}_{\s}$ serving as the homogeneity assignment.

Fix $\kappa > 0$, a second homogeneity assignment $|\cdot|_{\s}$, and a cumulant homogeneity $\c$ on $\mfL_{\CCum}$ consistent with $|\cdot|_{\s}$ with the following properties:   
\begin{itemize}
\item $|\cdot|_{\s}$ is $R$-consistent with $\wnorm{\cdot}_{\s}$.
\item For any $\alpha \in A$ with $\alpha < 0$, every semi-decorated tree $T^{\mfn}_{\mfe} \in \mcb{T}_{\alpha}$ is $(\c,|\cdot~|_{\s},\mfL_{\CCum})$-super regular and 
\[
|T^{\mfn}_{\mfe}|_{\s}
\ge
\wnorm{T^{\mfn}_{\mfe}}_{\s}
+
\kappa.
\]
\end{itemize}
Then the map $Z^{\bullet}_{\BPHZ}:\mcM(\Omega_{\infty},\mfL_{\CCum}) \rightarrow \mathscr{M}_{\rand}(\mathscr{T})$ has the following property.

For every compact set $\mathfrak{K} \subset \R^{d}$, $\alpha \in A$, and $p \in \N$ there exists $C$ such that for any $\xi, \bar{\xi} \in \CM(\Omega_\infty,\mfL_{\CCum})$ one has 
\begin{equs}
{}
&
\E \$ Z^{\xi}_{\BPHZ}; Z^{\bar{\xi}}_{\BPHZ} \$_{\alpha, \mathfrak{K}}^{2p}
\le
C
(\|\xi\|_{N,\c} \vee \|\bar \xi\|_{N,\c})
\big( 
\| \xi; \bar\xi\|_{N,\c}
\big)\;,\ \textnormal{\it where}\\
N \eqdef&\ 2 \Big(p \vee \Big\lceil \frac{|\s|}{\kappa} \Big\rceil \Big) \cdot \max \{ |L(T)|: T^{\mfn}_{\mfe} \in \mcb{T}_{\alpha}\; \textnormal{\it for } \alpha \in A,\ \alpha < 0 \}\;. 
\end{equs}
\end{theorem} 
The following theorem states that if we impose some mild regularity conditions then the map $\xi \mapsto Z^{\xi}_{\BPHZ}$ admits some nice continuity and regularity properties.
\begin{theorem}\label{main thm: continuity}
Fix $d \ge 1$, a scaling $\s$, finite sets of types $\Le$ and $\mfL_{+}$, a homogeneity assignment $\wwnorm{\cdot}_{\s}$ and a set of cumulants $\mfL_{\CCum}$ satisfying Assumption~\ref{assump - noise + kernel}, and a cumulant homogeneity $\c$ on $\mfL_{\CCum}$ consistent with $\wwnorm{\cdot}_{\s}$. Let $R$ be a rule which is complete and subcritical with respect to $\wwnorm{\cdot}_{\s}$.

Fix $\kappa$ satisfying
\[
0 < \kappa < 
\frac{1}{2}
\min_{(\mft,[2]) \in \mfL_{\CCum}}
\Big( \lceil - \wwnorm{\mft([2])}_{\s} \rceil + \wwnorm{\mft([2])}_{\s} \Big),
\] 
and small enough such that the homogeneity assignment    
\[
\wnorm{\mft}_{\s}
\eqdef \wwnorm{\mft}_{\s} \cdot (1 + \kappa \mathbbm{1}\{ \mft \in \Le\})\; \textnormal{\it for every }\mft \in \mfL\;.
\]
is $R$-consistent with $\wwnorm{\cdot}$. 
Denote by $\mathscr{T} = (A,\mcb{T},G)$ the corresponding reduced regularity structure generated by $R$ with $\wnorm{\cdot}_{\s}$ as its homogeneity assignment.

Assume that for every $\alpha \in A$ with $\alpha < 0$ and every $T^{\mfn}_{\mfe} \in \mcb{T}_{\alpha}$, the decorated tree $T^{\mfn}_{\mfe}$ is $(\tilde{\c}, \wnorm{\cdot}_{\s}, \mfL_{\CCum})$-super-regular where $\tilde{\c}$ is the $\kappa$-penalization\footnote{see Definition~\ref{def: cumulant homogeneity penalization}.} of $\c$.
Define
\begin{equ}
\mcM(\Omega_{0},\mfL_{\CCum},\c)
\eqdef
\{ 
\xi \in \mcM(\Omega_0,\mfL_{\CCum}):\ 
\|\xi\|_{N,\c,2}
<
\infty
\}\;,
\end{equ}
where
\begin{equ}
N \eqdef\
2 \cdot \Big( 2 \vee \Big\lceil \frac{2|\s|}{\kappa} \Big\rceil \Big) 
\max \{ |L(T)|: T^{\mfn}_{\mfe} \in \mcb{T}_{\alpha}\; 
\textnormal{\it for } \alpha \in A,\ \alpha < 0 \}\;.   
\end{equ}
Then there is a unique extension of $Z^{\bullet}_{\BPHZ}:\mcM(\Omega_{\infty},\mfL_{\CCum}) \rightarrow \mathscr{M}_{\rand}(\mathscr{T})$ to all of $\mcM(\Omega_{0},\mfL_{\CCum},\c)$ which satisfies the following continuity property: for any $(\xi_{n})_{n \in \N} \subset \mcM(\Omega_{0},\mfL_{\CCum},\c)$ and $\xi \in \mcM(\Omega_{0},\mfL_{\CCum},N,\c)$ such that 
\begin{claim}
\item[(i)] $\sup_{n \ge 0} \|\xi_{n}\|_{N,\c,1} < \infty$ and
\item[(ii)] $\xi_{n} \rightarrow \xi$ in probability on $\bigoplus_{\mft \in \Le} \mathcal{D}'(\R^{d})$
\end{claim}
$Z^{\xi_{n}}_{\BPHZ}$ converges to $Z^{\xi}_{\BPHZ}$ in probability on $\mathscr{M}_{0}(\mathscr{T})$. 
\end{theorem}
\begin{remark}
Of course, this extended map $Z^{\bullet}_{\BPHZ}$ also satisfies a weaker version of the above property  -- if (ii) above is replaced by the assumption that $\xi_{n} \rightarrow \xi$ in law on $\bigoplus_{\mft \in \Le} \mathcal{D}'(\R^{d})$ then $Z^{\xi_{n}}_{\BPHZ}$ converges to $Z^{\xi}_{\BPHZ}$ in law on $\mathscr{M}_{0}(\mathscr{T})$.
\end{remark}
\begin{remark}
We define $Z^{\xi}_{\BPHZ}$, for any $\xi \in \mcM(\Omega_{0},\mfL_{\CCum},N,\c)$, as the $L^{2}$ limit $ \lim_{n \rightarrow \infty} Z_{\BPHZ}^{\xi_{n}}$ where we set $\xi_{n} = \xi \ast \eta_{n}$ where $(\eta_{n})_{n \in \N}$ is a sequence of approximate identities converging in an appropriate sense to a Dirac delta function as $n \rightarrow \infty$.  
Thus we can see this extended $Z^{\bullet}_{\BPHZ}$ as a map from $\mcM(\Omega_{0},\mfL_{\CCum},N,\c)$ into the space of all measurable maps from $\Omega_{0}$ to $\mathscr{M}_{0}(\mathscr{T})$.
\end{remark} 

\begin{remark}
Henceforth all of our estimates are claimed to be uniform in $\lambda \in (0,1]$, even if this is not explicitly stated.
\end{remark}
For the rest of the paper we fix a cumulant homogeneity $\c$ on $\mfL_{\CCum}$ consistent with our fixed homogeneity assignment $|\cdot|_{\s}$ and a $(\c,|\cdot|_{\s},\mfL_{\CCum})$-super regular semi-decorated tree $\sT^{\sn}_{\se} \in B_{\circ}$ for which we will prove Theorem~\ref{upgraded thm - main theorem}.
We gather all the steps to prove Theorem~\ref{upgraded thm - main theorem} and Theorem~\ref{main thm: continuity} in Section~\ref{sec: proof of main theorem}. 
\section{Renormalization of combinatorial trees and random fields}\label{Sec: def of BPHZ model} 
\subsection{Colorings, a new decoration, more homogeneities, and identified forests}
\subsubsection{Colorings}
\begin{definition}
A colored forest is a pair $(F,\hat{F})$ where $F$ is a typed rooted forest and the ``coloring'' $\hat{F}$ is a pair $(\hat{F}_{1},\hat{F}_{2})$ of disjoint subforests of $F$ with the property that $\rho(\hat{F}_{2}) \subset \rho(F)$.

Note that $\hat{F}$ induces a map\footnote{Here $\hat{F}$ is an overloaded notation but whether it is being treated as a subforest or as a map should always be clear from context.} $\hat{F}:N_{F} \sqcup E_{F} \rightarrow \{0,1,2\}$ 
by setting $\hat F(u) = i$ if $u \in N_{\hat F_i} \sqcup E_{\hat F_i}$ and $\hat F(u) = 0$ otherwise.
Clearly, one can recover the tuple $\hat{F}$ from the map $\hat{F}(\cdot)$.
\end{definition}
There is a useful alternate notation for specifying colorings of a forest. 
Given $F$ and two disjoint subforests $A$ and $B$ of $F$ with $\rho(B) \subset \rho(F)$ we write $[A]_{1} \sqcup [B]_{2}$ for the coloring $\hat{F} = (A,B)$.
If one of the two forests is empty we may drop it, i.e. we may write $(F,[A]_{1})$ instead of $(F,[A]_{1} \sqcup [\mathbf{1}]_{2})$. 
Additionally, for $i \in \{0,1,2\}$ we write $(F,i)$ for the colored forest with the constant 
coloring $\hat{F}(\cdot) = i$. 
\subsubsection{Decorated colored forests}
A decorated colored forest consists of a colored forest $(F,\hat{F})$ and three maps: $\mfn: N(F) \rightarrow \N^{d}$, $\mfo: N(\hat{F}_{1}) \rightarrow \Z^{d} \oplus \Z(\Lab)$, and $\mfe: E_{F} \rightarrow \N^{d}$,
where $\Z(\Lab)$ denotes the free $\Z$-module generated by $\Lab$. 
Given a colored forest $(F,\hat{F})$ and decorations $\mfn$, $\mfo$, and $\mfe$ on $F$ we denote the corresponding decorated colored forest by $(F,\hat{F})^{\mfn,\mfo}_{\mfe}$. 

For the trivial tree we just write $(\bullet,i)^{\mfn,\mfo}$ and also observe that the empty forest $\mathbf{1}$ is automatically a decorated colored forest. 
When working with a forest $(F,\hat{F})$ where $\hat{F}_{1} = \mathbf{1}$ we abuse notation and write $\mfo = 0$ instead of writing $\mfo = \emptyset$. 
Moreover, if the $\mfo$-label vanishes or is given by $\emptyset$ we often drop it from the notation. 

Observe that given a decorated colored forest $(F,\hat{F})^{\mfn,\mfo}_{\mfe}$ and a subforest $A$ of $F$ there is, by taking restrictions, a corresponding decorated colored forest $(A,\hat{F})^{\mfn,\mfo}_{\mfe}$. 
Here we abuse notation by not making this restriction on our decorations or colorings explicit in our notation.

Clearly any semi-decorated tree $T^{\mfn}_{\mfe}$ naturally corresponds to a decorated colored tree $(T,0)^{\mfn}_{\mfe}$. 
Finally, we also remark that the forest product extends naturally to a product on decorated colored forests.
\subsubsection{More homogeneities}
With the introduction of coloring and  $|\cdot|_{s}$ we define define two homogeneities 
$|\cdot|_{+}$ and $|\cdot|_{-}$ on decorated colored forests as follows: 
given $\tau = (F,\hat{F})^{\mfn,\mfo}_{\mfe}$, we set 
\begin{equs}
|\tau|_{-}
\eqdef&\ 
\sum_{
\substack{
e \in E_{F}\\
\hat{F}(e)=0}
}
\left(
|\mft(e)|_{\s}
-
|\mfe(e)|_{\s}
\right)
+
\sum_{u \in N(F)}
|\mfn(u)|_{\s}\;,\\
|\tau|_{+}
\eqdef&\ 
\sum_{
\substack{
e \in E_{F}\\
\hat{F}(e) \not = 2
}
}
\left(
|\mft(e)|_{\s}
-
|\mfe(e)|_{\s}
\right)
+
\sum_{
\substack{
u \in N(F)\\
\hat{F}(u) 
\not = 2
}
}
(|\mfn(u)|_{\s} +
|\mfo(u)|_{\s})\;.
\end{equs}
\subsubsection{Identified forests}\label{subsec: identified forests}
An \textit{identified forest} (or i-forest) is a pair $\bo{\sigma} = (\sigma,\iota_{\bo{\sigma}})$ where $\sigma$ is a colored forest and $\iota_{\bo{\sigma}}$ is a forest morphism from $\sigma$ into $\sT$ which restricts to a tree monomorphism on \emph{each connected component} of $\sigma$. Note that $\iota_{\bo{\sigma}}$ is \emph{not} required to be globally injective and in most cases it will not be.
Two i-forests $\bo{\sigma_{1}}$, $\bo{\sigma_{2}}$ are isomorphic if there exists a forest isomorphism $\iota\colon \sigma_{1} \rightarrow \sigma_{2}$ such that $\iota_{\bo{\sigma_{2}}} \circ \iota = \iota_{\bo{\sigma_{1}}}$. 

Sometimes when talking about an i-forest $\bo{\sigma}$ we may write
\begin{equ}\label{explicit i-forest}
\sigma = (F^{(1)},\widehat{F^{(1)}}) \cdots (F^{(k)},\widehat{F^{(k)}}).
\end{equ}
In this case it should be understood that for $1 \le j \le k$ the $F^{(j)}$ are explicit subforests of $\sT$ and there is no need to mention $\iota_{\bo{\sigma}}$. 
Note that there is some subtlety here since the $F^{(j)}$ may not be \emph{disjoint} subforests of $\sT$. 

We make decorations explicit in our notation, so writing $\bo{\sigma}$ indicates the constituent $\sigma$ is undecorated while in the decorated case we write $\bo{\sigma}^{\mfn}_{\mfe}$. Note that the decorations of an i-forest
are not necessarily inherited from $\sT^{\sn}_{\se}$. 

Recall that when we say $S$ is a subtree of $\sT$ we are being more explicit with how $S$ is identified as a sub-object of $\sT$, that is $N_{S} \subset N_{\sT}$, $N(S) \subset N(\sT)$, and $E_{S} \subset E_{\sT}$ as concrete sets with the type map on $E_{S}$ being the restriction of the one on $E_{\sT}$. For an i-tree $\bo{\sigma}$ we usually write $\sigma = (S,0)$ where $S$ is a subtree of $\sT$. 

The set of all decorated i-forests is denoted by $\ident[\For_{2}]$.
The subset of $\ident[\For_{2}]$ of decorated i-trees is denoted by $\ident[\Tr_{2}]$.
For $i \in \{0,1\}$ we define $\ident[\Tr_{i}]$ as the collection of all decorated i-trees $(T,\hat{T})^{\mfn,\mfo}_{\mfe} \in \Tr_{2}$ with $\hat{T}(\cdot) \le i$ and define $\ident[\For_{i}]$ analogously.
Given any collection of decorated i-forests $\mathfrak{G}$ we denote by $\mathring{\mathfrak{G}}$ the corresponding set of undecorated objects.

We denote by $\label{def: set of identified forests} \langle \ident[\For_{2}] \rangle$ the unital algebra obtained by endowing
free vector space generated by $\ident[\For_{2}]$ with the (linear extension of the) forest product. 
For any subset $A \subset \ident[\For_{2}]$ we denote by $\langle A \rangle_{\for}$  and $\langle A \rangle$ the subalgebra of $\langle \ident[\For_{2}] \rangle $ generated by $A$ and the vector subspace of $\langle \ident[\For_{2}] \rangle$ generated by $A$, respectively. 
\subsection{Co-actions and twisted antipodes}\label{subsec: coprods and twisted antipodes}
Our algebraic description of positive and negative renormalizations both involve two steps. In the first step one uses a ``co-action'' extracts object(s) which will be later assigned a counterterm and in the second step one builds the counter-term by using a ``twisted antipode'' to perform recursive renormalization procedure \emph{within} the extracted object(s).  

Each coaction produces a linear combination of tensor products -- in each product one factor consists of the extracted object(s) to be assigned counter-terms while the other factor consists of the part of $\sT$ which is left over.
On the other hand each twisted antipode will generate a forest product. 
\subsubsection{Negative renormalization}
First, we define a set of decorated i-trees which, due to their power-counting, should be assigned counterterms for negative renormalization. 
We set
\[
\mathfrak{X}_{-} 
\eqdef 
\left\{
(T,0)^{\mfn}_{\mfe} \in \ident[\Tr_{0}]:\ 
\mfn(\rho_{T}) = 0 
\textnormal{ and }
|(T,0)^{\mfn}_{\mfe}|_{-} < 0
\right\}\;,
\]
and we define $\mathfrak{p}_{-}\colon \langle \ident[\Tr_{0}] \rangle_{\for} \rightarrow \langle \mathfrak{X}_{-} \rangle_{\for}$ to be the algebra homorphism given by projection onto $\langle \mathfrak{X}_{-} \rangle_{\for}$.
The co-action describing the negative renormalization is a linear map
\[
\tDelta_{-}
\colon
\langle \ident[\Tr_{0}] \rangle
\rightarrow
\langle \mathfrak{X}_{-} \rangle_{\for}
\otimes 
\langle \ident[\Tr_{1}] \rangle\;.
\] 
We now specify, for each given i-tree, what sorts of subforests we will try to extract when applying $\tDelta_{-}$. 
\begin{definition}
For $(F,0) \in \mathring{\For}_{0}$ we define $\mathfrak{A}_{1}(F,0)$ to be the collection of all subforests $G$ of $F$ with the property that $\Con(G)$ contains no trivial trees.
\end{definition}
Note that one always has $\mathbf{1} \in \mathfrak{A}_{1}(F,0)$. Now, for any $(T,0)_{\mfe}^{\mfn} \in \ident[\Tr_{0}]$, we set 
\begin{equ}\label{def: negative coprod}
\tDelta_{-}(T,0)_{\mfe}^{\mfn}
\eqdef
\sum_{G \in \mathfrak{A}_{1}(T,0)}
\sum_{\mfn_{G},\mfe_{G}}
\frac{1}{\mfe_{G}!}
\binom{\mfn}{\mfn_{G}}
\mfp_{-}\Big[(G,0)_{\mfe}^{\mfn_{G} + \chi \mfe_{G}}
\Big]
\otimes
\big(T, 
[G]_{1} 
\big)^{\mfn - \mfn_{G},\mfn_{G} + \chi \mfe_{G}}_{\mfe_{G} + \mfe}\;.
\end{equ}
Here in the second sum above we are summing over (i) all $\mfn_{G}: N(T) \rightarrow \N^{d}$ supported on $N(G)$ and (ii) all $\mfe_{G}: E_{T} \rightarrow \N^{d}$ supported on
\[
\partial(G,T)
\eqdef
\{
(e_{\p},e_{\ch}) \in E_{T}\setminus E_{G}:\ e_{\p} \in N_{G}\}\;.
\]
Also, for forest $T$ and any decoration $\tilde{\mfe}:E_{T} \rightarrow \N^{d}$ we define $\chi\tilde{\mfe}: N(T) \rightarrow \N^{d}$ as 
$\chi\tilde\mfe(u)
\eqdef
\sum_{
e \,:\,
e_{\p} = u
}
\tilde{\mfe}(e)$.
\begin{remark} The appearance of formulas like \eqref{def: negative coprod} will be quite common -- we adopt these conventions on the meaning of $\sum_{\mfn_G, \mfe_{G}}$ throughout this paper. Note that they make sense even if $T$ is replaced by a forest and $G$ is a subforest of $F$.
\end{remark}
Recursive negative renormalization is described by an algebra homomorphism 
\[
\tpode_{-}:
\langle \mathfrak{X}_{-}
\rangle_{\for}
\rightarrow
\langle \ident[\Tr_{1}]
\rangle_{\for}\;.
\]
One can imagine $\tpode_{-}$ as given by iteration of a ``reduced'' analogue of $\hat{\Delta}_{-}$. 
By reduced, we mean it is more constrained in what objects it extracts.
\begin{definition}
For $(F,0) \in \mathring{\For}_{0}$ we define $\overline{\mathfrak{A}}_{1}(F,0)$ to be the collection of all $G \in \mathfrak{A}_{1}(F,0)$ with the property that for every $T \in \Con(F)$ one has $T \not \in \Con(G)$. 
\end{definition}
We first define the map $\tpode_{-}$ for any $(F,0)^{\mfn}_{\mfe} \in \For_{0}$ which can be written as a forest product of elements of $\mathfrak{X}_{-}$. 

This definition will be inductive in $|E_{F}|$ with the base case given by setting $\tpode_{-} \mathbf{1} = \mathbf{1}$. The induction is given by setting
\begin{equ}
\tpode_{-} 
(F,0)^{\mfn}_{\mfe}
\eqdef
(-1)^{|\Con(F)|}
\sum_{
\substack{
G \in \overline{\mathfrak{A}}_{1}(F,0)\\
\mfn_{G},\mfe_{G}
}
}
\frac{1}{\mfe_{G}!}
\binom{\mfn}{\mfn_{G}}
\Big[
\tpode_{-}
\mfp_{-}
(G,0)_{\mfe}^{\mfn_{G} + \chi \mfe_{G}}
\Big]
\cdot
\big(F, 
[G]_{1} 
\big)^{\mfn - \mfn_{G}}_{\mfe_{G} + \mfe}\;.
\end{equ} 
\subsubsection{Positive renormalization}\label{subsubsec: positive renorm}
In order to define $\mathfrak{X}_{+}$ we first need some new notions.
Given a subtree $T$ of $\sT$ and an edge $e \in E_{T}$ we define $T_{\ge}(e)$ to be the subtree of $T$ formed by all edges $e' \in E_{T}$ with $e' \ge e$, the corresponding node set being given by the collection of all $u \in N_{T}$ with $u \ge e_{\p}$. 
Note that if $e \in K(T)$ then $e_{p} \not \in L(T_{\ge}(e))$ (see Remark~\ref{disappearing noises}). 

Since it will be useful later we also define $T_{\not \ge }(e)$ to be the subtree of $T$ determined by all edges $e' \in E_{\sT}$ with $e' \not \ge e$, the corresponding node set is given by the collection of all $u \in N_{\sT}$ with $u \not > e_{\p}$. 

We draw a picture to make these definitions clear. On the leftmost picture we draw some subtree $T$ of $\sT$ and specify an edge $e \in E_{T}$ by drawing a small bisector in the middle of it.
In the middle picture we have shaded in the subtree $T_{\ge}(e)$ while on the rightmost picture we have shaded in the subtree $T_{\not \ge}(e)$. 
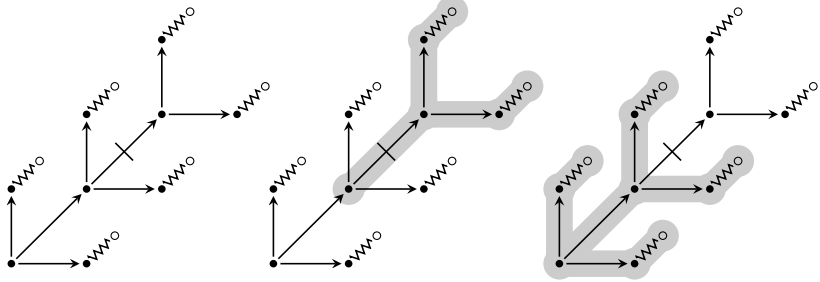
\[
\begin{tikzpicture}[scale=.7,baseline=0,rotate=-45]
    \node[dot] (root) at (0, 0) {};
    \node[dot] (bottom left noise) at (-1, 1) {};
    \node[circ] (bottom left anoise) at (-1, 1.75) {};
    \node[dot] (bottom right noise) at (1, 1) {};
    \node[circ] (bottom right anoise) at (1, 1.75) {};
    \node[dot] (middle) at (0, 2) {};
    \node[dot] (middle left noise) at (-1, 3) {};
    \node[circ] (middle left anoise) at (-1, 3.75) {};
    \node[dot] (top) at (0, 4) {};
    \node[dot] (middle right noise) at (1, 3) {};
    \node[circ] (middle right anoise) at (1, 3.75) {};
    \node[dot] (top left noise) at (-1, 5) {};
    \node[dot] (top right noise) at (1, 5) {};
    \node[circ] (top left anoise) at (-1, 5.75) {};
    \node[circ] (top right anoise) at (1, 5.75) {};

    \draw[kernel] (root) to (middle);
    \draw[kernel1] (middle) to (top);
    \draw[kernel] (root) to (bottom left noise);
    \draw[kernel] (root) to (bottom right noise);
    \draw[kernel] (middle) to (middle left noise);
    \draw[kernel] (middle) to (middle right noise);
    \draw[kernel] (top) to (top left noise) ;
    \draw[kernel] (top) to (top right noise);
    \draw[leaf] (bottom left noise) to (bottom left anoise);
    \draw[leaf] (bottom right noise) to (bottom right anoise);
    \draw[leaf] (middle left noise) to (middle left anoise);
    \draw[leaf] (middle right noise) to (middle right anoise);
    \draw[leaf] (top left noise) to (top left anoise);
    \draw[leaf] (top right noise) to (top right anoise);
\end{tikzpicture}
\begin{tikzpicture}[scale=.7,baseline=0,rotate=-45]
    \node[subtreenode] (subtree top) at (0, 4) {};
    \node[subtreenode] (subtree top left noise) at (-1, 5) {};
    \node[subtreenode] (subtree top right noise) at (1, 5) {};
    \node[subtreenode] (subtree top left anoise) at (-1, 5.75) {};
    \node[subtreenode] (subtree top right anoise) at (1, 5.75) {};
    \node[subtreenode] (subtree middle) at (0, 2) {};

    \draw[subtreeedge] (subtree top) to (subtree top left noise);
    \draw[subtreeedge] (subtree top) to (subtree top right noise);

    \draw[subtreeedge] (subtree top left noise) to (subtree top left anoise);
    \draw[subtreeedge] (subtree top right noise) to (subtree top right anoise);

    \draw[subtreeedge] (subtree middle) to (subtree top);

    \node[dot] (root) at (0, 0) {};
    \node[dot] (bottom left noise) at (-1, 1) {};
    \node[circ] (bottom left anoise) at (-1, 1.75) {};
    \node[dot] (bottom right noise) at (1, 1) {};
    \node[circ] (bottom right anoise) at (1, 1.75) {};
    \node[dot] (middle) at (0, 2) {};
    \node[dot] (middle left noise) at (-1, 3) {};
    \node[circ] (middle left anoise) at (-1, 3.75) {};
    \node[dot] (top) at (0, 4) {};
    \node[dot] (middle right noise) at (1, 3) {};
    \node[circ] (middle right anoise) at (1, 3.75) {};
    \node[dot] (top left noise) at (-1, 5) {};
    \node[dot] (top right noise) at (1, 5) {};
    \node[circ] (top left anoise) at (-1, 5.75) {};
    \node[circ] (top right anoise) at (1, 5.75) {};

    \draw[kernel] (root) to (middle);
    \draw[kernel1] (middle) to (top);
    \draw[kernel] (root) to (bottom left noise);
    \draw[kernel] (root) to (bottom right noise);
    \draw[kernel] (middle) to (middle left noise);
    \draw[kernel] (middle) to (middle right noise);
    \draw[kernel] (top) to (top left noise) ;
    \draw[kernel] (top) to (top right noise);
    \draw[leaf] (bottom left noise) to (bottom left anoise);
    \draw[leaf] (bottom right noise) to (bottom right anoise);
    \draw[leaf] (middle left noise) to (middle left anoise);
    \draw[leaf] (middle right noise) to (middle right anoise);
    \draw[leaf] (top left noise) to (top left anoise);
    \draw[leaf] (top right noise) to (top right anoise);
\end{tikzpicture} 
\begin{tikzpicture}[scale=.7,baseline=0,rotate=-45]

    \node[subtreenode] (subtree root) at (0,0) {};
    \node[subtreenode] (subtree bottom left noise) at (-1, 1) {};
    \node[subtreenode] (subtree bottom left anoise) at (-1, 1.75) {};
    \node[subtreenode] (subtree bottom right noise) at (1, 1) {};
    \node[subtreenode] (subtree bottom right anoise) at (1, 1.75) {};
    \node[subtreenode] (subtree middle) at (0, 2) {};
    \node[subtreenode] (subtree middle left noise) at (-1, 3) {};
    \node[subtreenode] (subtree middle left anoise) at (-1, 3.75) {};
    \node[subtreenode] (subtree middle right noise) at (1, 3) {};
    \node[subtreenode] (subtree middle right anoise) at (1, 3.75) {};

    \draw[subtreeedge] (subtree root) to (subtree middle);

    \draw[subtreeedge] (subtree root) to (subtree bottom left noise);
    \draw[subtreeedge] (subtree root) to (subtree bottom right noise);
 
    \draw[subtreeedge] (subtree bottom left noise) to (subtree bottom left anoise);
    \draw[subtreeedge] (subtree bottom right noise) to (subtree bottom right anoise);

    \draw[subtreeedge] (subtree middle) to (subtree middle left noise);
    \draw[subtreeedge] (subtree middle) to (subtree middle right noise);

    \draw[subtreeedge] (subtree middle left noise) to (subtree middle left anoise);
    \draw[subtreeedge] (subtree middle right noise) to (subtree middle right anoise);

    \node[dot] (root) at (0, 0) {};
    \node[dot] (bottom left noise) at (-1, 1) {};
    \node[circ] (bottom left anoise) at (-1, 1.75) {};
    \node[dot] (bottom right noise) at (1, 1) {};
    \node[circ] (bottom right anoise) at (1, 1.75) {};
    \node[dot] (middle) at (0, 2) {};
    \node[dot] (middle left noise) at (-1, 3) {};
    \node[circ] (middle left anoise) at (-1, 3.75) {};
    \node[dot] (top) at (0, 4) {};
    \node[dot] (middle right noise) at (1, 3) {};
    \node[circ] (middle right anoise) at (1, 3.75) {};
    \node[dot] (top left noise) at (-1, 5) {};
    \node[dot] (top right noise) at (1, 5) {};
    \node[circ] (top left anoise) at (-1, 5.75) {};
    \node[circ] (top right anoise) at (1, 5.75) {};

    \draw[kernel] (root) to (middle);
    \draw[kernel1] (middle) to (top);
    \draw[kernel] (root) to (bottom left noise);
    \draw[kernel] (root) to (bottom right noise);
    \draw[kernel] (middle) to (middle left noise);
    \draw[kernel] (middle) to (middle right noise);
    \draw[kernel] (top) to (top left noise) ;
    \draw[kernel] (top) to (top right noise);
    \draw[leaf] (bottom left noise) to (bottom left anoise);
    \draw[leaf] (bottom right noise) to (bottom right anoise);
    \draw[leaf] (middle left noise) to (middle left anoise);
    \draw[leaf] (middle right noise) to (middle right anoise);
    \draw[leaf] (top left noise) to (top left anoise);
    \draw[leaf] (top right noise) to (top right anoise);
\end{tikzpicture}
\]
\begin{definition} \label{def: dangling trees} Given a subtree $S$ of $\sT$ and a subtree $S'$ of $S$ 
we define $\Tr(S,S')$ to be the collection of subtrees of $S$ defined as follows
\[
\Tr(S,S')
\eqdef
\left\{ 
S_{\ge}(e):\ 
e \in K(S),\ 
e_{\p} \in N_{S'},\ e_{\ch} \not \in N_{S'}
\right\}\;.
\]
\end{definition}
We also define $\label{def: i-trees with 2 colored root}\ident[\widehat{\Tr}_{2}]
\eqdef
\{
(T,\hat{T})^{\mfn,\mfo}_{\mfe} \in \ident[\Tr_2]:\ 
\hat{T}_{2} \not = \mathbf{1}
\}$. 
Finally we set
\[
\mathfrak{X}_{+} 
\eqdef
\left\{
(T,\hat{T})^{\mfn,\mfo}_{\mfe} \in \ident[\widehat{\Tr}_{2}]:\ 
\forall S \in \Tr(T,\hat{T}_{2}),\ 
|\clearrootn(S,\hat{T})_{\mfe}^{\mfn,\mfo} |_{+} > 0
\right\}\;,
\]
where the map $\clearrootn: \ident[\For_{2}] \rightarrow \ident[\For_{2}]$ acts on decorated colored forests 
by setting their root $\mfn$ label to vanish, that is
\[
\clearrootn
(F,\hat{F})^{\mfn,\mfo}_{\mfe} 
\eqdef
(F,\hat{F})^{\tilde{\mfn},\mfo}_{\mfe}\;,\qquad 
\tilde{\mfn}(u) = \mfn(u) \mathbbm{1}\{ u \not \in \rho(F)\}\;.
\]
We define $\mathfrak{p}_{+}: \langle \ident[\widehat{\Tr}_{2}] \rangle_{\for} \rightarrow \langle \mathfrak{X}_{+} \rangle_{\for}$ to be the algebra homomorphism given by projection onto  $\langle \mathfrak{X}_{+} \rangle_{\for}$.
The co-action for positive renormalization is a linear map
\[
\tDelta_{+}
:\ 
\langle \ident[\Tr_{1}] \rangle
\rightarrow
\langle \ident[\Tr_{1}] \rangle
\otimes
\langle \mathfrak{X}_{+} \rangle\;.
\]
\begin{definition}
Given $(T,\hat{T}) \in \ident[\mathring{\Tr}_{2}]$ we define $\mathfrak{A}_{2}(T,\hat{T})$ to be the collection of all subtrees $S$ of $T$ with the following properties that $\rho_{S} = \rho_{T}$ and that for every $T' \in \Con(\hat{T}_{1})$ one has $T'$ is either a subtree of $S$ or disjoint with $S$.
\end{definition}
We then define, for any $(T,\hat{T})_{\mfe}^{\mfn,\mfo} \in \ident[\Tr_{1}]$,
\begin{equ}\label{def: positive coprod}
\tDelta_{+}(T,\hat{T})_{\mfe}^{\mfn,\mfo}
\eqdef
\sum_{
\substack{
S \in \mathfrak{A}_{2}(T,\hat{T})\\
\mfn_{S}, \mfe_{S}}
}
\frac{1}{\mfe_{S}!}
\binom{\mfn}{\mfn_{S}}
(S,\hat{T})_{\mfe}^{\mfn_{S} + \chi \mfe_{S}}
\otimes
\mathfrak{p}_{+}
\Big(T, 
[\hat{T}_{1} \setminus S]_{1} 
\sqcup 
[S]_{2}
\Big)^{\mfn - \mfn_{S},\mfo }_{\mfe_{S} + \mfe}\;.
\end{equ}
Recursive positive renormalization is similarly described by a linear map
\[
\tpode_{+}:
\langle \mathfrak{X}_{+}
\rangle
\rightarrow
\langle \ident[\widehat{\Tr}_{2}]
\rangle_{\for}\;.
\]
It suffices to define the map $\tpode_{+}$ for any $(T,\hat{T})^{\mfn,\mfo}_{\mfe} \in \mathfrak{X}_{+}$ and then extend linearly. 
This definition on $\mathfrak{X}_{+}$ will be inductive with respect to $|E_{T} \setminus  E_{\hat{T}_{2}}|$.
The base case, when $|E_{T} \setminus  E_{\hat{T}_{2}}| = 0$, is given by setting 
\[
\tpode_{+}(T,2)^{\mfn,\mfo}_{\mfe} \eqdef (-1)^{\sum_{u \in N(T)} |\mfn(u)|} (T,2)^{\mfn,0}_{\mfe}
.
\]
Before stating the inductive definition we need some more notation. 
\begin{definition}
Given $(T,\hat{T}) \in \underline{\mathring{\widehat{\Tr}}_{2}}$ we define $\overline{\mathfrak{A}}_{2}(T,\hat{T})$ to be the collection of all $S \in \mathfrak{A}_{2}(T,\hat{T})$ such that $\hat{T}_{2}$ is a subtree of $S$ and for every $T' \in \Tr(T,\hat{T}_{2})$ one has $E_{T'} \cap E_{S} \not = \emptyset$.
\end{definition}
\begin{definition}
Given $(T,\hat{T})^{\mfn,\mfo}_{\mfe} \in \ident[\widehat{\Tr}_{2}]$ we define $\mathfrak{E}[(T,\hat{T})^{\mfn,\mfo}_{\mfe}]$ to be the set of all edge decorations $\mfe_{T}$ on $K(\sT)$ supported on $\partial(\hat{T}_{2},T)$ such that $(T,\hat{T})^{\mfn,\mfo}_{\mfe + \mfe_{T}} \in \mathfrak{X}_{+}$. 
\end{definition} 
We then set, for any $(T,\hat{T})^{\mfn,\mfo}_{\mfe} \in \mathfrak{X}_{+}$ with $\hat{T} \not\equiv 2$, 
\begin{equs}
\tpode_{+} (T,\hat{T})^{\mfn,\mfo}_{\mfe}
\eqdef
&
(-1)^{|\Tr(T,\hat{T}_{2})|}
\sum_{S \in \overline{\mathfrak{A}}_{2}(T,\hat{T})}
\sum_{
\mfn_{S},\mfe_{S},\mff
}
\binom{\mfn - \hat{\mfn}}{\mfn_{S}}
\frac{(-1)^{\sum_{u \in N(T)} (|\hat{\mfn}(u)| + |\chi \mff(u)|)}}{(\mff + \mfe_{S})!}
\\
&
\quad
\cdot
(S,\hat{T})_{\mfe + \mff}^{\mfn_{S} + \hat{\mfn} + \chi(\mfe_{S} + \mff),0}
\cdot
\tpode_{+}
\mfp_{+}
\Big(T, 
[\hat{T}_{1} \setminus S]_{1} 
\sqcup 
[S]_{2}
\Big)^{\mfn - \mfn_{S} - \hat{\mfn},\mfo}_{\mfe_{S} + \mfe},
\end{equs}
where $\hat{\mfn}(u) \eqdef \mfn(u) \mathbbm{1}\{u \in N(\hat{T}_{2})\}$ and the sum over $\mff$ is a sum over all edge decorations $\mff \in \mathfrak{E}[(T,\hat{T})^{\mfn,\mfo}_{\mfe}]$.

Our objective in this section is to describe two ways to map the decorated i-forests of the previous section to space-time functions and random fields.

\subsection{From combinatorial trees to space-time functions}\label{subsec: mapping to analytic objects}

In this subsection we describe how, for each noise $\xi \in \Omega_{\infty}$, to map i-forests to functions of the space-time variables $(x_{v})_{v \in \allnodes}$ where we set $\allnodes \eqdef N(\sT) \sqcup \{\logof\}$.
Here $\label{def: basepoint vertex} \logof$ is a new node element and the corresponding $x_{\logof}$ is the variable that encodes the base-point (the $z$ of $\Pi_{z}$). 
For any finite set $A$ we write $\label{def: smooth functions 2}\mcb{C}_{A}$ for the algebra of smooth functions $\mcb{C}^\infty((\R^{d})^{A})$ and we also use the shorthand $\label{def: all smooth functions}\allf \eqdef \mcb{C}_{\allnodes}$.

Throughout this article, 
given some finite set $A$, some $x \in (\R^d)^A$, and some
$B \subset A$, we write $x_B$ for the element of $(\R^d)^B$ with $(x_B)_i = x_i$ for $i \in B$.
Similarly, we will always use the shorthands $\int dy_B$ or $\int_B dy$ instead of 
writing $\int_{(\R^d)^B} dy$.
Finally, for $B \subset A$, we will always view $\mcb{C}_B$ 
as a subspace of $\mcb{C}_A$ via the canonical injection
$\iota$ given by $(\iota f)(x) = f(x_B)$.
For two disjoint sets $A$ and $\bar A$ and elements $x \in (\R^d)^A$ and $\bar x \in (\R^d)^{\bar A}$,
we also write $x \sqcup \bar x$ for the element $y \in (\R^d)^{A \sqcup \bar A}$ given by
$y_i = x_i$ for $i \in A$ and $y_i = \bar x_i$ for $i \in \bar A$.
Sometimes, we will make the abuse of notation of writing $x_u$ instead of $x_{\{u\}}$
in expressions like ``$x_u \sqcup y_A$''.

First, for any i-tree $(T,\hat{T})$ we make the following definitions. We define set $K(T,\hat{T}) \eqdef K(T) \cap \hat{T}^{-1}(0)$ and $L(T,\hat{T}) \eqdef L(T) \cap \hat{T}^{-1}(0)$. 
We set
\[
N(T,\hat{T}) 
\eqdef
N(T) \cap
\big(
\hat{T}^{-1}(0)
\sqcup
\rho(\hat{T}_{1})
\big).
\]
We also define an associated map $\gv: N(T) \rightarrow N(T,\hat{T}) \sqcup \{\logof\}$ by setting, for each $u \in N(T)$, 
\[
\gv(u)
\eqdef
\begin{cases}
u&
\textnormal{ if } u \in N(T,\hat{T}),\\
\rho_{S}&
\textnormal{ if } u \in N(S) \textnormal{ for } S \in \Con(\hat{T}_{1}),\\
\logof & 
\textnormal{ if } u \in N(\hat{T}_{2}).
\end{cases}
\]
For any $\xi \in \Omega_{\infty}$ and decorated i-tree $(T,\hat{T})^{\mfn,\mfo}_{\mfe} \in \underline{\Tr_{2}}$ we define $\mathring{\Upsilon}^{\xi}[(T,\hat{T})^{\mfn,\mfo}_{\mfe}] \in \allf$ by 
\begin{equs}
\mathring{\Upsilon}^{\xi}[(T,\hat{T})^{\mfn,\mfo}_{\mfe}](x)
&\eqdef
\Big( 
\prod_{u \in N(T)}
(x_{\gv(u)})^{\mfn(u)}
\Big)
\Big(
\prod_{
e \in K(T,\hat{T})
}
D^{\mfe(e)}K_{\mft(e)}(x_{\gv(e_{\p})} - x_{\gv(e_{\ch})})
\Big)\\
&
\qquad
\qquad
\cdot
\Big(
\prod_{v
\in 
L(T,\hat{T})
}
\xi_{\mft(v)}(x_{v})
\Big)\;.
\end{equs} 
We then define an algebra homomorphism $\Upsilon^{\xi}: \langle \ident[\For_{2}] \rangle \rightarrow \allf$ by setting, for $(T,\hat{T})^{\mfn,\mfo}_{\mfe} \in \underline{\Tr_{2}}$, 
\[
\Upsilon^{\xi}[(T,\hat{T})^{\mfn,\mfo}_{\mfe}](x)
\eqdef
\int_{N(T,\hat{T})} \back\mathrm{d}y\ 
\mathring{\Upsilon}^{\xi}[(T,\hat{T})^{\mfn,\mfo}_{\mfe}](y \sqcup x_{\rho_{T}} \sqcup x_{\logof})\;,
\]
and then extending multipliciatively and linearly to all of $\Upsilon^{\xi}$ (in particular, $\Upsilon^{\xi}[\mathbf{1}] = 1$). 
We also make the important observation that $\Upsilon^{\xi}[(T,\hat{T})^{\mfn,\mfo}_{\mfe}](x)$ depends only 
on \emph{either} $x_{\rho_{T}}$ \emph{or} $x_{\logof}$ the former case holds when $\hat{T}_{2} = \mathbf{1}$
and the later when $\hat{T}_{2} \not = \mathbf{1}$.
\subsection{The BPHZ renormalized tree}\label{subsec: BPHZ renormalized tree}
We now make our setting stochastic by fixing for what follows an arbitrary random noise $\xi \in \CM(\Omega_\infty)$.

We define an algebra homomorphism $\bar{\Upsilon}^{\xi}: \langle \For_{1} \rangle_{\for} \rightarrow \R$ by setting, for any $\bo{\sigma}^{\mfn,\mfo}_{\mfe} \in \ident[\Tr_{1}]$,
\begin{equation}\label{def: c}
\bar{\Upsilon}^{\xi} [\bo{\sigma}^{\mfn,\mfo}_{\mfe}] 
\eqdef
\E (\Upsilon^{\xi} \bo{\sigma}^{\mfn,\mfo}_{\mfe})(0)\;,
\end{equation}
and extending this definition multiplicatively and linearly.
\begin{remark}
It is important to note here that $\bar{\Upsilon}^{\xi}$ is deterministic and depends solely on the 
\textit{law} of the random variable $\xi$, not on any random sample drawn from it.
\end{remark}
Finally, we define, for each realization $\xi(\omega) \in \Omega_{\infty}$ of $\xi$, the smooth function $\hat{\Upsilon}^{\xi}[\sT^{\sn}_{\se}] \in \allf$, again depending only on $x_{\rho_{\sT}}$ and $x_{\logof}$, by setting
\begin{equation}\label{def: BPHZ renormalized tree}
\hat{\Upsilon}^{\xi(\omega)}\big[\sT^{\sn}_{\se}\big]
\eqdef
\Big(
\bar{\Upsilon}^{\xi}[\tpode_{-} \cdot ]
\otimes 
\Upsilon^{\xi(\omega)}[\cdot]
\otimes 
\Upsilon^{\xi(\omega)}[\tpode_{+} \cdot]
\Big)
(\Id \otimes \tDelta_{+} )
\tDelta_{-} (\sT,0)^{\sn,0}_{\se} \;.
\end{equation}
Above we are committing an abuse of notation, the RHS is technically a sum of triple tensor products each consisting of a scalar and two elements of $\allf$, however by simply multiplying them together we view the whole sum as an element of $\allf$. 

We emphasize again that in this expression 
$\bar{\Upsilon}^{\xi}$ is \textit{deterministic} and depends on the law of the random variable $\xi$,
while $\Upsilon^{\xi(\omega)}$ is \textit{random} and depends on the specific sample $\xi(\omega)$.
Later on, we will however typically suppress the chance element $\omega$ in our notations
and we will simply write $\hat{\Upsilon}^{\xi}[\sT^{\sn}_{\se}](x)$ for the above expression.
We will also view this as a family of random distributions in $\mcD'(\R^{d})$, using the shorthand, for each $z \in \R^{d}$ and $\phi \in \mcD(\R^{d})$,
\begin{equ}
\hat{\Upsilon}^{\xi}_{z}[\sT^{\sn}_{\se}](\phi) \eqdef \int_{\{\mainroot,\logof\}} \back dx\  \hat{\Upsilon}^{\xi}[\sT^{\sn}_{\se}](x)\,\phi(x_{\mainroot}) \delta(x_{\logof} -z)\;,
\end{equ}
where here, and in what follows, we simply write $\mainroot$ instead of $\rho_{\sT}$ for the root of $\sT$.
We refer to the above random distribution as the 
BPHZ renormalized tree corresponding to $\sT^{\sn}_{\se}$ with basepoint $z$.




\section{A forest and cut expansion for the BPHZ renormalized tree}\label{sec: derivation of formula} 
The goal of this section is to rewrite the opaque formula for $\hat{\Upsilon}^{\xi(\omega)}[\sT^{\sn}_{\se}]$ given in the previous section in a more explicit form amenable to direct analysis. 
The fundamental result of this section is Lemma~\ref{explicit formula}. 


\subsection{A forest formula for negative renormalizations}
For most of this subsection we work with a set of subtrees of $\sT$. We see this set as carrying the inclusion partial order (for which throughout this article we use the convention  $\subset\ \longleftrightarrow\ \le$). 
In general, for any poset $A$ and subset $\mathcal{A} \subset A$ we write $\Max(\mathcal{A})$ or $\mmax{\mathcal{A}}$ for the set of maximal elements of $\mathcal{A}$ and $\Min(\mathcal{A})$ or $\mmin{\mathcal{A}}$ for the set of minimal elements of $\mathcal{A}$.\label{posets} 
As an exception to this notation, $\sT$ always refers to our fixed tree.

We write $\Div$\label{Div} for the set of all subtrees $S$ of $\sT$ with $\omega(S) \eqdef - |(S^{0}_{\se},0)|_{\s} > 0$. 
Note that $S \in \Div$ forces $S$ to be non-trivial.
\begin{remark}\label{divergent trees come with all leaves}
Note that as a consequence of Definition~\ref{def: strong subcriticality}, if one has $S \in \Div$ and $e \in \mathbb{L}(\sT)$ with $e_{\p} \in N(S)$ then it must be the case that $e \in \mathbb{L}(S)$.
\end{remark}
\begin{definition}\label{forest of subtrees}
We say that $\mcF$ is a forest of subtrees if it is a collection of subtrees of $\sT$ with the 
property that for any pair $S, S' \in \mcF$ one has either $N_{S} \subset N_{S'}$, or 
$N_{S'} \subset N_{S}$, or $N_{S} \cap N_{S'} = \emptyset$. 
\end{definition}
As an example, suppose that $\sT$ is given by
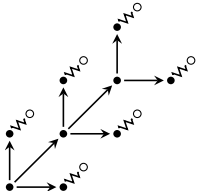
\[
\begin{tikzpicture}[scale=.5,rotate=-45]
    \node[dot] (root) at (0, 0) {};
    \node[dot] (bottom left noise) at (-1, 1) {};
    \node[circ] (bottom left anoise) at (-1, 1.75) {};
    \node[dot] (bottom right noise) at (1, 1) {};
    \node[circ] (bottom right anoise) at (1, 1.75) {};
    \node[dot] (middle) at (0, 2) {};
    \node[dot] (middle left noise) at (-1, 3) {};
    \node[circ] (middle left anoise) at (-1, 3.75) {};
    \node[dot] (top) at (0, 4) {};
    \node[dot] (middle right noise) at (1, 3) {};
    \node[circ] (middle right anoise) at (1, 3.75) {};
    \node[dot] (top left noise) at (-1, 5) {};
    \node[dot] (top right noise) at (1, 5) {};
    \node[circ] (top left anoise) at (-1, 5.75) {};
    \node[circ] (top right anoise) at (1, 5.75) {};

    \draw[kernel] (root) to (middle);
    \draw[kernel] (middle) to (top);
    \draw[kernel] (root) to (bottom left noise);
    \draw[kernel] (root) to (bottom right noise);
    \draw[kernel] (middle) to (middle left noise);
    \draw[kernel] (middle) to (middle right noise);
    \draw[kernel] (top) to (top left noise) ;
    \draw[kernel] (top) to (top right noise);
    \draw[leaf] (bottom left noise) to (bottom left anoise);
    \draw[leaf] (bottom right noise) to (bottom right anoise);
    \draw[leaf] (middle left noise) to (middle left anoise);
    \draw[leaf] (middle right noise) to (middle right anoise);
    \draw[leaf] (top left noise) to (top left anoise);
    \draw[leaf] (top right noise) to (top right anoise);
\end{tikzpicture}
\]
with the root being at the bottom.
Indicating subtrees of a tree by shading them, we consider the set $\{S_{1},\ldots,S_{6}\}$ of subtrees of $\sT$ given by
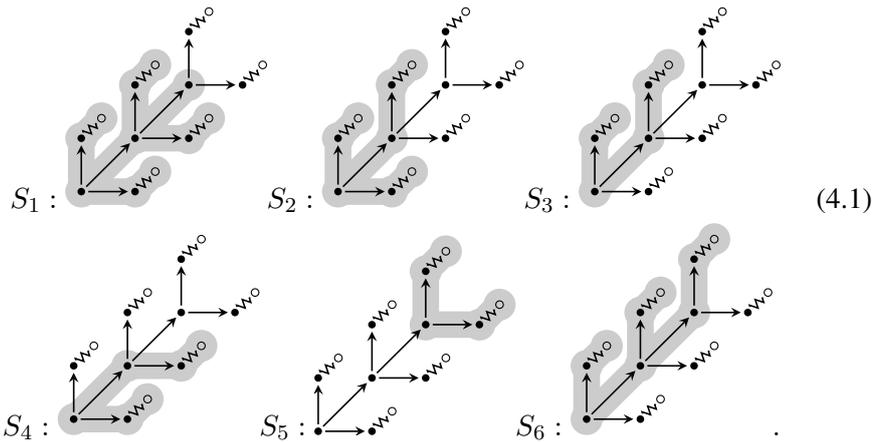
\begin{equs}\label{picture: examples of subtrees}
S_{1}:
\begin{tikzpicture}[scale=.5,rotate=-45]
    \node[subtreenode] (subtree root) at (0,0) {};
    \node[subtreenode] (subtree bottom left noise) at (-1, 1) {};
    \node[subtreenode] (subtree bottom left anoise) at (-1, 1.75) {};
    \node[subtreenode] (subtree bottom right noise) at (1, 1) {};
    \node[subtreenode] (subtree bottom right anoise) at (1, 1.75) {};
    \node[subtreenode] (subtree middle) at (0, 2) {};
    \node[subtreenode] (subtree middle left noise) at (-1, 3) {};
    \node[subtreenode] (subtree middle left anoise) at (-1, 3.75) {};
    \node[subtreenode] (subtree middle right noise) at (1, 3) {};
    \node[subtreenode] (subtree middle right anoise) at (1, 3.75) {};
    \node[subtreenode] (subtree top) at (0, 4) {};

    \draw[subtreeedge] (subtree root) to (subtree middle);
    \draw[subtreeedge] (subtree middle) to (subtree top);

    \draw[subtreeedge] (subtree root) to (subtree bottom left noise);
    \draw[subtreeedge] (subtree root) to (subtree bottom right noise);
    \draw[subtreeedge] (subtree bottom left noise) to (subtree bottom left anoise);
    \draw[subtreeedge] (subtree bottom right noise) to (subtree bottom right anoise);

    \draw[subtreeedge] (subtree middle) to (subtree middle left noise);
    \draw[subtreeedge] (subtree middle) to (subtree middle right noise);
    \draw[subtreeedge] (subtree middle left noise) to (subtree middle left anoise);
    \draw[subtreeedge] (subtree middle right noise) to (subtree middle right anoise);

    \node[dot] (root) at (0, 0) {};
    \node[dot] (bottom left noise) at (-1, 1) {};
    \node[circ] (bottom left anoise) at (-1, 1.75) {};
    \node[dot] (bottom right noise) at (1, 1) {};
    \node[circ] (bottom right anoise) at (1, 1.75) {};
    \node[dot] (middle) at (0, 2) {};
    \node[dot] (middle left noise) at (-1, 3) {};
    \node[circ] (middle left anoise) at (-1, 3.75) {};
    \node[dot] (top) at (0, 4) {};
    \node[dot] (middle right noise) at (1, 3) {};
    \node[circ] (middle right anoise) at (1, 3.75) {};
    \node[dot] (top left noise) at (-1, 5) {};
    \node[dot] (top right noise) at (1, 5) {};
    \node[circ] (top left anoise) at (-1, 5.75) {};
    \node[circ] (top right anoise) at (1, 5.75) {};

    \draw[kernel] (root) to (middle);
    \draw[kernel] (middle) to (top);
    \draw[kernel] (root) to (bottom left noise);
    \draw[kernel] (root) to (bottom right noise);
    \draw[kernel] (middle) to (middle left noise);
    \draw[kernel] (middle) to (middle right noise);
    \draw[kernel] (top) to (top left noise) ;
    \draw[kernel] (top) to (top right noise);
    \draw[leaf] (bottom left noise) to (bottom left anoise);
      \draw[leaf] (bottom right noise) to (bottom right anoise);
      \draw[leaf] (middle left noise) to (middle left anoise);
      \draw[leaf] (middle right noise) to (middle right anoise);
      \draw[leaf] (top left noise) to (top left anoise);
      \draw[leaf] (top right noise) to (top right anoise);

\end{tikzpicture}
S_{2}:
\begin{tikzpicture}[scale=.5,rotate=-45]
    \node[subtreenode] (subtree root) at (0,0) {};
    \node[subtreenode] (subtree bottom left noise) at (-1, 1) {};
    \node[subtreenode] (subtree bottom left anoise) at (-1, 1.75) {};
    \node[subtreenode] (subtree bottom right noise) at (1, 1) {};
    \node[subtreenode] (subtree bottom right anoise) at (1, 1.75) {};
    \node[subtreenode] (subtree middle) at (0, 2) {};
    \node[subtreenode] (subtree middle left noise) at (-1, 3) {};
    \node[subtreenode] (subtree middle left anoise) at (-1, 3.75) {};

    \draw[subtreeedge] (subtree root) to (subtree middle);

    \draw[subtreeedge] (subtree root) to (subtree bottom left noise);
    \draw[subtreeedge] (subtree root) to (subtree bottom right noise);

    \draw[subtreeedge] (subtree bottom left noise) to (subtree bottom left anoise);
    \draw[subtreeedge] (subtree bottom right noise) to (subtree bottom right anoise);

    \draw[subtreeedge] (subtree middle) to (subtree middle left noise);
    \draw[subtreeedge] (subtree middle left noise) to (subtree middle left anoise);

    \node[dot] (root) at (0, 0) {};
    \node[dot] (bottom left noise) at (-1, 1) {};
    \node[circ] (bottom left anoise) at (-1, 1.75) {};
    \node[dot] (bottom right noise) at (1, 1) {};
    \node[circ] (bottom right anoise) at (1, 1.75) {};
    \node[dot] (middle) at (0, 2) {};
    \node[dot] (middle left noise) at (-1, 3) {};
    \node[circ] (middle left anoise) at (-1, 3.75) {};
    \node[dot] (top) at (0, 4) {};
    \node[dot] (middle right noise) at (1, 3) {};
    \node[circ] (middle right anoise) at (1, 3.75) {};
    \node[dot] (top left noise) at (-1, 5) {};
    \node[dot] (top right noise) at (1, 5) {};
    \node[circ] (top left anoise) at (-1, 5.75) {};
    \node[circ] (top right anoise) at (1, 5.75) {};

    \draw[kernel] (root) to (middle);
    \draw[kernel] (middle) to (top);
    \draw[kernel] (root) to (bottom left noise);
    \draw[kernel] (root) to (bottom right noise);
    \draw[kernel] (middle) to (middle left noise);
    \draw[kernel] (middle) to (middle right noise);
    \draw[kernel] (top) to (top left noise) ;
    \draw[kernel] (top) to (top right noise);
    \draw[leaf] (bottom left noise) to (bottom left anoise);
      \draw[leaf] (bottom right noise) to (bottom right anoise);
      \draw[leaf] (middle left noise) to (middle left anoise);
      \draw[leaf] (middle right noise) to (middle right anoise);
      \draw[leaf] (top left noise) to (top left anoise);
      \draw[leaf] (top right noise) to (top right anoise);

\end{tikzpicture}
S_{3}:
\begin{tikzpicture}[scale=.5,rotate=-45]
    \node[subtreenode] (subtree root) at (0,0) {};
    \node[subtreenode] (subtree bottom left noise) at (-1, 1) {};
    \node[subtreenode] (subtree bottom left anoise) at (-1, 1.75) {};
    \node[subtreenode] (subtree middle) at (0, 2) {};
    \node[subtreenode] (subtree middle left noise) at (-1, 3) {};
    \node[subtreenode] (subtree middle left anoise) at (-1, 3.75) {};

    \draw[subtreeedge] (subtree root) to (subtree middle);

    \draw[subtreeedge] (subtree root) to (subtree bottom left noise);
 
    \draw[subtreeedge] (subtree bottom left noise) to (subtree bottom left anoise);

    \draw[subtreeedge] (subtree middle) to (subtree middle left noise);

    \draw[subtreeedge] (subtree middle left noise) to (subtree middle left anoise);

    \node[dot] (root) at (0, 0) {};
    \node[dot] (bottom left noise) at (-1, 1) {};
    \node[circ] (bottom left anoise) at (-1, 1.75) {};
    \node[dot] (bottom right noise) at (1, 1) {};
    \node[circ] (bottom right anoise) at (1, 1.75) {};
    \node[dot] (middle) at (0, 2) {};
    \node[dot] (middle left noise) at (-1, 3) {};
    \node[circ] (middle left anoise) at (-1, 3.75) {};
    \node[dot] (top) at (0, 4) {};
    \node[dot] (middle right noise) at (1, 3) {};
    \node[circ] (middle right anoise) at (1, 3.75) {};
    \node[dot] (top left noise) at (-1, 5) {};
    \node[dot] (top right noise) at (1, 5) {};
    \node[circ] (top left anoise) at (-1, 5.75) {};
    \node[circ] (top right anoise) at (1, 5.75) {};

    \draw[kernel] (root) to (middle);
    \draw[kernel] (middle) to (top);
    \draw[kernel] (root) to (bottom left noise);
    \draw[kernel] (root) to (bottom right noise);
    \draw[kernel] (middle) to (middle left noise);
    \draw[kernel] (middle) to (middle right noise);
    \draw[kernel] (top) to (top left noise) ;
    \draw[kernel] (top) to (top right noise);
    \draw[leaf] (bottom left noise) to (bottom left anoise);
      \draw[leaf] (bottom right noise) to (bottom right anoise);
      \draw[leaf] (middle left noise) to (middle left anoise);
      \draw[leaf] (middle right noise) to (middle right anoise);
      \draw[leaf] (top left noise) to (top left anoise);
      \draw[leaf] (top right noise) to (top right anoise);

\end{tikzpicture}\\
S_{4}:
\begin{tikzpicture}[scale=.5,rotate=-45]
    \node[subtreenode] (subtree root) at (0,0) {};
    \node[subtreenode] (subtree bottom right noise) at (1, 1) {};
    \node[subtreenode] (subtree bottom right anoise) at (1, 1.75) {};
    \node[subtreenode] (subtree middle) at (0, 2) {};
    \node[subtreenode] (subtree middle right noise) at (1, 3) {};
    \node[subtreenode] (subtree middle right anoise) at (1, 3.75) {};

    \draw[subtreeedge] (subtree root) to (subtree middle);

    \draw[subtreeedge] (subtree root) to (subtree bottom right noise);
 
    \draw[subtreeedge] (subtree bottom right noise) to (subtree bottom right anoise);

    \draw[subtreeedge] (subtree middle) to (subtree middle right noise);

    \draw[subtreeedge] (subtree middle right noise) to (subtree middle right anoise);

    \node[dot] (root) at (0, 0) {};
    \node[dot] (bottom left noise) at (-1, 1) {};
    \node[circ] (bottom left anoise) at (-1, 1.75) {};
    \node[dot] (bottom right noise) at (1, 1) {};
    \node[circ] (bottom right anoise) at (1, 1.75) {};
    \node[dot] (middle) at (0, 2) {};
    \node[dot] (middle left noise) at (-1, 3) {};
    \node[circ] (middle left anoise) at (-1, 3.75) {};
    \node[dot] (top) at (0, 4) {};
    \node[dot] (middle right noise) at (1, 3) {};
    \node[circ] (middle right anoise) at (1, 3.75) {};
    \node[dot] (top left noise) at (-1, 5) {};
    \node[dot] (top right noise) at (1, 5) {};
    \node[circ] (top left anoise) at (-1, 5.75) {};
    \node[circ] (top right anoise) at (1, 5.75) {};

    \draw[kernel] (root) to (middle);
    \draw[kernel] (middle) to (top);
    \draw[kernel] (root) to (bottom left noise);
    \draw[kernel] (root) to (bottom right noise);
    \draw[kernel] (middle) to (middle left noise);
    \draw[kernel] (middle) to (middle right noise);
    \draw[kernel] (top) to (top left noise) ;
    \draw[kernel] (top) to (top right noise);
    \draw[leaf] (bottom left noise) to (bottom left anoise);
      \draw[leaf] (bottom right noise) to (bottom right anoise);
      \draw[leaf] (middle left noise) to (middle left anoise);
      \draw[leaf] (middle right noise) to (middle right anoise);
      \draw[leaf] (top left noise) to (top left anoise);
      \draw[leaf] (top right noise) to (top right anoise);

\end{tikzpicture}
S_{5}:
\begin{tikzpicture}[scale=.5,rotate=-45]

    \node[subtreenode] (subtree top) at (0, 4) {};
    \node[subtreenode] (subtree top left noise) at (-1, 5) {};
    \node[subtreenode] (subtree top right noise) at (1, 5) {};
    \node[subtreenode] (subtree top left anoise) at (-1, 5.75) {};
    \node[subtreenode] (subtree top right anoise) at (1, 5.75) {};

    \draw[subtreeedge] (subtree top) to (subtree top left noise);
    \draw[subtreeedge] (subtree top) to (subtree top right noise);

    \draw[subtreeedge] (subtree top left noise) to (subtree top left anoise);
    \draw[subtreeedge] (subtree top right noise) to (subtree top right anoise);

    \node[dot] (root) at (0, 0) {};
    \node[dot] (bottom left noise) at (-1, 1) {};
    \node[circ] (bottom left anoise) at (-1, 1.75) {};
    \node[dot] (bottom right noise) at (1, 1) {};
    \node[circ] (bottom right anoise) at (1, 1.75) {};
    \node[dot] (middle) at (0, 2) {};
    \node[dot] (middle left noise) at (-1, 3) {};
    \node[circ] (middle left anoise) at (-1, 3.75) {};
    \node[dot] (top) at (0, 4) {};
    \node[dot] (middle right noise) at (1, 3) {};
    \node[circ] (middle right anoise) at (1, 3.75) {};
    \node[dot] (top left noise) at (-1, 5) {};
    \node[dot] (top right noise) at (1, 5) {};
    \node[circ] (top left anoise) at (-1, 5.75) {};
    \node[circ] (top right anoise) at (1, 5.75) {};

    \draw[kernel] (root) to (middle);
    \draw[kernel] (middle) to (top);
    \draw[kernel] (root) to (bottom left noise);
    \draw[kernel] (root) to (bottom right noise);
    \draw[kernel] (middle) to (middle left noise);
    \draw[kernel] (middle) to (middle right noise);
    \draw[kernel] (top) to (top left noise) ;
    \draw[kernel] (top) to (top right noise);
    \draw[leaf] (bottom left noise) to (bottom left anoise);
      \draw[leaf] (bottom right noise) to (bottom right anoise);
      \draw[leaf] (middle left noise) to (middle left anoise);
      \draw[leaf] (middle right noise) to (middle right anoise);
      \draw[leaf] (top left noise) to (top left anoise);
      \draw[leaf] (top right noise) to (top right anoise);

\end{tikzpicture}
S_{6}:
\begin{tikzpicture}[scale=.5,rotate=-45]
    \node[subtreenode] (subtree root) at (0,0) {};
    \node[subtreenode] (subtree bottom left noise) at (-1, 1) {};
    \node[subtreenode] (subtree bottom left anoise) at (-1, 1.75) {};
    \node[subtreenode] (subtree middle) at (0, 2) {};
    \node[subtreenode] (subtree middle left noise) at (-1, 3) {};
    \node[subtreenode] (subtree middle left anoise) at (-1, 3.75) {};
    \node[subtreenode] (subtree top) at (0, 4) {};
    \node[subtreenode] (subtree top left noise) at (-1, 5) {};
    \node[subtreenode] (subtree top left anoise) at (-1, 5.75) {};

    \draw[subtreeedge] (subtree root) to (subtree middle);
    \draw[subtreeedge] (subtree middle) to (subtree top);

    \draw[subtreeedge] (subtree root) to (subtree bottom left noise);
    \draw[subtreeedge] (subtree bottom left noise) to (subtree bottom left anoise);
    \draw[subtreeedge] (subtree middle) to (subtree middle left noise);
    \draw[subtreeedge] (subtree middle left noise) to (subtree middle left anoise);
    \draw[subtreeedge] (subtree top) to (subtree top left noise);
    \draw[subtreeedge] (subtree top left noise) to (subtree top left anoise);

    \node[dot] (root) at (0, 0) {};
    \node[dot] (bottom left noise) at (-1, 1) {};
    \node[circ] (bottom left anoise) at (-1, 1.75) {};
    \node[dot] (bottom right noise) at (1, 1) {};
    \node[circ] (bottom right anoise) at (1, 1.75) {};
    \node[dot] (middle) at (0, 2) {};
    \node[dot] (middle left noise) at (-1, 3) {};
    \node[circ] (middle left anoise) at (-1, 3.75) {};
    \node[dot] (top) at (0, 4) {};
    \node[dot] (middle right noise) at (1, 3) {};
    \node[circ] (middle right anoise) at (1, 3.75) {};
    \node[dot] (top left noise) at (-1, 5) {};
    \node[dot] (top right noise) at (1, 5) {};
    \node[circ] (top left anoise) at (-1, 5.75) {};
    \node[circ] (top right anoise) at (1, 5.75) {};

    \draw[kernel] (root) to (middle);
    \draw[kernel] (middle) to (top);
    \draw[kernel] (root) to (bottom left noise);
    \draw[kernel] (root) to (bottom right noise);
    \draw[kernel] (middle) to (middle left noise);
    \draw[kernel] (middle) to (middle right noise);
    \draw[kernel] (top) to (top left noise) ;
    \draw[kernel] (top) to (top right noise);
    \draw[leaf] (bottom left noise) to (bottom left anoise);
      \draw[leaf] (bottom right noise) to (bottom right anoise);
      \draw[leaf] (middle left noise) to (middle left anoise);
      \draw[leaf] (middle right noise) to (middle right anoise);
      \draw[leaf] (top left noise) to (top left anoise);
      \draw[leaf] (top right noise) to (top right anoise);
\end{tikzpicture}.
\end{equs}
Note here that while $S_3$ and $S_4$ are isomorphic as labelled trees, they represent different
subtrees of $\sT$ and are therefore not isomorphic as i-trees.
The list of all subsets of $\{S_{1},\dots,S_{6}\}$ which are forests of subtrees is given by:
\begin{equs}[0]
\emptyset,\ \{S_{1}\}, \{S_{2}\},\{S_{3}\}, \{S_{4}\}, \{S_{5}\}, \{S_{6}\},\\
\{S_{3},S_{6}\}, \{S_{2},S_{3}\},\ \{S_{1},S_{2}\},\ \{S_{1},S_{3}\},\ \{S_{1},S_{4}\},\ \{S_{2},S_{5}\},\\
\{S_{3},S_{5}\},\ \{S_{4},S_{5}\},\ \{S_{3},S_{6}\},\ \{S_{1},S_{2},S_{3}\},\textnormal{ and } \{S_{2},S_{3},S_{5}\}\;.
\end{equs} 
\begin{remark}\label{rem:forests}
We will use two notions of forest in what follows, one is that of an i-forest 
and the other is the notion of a forest of subtrees. Clearly any collection $\mcA$ of subtrees of $\sT$ canonically determines an i-forest where, instead of considering $\mcA$ as a set with trees as elements, 
we see it as the i-forest obtained by the disjoint union of these trees.  

In the above example $\{S_2,S_3,S_4\}$ can be viewed as an i-forest (and in this case we would write
it as $S_2\cdot S_3 \cdot S_4$), but it is
not a forest of subtrees since $S_4$ is not related to $S_2$ (or $S_3$ for that matter)
by inclusion. This illustrates that, with this identification, being a forest of subtrees 
is a strictly stronger property than being an i-forest.

We often use a calligraphic font to indicate forests of subtrees versus standard Roman letters for i-forests.  
Finally, we observe that in the special case when $\mcF$ is a forest of subtrees of $\sT$ of depth $1$ or $0$, the node set $N_{F}$ of the canonical i-forest $F$ associated to it can be identified as a subset of $N_{\sT}$. 
\end{remark} 
\begin{definition}
We say that a subtree $S$ is \emph{compatible} with a forest of subtrees $\mcF$ if $\{S \} \cup \mcF$ 
is again a forest of subtrees (note that if $S \in \mcF$, then this is automatically the case).
\end{definition}
We write $\mathbb{F}$\label{bbF} for the collection of all forests of subtrees $\CF$
such that each tree in $\CF$ belongs to $\Div$.
Given a forest of subtrees $\CF$ and a subtree $S$ compatible with $\mcF$
we define the forest of subtrees which are the (immediate) children of $S$ in $\mathcal{F}$ as
\[\label{def: immediate children in forest}
C_{\mathcal{F}}(S)
\eqdef
\mathrm{Max} \{T \in \mathcal{F}:\ T < S \}\;.
\]
and the branch of $S$ in $\mcF$ by $\mcF_{S} \eqdef \{ T \in \mcF:\ T \le S\}$. 

For non-empty $\mathcal{F} \in \mathbb{F}$ and $k \in \N$ we define subsets $D_{k}(\mathcal{F})$ as follows. We set $D_{0}(\mathcal{F}) \eqdef \emptyset$, $D_{1}(\mathcal{F}) \eqdef \mmax{\mathcal{F}}$, and for $k \ge 1$ we set 
\[
D_{k+1}(\mathcal{F})
=
\bigsqcup_{S \in D_{k}(\mathcal{F})}
C_{\mathcal{F}}(S).
\]
We define the depth of $\mathcal{F}$, denoted by $\mathrm{depth}(\mathcal{F})$, to be given by
\[
\mathrm{depth}(\mathcal{F}) 
\eqdef 
\inf 
\Big\{ 
k \in \N:\ \bigcup_{j=0}^{k}
D_{j}(\mathcal{F})  = \mathcal{F}
\Big\} = \inf 
\{ 
k \ge 0:\ 
D_{k+1}(\mathcal{F})  = \emptyset
\}\;.
\]
Recalling our previous pictorial example of forests, if we set $\mcF = \{S_{1},S_{2},S_{3}\}$ we have $D_{1}(\mcF) = \{S_{1}\}$, $D_{2}(\mcF) = \{S_{2}\}$, and $D_{3}(\mcF) = \{S_{3}\}$. 
On the other hand, if we set $\mcF = \{S_{2},S_{3},S_{5}\}$ then $D_{1}(\mcF) = \{S_{2},S_{5}\}$ and $D_{2}(\mcF) = \{S_{3}\}$. 

We write $\mathbb{F}^{\le 1}$ for the collection of elements of $\mathbb{F}$ of depth $1$ or $0$.  
We also define, for any $\mathcal{F} \in \mathbb{F}$, 
\begin{equs}\label{e:maxForest}
\mathbb{F}[\mathcal{F}]
\eqdef&
\{ \mathcal{G} \in \mathbb{F}:\ \mmax{\mathcal{G}} = \mathcal{F} \}\;,\\
\mathbb{F}_{<}[\mathcal{F}]
\eqdef&
\left\{ \mathcal{G} \in \mathbb{F}^{\le 1}:\ \forall S \in \mcG,\ \exists T \in \mcF
\textnormal{ with } S < T
\right\}\;, \textnormal{ and}\\
\mathbb{F}_{\le}[\mathcal{F}]
\eqdef&
\left\{ \mathcal{G} \in \mathbb{F}^{\le 1}:\ \forall S \in \mcG,\ \exists T \in \mcF
\textnormal{ with } S \le T
\right\}\;.
\end{equs}
Note that $\mathbb{F}[\mathcal{F}]$ is empty unless $\CF \in  \mathbb{F}^{\le 1}$.

Finally, in some of our statements and proofs it is useful to sum over the decorations of an i-forest 
while keeping the specific i-forest fixed. 
We do this by introducing the following family of projection maps.
\begin{definition}\label{def: shapeprojection}
for any \emph{undecorated} i-forest $\bo{\sigma} \in \ident[\mathring{\For}_{2}]$ we write $\projminusshape_{\bo{\sigma}}$ for the projection onto the subspace of $\langle \ident[\For_{2}] \rangle$ spanned by all decorated i-forests of the form $\bo{\sigma}^{\mfn,\mfo}_{\mfe}$.
\end{definition}
We now describe the class of i-forests produced by the action of $\tpode_{-}$. 
\begin{definition}\label{def: negative forest} 
To each $\mathcal{F} \in \mathbb{F}$ we associate an i-forest $\bo{\sigma}_{\mathcal{F}} \in \mathring{\ident[\For_{2}]}$ as follows. If $\mathcal{F} = \emptyset$ then we set $\bo{\sigma}_{\mathcal{F}} \eqdef \mathbf{1}$.
For $\mathcal{F}$ of depth $j > 0$ we set 
\[
\bo{\sigma}_{\mcF} \eqdef 
(F^{(j)},\widehat{F^{(j)}})
\cdots
(F^{(1)},\widehat{F^{(1)}})
\]
where for $1 \le k \le j$ we set $F^{(k)} \eqdef D_{k}(\mathcal{F})$ and $\widehat{F^{(k)}} = \left[ D_{k+1}(\mathcal{F}) \right]_{1}$. In particular, $\widehat{F^{(j)}} = 0$. 
\end{definition}
Below we pictorially present $\bo{\sigma}_{\mcF}$ in the context of our previous example for the case $\mcF = \{S_{1},S_{2},S_{3}\}$. 
The color blue corresponds to the color $1$. 
The forest product of trees is represented by placing the trees next to each other but 
note that the order does not matter.
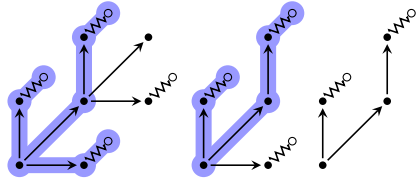
\[
\begin{tikzpicture}[scale=.6,baseline=0,rotate=-45]
    \node[smallsubtreenode,color=blue!40] (subtree root) at (0,0) {};
    \node[smallsubtreenode,color=blue!40] (subtree bottom left noise) at (-1, 1) {};
    \node[smallsubtreenode,color=blue!40] (subtree bottom left anoise) at (-1, 1.75) {};
    \node[smallsubtreenode,color=blue!40] (subtree bottom right noise) at (1, 1) {};
    \node[smallsubtreenode,color=blue!40] (subtree bottom right anoise) at (1, 1.75) {};
    \node[smallsubtreenode,color=blue!40] (subtree middle) at (0, 2) {};
    \node[smallsubtreenode,color=blue!40] (subtree middle left noise) at (-1, 3) {};
    \node[smallsubtreenode,color=blue!40] (subtree middle left anoise) at (-1, 3.75) {};

    \draw[smallsubtreeedge,color=blue!40] (subtree root) to (subtree middle);

    \draw[smallsubtreeedge,color=blue!40] (subtree root) to (subtree bottom left noise);
    \draw[smallsubtreeedge,color=blue!40] (subtree root) to (subtree bottom right noise);
    \draw[smallsubtreeedge,color=blue!40] (subtree bottom left noise) to (subtree bottom left anoise);
    \draw[smallsubtreeedge,color=blue!40] (subtree bottom right noise) to (subtree bottom right anoise);

    \draw[smallsubtreeedge,color=blue!40] (subtree middle) to (subtree middle left noise);

    \draw[smallsubtreeedge,color=blue!40] (subtree middle left noise) to (subtree middle left anoise);

    \node[dot] (root) at (0, 0) {};
    \node[dot] (bottom left noise) at (-1, 1) {};
    \node[circ] (bottom left anoise) at (-1, 1.75) {};
    \node[dot] (bottom right noise) at (1, 1) {};
    \node[circ] (bottom right anoise) at (1, 1.75) {};
    \node[dot] (middle) at (0, 2) {};
    \node[dot] (middle left noise) at (-1, 3) {};
    \node[circ] (middle left anoise) at (-1, 3.75) {};
    \node[dot] (top) at (0, 4) {};
    \node[dot] (middle right noise) at (1, 3) {};
    \node[circ] (middle right anoise) at (1, 3.75) {};

    \draw[kernel] (root) to (middle);
    \draw[kernel] (middle) to (top);
    \draw[kernel] (root) to (bottom left noise);
    \draw[kernel] (root) to (bottom right noise);
    \draw[kernel] (middle) to (middle left noise);
    \draw[kernel] (middle) to (middle right noise);
    \draw[leaf] (bottom left noise) to (bottom left anoise);
    \draw[leaf] (bottom right noise) to (bottom right anoise);
    \draw[leaf] (middle left noise) to (middle left anoise);
    \draw[leaf] (middle right noise) to (middle right anoise);
\end{tikzpicture}
\enskip
\begin{tikzpicture}[scale=.6,baseline=0,rotate=-45]
    \node[smallsubtreenode,color=blue!40] (subtree root) at (0,0) {};
    \node[smallsubtreenode,color=blue!40] (subtree bottom left noise) at (-1, 1) {};
    \node[smallsubtreenode,color=blue!40] (subtree bottom left anoise) at (-1, 1.75) {};

    \node[smallsubtreenode,color=blue!40] (subtree middle) at (0, 2) {};
    \node[smallsubtreenode,color=blue!40] (subtree middle left noise) at (-1, 3) {};
    \node[smallsubtreenode,color=blue!40] (subtree middle left anoise) at (-1, 3.75) {};

    \draw[smallsubtreeedge,color=blue!40] (subtree root) to (subtree middle);

    \draw[smallsubtreeedge,color=blue!40] (subtree root) to (subtree bottom left noise);

    \draw[smallsubtreeedge,color=blue!40] (subtree bottom left noise) to (subtree bottom left anoise);

    \draw[smallsubtreeedge,color=blue!40] (subtree middle) to (subtree middle left noise);
    \draw[smallsubtreeedge,color=blue!40] (subtree middle left noise) to (subtree middle left anoise);

    \node[dot] (root) at (0, 0) {};
    \node[dot] (bottom left noise) at (-1, 1) {};
    \node[circ] (bottom left anoise) at (-1, 1.75) {};
    \node[dot] (bottom right noise) at (1, 1) {};
    \node[circ] (bottom right anoise) at (1, 1.75) {};
    \node[dot] (middle) at (0, 2) {};
    \node[dot] (middle left noise) at (-1, 3) {};
    \node[circ] (middle left anoise) at (-1, 3.75) {};

    \draw[kernel] (root) to (middle);
    \draw[kernel] (root) to (bottom left noise);
    \draw[kernel] (root) to (bottom right noise);
    \draw[kernel] (middle) to (middle left noise);

    \draw[leaf] (bottom left noise) to (bottom left anoise);
    \draw[leaf] (bottom right noise) to (bottom right anoise);
    \draw[leaf] (middle left noise) to (middle left anoise);
\end{tikzpicture}
\enskip
\begin{tikzpicture}[scale=.6,baseline=0,rotate=-45]
    \node[dot] (root) at (0, 0) {};
    \node[dot] (bottom left noise) at (-1, 1) {};
    \node[circ] (bottom left anoise) at (-1, 1.75) {};

    \node[dot] (middle) at (0, 2) {};
    \node[dot] (middle left noise) at (-1, 3) {};
    \node[circ] (middle left anoise) at (-1, 3.75) {};

    \draw[kernel] (root) to (middle);
    \draw[kernel] (root) to (bottom left noise);

    \draw[kernel] (middle) to (middle left noise);

    \draw[leaf] (bottom left noise) to (bottom left anoise);

    \draw[leaf] (middle left noise) to (middle left anoise);
\end{tikzpicture}
\]
Given a forest of subtrees $\mcF$ we define $E_{\mcF} \subset E_{\sT}$ via $E_{\mcF} = \bigsqcup_{S \in \overline{\mcF}} E_{S}$. 
The following lemma states how the action of $\tpode_{-}$ admits an expansion into forests. 
\begin{lemma}\label{lemma: negative forest expansion}
Let $\mcF \in \mathbb{F}^{\le 1}$ and $F$ be the i-forest corresponding to $\mcF$.
Then for any node labeling $\mfn$ on $N(F)$ one has
\begin{equ}\label{eq: negative forest expansion}
\Big(
\sum_{\mathcal{G} \in \mathbb{F}[\mcF]}
\projminusshape_{\bo{\sigma}_{\mathcal{G}}}
\Big)
\tpode_{-} 
(F,0)^{\mfn}_{\se}
=
\tpode_{-}
(F,0)^{\mfn}_{\se}\;.
\end{equ}
\end{lemma}
\begin{proof} 
We prove the statement via induction in $|E_{F}|$. 
The base case, when $\mcF = \emptyset$ and $F = \mathbf{1}$, is trivial (the sum on the 
LHS of \eqref{eq: negative forest expansion} just gives $\projminusshape_{\emptyset}$). 

Now suppose $F$ has $k > 0$ edges and the claim has been proven for all $\mcF' \in \mathbb{F}^{\le 1}$ with $|E_{\mcF'}| < k$.
Then $(-1)^{|\Con(F)|}\tpode_{-}
(F,0)^{\mfn}_{\mfe}$ is given by 
\begin{equs}
{}
&
\sum_{
\mcG \in \mathbb{F}_{<}[\mcF]
}
\sum_{\mfn_{G}, \mfe_{G}}
\frac{1}{\mfe_{G}!}\binom{\mfn}{\mfn_{G}}
\left(
\tpode_{-}
\mfp_{-}
(G,0)^{\mfn_{G} + \chi \mfe_{G}}_{\se} 
\right)
\cdot
(F,[G]_{1})^{\mfn - \mfn_{G}}_{\se + \mfe_{G}}\\
=
& 
\sum_{
\mcG \in \mathbb{F}_{<}[\mcF]
}
\sum_{\mfn_{G}, \mfe_{G}}
\frac{1}{\mfe_{G}!}\binom{\mfn}{\mfn_{G}}
\Big(
\sum_{\mathcal{G}' \in \mathbb{F}[\mcG]}
\projminusshape_{\sigma_{\mathcal{G}'}}
\tpode_{-}
\mfp_{-}
(G,0)^{\mfn_{G} + \chi \mfe_{G}}_{\se} 
\Big)
\cdot
(F,[G]_{1})^{\mfn - \mfn_{G}}_{\se + \mfe_{G}}\\
=& 
\sum_{
\mcG \in \mathbb{F}_{<}[\mcF]
}
\sum_{\mathcal{G}' \in \mathbb{F}[\mcG]}
\projminusshape_{\sigma_{\mathcal{G}' \sqcup \mcF}}
\sum_{\mfn_{G}, \mfe_{G}}
\frac{1}{\mfe_{G}!}\binom{\mfn}{\mfn_{G}}
\Big(
\tpode_{-}
\mfp_{-}
(G,0)^{\mfn_{G} + \chi \mfe_{G}}_{\se} 
\Big)
\cdot
(F,[G]_{1})^{\mfn - \mfn_{G}}_{\se + \mfe_{G}}\\
=& 
\sum_{
\substack{
\mcG \in \mathbb{F}_{<}[\mcF]\\
\mathcal{G}' \in \mathbb{F}[\mcG]
}}
\projminusshape_{\sigma_{\mathcal{G}' \sqcup \mcF}}
\Big[
\sum_{
\substack{
H \in \overline{\mathfrak{A}}_{1}(F,0)\\
\mfn_{H}, \mfe_{H}
}
}
\frac{1}{\mfe_{H}!}\binom{\mfn}{\mfn_{H}}
\Big(
\tpode_{-}
\mfp_{-}
(H,0)^{\mfn_{H} + \chi \mfe_{H}}_{\se} 
\Big)
\cdot
(F,[H]_{1})^{\mfn - \mfn_{H}}_{\se + \mfe_{H}}
\Big]\\
=&
(-1)^{|\Con(F)|}
\Big(
\sum_{
\substack{
\mcG \in \mathbb{F}_{<}[\mcF]\\
\mathcal{G}' \in \mathbb{F}[\mcG]
}
}
\projminusshape_{\sigma_{\mathcal{G}' \sqcup \mcF}}
\Big)
\tpode_{-}(F,0)^{\mfn}_{\mfe}
\end{equs}
where for each $\mcG \in \mathbb{F}_{<}[\mcF] $ we are writing $G$ for the i-forest corresponding to $\mcG$ (which is certainly an i-subforest of $F$). 
Noting that $
\bigsqcup_{
\mcG \in \mathbb{F}_{<}[\mcF] 
}
\bigsqcup_{\mathcal{G}' \in \mathbb{F}[\mcG]}
\left( \mathcal{G}' \sqcup \mcF \right)
=
\mathbb{F}[\mcF]$ completes the proof.
\end{proof}
We now mix our expansion into forests with a cumulant expansion.
For any finite set $L$ we denote by $\bar \CP(L)$ the set of all partitions $\pi$ of $L$.
We also define $\CP(L) \eqdef \bigcup_{A \subset L}\bar \CP(A)$. 
We remind the reader that $\bar{\mathcal{P}}(\emptyset) = \{ \emptyset \}$. 
\begin{definition}
Given $\pi \in \allpon{L(\sT)}$ and a forest of subtrees $\mathcal{F}$, we say $\mathcal{F}$ and $\pi$ are \emph{compatible} if, for each $S \in \mathcal{F}$, $L(S)$ can be written as union of blocks of $\pi$.
\end{definition}
In sums over partitions or forests we sometimes write, as a subscript,  ``$\pi \textnormal{ comp. } \mathcal{F}$'' to restrict the sum to compatible partitions. 
We also write $\mathbbm{1}_{\mathrm{comp}}(\mathcal{F},\pi)$ for the indicator 
function of the condition that $\mathcal{F}$ and $\pi$ are compatible. Given a partition 
$\pi$, we write $\mathbb{F}_{\pi}$ for the collection of elements of $\mathbb{F}$ which are compatible with $\pi$. 

We now introduce some notions to allow us to explicitly write down the negative renormalization procedure. For any subtree $S$, we set
\begin{equation*}\label{def: some kernel edge sets}
\begin{split}
K^{\downarrow}(S) &\eqdef 
\{
e \in K(\bar{T}):\ 
e_{\p} \in N(S),\ e_{\ch} \not \in N(S)  
\},\\
\bar{K}^{\downarrow}(S) 
&\eqdef
K^{\downarrow}(S) \sqcup K(S),
\textnormal{ and } 
N^{\downarrow}(S)
\eqdef
e_{\ch}(K^{\downarrow}(S))\;.
\end{split}
\end{equation*}
We include a picture to make some of this notions clear.
Below we shaded in a subtree $S$ as before, but we additionally coloured $\rho_{S}$ in dark blue, the elements of $\tilde{N}(S)$ in light blue (recall that $\tilde N(S) = N(S) \setminus \{\rho_S\}$), the elements of $N^{\downarrow}(S)$ in red, and the edges of $K^{\downarrow}(S)$ in light green.
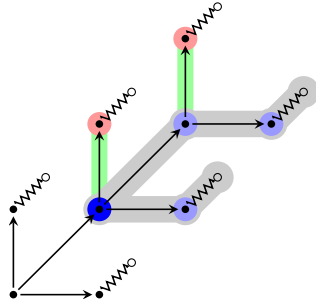
\[
\begin{tikzpicture}[scale = 0.8,rotate=-45]
    \draw[shadededge,color=green!40] (0,2) to (-1,3);
    \draw[shadededge,color=green!40] (0,4) to (-1,5);

    \node[subtreenode] (subtree middle) at (0, 2) {};
    \node[subtreenode] (subtree middle right noise) at (1, 3) {};
    \node[subtreenode] (subtree middle right anoise) at (1, 3.75) {};
    \node[subtreenode] (subtree top) at (0, 4) {};
    \node[subtreenode] (subtree top right noise) at (1, 5) {};
    \node[subtreenode] (subtree top right anoise) at (1, 5.75) {};

    \draw[subtreeedge] (subtree middle) to (subtree top);
    \draw[subtreeedge] (subtree middle) to (subtree middle right noise);
    \draw[subtreeedge] (subtree middle right noise) to (subtree middle right anoise);
    \draw[subtreeedge] (subtree top) to (subtree top right noise);
    \draw[subtreeedge] (subtree top right noise) to (subtree top right anoise);

      \node [shadednode,color=blue] at (0,2) {};
      \node [shadednode] at (1,3) {};
      \node [shadednode] at (0,4) {};
      \node [shadednode] at (1,5) {};
      \node [shadednode,color=red!40] at (-1,3) {};
      \node [shadednode,color=red!40] at (-1,5) {};

    \node[dot] (root) at (0, 0) {};
    \node[dot] (bottom left noise) at (-1, 1) {};
    \node[circ] (bottom left anoise) at (-1, 1.75) {};
    \node[dot] (bottom right noise) at (1, 1) {};
    \node[circ] (bottom right anoise) at (1, 1.75) {};
    \node[dot] (middle) at (0, 2) {};
    \node[dot] (middle left noise) at (-1, 3) {};
    \node[circ] (middle left anoise) at (-1, 3.75) {};
    \node[dot] (top) at (0, 4) {};
    \node[dot] (middle right noise) at (1, 3) {};
    \node[circ] (middle right anoise) at (1, 3.75) {};
    \node[dot] (top left noise) at (-1, 5) {};
    \node[dot] (top right noise) at (1, 5) {};
    \node[circ] (top left anoise) at (-1, 5.75) {};
    \node[circ] (top right anoise) at (1, 5.75) {};

    \draw[kernel] (root) to (middle);
    \draw[kernel] (middle) to (top);
    \draw[kernel] (root) to (bottom left noise);
    \draw[kernel] (root) to (bottom right noise);
    \draw[kernel] (middle) to (middle left noise);
    \draw[kernel] (middle) to (middle right noise);
    \draw[kernel] (top) to (top left noise);
    \draw[kernel] (top) to (top right noise);
    \draw[leaf] (bottom left noise) to (bottom left anoise);
      \draw[leaf] (bottom right noise) to (bottom right anoise);
      \draw[leaf] (middle left noise) to (middle left anoise);
      \draw[leaf] (middle right noise) to (middle right anoise);
      \draw[leaf] (top left noise) to (top left anoise);
      \draw[leaf] (top right noise) to (top right anoise);

%

\end{tikzpicture}
\]
Note that the zigzag at the top right does not belong to $\bar K^{\downarrow}(S)$ because
$K(\sT)$ only contains ``kernel edges'' by definition.

For a pair of subtrees $S,T$ with $T \le S$ we define $\label{def: incoming edges} K^{\partial}_{T}(S) \eqdef K^{\downarrow}(T) \cap K(S)$.

Given a forest of subtrees $\mcF$ and a subtree $S$ we define
\begin{equation*}\label{def: some modulo forest node sets}
\tilde{N}_{\mcF}(S) \eqdef \tilde{N}(S) \setminus \Big(\bigsqcup_{T \in C_{\mcF}(S)} \tilde{N}(T) \Big), \quad
N_{\mcF}(S) \eqdef \nrmod[\mcF,S] \sqcup \{\rho_{S}\}\;.
\end{equation*}
We also use kernel edge decomposition
\[
K(S)
=
\Big( \bigsqcup_{T \in C_{\mcF}(S)} K(T) \Big)
\sqcup 
\mathring{K}_{\mcF}(S) 
\sqcup 
K^{\partial}_{\mcF}(S)
\] 
where
\begin{equation*}\label{def: some modulo forest kernel sets}
\mathring{K}_{\mcF}(S) 
\eqdef K(S) \setminus \Big( \bigsqcup_{T \in C_{\mcF}(S)} \bar{K}^{\downarrow}(T) \Big), \textnormal{  and  } 
K^{\partial}_{\mcF}(S) \eqdef \bigsqcup_{T \in C_{\mcF}(S)} K^{\partial}_{T}(S).
\end{equation*}
To keep track of leaf edges we set 
\[\label{def: leaves mod forest}
L_{\mcF}(S)
\eqdef
L(S)
\setminus
\Big(
\bigsqcup_{T \in C_{\mcF}(S)}
L(T)
\Big).\]
To make some of these definitions more concrete we look at the example of $\mcF = \{S_{3},S_{6}\}$, referring to \eqref{picture: examples of subtrees}. Below we have shaded in $S_{6}$ in light gray and, on top of this, $S_3$ in dark gray. We also shaded the nodes of $\tilde{N}_{\mcF}(S_{6})$ in light blue, 
the edge of $\mathring{K}_{\mcF}(S_{6})$ in light green, and the edge of $K^{\partial}_{\mcF}(S_{6}) = K^{\partial}_{S_{3}}(S_{6})$ in red. 
$L_{\mcF}(S_{6})$ consists of just one edge which is the left uppermost zigzag.
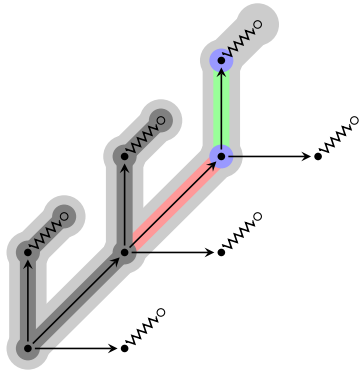
\[
\begin{tikzpicture}[scale=.9,rotate=-45,subtreenode/.style={circle,fill=gray!40,inner sep=0pt,minimum size=16pt},
  subtreeedge/.style={line width=14pt,gray!40, shorten >= -3pt,shorten <=-3pt}]
    \node (subtree root) at (0,0) {};
    \node (subtree bottom left noise) at (-1, 1) {};
    \node (subtree bottom left anoise) at (-1, 1.75) {};
    \node (subtree middle) at (0, 2) {};
    \node (subtree middle left noise) at (-1, 3) {};
    \node (subtree middle left anoise) at (-1, 3.75) {};
    \node (subtree top) at (0, 4) {};
    \node (subtree top left noise) at (-1, 5) {};
    \node (subtree top left anoise) at (-1, 5.75) {};

    \node[subtreenode] at (subtree root) {};
    \node[subtreenode] at (subtree bottom left noise) {};
    \node[subtreenode] at (subtree bottom left anoise) {};
    \node[subtreenode] at (subtree middle) {};
    \node[subtreenode] at (subtree middle left noise) {};
    \node[subtreenode] at (subtree middle left anoise) {};
    \node[subtreenode] at (subtree top) {};
    \node[subtreenode] at (subtree top left noise) {};
    \node[subtreenode] at (subtree top left anoise) {};

    \draw[subtreeedge] (subtree root) to (subtree middle);
    \draw[subtreeedge] (subtree middle) to (subtree top);
    \draw[subtreeedge] (subtree root) to (subtree bottom left noise);
    \draw[subtreeedge] (subtree bottom left noise) to (subtree bottom left anoise);
    \draw[subtreeedge] (subtree middle) to (subtree middle left noise);
    \draw[subtreeedge] (subtree middle left noise) to (subtree middle left anoise);
    \draw[subtreeedge] (subtree top) to (subtree top left noise);
    \draw[subtreeedge] (subtree top left noise) to (subtree top left anoise);

    \draw[shadededge,green!40] (subtree top) to (subtree top left noise);
    \draw[shadededge,red!40] (subtree top) to (subtree middle);
    \node [shadednode] at (0,4) {};
    \node [shadednode] at (-1,5) {};

    \node[shadednode,gray] (subtree root) at (0,0) {};
    \node[shadednode,gray] (subtree bottom left noise) at (-1, 1) {};
    \node[shadednode,gray] (subtree bottom left anoise) at (-1, 1.75) {};
    \node[shadednode,gray] (subtree middle) at (0, 2) {};
    \node[shadednode,gray] (subtree middle left noise) at (-1, 3) {};
    \node[shadednode,gray] (subtree middle left anoise) at (-1, 3.75) {};

    \draw[shadededge,gray] (subtree root) to (subtree middle);
    \draw[shadededge,gray] (subtree root) to (subtree bottom left noise);
    \draw[shadededge,gray] (subtree bottom left noise) to (subtree bottom left anoise);
    \draw[shadededge,gray] (subtree middle) to (subtree middle left noise);
    \draw[shadededge,gray] (subtree middle left noise) to (subtree middle left anoise);

    \node[dot] (root) at (0, 0) {};
    \node[dot] (bottom left noise) at (-1, 1) {};
    \node[circ] (bottom left anoise) at (-1, 1.75) {};
    \node[dot] (bottom right noise) at (1, 1) {};
    \node[circ] (bottom right anoise) at (1, 1.75) {};
    \node[dot] (middle) at (0, 2) {};
    \node[dot] (middle left noise) at (-1, 3) {};
    \node[circ] (middle left anoise) at (-1, 3.75) {};
    \node[dot] (top) at (0, 4) {};
    \node[dot] (middle right noise) at (1, 3) {};
    \node[circ] (middle right anoise) at (1, 3.75) {};
    \node[dot] (top left noise) at (-1, 5) {};
    \node[dot] (top right noise) at (1, 5) {};
    \node[circ] (top left anoise) at (-1, 5.75) {};
    \node[circ] (top right anoise) at (1, 5.75) {};

    \draw[kernel] (root) to (middle);
    \draw[kernel] (middle) to (top);
    \draw[kernel] (root) to (bottom left noise);
    \draw[kernel] (root) to (bottom right noise);
    \draw[kernel] (middle) to (middle left noise);
    \draw[kernel] (middle) to (middle right noise);
    \draw[kernel] (top) to (top left noise) ;
    \draw[kernel] (top) to (top right noise);
    \draw[leaf] (bottom left noise) to (bottom left anoise);
      \draw[leaf] (bottom right noise) to (bottom right anoise);
      \draw[leaf] (middle left noise) to (middle left anoise);
      \draw[leaf] (middle right noise) to (middle right anoise);
      \draw[leaf] (top left noise) to (top left anoise);
      \draw[leaf] (top right noise) to (top right anoise);

%

\end{tikzpicture}
\]
\begin{remark}\label{remark: overload from subtree to collection}
For the rest of the paper we adopt the convention that any notation which takes a subtree of $\sT$ as an argument can also take a collection of subtrees of $\sT$ (not necessarily a forest) as an argument by taking unions over elements of that collection.
For example, for a collection $\mcA$ of subtrees we define
\[
\tilde{N}(\mcA) \eqdef \bigcup_{S \in \mcA} \tilde{N}(S)\;,\quad
L(\mcA) 
\eqdef 
\bigcup_{S \in \mcA} L(S)\;,\quad
K^{\downarrow}(\mcA) \eqdef 
\bigcup_{S \in \mcA}
K^{\downarrow}(S)\;,
\]
and so on.
\end{remark}
We now introduce shorthand for the functions appearing in our integrands.
For any set of kernel edges $E \subset K(\sT)$ and edge decoration $\mfe: K(\sT) \rightarrow \N^{d}$, we define the function $\ke{E}{\mfe} \in \smooth_A$ with $A = e_{\p}(E) \cup e_{\ch}(E)$, by
\[
\ke{E}{\mfe}(x) \eqdef
\prod_{e \in E}
D^{\se(e) + \mfe(e)}K_{\mft(e)}(x_{e_{\p}} - x_{e_{\ch}}).
\] 
Given $u \in \allnodes$, we also define $\kerec{E}{\mfe}{u} \in \smooth_A$
with $A = \{u\} \cup e_{\ch}(E)$, by 
\[
\kerec{E}{\mfe}{u}(x) 
\eqdef
\prod_{e \in E} D^{\se(e) + \mfe(e)}K_{\mft(e)}(x_{u} - x_{e_{\ch}}).
\]
For any $L \subset L(\sT)$ and $\pi \in \allpon{L(\sT)}$ we define the function $\mathrm{Cu}^{L}_{\pi} \in \smooth_{L}$ by
\[
\mathrm{Cu}^{L}_{\pi}(x) 
\eqdef
\prod_{
\substack{
B \in \pi\\
B \subset L
}}
\Cum[\{ \xi_{\mft(u)}(x_{u}) \}_{u \in B}],
\]
where $\Cum$ denotes joint cumulants as before.

Finally, for any $N \subset N(\sT)$, $\mfn:N \rightarrow \N^{d}$, and distinct $v,v' \in \allnodes$ we define the functions $\pow{N}{\mfn}$, $\powroot{N}{\mfn}{v}$, $\powquot{N}{\mfn}{v}$, and $\powrootquot{N}{\mfn}{v}{v'} \in \allf$ via
\begin{equs}[2]\label{def: various power functions}
\pow{N}{\mfn}
(x)
&\eqdef
\prod_{u \in N}
x_{u}^{\mfn(u)}, &\quad
\powroot{N}{\mfn}{v}
(x) 
&\eqdef
\prod_{u \in N}
(x_{u} - x_{v})^{\mfn(u)},\\
\powquot{N}{\mfn}{v}
(x)
&\eqdef
(x_{v})^{\sum_{u \in N} \mathfrak{n}(u)}, &\quad
\powrootquot{N}{\mfn}{v}{v'}(x)
&=
(x_{v'} - x_{v})^{\sum_{u \in N} \mathfrak{n}(u)}\;.
\end{equs}
We also set
$\fpowquot{N}{\mfn}{v}
(x) 
\eqdef
(-x_{v})^{\sum_{u \in N} \mfn(u)}$.
\subsubsection{An inductive definition of negative renormalization counterterms}
We will begin using multivariable multi-index notation more frequently: we take as a universe of multi-indices $(\N^{d})^{\allnodes \sqcup K(\sT)}$. We define $|\cdot|_{\s}$ and $|\cdot|$ for tuples of multi-indices by summation over entries. 
\begin{definition}\label{def:collapse}
Given a subtree $S$ and any subset $A \subset \allnodes$ with $\tilde N(S)\subset A$,
we define the ``collapsing map'' $\Coll_S$ onto the root of $S$ as the map
from $(\R^d)^A$ to itself given by
\begin{equ}
\Coll_S(x)_u \eqdef 
\left\{\begin{array}{cl}
	x_{\rho(S)} & \text{if $u \in \tilde N(S)$,} \\
	x_u & \text{otherwise.}
\end{array}\right.
\end{equ}
\end{definition}
This allows us to give the following definition.
\begin{definition}\label{set of derivatives}
For any subtree $S \in \Div$ we define $\mathrm{Der}(S)$ to be the set of all multi-indices $k$ supported on $\tilde{N}(S)$ with $|k|_{\s} < \omega(S)$. We also define an operator $\mathscr{Y}_{S}: \allf \mapsto \allf$ via 
\[
\left[
\mathscr{Y}_{S}
\phi
\right](x)\ 
=\ 
\sum_{k \in \mathrm{Der}(S)}
\frac{(x - \Coll_{S}(x))^{k}}{k!}
\big(D^{k} \phi \big)\big(\Coll_S(x)\big)\;.
\]
Here and below, given any set $M \subset N^*$, we write $M^c$
for its complement in $N^*$.
\end{definition}
We now inductively define, for each fixed $\pi \in \allpon{L(\sT)}$ and $\mcF \in \mathbb{F}_{\pi}$, operators $\genvert{\pi,\mcF,S}:\allf \rightarrow \allf$ for each $S \in \mcF$. 
This induction will be with respect to depth of the forest of subtrees $\{T \in \mcF:\ T < S \}$. 
Moreover, for any $M \subset \allnodes$, the map $\genvert{\pi,\mcF,S}$ maps 
$\mcb{C}_M$ to $\mcb{C}_{M \setminus \tilde{N}(S)}$.

The base case for our definition occurs when $C_{\mcF}(S) = \emptyset$ and, in that case, 
we set for any $\phi \in \allf$
\begin{equ}\label{def: generalized vertex base case}
[\genvert{\pi,\mcF,S}(\phi)](x)\ 
\eqdef
\int_{{\tilde{N}(S)}} \back dy\ 
\mathrm{Cu}^{L(S)}_{\pi}(y \sqcup x_{\rho_{S}})\,
\ke{K(S)}{0}(y \sqcup x_{\rho_{S}})
(-\mathscr{Y}_{S}\phi)(
x_{\tilde{N}(S)^c} \sqcup
y). 
\end{equ}
\begin{remark}
It is not hard to see that the integral above is absolutely convergent for fixed $x$.
For each $e \in K(S)$ we have a uniform bound of the type
\[
|
D^{\bar{e}}K_{\mft(e)}(y_{e_{\p}} - y_{e_{\ch}})|
\lesssim
\mathbbm{1}\left\{ | y_{e_{\p}} - y_{e_{\ch}}| \le 1 \right\}
\,
|y_{e_{\p}} - y_{e_{\ch}}|^{-|\s| + |\mft(e)|_{\s} - |\se(e)|_{\s}}
\] 
where $|\mft(e)|_{\s} - |\se(e)|_{\s} > 0$ and any occurrence of $y_{\rho_{S}}$ should be replaced by $x_{\rho_{S}}$.
The edges of $K(S)$ form a tree on the vertices of $\tilde{N}(S)$ so the support of $\ke{K(S)}{0}(y \sqcup x_{\rho_{S}})$, 
when seen as a function of $y_{\tilde{N}(S)}$ for fixed $x_{\rho_{S}}$, is compact.
In particular, if the other parts of the integrand of \eqref{def: cumulantbound} can be bounded uniformly in $y$ by some constant, then the entire integral can be bounded by this constant times the the product of $L_{1}(\R)$ norms of the functions $D^{\bar{e}}K_{\mft(e)}(\cdot)$. 
This is clearly the case as both $\mathrm{Cu}^{L}_{\pi}(y)$ and $(-\mathscr{Y}_{S}\phi)$ are continuous, and therefore bounded on compact sets. 
\end{remark}
\begin{remark}\label{remark: explanation of negative renormalization} A second question the reader may ask is what the definition \eqref{def: generalized vertex base case} is trying to accomplish. 
The quantity defined from this formula is a renormalization counterterm. 
One interpretation of the need for renormalization is the following -- while working on a problem one may see the appearance of a function $F$ which fails to belong to $L^1_\loc$ in some region (say at the origin for a distribution on $\R^{d}$ or on the diagonal for a translation invariant distribution on $\R^{d} \times \R^{d}$), so that it cannot be canonically identified with a distribution. 

The insertion of a regularization parameter $\eps > 0$ which disappears in the $\eps \downarrow 0$ limit may allow one to define a family of bonafide distributions $F_{\eps}$ but these may fail to converge in any meaningful sense as one takes $\eps \downarrow 0$.

In many cases it is then possible to obtain a convergent family of distributions $\tilde{F}_{\eps}$ by subtracting from each $F_{\eps}$ a well-chosen, $\eps$-dependent linear combination of $\delta$ functions and their derivatives situated at these singular regions of $F$. 
A simple example is encountered when one tries to make sense of $F(t) = |t|^{-\alpha}$ for $\alpha > 1$ with $\alpha \not \in \N$ as a distribution on $\R$.
A convergent family of distributions can be obtained by defining
\[
\tilde{F}_{\eps}(t)
\eqdef
(|t| \vee \eps)^{-\alpha}
-
\sum_{j=0}^{\lfloor \alpha \rfloor - 1}
\Big(\int_{\R}ds\ \frac{s^{j}}{j!} (|s| \vee \eps)^{-\alpha} \Big)\delta^{(j)}(t)\;,
\]
where $\delta^{(j)}$ denotes the $j$-th weak derivative of the delta function $\delta$ on $\R$. 
To derive good (uniform in $\eps$) bounds on $\tilde{F}_{\eps}(f)$ for a test function $f$ one exploits a Taylor remainder estimate for $f$.

For an example closer to our setting, suppose that we have a single divergent subtree $S$ 
which contains no divergent subtree. 
Then $S$ can be associated to the distribution on $(\R^{d})^{N(S)}$ given by integration against
the function $F_\eps$ defined as
\[
F_{\eps}(x)
\eqdef
\mathrm{Cu}^{L}_{\eps,\pi}(x)
\ke{K(S)}{0}(x).
\] 
We have made the regularization parameter explicit on the RHS. In practice all of our constructions will be applied to probability measures $\mathbb{P}_{\eps}$ corresponding to some $\xi_{\eps} \in \mcM(\Omega_{\infty})$ which converge to a limiting measure $\mathbb{P}$ with cumulants that are singular at coinciding points -- in such a limit $F_{\eps}$ fails to be in $L^{1}_{\mathrm{loc}} [(\R^{d})^{N^{\exte}(S)}]$ as $\eps \downarrow 0$.
The singular region occurs on the ``small'' diagonal -- those $(x_{v})_{v \in N(S)}$ with $x_{v} = x_{\rho_{S}}$ for all $v$. 
The solution is to subtract an appropriate linear combination of delta functions and their derivatives on this diagonal, i.e. $F_{\eps}(x) - \tilde{F}_{\eps}(x)$ is given by
\begin{equs}
\sum_{k \in \mathrm{Der}(S)}
\Big(
\int_{{\tilde{N}(S)}}
\back dy\ 
\frac{(y - \Coll_{S}(x))^{k}}{k!}
F_\eps(x_{\rho_{S}}\sqcup y)
\Big)
\prod_{u \in \tilde{N}(S)}\delta^{(k_{u})}(x_{u} - x_{\rho_{S}}).
\end{equs} 
\end{remark}
In general, $\genvert{\pi,\mcF,S}$
is given recursively by
\begin{equation*}\label{def: renormalized kernel 1}
\begin{split}
[\genvert{\pi,\mcF,S}(\phi)](x)\ 
&\eqdef 
\int_{{\nrmod[\mcF,S]}} \back dy \
\mathrm{Cu}_{\pi}^{L_{\mcF}(S)}(y)  
\ke{\mathring{K}_{\mcF}(S)}{0}(y \sqcup x_{\rho_{S}})\\
& \quad
\cdot
\genvert{\pi,\mcF, C_{\mcF}(S)}
\left[
\ke{K^{\partial}_{\mcF}(S)}{0}
\cdot
(-\mathscr{Y}_{S}\phi)
\right](x_{\tilde{N}(S)^c} \sqcup
y)\;,
\end{split}
\end{equation*}
where $\genvert{\pi,\mcF,C_{\mcF}(S)}$ denotes the composition of 
the operators $\genvert{\pi,\mcF,T}$ for all $T \in C_{\mcF}(S)$. 
No order needs to be prescribed for this composition since, for any $T_{1},T_{2} \in \mcF$ which are disjoint, the operators $H_{\pi,\mcF,T_{1}}$ and $H_{\pi,\mcF,T_{2}}$ commute. 

Let us sketch the argument showing that this is indeed the case.
We first observe that for any $T \in \mcF$ and any $\psi \in \allf$ which depends only on $x_v$ with $v \not \in \tilde{N}(T)$ one has 
\begin{equ}\label{eq: factorization}
H_{\pi,\mcF,T}[\phi \psi] 
=
\psi 
H_{\pi,\mcF,T}[\phi].
\end{equ}
Then one can unravel the inductive definition of $H_{\pi,\mcF,T}$ to arrive at an expansion of the form
\begin{equ}\label{eq: H expansion}
H_{\pi,\mcF,T}[\phi](x)
\eqdef
\sum_{k 
\in 
\widetilde{\mathrm{Der}}(T)
}
(D^{k}\phi)(\Coll_T(x))
\int_{\tilde{N}(T)} \back dy\ 
F_{T,k}(x_{N^{\downarrow}(T)} \sqcup y)\;,
\end{equ}
where $\widetilde{\mathrm{Der}}(T)$ is some set of multi-indices supported on $\tilde{N}(T)$ and, for each $k$, $F_{T,k}$ is a product of combinatorial factors, derivatives of $\{D^{\se(e)}K_{e}\}_{e \in K(T)}$ evaluated on differences of $(x_{v})_{v \in N(T)}$, and polynomials of those same differences. 
Using the combination of the observation \eqref{eq: factorization}, the symmetry of partial derivatives, and the fact that for distinct $i,j \in \{1,2\}$ one has $N^{\downarrow}(T_{i}) \cap \tilde{N}(T_{j}) = \tilde{N}(T_{i}) \cap \tilde{N}(T_{j}) = \emptyset$ it is straightforward to use the expansion \eqref{eq: H expansion} to prove the claimed commutation rule. 

More generally, given $\mcF \in \mathbb{F}_{\pi}$ and a subforest $\mcG \subset \mcF$ with $\mathrm{depth}(G) \le 1$ we set $H_{\pi,\mcF,\mcG} \eqdef \circ_{T \in G} H_{\pi,\mcF,T}$ where on the RHS we are denoting a composition.
We will also use this convention for variants of these $H$ operators we introduce later.
\subsubsection{Deriving an explicit formula for negative renormalizations}
In what follows, for any multi-index $\mfm$ and real number $r$ we write
\[
\mathbbm{1}_{< r}(\mfm) 
\eqdef 
\mathbbm{1} \left\{ |\mfm|_{\s}
< r \right\}\;.
\] 
We also define, for any three multi-indices $\mfn$, $\mfm$, $\mfe$ and real number $r$,  
\[
\combplus[\mfn,\mfm,\mfe,r]
\eqdef
\frac{1}{\mfe!}\binom{\mfn}{\mfm}\mathbbm{1}_{< r}(\mfm + \mfe) 
\] 
For any $\mcF \in \mathbb{F}$, any node decoration $\mfn$ on $\sT$, and any $\pi \in \allpon{L(\sT)}$ compatible with $\mcF$, we define a function $\mathring{\varpi}^{\mfn}_{\pi}[\mathcal{F}] \in \allf$ as follows. 
We set $\mathring{\varpi}^{\mfn}_{\pi}[\emptyset] = 1$ and, for $\mcF$ non-empty, we set
\begin{equation*}
\begin{split}
\mathring{\varpi}^{\mfn}_\pi[\mathcal{F}] \eqdef&
\left[
\prod_{S \in \mmax{\mcF}} 
\mathrm{Cu}_{\pi}^{L_{\mcF}(S)} 
\ke{\mathring{K}_{\mcF}(S)}{0}
\genvert{\pi,\mcF, C_{\mcF}(S)}
\left[
X_{\mfn}^{\tilde{N}(S)}
\ke{K^{\partial}_{\mcF}(S)}{0}
\right]
\right].
\end{split}
\end{equation*}
\begin{lemma}\label{lemma: negatively renormalized integrand} 
Let $\mathcal{F} \in \mathbb{F}$ and let $F$ be the i-forest corresponding to $\mmax{\mcF}$.
Let $\mfn$ be a node decoration on $F$ vanishing on $\rho(F)$. Then one has
\begin{equ}\label{eq: negatively renormalized integrand}
\bar{\Upsilon}^{\mathbb{P}}
\left[ 
\projminusshape_{\bo{\sigma}_{\mcF}}
\tpode_{-}
(F,0)^{\mfn}_{\se}
\right]
=
(-1)^{|\mmax{\mcF}|}
\sum_{
\substack{
\pi \in \fillingpon{L(\mcF)}\\
\pi \textnormal{ comp. } \mcF
}
}
\int_{{N_\CF}}
\back dy\,
\delta(y_{\rho(\mmax{\mcF})})
\mathring{\varpi}^{\mfn}_{\pi}[\mathcal{F}](y).
\end{equ}
Additionally, if $\mfn$ does not vanish on $\rho(F)$ then the RHS of \eqref{eq: negatively renormalized integrand} vanishes.
\end{lemma}
\begin{proof}
The case where $\mathcal{F} = \emptyset$ holds trivially and we ignore this case below. 

We work with the RHS of \eqref{eq: negatively renormalized integrand} and try to transform it to the LHS, working inductively in the depth of $\mcF$. 

We first observe that \eqref{eq: negatively renormalized integrand} is forest multiplicative, that is we can write 
\[
\projminusshape_{\bo{\sigma}_{\mcF}}
\tpode_{-}
(F,0)^{\mfn}_{\se}
=
\prod_{S \in \overline{\mcF}}
\projminusshape_{\bo{\sigma}_{\mcF_{S}}}
\tpode_{-}
(S,0)^{\mfn}_{\se}\;.
\]
Both the integral and the sum over compatible partitions appearing in
the RHS of \eqref{eq: negatively renormalized integrand} factorize analogously. 
It follows that for both the base step and the inductive step it suffices to treat the situation where there is a single maximal tree in $\mcF$, so $F = S$ for some subtree $S$.

We first treat the base case when $\mcF$ has depth $1$. 
We then have
\[
\projminusshape_{\bo{\sigma}_{\{S\}}}
\tpode_{-}
(S,0)^{\mfn}_{\se}
=
-(S,0)^{\mfn}_{\se}
\]
and it is straightforward to see that 
\[
\bar{\Upsilon}^{\mathbb{P}}[(S,0)^{\mfn}_{\se}] 
=
-
\sum_{\pi \in \fillingpon{L(S)}}
\int_{{N(S)}} \back dy\,
\delta(y_{\rho_{S}})
\mathring{\varpi}^{\mfn}_{\pi}[\{S\}](y).
\]
We now treat the inductive case, we assume \eqref{eq: negatively renormalized integrand} is true for all forests of depth less than $j$ and prove it for a forest $\mcF$ of depth $j$ with $\mmax{\mcF} = \{S\}$. 

We write $\mathcal{G} \eqdef \mcF \setminus \{S\}$ and $G$ for the i-forest corresponding to $\mmax{\mathcal{G}}$. 
For $T \in \mmax{\mathcal{G}}$ we also write $\mathcal{G}_{T} = \{U \in \mathcal{F}:\ U \le T\}$, observe that $\mathcal{G}_{T}$ is a forest of subtrees of depth less than $j$. 

With this notation in hand we have that $\bar{\Upsilon}^{\mathbb{P}}
[
\projminusshape_{\bo{\sigma}_{\mcF}}
\tpode_{-} 
(S,0)^{\mfn}_{\se}
]$ is given by
\begin{equation*}
\begin{split}
-
\sum_{
\mfn_{G},\ 
\mfe_{G}
}
\bar{\Upsilon}^{\mathbb{P}}[
\left(S,[ G]_{1} \right)^{\mfn - \mfn_{G}}_{
\se + \mfe_{G}
} 
]
\cdot
\Big[
\prod_{T \in \mmax{\mathcal{G}}}
\combplus[\mfn,\mfn_{T},\mfe_{T},\omega(T)]
\bar{\Upsilon}^{\mathbb{P}}
[
\projminusshape_{\bo{\sigma}_{\mcG_{T}}}
\tpode_{-}
(T,0)^{\mfn_{T} + \chi \mfe_{T}}_{\se}
]
\Big],
\end{split}
\end{equation*}
where above the sum is over node decorations $\mfn_{G}$ on $N(G) \setminus \rho(G)$ and edge decorations $\mfe_{G}$ on $K^{\partial}_{\mcF}(S)$ -- the edge decorations $\mfe_{T}$ and node decorations $\mfn_{T}$ are given by the restrictions of $\mfe_{G}$ to $K^{\partial}_{T}(S)$ and $\mfn_{G}$ to $\tilde{N}(T)$, respectively.  

Expanding now the RHS of \eqref{eq: negatively renormalized integrand} and using the fact that
every partition $\pi$ of $L(S)$ compatible with $\CF$ consists of a partition
$\bar \pi$ of $L_\CF(S)$ and, for every $T \in \bar \CG$, a partition $\pi_T$ compatible with $\CG_T$,
we obtain
\begin{equs}\label{work: negatively renormalized integrand - inductive step}
{}
&
-\sum_{
\substack{
\pi \in \fillingpon{L(S)}\\
\pi \textnormal{ comp. } \mcF
}
}
\int_{{N_{\mcF}(S)}} \back dy\,
\delta(y_{\rho_{S}} )
\mathring{\varpi}^{\mfn}_\pi[\mathcal{F}](y)\\
&=
-\sum_{
\bar{\pi} \in \fillingpon{L_{\mcF}(S)}
}
\int_{{N_{\mcF}(S)}} \back dy \,
\delta(y_{\rho_{S}}) 
\mathrm{Cu}_{\bar{\pi}}^{L_{\mcF}(S)}(y)
\ke{\mathring{K}_{\mcF}(S)}{0}(y)
X_{\mfn}^{N_{\mathcal{F}}(S)}(y)\\
& \qquad
\cdot
\prod_{T \in \mmax{\mcG}}
\bigg[
\sum_{
\substack{
\pi_{T} \in \fillingpon{L(T)}\\
\pi_{T} \textnormal{ comp. } \mathcal{G}_{T} 
}
}
\genvert{\pi_{T},\mcG_{T}, T}
\left[
\pow{\tilde{N}(T)}{\mfn}
\ke{K^{\partial}_{T}(S)}{0}
\right](y)
\bigg]\;.
\end{equs}
We want to expand the action of the $\genvert{\pi_{T},\mcG_{T}, T}$ above. 
For any $T \in \mmax{\mcG}$ one has, by the rules of of differential calculus and some manipulation of binomial coefficients,
\begin{equs}
\mathscr{Y}_{T}
\Bigg[
\pow{\tilde{N}(T)}{\mfn} 
\ke{K^{\partial}_{T}(S)}{0}
\Bigg]
=&
\sum_{
\substack{
j \in \mathrm{Der}(T)\\
\mfe_{T}}
}
\frac{\powroot{\tilde{N}(T)}{j}{\rho_{T}}}{j!}
\binom{j}{\chi \mfe_{T}}
\binom{\mfn}{j - \chi \mfe_{T}}
\powquot{\tilde{N}(T)}{\mfn - j + \chi \mfe_{T}}{\rho_{T}}
\kerec{K^{\partial}_{T}(S)}{\mfe_{T}}{\rho_{T}}\\[2ex]
=&
\sum_{ \mfn_{T}, \mfe_{T}}
\combplus[\mfn,\mfn_{T},\mfe_{T},\omega(T)]
\powroot{\tilde{N}(T)}{\mfn_{T}}{\rho_{T}}
\powquot{\tilde{N}(T)}{\mfn - \mfn_{T}}{\rho_{T}}
\kerec{K^{\partial}_{T}(S)}{\mfe_T}{\rho_T} 
\end{equs}
where we view $j \in \mathrm{Der}(T)$ as node decorations on $\tilde{N}(T)$. We can then write
\begin{equs}
{}&
\sum_{
\substack{
\pi_{T} \in \fillingpon{L(T)}\\
\pi_{T} \textnormal{ comp. } \mathcal{G}_{T} 
}
}
\genvert{\pi_{T},\mcG_{T}, T}
\Bigg[
\pow{\tilde{N}(T)}{\mfn} 
\ke{K^{\partial}_{T}(S)}{0}
\Bigg](y)\\
&= -
\sum_{
\substack{
\pi_{T} \in \fillingpon{L(T)}\\
\pi_{T} \textnormal{ comp. } \mathcal{G}_{T} 
}
}
\sum_{ \mfn_{T}, \mfe_{T}}
\combplus[\mfn,\mfn_{T},\mfe_{T},\omega(T)] 
\powquot{\tilde{N}(T)}{\mfn - \mfn_{T}}{\rho_{T}}(y)
\kerec{K^{\partial}_{T}(S)}{\mfe_T}{\rho_T}(y)\\
& 
\cdot
\int_{{\tilde{N}_{\mcG_{T}}(T)}} \back dz\, 
\mathrm{Cu}_{\pi_{T}}^{L_{\mcG_{T}}(T)}(z)  
\ke{\mathring{K}_{\mcG_{T}}(T)}{0}(z \sqcup y_{\rho_{T}})\\
& 
\qquad\cdot
\genvert{\pi_{T},\mcG_{T}, C_{\mcG_{T}}(T)}
\Bigg[
\ke{K^{\partial}_{\mcG_{T}}(T)}{0}
\mathrm{X}^{\tilde{N}(T)}_{\mfn_{T},\rho_{T}}
\Bigg]
(y_{N^{\partial}_{T}(S)} \sqcup y_{\rho_{T}} \sqcup z).
\end{equs}
By translation invariance, the final integral does not depend on the value of $y_{\rho_{T}}$. 
We can therefore replace $y_{\rho_{T}}$ with $0$, which makes $\mathring{\varpi}_{\pi_{T}}^{\mfn_{T}}$ appear. By our induction hypothesis, the above quantity is thus equal to
\[
-
\sum_{ \mfn_{T}, \mfe_{T}}
\combplus[\mfn,\mfn_{T},\mfe_{T},\omega(T)]
\powquot{\tilde{N}(T)}{\mfn - \mfn_{T}}{\rho_{T}}(y)
\,\kerec{K^{\partial}_{T}(S)}{\mfe_T}{\rho_T}(y)
\,\bar{\Upsilon}^{\mathbb{P}}
\left[ 
\projminusshape_{\bo{\sigma}_{\mcG_{T}}}
\tpode_{-}
(T,0)^{\mfn_{T}}_{\se}
\right].
\]
Inserting this into the last line of \eqref{work: negatively renormalized integrand - inductive step} for each $T \in \mmax{\mcG}$ gives
\begin{equation*}
\begin{split}
&
\sum_{
\bar{\pi} \in \fillingpon{L_{\mcF}(S)}
}
\int_{{N_{\mcF}(S)}} \back dy\,
\delta(y_{\rho_{S}}) 
\,\mathrm{Cu}_{\bar{\pi}}^{L_{\mcF}(S)}(y)
\,\ke{\mathring{K}_{\mcF}(S)}{0}(y)
\,X_{\mfn}^{N_{\mathcal{F}}(S)}(y)\\
& \enskip
\cdot
\Bigg[
\prod_{T \in \mmax{\mcG}}
\sum_{ \mfn_{T}, \mfe_{T}}
\combplus[\mfn,\mfn_{T},\mfe_{T},\omega(T)]
\,\powquot{\tilde{N}(T)}{\mfn - \mfn_{T}}{\rho_{T}}(y)
\,\kerec{K^{\partial}_{T}(S)}{\mfe_T}{\rho_T}(y)
\,\bar{\Upsilon}^{\mathbb{P}}
\left[ 
\projminusshape_{\bo{\sigma}_{\mcG_{T}}}
\tpode_{-}
(T,0)^{\mfn_{T} + \chi \mfe_{T}}_{\se}
\right]
\Bigg]\\
&= 
-
\sum_{
\mfn_{G},\ \mfe_{G}
}
\bar{\Upsilon}^{\mathbb{P}}
\left[
\left(S, [G]_{1} \right)^{\mfn - \mfn_{G}}_{
\se + \mfe_{G}} 
\right]
\,\Bigg(
\prod_{T \in \mmax{\mcG}}
\combplus[\mfn,\mfn_{T},\mfe_{T},\omega(T)]
\bar{\Upsilon}^{\mathbb{P}}
\left[ 
\projminusshape_{\bo{\sigma}_{\mcG_{T}}}
\tpode_{-}
(T,0)^{\mfn_{T} + \chi \mfe_{T}}_{\se}
\right]
\Bigg).
\end{split}
\end{equation*}
\end{proof}
\subsection{A cutting formula for positive renormalizations}
In the previous subsection we saw that the relevant substructures of $\sT$ with regards to negative renormalizations were the subtrees of $\sT$ of negative homogeneity. 
The analogous substructures of $\sT$ for positive renormalizations will be edges in $K(\sT)$ which are ``trunks'' of subtrees of $\sT$ of positive homogeneity. 

We frequently reference our chosen poset structure on $K(\sT)$, namely $e \le \bar{e}$ if and only if the unique path of edges from $\bar{e}_{\ch}$ to the root $\mainroot$ of $\sT$ contains $e$. 

For $\cC \subset K(\sT)$, $e \in K(\sT)$, the set of immediate children of $e$ in $\cC$ is given by
\[
\label{def: immediate children - cuts}
C_{\cC}(e)
\eqdef
\mathrm{Min}
\left(
\{ \bar{e} \in \cC:\ \bar{e} > e \}
\right).
\] 
For a subset $\cC \subset K(\sT)$ we define the subtrees and sub-forests
\[\label{def: subtrees from cut sets}
\sT_{\not \ge}[\cC] 
\eqdef 
\bigcap_{e \in \cC} \sT_{\not \ge}(e) \textnormal{  and  } 
\Tr_{\ge}[\cC]
\eqdef
\Tr(\sT,\sT_{\not \ge}(\cC))\;,
\] 
where we adopt the convention $\sT_{\not \ge}[\emptyset] \eqdef \sT$ and $\Tr[\cdot,\cdot]$ is as in Definition~\ref{def: dangling trees}. 
Note that while $\Tr_{\ge}[\cC]$ is a collection of subtrees of $\sT$ it will in general \emph{not} be a forest of subtrees since distinct elements of $\Tr_{\ge}[\cC]$ may share roots.

Observe that both $\sT_{\not \ge}[\cC]$ and $\Tr_{\ge}[\cC]$ depend on $\cC$ only through $\mathrm{Min}(\cC)$ and that $\sT_{\not \ge}[\cC]$ always contains as a subtree the trivial tree consisting only of $\mainroot$.
\begin{remark}
In what follows we will begin writing co-action extraction / colorings as sums over sets of edges. It then makes sense to label the resulting sums over decorations with sets of edges rather than the subtrees they correspond to. Namely, if $\cC \subset K(\sT)$ then a sum over $\mfn_{\cC}$ corresponds to a sum over node decorations which are supported on $N(T_{\not \ge}[\cC])$ while a sum over $\mfe_{\cC}$ corresponds a sum over edge-derivative decorations which are supported on $\underline{\cC}$.
\end{remark}
\begin{definition} We say $e \in K(\sT)$ is a \emph{positive cut} if 
\[
\left|
\clearrootn(\sT_{\ge}(e),0)^{\sn}_{\se} 
\right|_{+}
> 0\;,
\]
where $\clearrootn$ is the map mentioned before which zeroes node labels at roots (this map naturally induces a map on i-forests). 

We denote by $\cut \subset K(\sT)$\label{def: set of cuts} the collection of all positive cuts and define a function $\gamma: \cut \rightarrow \N$ via 
\[
\gamma(e)
\eqdef  
\left\lceil
 \left|
\clearrootn(\sT_{\ge}(e),0)^{\sn}_{\se} 
\right|_{+}
\right\rceil.
\]
\end{definition}
The key role of the extended node-label $\mfo$ is to store information so that the negative renormalizations from $\tDelta_{-}$ do not cause $\tDelta_{+}$ or $\tpode_{+}$ to make \emph{additional} positive renormalizations. 

We now take advantage of this to help us unravel the action of $\tpode_{+}$. First recall that any i-tree can be realized as a subtree of $\sT$. We define $\Tr_{+}$ to be the collection of all decorated i-trees of the form $(\sT,\hat{T})^{\mfn,\mfo}_{\mfe} \in \ident[\mfT_{2}]$ with the properties that (i) $\mfe \ge \se$ and (ii) for every $e \in K(\sT)$ with $\hat{T}(e_{\p}) \not = 2$ and $\hat{T}(e) = 0$ one has
\[
 \left|
\clearrootn(\sT_{\ge}(e),0)^{\sn}_{\se} 
\right|_{+}
=
 \left|
\clearrootn(\sT_{\ge}(e),\hat{T})^{\mfn,\mfo}_{\mfe} 
\right|_{+}
+
|\mfe(e) - \se(e)|_{\s}\;.
\]
We then have the following lemma.
\begin{lemma}\label{lemma: extended labels do their job.}
One has
\begin{equs}
\tDelta_{1}(\sT,0)^{\sn}_{\se}
\in&
\langle 
\For_{0}
\rangle
\otimes
\langle 
\ident[\Tr_{1}]
\cap \Tr_{+}
\rangle\\
\textnormal{ and }
(\Id \otimes \tDelta_{2} ) \tDelta_{1}(\sT,0)^{\sn}_{\se}
\in&
\langle 
\For_{0}
\rangle
\otimes
\langle 
\ident[\Tr_{1}]
\rangle
\otimes
\langle
\Tr_{+}
\cap
\ident[\hat{\Tr}_{2}]
\rangle.
\end{equs}
\end{lemma}
Given $\cC \subset \cut$ and $k \in \N$ we define subsets $D_{k}(\cC)$ of $\cC$ as follows. We set $D_{0}(\cC) \eqdef \emptyset$, $D_{1}(\cC) \eqdef \underline{\cC}$, and for $k > 1$ we set
\[
D_{k}(\cC) \eqdef
\bigsqcup_{e \in D_{k-1}(\cC)}
C_{\cC}(e). 
\]
We define the depth of $\cC \subset \cut$ via 
\[
\mathrm{depth}(\cC) \eqdef \inf \Big\{ k \in \N:\ \bigsqcup_{j=1}^{k} D_{j}(\cC)  =\cC \Big\}
= \inf \{ k \ge 0:\  D_{k+1}(\cC)  =\emptyset\}\;.
\]
Below we draw a picture to clarify these notions.
We have drawn a tree $\sT$, but have neglected to draw any leaf edges or any of the fictitious nodes attached to their ends. 
The set $\cC \subset K(\sT)$ consists of all edges which have a non-zero number of tick marks. 
For $k \ge 1$ the set $D_{k}(\cC)$ then consists of those edges with $k$ tick marks.
We see that the set $\cC$ has depth $3$.
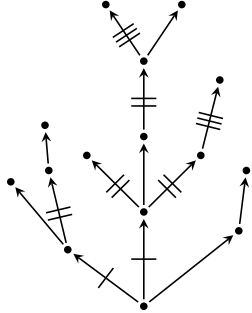
\[
\begin{tikzpicture}
    \node [style=dot] (null) at (0, -2) {};
    \node [style=dot] (2) at (0, -0.75) {};
    \node [style=dot] (3) at (1.25, -1) {};
    \node [style=dot] (1) at (-1, -1.25) {};
    \node [style=dot] (22) at (0, 0.25) {};
    \node [style=dot] (23) at (.75, 0) {};
    \node [style=dot] (231) at (1, 1) {};
    \node [style=dot] (21) at (-.75, 0) {};
    \node [style=dot] (31) at (1.35, -0.2) {};
    \node [style=dot] (221) at (0, 1.25) {};
    \node [style=dot] (2211) at (-0.5, 2) {};
    \node [style=dot] (2212) at (0.5, 2) {};
    \node [style=dot] (11) at (-1.75, -0.35) {};
    \node [style=dot] (12) at (-1.25, -0.2) {};
    \node [style=dot] (121) at (-1.3, 0.4) {};

    \draw[kernel1] (null) to (1);
    \draw[kernel1] (null) to (2);
    \draw[kernel] (null) to (3);

    \draw[kernel] (1) to (11);
    \draw[kernel2] (1) to (12);

    \draw[kernel] (12) to (121);

    \draw[kernel2] (2) to (21);
    \draw[kernel] (2) to (22);
    \draw[kernel2] (2) to (23);

    \draw[kernel2] (22) to (221);

    \draw[kernel3] (221) to (2211);
    \draw[kernel] (221) to (2212);

    \draw[kernel3] (23) to (231);

    \draw[kernel] (3) to (31);
   
\end{tikzpicture}
\]
For any $\cC \subset \cut$ we define $\Div_{\cC}$ to be the collection of those $T \in \Div$ such that $\cC \cap K(T) = \emptyset$. 
We define $\mathbb{F}_{\cC}$ to be the set of all forests $\mcF \in \mathbb{F}$ with $\mcF \subset \Div_{\cC}$.

Conversely, for any $\mathcal{F} \in \mathbb{F}$ we define $\cut_{\mathcal{F}}
\eqdef
\cut \setminus 
K(\mcF)$.
so that $\CF \in \mathbb{F}_{\cC}$ if and only if $\cC \subset \cut_{\mathcal{F}}$.

For any $\cC \subset \cut$ and $\mathcal{F} \in \Div_{\cC}$ we also define 
\[
\mathcal{F}[\cC] \eqdef \{ 
T \in \mathcal{F}: T \le S \textnormal{ for some } S \in \Tr_{ \ge}[\cC]
\}\;.
\]
Additionally, for $\mcF \in \mathbb{F}$ and $\cC,\cD \subset \cut_{\mcF}$ we define
\[
\mathcal{F}[\cC,\cD] \eqdef \{ 
T \in \mathcal{F}[\cC]: T \le \sT_{\not \ge}[\cD]
\}\;.
\]
\begin{definition}\label{def: positive cutting} 
Given $\cC \subset \cut$ and $\mathcal{F} \in \Div_{\cC}$ we define $\bo{\sigma}_{\cC,\mathcal{F}} \in \mathring{\tilde{\Tr}}$ as follows. 

If $\cC = \emptyset$ then we set $\bo{\sigma}_{\cC,\mcF} \eqdef (\sT,2)$.
Otherwise $\cC$ must be of depth $k \ge 1$ in which case we set 
\begin{equ}\label{eq: positive cutting}
\bo{\sigma}_{\cC,\mcF} 
\eqdef 
(T^{(1)},\widehat{T^{(1)}})
\cdots
(T^{(k)},\widehat{T^{(k)}}) 
,
\end{equ}
where each of the trees $T^{(j)}$ are subtrees of $\sT$ and
for $1 \le j \le k$ one has
\begin{itemize}
\item $T^{(j)} = \sT_{\not \ge}[D_{j+1}(\cC)]$,
\item $\widehat{T^{(j)}} = \big[\sT_{\not \ge}[D_{j}(\cC)]\big]_{2}
\sqcup \Big[ \overline{\mcF[D_{j}(\cC),D_{j+1}(\cC)]} \Big]_{1}$.
\end{itemize}
\end{definition}
We present a pictorial example to clarify the above definitions.
We let $\sT$ and $\cC$ be as in the immediately previous pictorial example and set $\mcF = \emptyset$.  
The tree $\bo{\sigma}_{\cC,\emptyset}$ is then given by the forest product (written in the same order as in \eqref{eq: positive cutting})
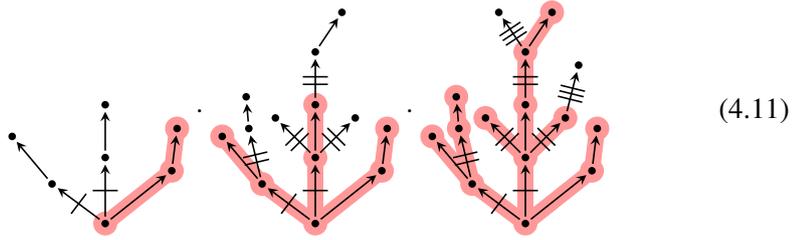
\begin{equ}\label{crimson tide}
\begin{tikzpicture}[scale=.7,baseline=0]
    
    \node [smallsubtreenode,color=red!40] (snull) at (0, -2) {};

    \node [smallsubtreenode,color=red!40] (s3) at (1.25, -1) {};

    \node [smallsubtreenode,color=red!40] (s31) at (1.35, -0.2) {};

    \draw[smallsubtreeedge,color=red!40] (snull) to (s3);

    \draw[smallsubtreeedge,color=red!40] (s3) to (s31);

    \node [style=dot] (null) at (0, -2) {};
    \node [style=dot] (2) at (0, -0.75) {};
    \node [style=dot] (3) at (1.25, -1) {};
    \node [style=dot] (1) at (-1, -1.25) {};
    \node [style=dot] (22) at (0, 0.25) {};

    \node [style=dot] (31) at (1.35, -0.2) {};

    \node [style=dot] (11) at (-1.75, -0.35) {};

    \draw[kernel1] (null) to (1);
    \draw[kernel1] (null) to (2);
    \draw[kernel] (null) to (3);

    \draw[kernel] (1) to (11);

    \draw[kernel] (2) to (22);

    \draw[kernel] (3) to (31);
   
\end{tikzpicture}
\cdot
\begin{tikzpicture}[scale=.7,baseline=0]
    
    \node [smallsubtreenode,color=red!40] (snull) at (0, -2) {};
    \node [smallsubtreenode,color=red!40] (s2) at (0, -0.75) {};
    \node [smallsubtreenode,color=red!40] (s3) at (1.25, -1) {};
    \node [smallsubtreenode,color=red!40] (s1) at (-1, -1.25) {};
    \node [smallsubtreenode,color=red!40] (s22) at (0, 0.25) {};

    \node [smallsubtreenode,color=red!40] (s31) at (1.35, -0.2) {};

    \node [smallsubtreenode,color=red!40] (s11) at (-1.75, -0.35) {};

    \draw[smallsubtreeedge,color=red!40] (snull) to (s1);
    \draw[smallsubtreeedge,color=red!40] (snull) to (s2);
    \draw[smallsubtreeedge,color=red!40] (snull) to (s3);

    \draw[smallsubtreeedge,color=red!40] (s1) to (s11);

    \draw[smallsubtreeedge,color=red!40] (s2) to (s22);

    \draw[smallsubtreeedge,color=red!40] (s3) to (s31);

    \node [style=dot] (null) at (0, -2) {};
    \node [style=dot] (2) at (0, -0.75) {};
    \node [style=dot] (3) at (1.25, -1) {};
    \node [style=dot] (1) at (-1, -1.25) {};
    \node [style=dot] (22) at (0, 0.25) {};
    \node [style=dot] (23) at (.75, 0) {};

    \node [style=dot] (21) at (-.75, 0) {};
    \node [style=dot] (31) at (1.35, -0.2) {};
    \node [style=dot] (221) at (0, 1.25) {};

    \node [style=dot] (2212) at (0.5, 2) {};
    \node [style=dot] (11) at (-1.75, -0.35) {};
    \node [style=dot] (12) at (-1.25, -0.2) {};
    \node [style=dot] (121) at (-1.3, 0.4) {};

    \draw[kernel1] (null) to (1);
    \draw[kernel1] (null) to (2);
    \draw[kernel] (null) to (3);

    \draw[kernel] (1) to (11);
    \draw[kernel2] (1) to (12);

    \draw[kernel] (12) to (121);

    \draw[kernel2] (2) to (21);
    \draw[kernel] (2) to (22);
    \draw[kernel2] (2) to (23);

    \draw[kernel2] (22) to (221);

    \draw[kernel] (221) to (2212);

    \draw[kernel] (3) to (31);
   
\end{tikzpicture}
\cdot
\begin{tikzpicture}[scale=.7,baseline=0]
    
    \node [smallsubtreenode,color=red!40] (snull) at (0, -2) {};
    \node [smallsubtreenode,color=red!40] (s2) at (0, -0.75) {};
    \node [smallsubtreenode,color=red!40] (s3) at (1.25, -1) {};
    \node [smallsubtreenode,color=red!40] (s1) at (-1, -1.25) {};
    \node [smallsubtreenode,color=red!40] (s22) at (0, 0.25) {};
    \node [smallsubtreenode,color=red!40] (s23) at (.75, 0) {};
    \node [smallsubtreenode,color=red!40] (s21) at (-.75, 0) {};
    \node [smallsubtreenode,color=red!40] (s31) at (1.35, -0.2) {};
    \node [smallsubtreenode,color=red!40] (s221) at (0, 1.25) {};

    \node [smallsubtreenode,color=red!40] (s2212) at (0.5, 2) {};
    \node [smallsubtreenode,color=red!40] (s11) at (-1.75, -0.35) {};
    \node [smallsubtreenode,color=red!40] (s12) at (-1.25, -0.2) {};
    \node [smallsubtreenode,color=red!40] (s121) at (-1.3, 0.4) {};

    \draw[smallsubtreeedge,color=red!40] (snull) to (s1);
    \draw[smallsubtreeedge,color=red!40] (snull) to (s2);
    \draw[smallsubtreeedge,color=red!40] (snull) to (s3);

    \draw[smallsubtreeedge,color=red!40] (s1) to (s11);

    \draw[smallsubtreeedge,color=red!40] (s1) to (s12);
    \draw[smallsubtreeedge,color=red!40] (s12) to (s121);

    \draw[smallsubtreeedge,color=red!40] (s2) to (s21);
    \draw[smallsubtreeedge,color=red!40] (s2) to (s22);
    \draw[smallsubtreeedge,color=red!40] (s2) to (s23);

    \draw[smallsubtreeedge,color=red!40] (s22) to (s221);

    \draw[smallsubtreeedge,color=red!40] (s221) to (s2212);

    \draw[smallsubtreeedge,color=red!40] (s3) to (s31);

    \node [style=dot] (null) at (0, -2) {};
    \node [style=dot] (2) at (0, -0.75) {};
    \node [style=dot] (3) at (1.25, -1) {};
    \node [style=dot] (1) at (-1, -1.25) {};
    \node [style=dot] (22) at (0, 0.25) {};
    \node [style=dot] (23) at (.75, 0) {};
    \node [style=dot] (231) at (1, 1) {};
    \node [style=dot] (21) at (-.75, 0) {};
    \node [style=dot] (31) at (1.35, -0.2) {};
    \node [style=dot] (221) at (0, 1.25) {};
    \node [style=dot] (2211) at (-0.5, 2) {};
    \node [style=dot] (2212) at (0.5, 2) {};
    \node [style=dot] (11) at (-1.75, -0.35) {};
    \node [style=dot] (12) at (-1.25, -0.2) {};
    \node [style=dot] (121) at (-1.3, 0.4) {};

    \draw[kernel1] (null) to (1);
    \draw[kernel1] (null) to (2);
    \draw[kernel] (null) to (3);

    \draw[kernel] (1) to (11);
    \draw[kernel2] (1) to (12);

    \draw[kernel] (12) to (121);

    \draw[kernel2] (2) to (21);
    \draw[kernel] (2) to (22);
    \draw[kernel2] (2) to (23);

    \draw[kernel2] (22) to (221);

    \draw[kernel3] (221) to (2211);
    \draw[kernel] (221) to (2212);

    \draw[kernel3] (23) to (231);

    \draw[kernel] (3) to (31);
   
\end{tikzpicture}
\end{equ}
We now want to state the analogues of Lemmas~\ref{lemma: negative forest expansion} and~\ref{lemma: negatively renormalized integrand} for positive renormalizations.

We denote by $\cut^{\le 1}$ the collection of all subsets of $\cut$ with depth $0$ or $1$ and also set $\cut^{\le 1}_{\mcF} \eqdef \cut^{\le 1} \cap 2^{\cut_{\mcF}}$. 
For any $\cC \subset \cut$ and $\mathcal{F} \in \mathbb{F}_{\cC}$ we define a coloring on $\sT$ given by
\[
\hat{\sT}[\cC,\mathcal{F}]
\eqdef
\Big[\sT_{\not \ge}[\cC]\Big]_{2}
\sqcup 
\Big[ \overline{\mathcal{F}[\cC]} \Big]_{1}\;.
\] 
and use the shorthand $\sT[\cC,\mathcal{F}] \eqdef (\sT,\hat{\sT}[\cC,\mcF])$. 

For our previous example with $\sT$ and $\cC$ one has $(\sT,\hat{\sT}[\cC,\emptyset])$ given by
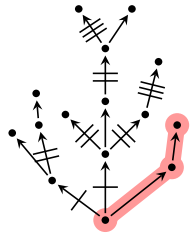
\[
\begin{tikzpicture}[scale=.7,baseline=0]
    
    \node [smallsubtreenode,color=red!40] (snull) at (0, -2) {};
    
    \node [smallsubtreenode,color=red!40] (s3) at (1.25, -1) {};

    \node [smallsubtreenode,color=red!40] (s31) at (1.35, -0.2) {};

    \draw[smallsubtreeedge,color=red!40] (snull) to (s3);
    \draw[smallsubtreeedge,color=red!40] (s3) to (s31);

    \node [style=dot] (null) at (0, -2) {};
    \node [style=dot] (2) at (0, -0.75) {};
    \node [style=dot] (3) at (1.25, -1) {};
    \node [style=dot] (1) at (-1, -1.25) {};
    \node [style=dot] (22) at (0, 0.25) {};
    \node [style=dot] (23) at (.75, 0) {};
    \node [style=dot] (231) at (1, 1) {};
    \node [style=dot] (21) at (-.75, 0) {};
    \node [style=dot] (31) at (1.35, -0.2) {};
    \node [style=dot] (221) at (0, 1.25) {};
    \node [style=dot] (2211) at (-0.5, 2) {};
    \node [style=dot] (2212) at (0.5, 2) {};
    \node [style=dot] (11) at (-1.75, -0.35) {};
    \node [style=dot] (12) at (-1.25, -0.2) {};
    \node [style=dot] (121) at (-1.3, 0.4) {};

    \draw[kernel1] (null) to (1);
    \draw[kernel1] (null) to (2);
    \draw[kernel] (null) to (3);

    \draw[kernel] (1) to (11);
    \draw[kernel2] (1) to (12);

    \draw[kernel] (12) to (121);

    \draw[kernel2] (2) to (21);
    \draw[kernel] (2) to (22);
    \draw[kernel2] (2) to (23);

    \draw[kernel2] (22) to (221);

    \draw[kernel3] (221) to (2211);
    \draw[kernel] (221) to (2212);

    \draw[kernel3] (23) to (231);

    \draw[kernel] (3) to (31);
   
\end{tikzpicture}
\]
Such i-trees will be produced by the action of $\tDelta_{+}$, the action of $\tpode_{+}$ corresponds to iteratively pulling out more structures, always acting to the right. 

Heuristically the tree $\bo{\sigma}_{\cC,\emptyset}$ is one of the terms appearing in the expansion of $\tpode_{+}(\sT,\hat{\sT}[\cC,\emptyset])$. 
Working through the action of $\tpode_{+}$ on $(\sT,\hat{\sT}[\cC,\emptyset])$ as specified through its recursive definition, after one step of the recursion many terms will be generated but the term giving rise to $\bo{\sigma}_{\cC,\emptyset}$ is
\begin{equ}\label{intermediate tide}
\begin{tikzpicture}[scale=.7,baseline=-0.8cm]
    
    \node [smallsubtreenode,color=red!40] (snull) at (0, -2) {};

    \node [smallsubtreenode,color=red!40] (s3) at (1.25, -1) {};

    \node [smallsubtreenode,color=red!40] (s31) at (1.35, -0.2) {};

    \draw[smallsubtreeedge,color=red!40] (snull) to (s3);

    \draw[smallsubtreeedge,color=red!40] (s3) to (s31);

    \node [style=dot] (null) at (0, -2) {};
    \node [style=dot] (2) at (0, -0.75) {};
    \node [style=dot] (3) at (1.25, -1) {};
    \node [style=dot] (1) at (-1, -1.25) {};
    \node [style=dot] (22) at (0, 0.25) {};

    \node [style=dot] (31) at (1.35, -0.2) {};

    \node [style=dot] (11) at (-1.75, -0.35) {};

    \draw[kernel1] (null) to (1);
    \draw[kernel1] (null) to (2);
    \draw[kernel] (null) to (3);

    \draw[kernel] (1) to (11);

    \draw[kernel] (2) to (22);

    \draw[kernel] (3) to (31);
   
\end{tikzpicture}
\cdot
\tpode_{+}
\begin{tikzpicture}[scale=.7,baseline=-0.8cm]
    
    \node [smallsubtreenode,color=red!40] (snull) at (0, -2) {};
    \node [smallsubtreenode,color=red!40] (s2) at (0, -0.75) {};
    \node [smallsubtreenode,color=red!40] (s3) at (1.25, -1) {};

    \node [smallsubtreenode,color=red!40] (s22) at (0, 0.25) {};

    \node [smallsubtreenode,color=red!40] (s31) at (1.35, -0.2) {};

    \node [smallsubtreenode,color=red!40] (s1) at (-1, -1.25) {};
    \node [smallsubtreenode,color=red!40] (s11) at (-1.75, -0.35) {};

    \draw[smallsubtreeedge,color=red!40] (snull) to (s2);
    \draw[smallsubtreeedge,color=red!40] (snull) to (s3);

    \draw[smallsubtreeedge,color=red!40] (snull) to (s1);
    \draw[smallsubtreeedge,color=red!40] (s1) to (s11);

    \draw[smallsubtreeedge,color=red!40] (s2) to (s22);

    \draw[smallsubtreeedge,color=red!40] (s3) to (s31);

    \node [style=dot] (null) at (0, -2) {};
    \node [style=dot] (2) at (0, -0.75) {};
    \node [style=dot] (3) at (1.25, -1) {};
    \node [style=dot] (1) at (-1, -1.25) {};
    \node [style=dot] (22) at (0, 0.25) {};
    \node [style=dot] (23) at (.75, 0) {};
    \node [style=dot] (231) at (1, 1) {};
    \node [style=dot] (21) at (-.75, 0) {};
    \node [style=dot] (31) at (1.35, -0.2) {};
    \node [style=dot] (221) at (0, 1.25) {};
    \node [style=dot] (2211) at (-0.5, 2) {};
    \node [style=dot] (2212) at (0.5, 2) {};
    \node [style=dot] (11) at (-1.75, -0.35) {};
    \node [style=dot] (12) at (-1.25, -0.2) {};
    \node [style=dot] (121) at (-1.3, 0.4) {};

    \draw[kernel1] (null) to (1);
    \draw[kernel1] (null) to (2);
    \draw[kernel] (null) to (3);

    \draw[kernel] (1) to (11);
    \draw[kernel2] (1) to (12);

    \draw[kernel] (12) to (121);

    \draw[kernel2] (2) to (21);
    \draw[kernel] (2) to (22);
    \draw[kernel2] (2) to (23);

    \draw[kernel2] (22) to (221);

    \draw[kernel3] (221) to (2211);
    \draw[kernel] (221) to (2212);

    \draw[kernel3] (23) to (231);

    \draw[kernel] (3) to (31);
   
\end{tikzpicture}
\end{equ}
The term $\bo{\sigma}_{\cC,\emptyset}$ then appears from the above term when one goes one step further in the recursive definition of $\tpode_{+}$, namely by pulling out the middle tree of \eqref{crimson tide} from the second tree of \eqref{intermediate tide}.
We define the collection of cut sets
\begin{equs}\label{positive cuttings - lot of defs}
\cut[\cC,\mathcal{F}]
\eqdef&
\left\{
\cD \subset \cut_{\mcF}:\ 
\underline{\cD} = \cC
\right\},\\[1.5ex]
\cut_{ < }[\cC,\mathcal{F}]
\eqdef&
\left\{
\cD \in \cut^{\le 1}_{\mcF}:\
\forall e \in \cD,\ \exists e' \in \cC 
\textnormal{ with } e' < e 
\right\}\;.
\end{equs}
Also, for any $\cC \subset \cut$, $\mcF \in \mathbb{F}_{\cC}$, and $\cD \in \cut_{<}[\cC,\mcF]$ we set
\[
\sT[\cC,\cD,\mcF] 
\eqdef 
\Big(
\sT_{\not \ge}[\cD], 
\big[
\overline{\mcF[\cC,\cD]} 
\big]_{1}
\sqcup
\big[\sT_{\not \ge}[\cC]\big]_{2}
\Big).
\] 
With all this notation in hand we have the following lemma.
\begin{lemma}\label{lemma: positive antipode formula} 
Let $\mathcal{F} \in \mathbb{F}$ and $\cC \in \cut^{\le 1}_{\mcF}$.
Then for any node decoration $\mfn$, edge decoration $\mfe$, and extended node label $\mfo$ on $\sT$ such that $\sT[\cC,\mcF]^{\mfn,\mfo}_{\se + \mfe} \in \Tr_{+}$ one has
\[
\Big(
\sum_{\cD \in \cut[\cC,\mathcal{F}]}
\projplusshape_{\bo{\sigma}_{\cD,\mathcal{F}}}
\Big)
\tpode_{+}
\sT[\cC,\mcF]^{\mfn,\mfo}_{\se + \mfe}
=
\tpode_{+} 
\sT[\cC,\mcF]^{\mfn,\mfo}_{\se + \mfe}\;.
\]
\end{lemma} 
\begin{proof}
Our proof operates inductively in the quantity 
\begin{equ}\label{eq: number of 2 edges}
|E_{\sT} \setminus E_{\sT_{\not \ge}[\cC]}|,
\end{equ}
with $\mcF$ being fixed for the entire proof. 
The base case of this induction, which occurs when the quantity \eqref{eq: number of 2 edges} is $0$, is immediate since we then have $\sT[\cC,\mcF] = (\sT,2)$ -- we then have $ 
\tpode_{+}
(\sT,2)^{\mfn,\mfo}_{\se + \mfe}
=
(-1)^{|\mfn|}(\sT,2)^{\mfn,0}_{\se + \mfe}$.

We now turn to the proving the inductive step. 
Fix $j \in \N$, and suppose the claim has been proven for any $\cC$ with $\cC \in \cut^{\le 1}_{\mcF}$ and \eqref{eq: number of 2 edges} less than $j$.

Then suppose we are given $\cC \in \cut^{\le 1}_{\mcF}$ with \eqref{eq: number of 2 edges} equal to $j+1$.
Then we have $(-1)^{|\cC|+ |\hat{\mfn}|} \tpode_{+} 
\sT[\cC,\mcF]^{\mfn,\mfo}_{\se + \mfe}$ is given by \ajay{Need to say more about going to first line}
\begin{equs}
{} 
&
\sum_{
\cD \in \cut_{<}[\cC,\mathcal{F}]
}
\sum_{
\mfe_{\cD},
\mfn_{\cD},
\mff_{\cC}
}
\frac{(-1)^{|\mff_{\cC}|}}{(\mff_{\cC} + \mfe_{\cD})!}
\binom{\mfn - \hat{\mfn}}{\mfn_{\cD}}
\sT[\cC,\cD,\mcF]^{
\hat{\mfn}
+
\mfn_{\cD} +  
\chi(\mfe_{\cD} + \mff_{\cC})
}_{\se + \mfe + \mff_{\cC}}\\
&
\qquad
\qquad
\qquad
\qquad
\qquad
\qquad
\cdot
\tpode_{+}
\mfp_{+}
\sT[\cD,\mcF]^{
\mfn - \mfn_{\cD} - \hat{\mfn},
\mfo
}_{\se + \mfe + \mfe_{\cD}}\\
=&
\sum_{
\cD \in \cut_{<}[\cC,\mathcal{F}]
}
\sum_{
\mfe_{\cD},
\mfn_{\cD},
\mff_{\cC}
}
\frac{(-1)^{|\mff_{\cC}|}}{(\mfe_{\cD} + \mff_{\cC})!}
\binom{\mfn - \hat{\mfn}}{\mfn_{\cD}}
\sT[\cC,\cD,\mcF]^{
\mfn_{\cD} +  \hat{\mfn} +
\chi(\mfe_{\cD} + \mff_{\cC})
}_{\se + \mfe + \mff_{\cC}}\\
&
\qquad
\qquad
\qquad
\qquad
\cdot
\sum_{\cE \in \cut[\cD,\mathcal{F}]}
\projplusshape_{\sigma_{\cE,\mcF}}
\sT[\cD,\mcF]^{
\mfn - \mfn_{\cD} - \hat{\mfn},
\mfo
}_{\se + \mfe + \mfe_{\cD}}\\
=&
\sum_{
\substack{
\cD \in \cut_{<}[\cC,\mathcal{F}]\\
\cE \in \cut[\cD,\mathcal{F}]
}
}
\projplusshape_{\sigma_{\cE \sqcup \cC,\mcF}}
\sum_{
\mfe_{\cD},
\mfn_{\cD},
\mff_{\cC}
}
\frac{(-1)^{|\mff_{\cC}|}}{(\mfe_{\cD} + \mff_{\cC})!}
\binom{\mfn}{\mfn_{\cD}}
\sT[\cC,\cD,\mcF]^{
\hat{\mfn}
+
\mfn_{\cD} +  
\chi(\mfe_{\cD} + \mff_{\cC})
}_{\se + \mfe + \mff_{\cC}}\\
&
\qquad
\qquad
\qquad
\qquad
\cdot
\tpode_{+}
\mfp_{+}
\sT[\cD,\mcF]^{
\mfn - \mfn_{\cD} - \hat{\mfn},
\mfo
}_{\se + \mfe + \mfe_{\cD}}\\
=&
(-1)^{|\cC|+ |\hat{\mfn}|}
\Big(
\sum_{
\substack{
\cD \in \cut_{<}[\cC,\mathcal{F}]\\
\cE \in \cut[\cD,\mathcal{F}]
}
}
\projplusshape_{\sigma_{\cE \sqcup \cC,\mcF}}
\Big)
\tpode_{+} 
\sT[\cC,\mcF]^{\mfn,\mfo}_{\se + \mfe},
\end{equs}
where $\hat{\mfn} = \mfn(N(T_{\not \ge}[\cC]))$ and the sum over $\mff_{\cC}$ is a sum over edge decorations $\mff_{\cC} \in \mathfrak{E}[\sT[\cC,\mcF]^{\mfn,\mfo}_{\se + \mfe}]$. 
Since 
\[
\bigsqcup_{
\cD \in \cut_{<}[\cC,\mathcal{F}]
}
\bigsqcup_{\cE \in \cut[\cD,\mathcal{F}]}
(\cC \sqcup \cE)
=
\cut[\cC,\mathcal{F}]\;,
\]
the result follows.
\end{proof}
\begin{remark}
Clearly, the statement of Lemma~\ref{lemma: positive antipode formula} holds if we relax the condition $\cC \in \cut^{\le 1}_{\mcF}$ to just that $\cC \in \cut_{\mcF}$ provided one modifies what's written on the RHS and sums over $\cD \in \cut[\underline{\cC},\mcF]$.
\end{remark}
We now introduce some more notation to facilitate writing explicit formulas for the integrands corresponding to the output of the positive twisted antipode.

For any $\cC \subset \cut$, $\mcF \in \mathbb{F}_{\cC}$, and $\cD \in \cut_{<}[\cC,\mcF]$
we define, using the convention of Remark~\ref{remark: overload from subtree to collection}, 
\begin{equation*}
\begin{split}
&
\tilde{N}[\cC,\mathcal{F}] 
\eqdef
\tilde{N}(\Tr_{\ge}[\cC]) 
\setminus 
\tilde{N}(\mathcal{F}),
K[\cC,\mathcal{F}] 
\eqdef
K(\Tr_{\ge}[\cC])
\setminus
(K(\mcF) \sqcup \cC),\\
& L_{\cC,\mathcal{F}}
\eqdef
L(\Tr_{\ge}[\cC])
\setminus 
L(\mcF)\;.   
\end{split}
\end{equation*}
For any $\cD \subset \cut$ and edge decoration $\mfe$ we define $\RKer_{\mfe}^{\cD} \in \allf$, depending on $x_{v}$ with $v \in e_{\ch}(\cD) \cup e_{\p}(\cD) \cup \{\logof\}$, via
\begin{equation*}
\begin{split}
&\RKer_{\mfe}^{\cD}(
x
)
\eqdef
(-1)^{|\cD|}
\prod_{e \in \cD}
\sum_{|k|_{\s} < \gamma(e) - |\mfe(e)|_{\s} }
\frac{(x_{e_{\p}} - x_{\logof})^{k}}{k!}
\,D^{k + \se(e) + \mfe(e)}K_{\mft(e)}(x_{\logof} - x_{e_{\ch}}).
\end{split}
\end{equation*} 
For any $\cC \subset \cut$ and edge decoration $\mfe$ on $K(\sT)$ we define $\mathrm{Der}(\cC,\mfe)$
to be the set of all edge decorations $\mff_{\cC}$ supported on $\underline{\cC}$ with the property that for every $e \in \underline{\cC}$ one has $|\mff_{\cC}(e)|_{\s} \le \gamma(e) - |\mfe(e)|_{\s}$. 

For any node decoration $\mfn$ on $N(T)$ with $\mfn \le \sn$ and edge decoration $\mfe$ on $K(\sT)$ we define $\ddot{\varpi}^{\mfn}_{\mfe}[\cC,\mathcal{F}](x) \in \allf$, depending on $x_{v}$ with $v \in \tilde{N}[\cC,\mathcal{F}] \sqcup \{\logof\}$, by setting
\begin{equation*} 
\begin{split}
\ddot{\varpi}^{\mfn}_{\mfe}[\cC,\mathcal{F}]
\eqdef&\ 
\RKer_{\mfe}^{\cC \setminus \underline{\cC}}
\cdot
\ke{K[\cC,\mathcal{F}]}{\mfe}
\cdot
\Big(
\sum_{\mff_{\cC} \in \mathrm{Der}(\cC,\mfe)}
\frac{\kerec{\underline{\cC}}{\mfe + \mff_{\cC}}{\logof}}
{\mff_{\cC}!}
\fpowquot{e_{\p}(\underline{\cC})}{\chi \mff_{\cC}}{\logof}
\Big)
\\
& \quad
\cdot
\Big(
\prod_{
T \in \overline{
\mathcal{F}[\cC]
}
}
\powrootquot{\tilde{N}(T)}{\mfn}{\logof}{\rho_{T}}
\Big)
\cdot
\powroot{
\tilde{N}[\cC,\mathcal{F}] 
}{\mfn}{\logof}
\cdot
\fpowquot{N(\sT_{\not \ge}[\cC])}{\mfn}{\logof}\;.
\end{split}
\end{equation*}
Finally, for an edge decoration $\mfe$ and $\cC \subset \cut$ we define an indicator function
\[
\mathbbm{1}_{<\gamma}^{\cC}(\mfe)
\eqdef
\prod_{e \in \cC}
\mathbbm{1}\{ |\mfe(e)|_{\s} < \gamma(e)\}
\]
\begin{lemma} \label{lemma: positively renormalized integrand} 
Let $\cC \subset \cut$ be non-empty and $\mathcal{F} \in \mathbb{F}_{\cC}$. 
Then for any node decoration $\mfn$ on $N(T)$ with $\mfn \le \sn$, edge decoration $\mfe$ on $K(\sT)$, and extended node label $\mfo$ on $\sT$ such that $\sT[\cC,\mcF]^{\mfn,\mfo}_{\se + \mfe} \in \Tr_{+}$ one has
\begin{equs}\label{eq: positive pseudopode identity}
\Upsilon^{\xi}
&\left[ 
\projplusshape_{\bo{\sigma}_{\cC,\mathcal{F}}}
\tpode_{+}
\sT[\cC,\mathcal{F}]^{\mfn, \mfo}_{\se + \mfe}
\right](x_{\logof})\\
&=
(-1)^{|\underline{\cC}|}
\int_{\tilde{N}[\cC,\mathcal{F}]} \back dy\  
\ddot{\varpi}^{\mfn}_{\mfe}
[\cC,\mathcal{F}](y \sqcup x_{\logof})
\Big( \prod_{u \in L_{\cC,\mathcal{F}}} \xi_{\mft(u)}(y_{u}) \Big)
\end{equs}
\end{lemma} 
\begin{proof}
We prove the statement by induction in $\mathrm{depth}(\cC)$ for fixed $\mcF$. 
The base case, when $\cC = \emptyset$, is a straightforward computation.  

We move to the inductive step. 
Fix $j \in \N$ and suppose the claim has been proven for any $\cC \subset \cut_{\mcF}$ with $\mathrm{depth}(\cC) \le j$.
Now suppose we are given $\cC \subset \cut_{\mcF}$ with $\mathrm{depth}(\cC) = j+1$. 
Writing $\cD = \cC \setminus \underline{\cC}$ and $\hat{\mfn} = \sn(T_{\not \ge}[\cC])$, we then have
\begin{equs}\label{eq: inductive step for positive antipode, goal} 
\projplusshape_{\bo{\sigma}_{\cC,\mathcal{F}}}
&\tpode_{+}
\hat{T}[\cC,\mathcal{F}]^{\mfn,\mfo}_{\se + \mfe}\\
&=
(-1)^{|\cC|+ |\hat{\mfn}|}
\sum_{
\mfe_{\cD},
\mfn_{\cD},
\mff_{\cC}
}
\frac{(-1)^{|\mff_{\cC}|}}{(\mff_{\cC} + \mfe_{\cD})!}
\binom{\mfn - \hat{\mfn}}{\mfn_{\cD}}
\sT[\cC,\cD,\mcF]^{
\mfn_{\cD} +  
\chi(\mfe_{\cD} + \mff_{\cC}) + 
\hat{\mfn}
}_{\se + \mfe + \mff_{\cC}}\\
&
\qquad
\qquad
\cdot
\tpode_{+}
\mfp_{+}
\sT[\cD,\mcF]^{
\mfn - \hat{\mfn} - \mfn_{\cD} ,
\mfo
}_{\se + \mfe + \mfe_{\cD}}
\cdot
\mathbbm{1}_{< \gamma}^{\underline{\cD}}
\left(
\mfe_{\cD} + 
\mfe
\right)
\cdot
\mathbbm{1}_{< \gamma}^{\underline{\cC} }
\left(
\mff_{\cC}
+ 
\mfe
\right)
.
\end{equs}
We rewrite the corresponding RHS of \eqref{eq: positive pseudopode identity}. By the binomial identity one has
\begin{equation}\label{eq: positive antipode binomial}
\begin{split}
&\fpowquot{N(\sT)}
{\hat{\mfn}}{\logof}
\powroot{
\tilde{N}[\cC,\mathcal{F}]
}{\mfn - \hat{\mfn}}{\logof}
\Big(
\prod_{
T \in \mmax{
\mathcal{F}[\cC]
}
}
\powrootquot{\tilde{N}(T)}{\mfn - \hat{\mfn}}{\logof}{\rho_{T}}
\Big)\\
=& 
\sum_{\mfn_{\cD}}
\binom{\mfn - \hat{\mfn}}{\mfn_{\cD}}
\pow{\tilde{N}[\cC,\mathcal{F}] \cap N(\sT_{\not \ge}[\cD])}{\mfn_{\cD}}
\Big(
\prod_{
T \in \mmax{
\mathcal{F}[\cC,\cD]
}
}
\powquot{\tilde{N}(T)}{\mfn_{\cD}}{\rho_{T}}
\Big)
\fpowquot{N(\sT)}
{\hat{\mfn}}{\logof}
\\
&
\quad
\cdot
\fpowquot{N(\sT_{\not \ge}[\cD])}
{\mfn - \hat{\mfn} - \mfn_{\cD}}{\logof}
\cdot
\powroot{
\tilde{N}[\cC,\mcF]
}
{\mfn - \hat{\mfn} - \mfn_{\cD}}{\logof}
\cdot
\Big(
\prod_{
T \in \mmax{
\mathcal{F}[\cD]
}
}
\powrootquot{\tilde{N}(T)}{\mfn - \hat{\mfn} - \mfn_{\cD}}
{\logof}{\rho_{T}}
\Big)\;,
\end{split}
\end{equation}
where in keeping with our convention, the sum over $\mfn_{\cD}$ is a sum over node decorations supported on $N(\sT_{\not \ge}[\cD])$ -- however that the binomial coefficient forces it to be supported on $N(\sT_{\not \ge}[\cD]) \setminus N(\sT_{\not \ge}[\cC])$.

We also have
\begin{equs}
\RKer_{\mfe}^{\underline{\cD}}
&=
(-1)^{|\underline{\cD}| }
\sum_{\mfe_{\cD}}
\frac{1}{\mfe_{\cD}!}
\powroot{e_{\mathrm{p}}(\underline{\cD})}{\chi \mfe_{\cD}}{\logof}
\cdot
\Ker^{\underline{\cD},\logof}_{\mfe + \mfe_{\cD}}
\mathbbm{1}_{< \gamma}^{\underline{\cD}}
\left(
\mfe_{\cD} + 
\mfe
\right)\\
&=
(-1)^{|\underline{\cD}| }
\sum_{\mff_{\cD},\mfe_{\cD}}
\frac{1}{\mfe_{\cD}!}
\binom{\mfe_{\cD}}{\mff_{\cD}}
\powroot{e_{\mathrm{p}}(\underline{\cD})}{\chi(\mfe_{\cD} - \mff_{\cD})}{\logof}
\fpowquot{e_{\mathrm{p}}(\underline{\cD})}{\chi \mff_{\cD}}{\logof}
\Ker^{\underline{\cD},\logof}_{\mfe + \mfe_{\cD}}
\mathbbm{1}_{< \gamma}^{\underline{\cD}}
\left(
\mfe_{\cD} + 
\mfe
\right)\\
&=\label{eq: positive antipode expanding renormalized edges}
(-1)^{|\underline{\cD}| }
\sum_{\mff_{\cD},
\mfe_{\cD}}
\frac{1}{\mfe_{\cD}!}
\frac{1}{\mff_{\cD}!}
\powroot{e_{\mathrm{p}}(\underline{\cD})}{\chi \mfe_{\cD}}{\logof}
\fpowquot{e_{\mathrm{p}}(\underline{\cD})}{\chi \mff_{\cD}}{\logof}
\Ker^{\underline{\cD},\logof}_{\mfe + \mfe_{\cD} + \mff_{\cD}}\\
&
\qquad
\qquad
\qquad
\cdot
\mathbbm{1}_{< \gamma}^{\underline{\cD}}
\left(
\mfe_{\cD} + \mff_{\cD}
+
\mfe\right)\\ 
&=
(-1)^{|\underline{\cD}| }
\sum_{\mfe_{\cD}}
\frac{1}{\mfe_{\cD}!}
\powroot{e_{\mathrm{p}}(\underline{\cD})}{\chi \mfe_{\cD}}{\logof}
\Big(
\sum_{\mff_{\cD} \in \mathrm{Der}(\cD,\mfe + \mfe_{\cD})}
\frac{\Ker^{\underline{\cD},\logof}_{\mfe + \mfe_{\cD} + \mff_{\cD}}}{\mff_{\cD}!}
\fpowquot{e_{\mathrm{p}}(\underline{\cD})}{\chi \mff_{\cD}}{\logof}
\Big)
\end{equs}
where in going to the second line we used the binomial identity and in going to the third line we performed a change of summation $\mfe_{\cD}' \eqdef \mfe_{\cD} - \mff_{\cD}$. 
Note that there is an indicator function $\mathbbm{1}_{< \gamma}^{\underline{\cD}}
\left(
\mfe_{\cD} +
\mfe
\right)$ implicit in the last line, when this vanishes $\mathrm{Der}(\cD,\mfe + \mfe_{\cD})$ is empty. 

Using \eqref{eq: positive antipode expanding renormalized edges} we have that $\ke{K[\cC,\mathcal{F}]}{\mfe}
\RKer_{\mfe}^{\cC \setminus \underline{\cC}}$ is equal to
\begin{equs}\label{eq: positive antipode kernels}
{}
&
\sum_{\mfe_{\cD}}
\mathbbm{1}_{< \gamma}^{\underline{\cD}}
\left(
\mfe_{\cD} +
\mfe
\right)
\frac{1}{\mfe_{\cD}!}
\ke{K[\cD,\mathcal{F}]}{\mfe}
\RKer_{\mfe}^{\cD \setminus \underline{\cD}}
\\
&
\qquad
\cdot
\ke{K[\cC,\mathcal{F}] \cap K(T_{\not \ge}[\cD])}{\mfe}
(-1)^{|\underline{\cD}| }
\powroot{e_{\mathrm{p}}(\underline{\cD})}{\chi\mfe_{\cD}}{\logof}
\Big(
\sum_{\mff_{\cD} \in \mathrm{Der}(\cD,\mfe + \mfe_{\cD})}
\frac{\Ker^{\underline{\cD},\logof}_{\mfe + \mfe_{\cD} + \mff_{\cD}}}{\mff_{\cD}!}
\fpowquot{e_{\mathrm{p}}(\underline{\cD})}{\chi \mff_{\cD}}{\logof}
\Big)
\end{equs}
We then put this all together to see that the RHS of \eqref{eq: positive pseudopode identity} is equal to
\begin{equs}
&
(-1)^{|\cC|}
\sum_{
\mfe_{\cD},
\mfn_{\cD},
\mff_{\cC}
}
\frac{1}{(\mfe_{\cD} + \mff_{\cC})!}
\binom{\mfn}{\mfn_{\cD}}
\mathbbm{1}_{< \gamma}^{\underline{\cD}}
\left(
\mfe_{\cD} + 
\mfe
\right)
\cdot
\mathbbm{1}_{< \gamma}^{\underline{\cC}}
\left(
\mff_{\cC}(e) + 
\mfe(e)
\right)
\\
&
\cdot
\int dy_{\tilde{N}[\cC,\mathcal{F}] \setminus \tilde{N}[\cD,\mathcal{F}] }\ 
\powroot{e_{\mathrm{p}}(\underline{\cD})}{\chi \mfe_{\cD}}{\logof}
(
y
\sqcup 
x_{\logof}
)
\Ker_{\mfe}^{\underline{\cD},\logof}
(y \sqcup x_{\logof})
\ke{K[\cC,\mathcal{F}] \cap K(T_{\not \ge}[\cD])}{\mfe}
(y)
\\
&\quad 
\cdot
\pow{\tilde{N}[\cC,\mcF] \cap N(T_{\not \ge }[\cD])}
{\mfn_{\cD}}(y)
\Big(
\prod_{
T \in \overline{\mathcal{F}[\cC,\cD]}
}
\powquot{\tilde{N}(T)}
{\mfn_{\cD}}
{\rho_{T}}(y)
\Big)
\fpowquot{N(\sT)}
{\hat{\mfn}}{\logof}
\\
&
\quad
\cdot
\kerec{\underline{\cC}}{\mfe + \mff_{\cC}}{\logof}
\fpowquot{e_{\p}(\underline{\cC})}{\chi \mff_{\cC}}{\logof}
\Big( 
\prod_{u \in L_{\cC,\mathcal{F}} \setminus L_{\cD,\mathcal{F}}} 
\xi_{\mft(u)}(y_{u}) 
\Big)
\\
&
\cdot
(-1)^{
|\underline{\cD}|
}
\int dy_{ \tilde{N}[\cD,\mathcal{F}] }
\ddot{\varpi}^{\mfn - \hat{\mfn} -
\mfn_{\cD}
}_{\mfe + \mfe_{\cD}}
[\cD,\mathcal{F}](y \sqcup x_{\logof}).
\end{equs}
Here $\ddot{\varpi}^{\mfn - \hat{\mfn} - 
\mfn_{\cD}
}_{\mfe + \mfe_{\cD}}
[\cD,\mathcal{F}]$ is built using the second lines of \eqref{eq: positive antipode binomial} and \eqref{eq: positive antipode kernels}.

Using our inductive hypothesis for the very last line we see that the expression above is equal to
\begin{equation*}
\begin{split}
&
(-1)^{|\cC|}
\sum_{
\mfe_{\cD},
\mfn_{\cD},
\mff_{\cC}
}
\frac{1}{(\mff_{\cC} + \mfe_{\cD})!}
\binom{\mfn}{\mfn_{\cD}}
\mathbbm{1}^{\underline{\cD}}_{< \gamma}
(
\mfe_{\cD} 
\mfe )\\
&
\quad
\cdot
\Upsilon^{\xi}
\left[ 
(-1)^{|\mff_{\cC}| + |\hat{\mfn}|}
\sT[\cC,\cD,\mcF]^{\mfn_{\cD} + \hat{\mfn} + \chi(\mfe_{\cD} + \mff_{\cC})}_{\se + \mfe + \mff}
\right](x_{\logof})\\
&
\quad
\cdot
\Upsilon^{\xi}
\left[ 
\projplusshape_{\sigma_{\cD,\mathcal{F}}}
\tpode_{+}
\sT[\cD,\mathcal{F}]^{\mfn - \hat{\mfn} - \mfn_{\cD},\mfo}_{\se 
+
\mfe_{\cD}
}
\right]
(x_{\logof}).
\end{split}
\end{equation*}
Comparing this with \eqref{eq: inductive step for positive antipode, goal} we arrive at the desired result.
\end{proof}
\subsection{An explicit formula for the BPHZ renormalized tree}
In this subsection we put together the results of the previous two subsections and prove a proposition on the explicit form of chaos kernels for the BPHZ formula.
First, for any forest of subtrees $\mcF$ we define
\begin{equs}
\nodesleft{\mcF}{\sT}
\eqdef&
\tilde{N}(\sT)
\setminus
\tilde{N}(\mcF),\; 
N(\mcF,\sT)
\eqdef
\tilde{N}(\sT)
\setminus
\tilde{N}(\mcF),\\
\kernelsleft{\mcF}{\sT}
\eqdef&
K(\sT) \setminus \bar{K}^{\downarrow}(\overline{\mcF}),\;
\textnormal{and}\;
\leavesleft{\mcF}{\sT}
\eqdef
L(\sT) 
\setminus 
L(\mcF)\;.
\end{equs}
Then, for any $\wickleaves \subset L(\sT)$, $\pi \in \allpon{L(\sT) \setminus \wickleaves}$, $\cC \subset \cut$, and $\mathcal{F} \in \mathbb{F}_{\cC}$, we define 
$\mathcal{W}_{L}^{\pi}[\mcF, \cC] \in \mcb{C}_{\wickleaves \cup \{\mainroot,\logof\}}$ via the following integral formula where we write $ y\in (\R^{d})^{\nodesleft{\mcF}{\sT} \setminus \wickleaves}$, $x \in (\R^{d})^{\wickleaves \cup \{\mainroot,\logof\}}$, and $z = x \sqcup y$.
\begin{equs}{}
&\mathcal{W}_{\wickleaves}^{\pi}[\mcF, \cC]( x_{\wickleaves \cup \{\mainroot,\logof\}})\\
&\eqdef 
\int_{\nodesleft{\mcF}{\sT} \setminus \wickleaves} 
dy\ 
\mathrm{Cu}^{\leavesleft{\mcF}{\sT}}_{\pi}(z)
\cdot
\ke{\kernelsleft{\mcF}{\sT} \setminus \cC}{0}(z) 
\cdot
\RKer_{0}^{\kernelsleft{\mcF}{\sT} \cap \cC}(z)
\cdot
\powroot{N(\mcF,\sT)}{\sn}{\logof}(z)
\\
& 
\quad 
\cdot 
\Big( 
\prod_{S \in \mmax{\mcF}}
H_{\pi,\mcF,S} \left[ 
\RKer_{0}^{K^{\downarrow}(S) \cap \cC}
\cdot 
\ke{K^{\downarrow}(S) \setminus \cC}{0}
\powroot{\tilde{N}(S)}{\sn}{\logof}
\right](z_{N^{\downarrow}(S)} \sqcup z_{\rho_{S}})
\Big)\;.
\end{equs} 
\begin{proposition}\label{explicit formula} For any $\xi \in \mcM(\Omega_{\infty})$ and $x \in (\R^{d})^{\{\logof,\mainroot\}}$,  
\begin{equ}\label{eq: chaos decomp}
\hat{\Upsilon}^{\xi}[\sT^{\sn}_{\se}](x)
= 
\sum_{
\substack{
\wickleaves \subset L(\sT)\\
\pi \in \fillingpon{L(\sT) \setminus \wickleaves}\\
\mcG \in \mathbb{F}_{\pi},\ 
\cD \subset \cut_{\mcG}
}
}
\int_{\wickleaves\setminus \{\mainroot\}} dy\  
\mathcal{W}_{\wickleaves}^{\pi}[ \mcG, \cD]
(z) \cdot
\wwick{ \{\xi_{\mft(u)}(z_{u}) \}_{u \in \wickleaves}}
\end{equ}
where $z = x \sqcup y$ and, for any multiset $A$ of random variables, $\wwick{A}$ denotes the Wick product of the elements of $A$ (see \cite[Def.~4.5]{HS15} or \cite[Appendix B]{AT06}). 
\end{proposition}
\begin{proof}
Starting from the LHS, we have
\begin{equ}\label{eq: explicit formula mess 1}
(\id \otimes \tDelta_{+} )
\tDelta_{-}
(\sT,0)^{\sn}_{\se}
=
\sum_{
\substack{
\mcF \in \mathbb{F}^{\le 1}\\
\cC \in \cut^{\le 1}_{\mcF}
}
}
\sum_{
\substack{
\mfn_{F},\mfe_{F}\\
\mfn_{\cC},\mfe_{\cC}
}
}
A(\mcF,\cC,\mfn_{F},\mfe_{F},
\mfn_{\cC},\mfe_{\cC})\;,
\end{equ} 
where $
A(\mcF,\cC,\mfn_{F},\mfe_{F},
\mfn_{\cC},\mfe_{\cC})$
is given by
\begin{equs}\label{eq: explicit formula mess 2}
&
\Big( 
\prod_{S \in \mcF} 
\mathbbm{1}_{< \omega(S)}(\mfn_{S} + \mfe_{S}) 
\Big)
\binom{\sn}{\mfn_{F}}\frac{1}{\mfe_{F}!}
(F,0)^{\mfn_{F} + \chi \mfe_{F},0}_{\se}
\\
&
\quad
\otimes
\binom{\sn - \mfn_{F}}{\mfn_{\cC}}
\frac{1}{\mfe_{\cC}!}
\mathbbm{1}_{< \gamma}^{\cC}
(
\mfe_{\cC}
+
\mfe_{F} 
)
\big(\sT_{\not \ge}[\cC], [F]_{1} \big)^{ \mfn_{\cC} + \chi \mfe_{\cC},0}_{\se + \mfe_{F}}\\
&
\quad
\otimes
\sT[\cC,\mcF]^{\mfn - \mfn_{F} - \mfn_{\cC},\mfo_{F}}_{\se + \mfe_{F} + \mfe_{\cC}}\;,
\end{equs}
where $\mfo_{F} \eqdef \mfn_{F} + \chi \mfe_{F}$. 
Note that we used Lemma~\ref{lemma: extended labels do their job.} here. 
Above, $F$ is the i-forest corresponding to $\mcF \in \mathbb{F}^{\le 1}$, the sum over decorations in \eqref{eq: explicit formula mess 1} is a sum over node decorations $\mfn_{F}$ on $N(F) \setminus \rho(F)$, $\mfn_{\cC}$ on $\sT_{ \not \ge}[\cC]$, edge decorations $\mfe_{F}$ on $K^{\downarrow}(\mcF)$ and $\mfe_{\cC}$ on $\cC$. The edge decorations $\mfe_{S}$ and node decorations $\mfn_{S}$ appearing in \eqref{eq: explicit formula mess 2} are given by the restrictions of $\mfe_{F}$ to $K^{\partial}_{S}(\sT)$ and $\mfn_{F}$ to $\tilde{N}(S)$, respectively. 

Observe that we have
\begin{equs}
&\bigsqcup_{\wickleaves \subset L(\sT)}
\bigsqcup_{\pi \in \fillingpon{L(\sT) \setminus \wickleaves}}
\bigsqcup_{
\substack{
\mcF \in \mathbb{F}_{\pi}\\
\cC \subset \cut_{\mcF}}
}
\{
(\mcF,\cC,\pi)
\}
=
\bigsqcup_{
\substack{
\mcF \in \mathbb{F}\\
\cC \subset \cut_{\mcF}}
}
\bigsqcup_{
\substack{
\pi \in \fillingpon{L(\mcF)} \\
\pi \textnormal{ comp. } \mcF
}
}
\bigsqcup_{
\bar{\pi} \in \allpon{\leavesleft{\mcF}{\sT} \setminus \wickleaves}
}
\{
(\mcF,\cC,\pi \sqcup \bar{\pi})
\},\\
& 
\sum_{\bar{\pi} \in \allpon{\leavesleft{\mcF}{\sT}}}
\mathrm{Cu}^{\leavesleft{\mcF}{\sT}}_{\pi}(y)
\wwick{ \{\xi_{\mft(u)}(y_{u}) \}_{u \in \leavesleft{\mcF}{\sT} \setminus \supp(\pi) }}
=
\prod_{u \in \leavesleft{\mcF}{\sT}}\xi_{\mft(u)}(y_{u}),
\end{equs}
where the second inequality holds for any $\mcF \in \mathbb{F}$. 
We can use the two identities above to interchange summations and eliminate the Wick monomials on the RHS of \eqref{eq: chaos decomp} to see that, again writing $z = y_{\nodesleft{\mcF}{\sT}} \sqcup x_{\{\mainroot,\logof\}}$, it is equal to
\begin{equs}\label{eq: explicit formula mess 4}
{}
&
\sum_{
\substack{
\mcF \in \mathbb{F}^{\le 1}\\
\cC \in \cut^{\le 1}_{\mcF}
}
}
\sum_{
\substack{
\mcG \in \mathbb{F}[\mcF]\\
\cD \in \cut[\cC,\mcF]
}
}
\sum_{
\substack{
\pi \in \fillingpon{L(\mcG)} \\
\pi \textnormal{ comp. } \mcG
}
}
\int_{\nodesleft{\mcF}{\sT}} 
dy\ 
\Big( 
\prod_{u \in \leavesleft{\mcF}{\sT}}\xi_{\mft(u)}(z)
\Big)
\powroot{N(\mcF,\sT)}{\sn}{\logof}(z)\\
&
\qquad
\cdot
\ke{\kernelsleft{\mcF}{\sT} \setminus \cD}{0}(z) 
\cdot
\RKer_{0}^{\kernelsleft{\mcF}{\sT} \cap \cD}(z)
\\
& 
\qquad 
\cdot 
\Big( 
\prod_{S \in \mcF}
H_{\pi,\mcG,S} \left[ 
\RKer_{0}^{K^{\downarrow}(S) \cap \cD}
\cdot 
\ke{K^{\downarrow}(S) \setminus \cD}{0}
\powroot{\tilde{N}(S)}{\sn}{\logof}(z)
\right](z_{N^{\downarrow}(S)} \sqcup z_{\rho_{S}})
\Big)\;.
\end{equs}
For the time being we fix $\mcF \in \mathbb{F}^{\le 1}$, 
$\cC \in \cut^{\le 1}_{\mcF}$, and $\cD \in \cut[\cC,\mcF]$. 
Then by straightforward computation, 
\begin{equs}\label{eq: explicit formula mess 5}
&
\sum_{
\substack{
\mcG \in \mathbb{F}[\mcF]\\
\pi \in \fillingpon{L(\mcG)} \\
\pi \textnormal{ comp. } \mcG
}
}
\Big( 
\prod_{S \in \mcF}
H_{\pi,\mcG,S} \left[ 
\RKer_{0}^{K^{\downarrow}(S) \cap \cD}
\cdot
\ke{K^{\downarrow}(S) \setminus \cD}{0}
\powroot{\tilde{N}(S)}{\sn}{\logof}
\right](z)
\Big)
\\
&=
\sum_{
\substack{
\mcG \in \mathbb{F}[\mcF]\\
\pi \in \fillingpon{L(\mcG)} \\
\pi \textnormal{ comp. } \mcG,\ 
}
}
\sum_{\tilde{\mfn}_{F}}
\binom{\sn}{\tilde{\mfn}_{F}}
\RKer_{0}^{\kernelsleft{\mcF}{\sT} \cap \cD}(z)\\
&
\enskip
\cdot
\Big( 
\prod_{S \in \mcF}
\powrootquot{\tilde{N}(S)}{\tilde{\mfn}_{F}}{\logof}{\rho_{S}}(z)
H_{\pi,\mcG,S} \left[ 
\RKer_{0}^{K^{\downarrow}(S) \cap \cD}
\ke{K^{\downarrow}(S) \setminus \cD}{0}
\pow{\tilde{N}(S)}{\sn - \tilde{\mfn}_{F}}
\right](z)
\Big)
\\
&=
(-1)^{|\mcF|}
\sum_{
\substack{
\tilde{\mfn}_{F}, \mfn_{F},\\
\mfe_{F}}
}
\binom{\sn}{\tilde{\mfn}_{F}}
\Big(
\prod_{S \in \mcF}
\powrootquot{\tilde{N}(S)}{\tilde{\mfn}_{F}}{\logof}{\rho_{S}}(z)
\powrootquot{\tilde{N}(S)}{\mfn - \tilde{\mfn}_{F} - \mfn_{F}}{\logof}{\rho_{S}}(z)
\combplus[\sn - \tilde{\mfn}_{F}, \mfn_{S},\mfe_{S}, \omega(S)]
\Big)\\
&
\enskip
\cdot
\RKer_{\mfe_{F}}^{K^{\downarrow}(\mcF) \cap \cD}(z)
\ke{K^{\downarrow}(\mcF) \setminus \cD}{\mfe_{F}}(z)
\sum_{
\substack{
\mcG \in \mathbb{F}[\mcF]\\
\pi \in \fillingpon{L(\mcG)} \\
\pi \textnormal{ comp. } \mcG
}
}
\int_{N_{\mcG}(\overline{\mcF})} dw
\delta(w_{\rho(\mcF)})
\mathring{\varpi}^{\mfn_{F} + \chi \mfe_{F}}_{\pi}[\mcG](w)\\
&=
\sum_{
\substack{
\tilde{\mfn}_{F}, \mfn_{F},\\
\mfe_{F}}
}
\binom{\sn}{\tilde{\mfn}_{F}}
\binom{\sn - \tilde{\mfn}_{F}}{\mfn_{F}}\frac{1}{\mfe_{F}!}
\Big(
\prod_{S \in \mcF}
\powrootquot{\tilde{N}(S)}{\sn - \mfn_{F}}{\logof}{\rho_{S}}(z)
\mathbbm{1}_{< \omega(S)}(\mfn_{S} + \mfe_{S})
\Big)\\
&
\enskip
\cdot
\RKer_{\mfe_{F}}^{K^{\downarrow}(\mcF) \cap \cD}(z)
\ke{K^{\downarrow}(\mcF) \setminus \cD}{\mfe_{F}}(z)
\cdot
\bar{\Upsilon}^{\mathbb{P}} \left[ \tpode_{-} 
(F,0)^{\mfn_{F} + \chi \mfe_{F}}_{\se}
\right]\;.
\end{equs}
In going to the second line we used the binomial identity and the fact that we can replace an $\powroot{\tilde{N}(S)}{\sn - \mfn_{F}}{\rho_{S}}$ with $\pow{\tilde{N}(S)}{\sn - \tilde{\mfn}_{F}}$ in the argument of $H_{\pi,\mcG,S}$ thanks to translation invariance. 
In going to the last line we used Lemmas~\ref{lemma: negative forest expansion} and~\ref{lemma: negatively renormalized integrand}. By the the Chu-Vandermonde identity we have that for each fixed $\mfn_{F}$,
\[
\sum_{\tilde{\mfn}_{F}}
\binom{\sn}{\tilde{\mfn}_{F}}
\binom{\sn - \tilde{\mfn}_{F}}{\mfn_{F}}
=
\binom{\sn}{\mfn_F}\;.
\]
It follows that \eqref{eq: explicit formula mess 4} is equal to
\begin{equation}\label{eq: explicit formula mess 4 v2}
\begin{split}
&
\sum_{
\substack{
\mcF \in \mathbb{F}^{\le 1}\\
\cC \in \cut^{\le 1}_{\mcF}\\
\cD \in \cut[\cC,\mcF]
}
}
\sum_{
\mfn_{F},\ 
\mfe_{F}
}
\binom{\sn}{\mfn_{F}}
\frac{1}{\mfe_{F}!}
\int_{\tilde{N}(\mcF,\sT)} dy\ 
\ke{\kernelsleft{\mcF}{\sT} \setminus \cD}{0}(z) 
\cdot 
\RKer_{0}^{\kernelsleft{\mcF}{\sT} \cap \cD}(z)\\
&
\quad
\cdot
\powroot{N(\mcF,\sT)}{\sn}{\logof}(z)
\cdot
\Big(
\prod_{S \in \mcF}
\powrootquot{\tilde{N}(S)}{\mfn - \mfn_{F}}{\logof}{\rho_{S}}(z)
\mathbbm{1}_{< \omega(S)}(\mfn_{S} + \mfe_{S})
\Big)
\cdot
\Big(
\prod_{u \in \leavesleft{\mcF}{\sT}}\xi_{\mft(u)}(z_{u})
\Big)
\\
& 
\quad 
\cdot
\RKer_{\mfe_{F}}^{K^{\downarrow}(\mcF) \cap \cD}(z)
\ke{K^{\downarrow}(\mcF) \setminus \cD}{\mfe_{F}}(z)
\cdot
\bar{\Upsilon}^{\mathbb{P}} \left[ \tpode_{-} 
(F,0)^{\mfn_{F} + \chi \mfe_{F}}_{\se}
\right]
\;.
\end{split}
\end{equation}
Next we observe that for fixed $
\mcF \in \mathbb{F}^{\le 1}$, 
$\cC \in \cut^{\le 1}_{\mcF}$,
$\mfn_{F}$, and $\mfe_{F}$ we have
\begin{equation*}\label{eq: explicit formula mess 6}
\begin{split}
\RKer_{0}^{\kernelsleft{\mcF}{\sT} \cap \cC} &\cdot
\RKer_{\mfe_{F}}^{K^{\downarrow}(\mcF) \cap \cC}\\
&=
(-1)^{|\cC|}
\sum_{\mfe_{\cC}}
\frac{1}{\mfe_{\cC}!}
\powroot{e_{\mathrm{p}}(\cC)}{\mfe_{\cC}}{\logof}
\Big(
\sum_{\mff_{\cC} \in \mathrm{Der}(\cC,\mfe + \mfe_{F})}
\frac{\Ker^{\logof,\cC}_{\mfe + \mfe_{\cC} + \mff_{\cC}}}{\mff_{\cC}!}
\fpowquot{e_{\mathrm{p}}(\cC)}{\chi \mff_{\cC}}{\logof}
\Big)
\end{split}
\end{equation*}
and we also have
\begin{equs}
\powroot{N(\mcF,\sT)}{\sn}{\logof}
\Big(
\prod_{S \in \mcF}
\powrootquot{\tilde{N}(S)}{\mfn - \mfn_{F}}{\logof}{\rho_{S}}
\Big)
&=
\sum_{\mfn_{\cC}}
\binom{\mfn - \mfn_{F}}{\mfn_{\cC}}
\pow{N_{\mcF}(\sT_{\not \ge}[\cC])}{\mfn_{\cC}}
\Big(
\prod_{
\substack{
S \in \mcF\\
S \le \sT_{\not \ge}[\cC]}
}
\powquot{\tilde{N}(S)}{\mfn_{\cC}}{\rho_{S}}
\Big)\\
&
\quad
\cdot
\fpowquot{N(\sT_{\not \ge}[\cC])}{\sn - \mfn_{F} - \mfn_{\cC}}{\logof}
\powroot{\tilde{N}[\cC,\mcF]}{\sn}{\logof}
\Big(
\prod_{S \in \mcF[\cC]}
\powrootquot{\tilde{N}(S)}{\sn - \mfn_{F} - \mfn_{\cC}}{\logof}{\rho_{S}}
\Big)\;.
\end{equs}
With these facts in hand one can carefully factorize the integral of \eqref{eq: explicit formula mess 4 v2}, and then use Lemmas~\ref{lemma: positive antipode formula} and~\ref{lemma: positively renormalized integrand} to write \eqref{eq: explicit formula mess 4 v2}
as
\begin{equation*}\label{eq: explicit formula mess 4 v3}
\begin{split}
&
\sum_{
\substack{
\mcF \in \mathbb{F}^{\le 1}\\
\cC \in \cut^{\le 1}_{\mcF}
}
}
\sum_{
\substack{
\mfn_{F},
\mfe_{F},\\
\mfe_{\cC},
\mfn_{\cC}
}
}
\binom{\sn}{\mfn_{F}}
\binom{\sn - \mfn_{F}}{\mfn_{\cC}}\frac{1}{\mfe_{F}!}
\frac{1}{\mfe_{\cC}!}
\Big(
\prod_{S \in \mcF}
\mathbbm{1}_{< \omega(S)}(\mfn_{S} + \mfe_{S})
\Big)\\
&
\qquad
\cdot
\mathbbm{1}_{< \gamma}^{\cC}(\mfe_{\cC} + \mfe_{F})
\Upsilon^{\zeta} 
\left[
\big(\sT_{\not \ge}[\cC], [F]_{1} \big)^{\mfn_{\cC} + \chi \mfe_{\cC},0}_{\se + \mfe_{F}}
\right]\\
&
\qquad
\cdot
\Upsilon^{\zeta} 
\left[
\tpode_{+}
\sT[\cC,\mcF]^{\sn - \mfn_{F} - \mfn_{\cC},
\mfo_{F}}_{\se + \mfe_{F} + \mfe_{\cC}}
\right]
\cdot
\bar{\Upsilon}^{\mathbb{P}} \left[ \tpode_{-} 
(F,0)^{\mfn_{F} + \chi \mfe_{F},0}_{0}
\right]
.
\end{split}
\end{equation*}
The proof is finished upon observing that the above expression is what one obtains by applying to \eqref{eq: explicit formula mess 1} the operator $\bar{\Upsilon}^{\xi}[ \cdot ]
\otimes 
\Upsilon^{\xi(\omega)}[\cdot]
\otimes 
\Upsilon^{\xi(\omega)}[ \cdot]$
.
\end{proof}
\subsection{From the BPHZ renormalized tree to the BPHZ model}
We close this section by checking that our BPHZ renormalized tree agrees with the BPHZ renormalized model, using freely the notations and terminology of \cite{BHZalg}.
\begin{lemma}\label{lem: BPHZ model and tree} Fix a realization $\xi(\omega) \in \Omega_{\infty}$. 
Let $Z^{\xi(\omega)}_{\BPHZ} = (\Pi^{\xi(\omega)},\Gamma^{\xi(\omega)})$ be the ``BPHZ model'' which is defined as the restriction to the reduced regularity structure of the model obtained by applying the BPHZ renormalization procedure of \cite[Thm~6.17]{BHZalg} to the random model corresponding to the canonical lift of $\xi(\omega)$ on the extended regularity structure. 
Then, for every $z \in \R^{d}$, $\big(\Pi_{z}\sT^{\sn}_{\se}\big)(\cdot)$ and $\hat{\Upsilon}^{\xi}_{z}[\sT^{\sn}_{\se}](\cdot)$ as defined in \eqref{def: BPHZ renormalized tree} coincide as random distributions. 
\end{lemma}
\begin{proof}
There is little to prove here other than to describe why the minor differences between our framework and that of \cite{BHZalg} lead to absolutely no difference in the resulting analytic expressions.

The process of going from elements of our ``underlined'' spaces of identified objects to their appropriate 
representatives in the corresponding ``un-identified'' and non-underlined spaces imported from \cite{BHZalg} 
entails forgetting the identification map, performing a contraction, and possibly replacing a forest 
product with a tree product. 
This ``forgetting'' is denoted by an algebra homomorphism $\mcb{U}: \langle \ident[\For_{2}] \rangle \rightarrow  \langle \For_{2}\rangle$ given by dropping the identification data from any element in $\ident[\For_{2}]$ and then extending by linearity.

As in \cite{BHZalg} we use a variety of contraction maps. 
We start with a map $\mcb{k}: \langle \mfF_{2} \rangle \rightarrow \langle \mfF_{2} \rangle$ defined exactly as $\mcb{K}$ is in \cite[Def.~3.14]{BHZalg} but with a minor difference in how decorations are treated -- in the definition of $\mcb{k}$ the new $[\mfo]$ label should be replaced by $[\mfo]'$ given by
\[
[\mfo]'(x)
\eqdef
\sum_{y \in x}
\mfo(y)
+
\sum_{e \in E_{F} \cap \hat{F}^{-1}(2) \cap x }
(\mft(e) - \mfe(e)).
\]

We then define, for $i \in \{1,2\}$, maps $\mcb{k}_{i} \eqdef \mcb{k} \circ \Phi_{i}$ and $\widehat{\mcb{k}}_{i} \eqdef \mcb{k} \circ \widehat{\Phi}_{i}$ where $\Phi_{i}, \widehat{\Phi}_{i} : \langle \mfF_{2} \rangle \rightarrow \langle \mfF_{2} \rangle$ are defined by \cite[Eq.~3.22]{BHZalg} along with the immediately preceding and following paragraphs. 

Finally, we define $\mcb{O} : \langle \mfF_{2} \rangle \rightarrow \langle \mfF_{2} \rangle$ to be the map that sets all extended node labels $\mfo$ that appear to zero. 
Then it is straightforward to check that
\[
\Big(
\mcb{U}
\otimes 
\mcb{k} \mcb{U}
\otimes
\widehat{\mcb{k}}_{2}
\mcb{U}
\Big)
(\Id \otimes \tDelta_{+})
\tDelta_{-}
(\sT,0)^{\sn}_{\se}
=
(\Id \otimes \mcb{O} \otimes \Id)
(\Id \otimes \Delta_{2})
\Delta_{1}
\sT^{\sn}_{\se},
\]
where we are implicitly using that the RHS can be identified as an element of $\mcb{T}^{\ex}_{-} \otimes \mcb{T}^{\ex} \otimes \mcb{T}^{\ex}_{+}$.

For the negative twisted antipodes $\ttpode_{-}:\mcb{T}^{\ex}_{-} \rightarrow \widehat{\mcb{T}}^{\ex}_{-}$ it is again straightforward to see that, for every $(F,0)^{\mfn}_{\mfe} \in \ident[\For_{0}]$ which, $(i)$ can be written as a forest product of elements of $\mathfrak{X}_{-}$ and $(ii)$ satisfies $\mcb{k}_{1} \mcb{U}(F,0)^{\mfn}_{\mfe} \in \mcb{T}^{\ex}_{-}$, one has the identity
\begin{equ}
\mcb{k}_{1} 
\mcb{U}
\tpode_{-} 
(F,0)^{\mfn}_{\mfe}
=
\mcb{O} 
\ttpode_{-} 
\mcb{U}
(F,0)^{\mfn}_{\mfe}\;.
\end{equ}
This claim can be verified inductively in $|E_{F}|$.

We turn to positive twisted antipode -- for any i-tree of the form $(\sT,\hat{T})^{\mfn,\mfo}_{\mfe}$ with $\widehat{\mcb{k}}_{2}\mcb{U}(\sT,\hat{T})^{\mfn,\mfo}_{\mfe} \in \mcb{T}^{\ex}_{+}$ one has
\begin{equ}
\mathscr{J} \widehat{\mcb{k}}_{2}\mcb{U}\tpode_{+}(\sT,\hat{T})^{\mfn,\mfo}_{\mfe}
=
\mcb{O}
\ttpode
\widehat{\mcb{k}}_{2}\mcb{U}(\sT,\hat{T})^{\mfn,\mfo}_{\mfe},
\end{equ}
where $\mathscr{J}: \langle \widehat{\Tr}_{2} \rangle_{\for} \rightarrow \langle \widehat{\Tr}_{2} \rangle $ is the operation of joining roots as given in \cite[Def.~4.6]{BHZalg} and the equality above is in the sense of elements of $\widehat{\mcb{T}}^{\ex}_{+}$. 
This claim can be verified inductively in $|E_{\sT} \setminus E_{\sT_{\not \ge}[\cC]}|$. 

It follows that one has equality
\begin{equs}\label{eq: almost at BPHZ model}
&
\Big(
\mcb{k}_{1}
\mcb{U}
\tpode_{-}
\otimes 
\mcb{k} \mcb{U}
\otimes
\mathscr{J}
\widehat{\mcb{k}}_{2}
\mcb{U}
\tpode_{+}
\Big)
(\Id \otimes \tDelta_{+})
\tDelta_{-}
(\sT,0)^{\sn}_{\se}\\
&
\qquad
=
(\mcb{O} \ttpode_{-} \otimes \mcb{O} \otimes \mcb{O} \ttpode_{+})
(\Id \otimes \Delta_{2})
\Delta_{1}
\sT^{\sn}_{\se}
\end{equs}
as elements of $\widehat{\mcb{T}}^{\ex}_{-} \otimes \mcb{T}^{\ex}\otimes \widehat{\mcb{T}}^{\ex}_{+}$.

If we denote by $\PPi^{\xi(\cdot)}:\mcb{T}^{\ex} \rightarrow \mcb{C}^{\infty}$ the unique \emph{multiplicative} admissible random map with $\PPi^{\xi(\omega)}[\Xi_{\mft}] \eqdef
\xi_{\mft}(\omega)$ for all $\mft \in \Le$ and again write $Z^{\xi(\omega)}_{\BPHZ} = (\Pi^{\xi(\omega)},\Gamma^{\xi(\omega)})$ then for any $x_{\logof},x_{\mainroot} \in \R^{d}$ one has
\begin{equs}\label{def of BPHZ model}
{}
&
\big(
\Pi^{\xi(\omega)}_{x_{\logof}} 
\sT^{\sn}_{\se}\big)(x_{\mainroot})\\
=&\ 
\big(
g^{-}(\PPi^{\xi(\cdot)})
\otimes
\PPi^{\xi(\omega)}[\cdot](x_{\mainroot})
\otimes
g^{+}_{x_{\logof}}(\PPi^{\xi(\omega)})
\big)
( \ttpode_{-} \otimes \Id \otimes \ttpode_{+})
(\Id \otimes \Delta_{2})
\Delta_{1}
\sT^{\sn}_{\se}\\
=&\ 
\big(
g^{-}(\PPi^{\xi(\cdot)})
\otimes
\PPi^{\xi(\omega)}[\cdot](x_{\mainroot})
\otimes
g^{+}_{x_{\logof}}(\PPi^{\xi(\omega)})
\big)
(\mcb{O} \ttpode_{-} \otimes \mcb{O} \otimes \mcb{O} \ttpode_{+})
(\Id \otimes \Delta_{2})
\Delta_{1}
\sT^{\sn}_{\se}\;.
\end{equs}
The first equality is by definition, and the second holds because the maps $g^{-}(\PPi^{\xi(\cdot)})$, $\PPi^{\xi(\omega)}[\cdot](x_{\mainroot})$, and $g^{+}_{x_{\logof}}(\PPi^{\xi(\omega)})$ all ignore the extended node label. 

We then observe the following.
\begin{itemize} 
\item For any $(\sT,\hat{T})^{\mfn}_{\mfe} \in \ident[\Tr_{1}]$
with
$\mcb{k}
\mcb{U}(\sT,\hat{T})^{\mfn}_{\mfe}](x_{\mainroot})
\in
\mcb{T}^{\ex}$
\[
\PPi^{\xi(\omega)}[
\mcb{k}
\mcb{U}(\sT,\hat{T})^{\mfn}_{\mfe}](x_{\mainroot})
=
\Upsilon^{\xi(\omega)}[(\sT,\hat{T})^{\mfn}_{\mfe}](x_{\mainroot}).
\]
\item For any
$(F,\hat{F})^{\mfn}_{\mfe} \in \langle \ident[\widehat{\Tr}_{2}] \rangle_{\for}$
with
$\mathscr{J}
\widehat{\mcb{k}}_{2}
\mcb{U}
(F,\hat{F})^{\mfn}_{\mfe}
\in
\widehat{\mcb{T}}_{+}^{\ex}$
\[
g^{+}_{x_{\logof}}(\PPi^{\xi(\omega)})
\mathscr{J}
\widehat{\mcb{k}}_{2}
\mcb{U}
(F,\hat{F})^{\mfn}_{\mfe}
=
\Upsilon^{\xi(\omega)}[(F,\hat{F})^{\mfn}_{\mfe}](x_{\logof}).
\]
\item For any $(F,\hat{F})^{\mfn}_{\mfe} \in \ident[\mfF_{1}]$ with $\mcb{k}_{1}
\mcb{U}(F,\hat{F})^{\mfn}_{\mfe} \in \widehat{\mcb{T}}^{\ex}_{-}$
\[
g^{-}(\PPi^{\xi(\cdot)})
\mcb{k}_{1}
\mcb{U}(F,\hat{F})^{\mfn}_{\mfe}
=
\bar{\Upsilon}^{\xi}[(F,\hat{F})^{\mfn}_{\mfe}].
\]
\end{itemize}
Using these identities, all that is left is comparing \eqref{def: BPHZ renormalized tree} to
\eqref{def of BPHZ model}.
\end{proof}
\subsection{An example}\label{subsection - an example}
We have verified that in our setting we have recovered the BPHZ model of \cite{BHZalg}. 
In fact, the BPHZ model has already appeared in previous works using the framework of 
regularity structures, but not under that name. 
In this subsection we compare, for a single symbol, our formulas for the BPHZ model with those of the renormalized model appearing in \cite{KPZJeremy}. 
The symbol we choose to look at is $\<211>$.

The context here is that of the KPZ equation, we we are working on $\R^{2}$ with space-time scaling $\s = (2,1)$. 
We have a single kernel type $\mft$ which corresponds to the spatial derivative of the space-time heat kernel 
on $\R^{2}$ and a single noise type $\mfl$ which corresponds to the driving noise. 
We specify a homogeneity assignment $|\cdot |_{\s}$ by setting $|\mft|_{\s} = 1$ and $|\mfl|_{\s} = -3/2 - \kappa$ where $\kappa \in (0,1/10)$. We also fixed a random, smooth noise time map given by setting $\xi_{\eps} = \xi \ast \rho_{\eps}$. 

Using the symbolic notation for kernels and integrals of \cite{KPZJeremy} one has
\begin{equ}\label{eq: from KPZJeremy}
\big(\hat \Pi_{x_{\logof}}^{(\eps)}\<211>\big)(\phi_{x_{\logof}}^\lambda)
=
\begin{tikzpicture}[scale=0.35,baseline=0.8cm]
  \node at (0,-0.8)  [logof] (root) {};
  \node at (-2,1)  [root] (left) {};
  \node at (-2,3)  [dot] (left1) {};
  \node at (-2,5)  [dot] (left2) {};
  \node at (0,1) [var] (variable1) {};
  \node at (0,3) [var] (variable2) {};
  \node at (0,4.3) [var] (variable3) {};
  \node at (0,5.7) [var] (variable4) {};
  
  \draw[testfcn] (left) to (root);

  \draw[kernel1] (left1) to   (left);
  \draw[kernel] (left2) to  (left1);
  \draw[keps] (variable2) to  (left1); 
  \draw[keps] (variable1) to   (left); 
  \draw[keps] (variable3) to   (left2); 
  \draw[keps] (variable4) to   (left2);
\end{tikzpicture}
\;+\;
\begin{tikzpicture}[scale=0.35,baseline=0.8cm]
  \node at (0,-0.8) [logof] (root) {};
  \node at (-2,1)  [root] (left) {};
  \node at (-2,3)  [dot] (left1) {};
  \node at (-2,5)  [dot] (left2) {};
  \node at (0,4.3) [var] (variable3) {};
  \node at (0,5.7) [var] (variable4) {};
  
  \draw[testfcn] (left) to (root);

  \draw[kernelBig] (left1) to   (left);
  \draw[kernel] (left2) to  (left1);
  \draw[keps] (variable3) to   (left2); 
  \draw[keps] (variable4) to   (left2);
\end{tikzpicture}
\;-\;
\begin{tikzpicture}[scale=0.35,baseline=0.8cm]
  \node at (0,-0.8)  [logof] (root) {};
  \node at (-2,1)  [root] (left) {};
  \node at (0,3)  [dot] (left1) {};
  \node at (0,5)  [dot] (left2) {};
  \node at (-2,3) [dot] (variable1) {};
  \node at (-2,4.3) [var] (variable3) {};
  \node at (-2,5.7) [var] (variable4) {};
  
  \draw[testfcn] (left) to (root);

  \draw[kernel] (left1) to   (root);
  \draw[kernel] (left2) to  (left1);
  \draw[keps] (variable1) to  (left1); 
  \draw[keps] (variable1) to   (left); 
  \draw[keps] (variable3) to   (left2); 
  \draw[keps] (variable4) to   (left2);
\end{tikzpicture}
\;+2\;
\begin{tikzpicture}[scale=0.35,baseline=0.8cm]
  \node at (0,-0.8)  [logof] (root) {};
  \node at (-2,1)  [root] (left) {};
  \node at (-2,3)  [dot] (left1) {};
  \node at (-2,5)  [dot] (left2) {};
  \node at (0,1) [var] (variable1) {};
  \node at (0,5.7) [var] (variable4) {};
  
  \draw[testfcn] (left) to (root);

  \draw[kernel1] (left1) to   (left);
  \draw[kernelBig] (left2) to  (left1);
  \draw[keps] (variable1) to   (left); 
  \draw[keps] (variable4) to   (left2);
\end{tikzpicture}
\;+2\;
\begin{tikzpicture}[scale=0.35,baseline=0.8cm]
  \node at (0,-0.8)  [logof] (root) {};
  \node at (-2,1)  [root] (left) {};
  \node at (0,3)  [dot] (left1) {};
  \node at (-2,5)  [dot] (left2) {};
  \node at (0,5) [var] (variable1) {};
  \node at (-2,3) [dot] (variable3) {};
  \node at (-0.5,6) [var] (variable4) {};
  
  \draw[testfcn] (left) to (root);

  \draw[kernel1] (left1) to   (left);
  \draw[kernel] (left2) to  (left1);
  \draw[keps] (variable3) to  (left); 
  \draw[keps] (variable1) to   (left1); 
  \draw[keps] (variable3) to   (left2); 
  \draw[keps] (variable4) to   (left2);
\end{tikzpicture}
\;-2\;
\begin{tikzpicture}[scale=0.35,baseline=0.8cm]
  \node at (0,-0.8)  [logof] (root) {};
  \node at (-2,1)  [root] (left) {};
  \node at (0,3)  [dot] (left1) {};
  \node at (-2,5)  [dot] (left2) {};
  \node at (-2,3) [dot] (variable3) {};
  
  \draw[testfcn] (left) to (root);

  \draw[kernel] (left1) to   (root);
  \draw[kernelBig] (left2) to  (left1);
  \draw[keps] (variable3) to  (left); 
  \draw[keps] (variable3) to   (left2); 
\end{tikzpicture}\;.
\end{equ}
We have modified the notations of \cite{KPZJeremy} slightly, the new blue vertex corresponds to the base point $\logof$ of the model and the green dot now represents the variable for the root of the tree. However all the kernel notation remains exactly the same. 

Switching back to our conventions the combinatorial tree $T^{0}_{0}$ corresponding to $\<211>$ is given by
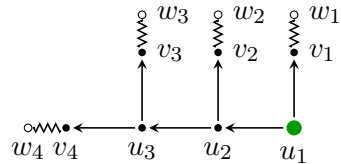
\begin{equ}
\begin{tikzpicture}
    \node [circ, label=below:$w_{4}$] (fifth anoise) at (-.5, 0) {};
    \node [circ, label=right:$w_{3}$] (fourth anoise) at (1, 1.5) {};

    \node [dot, label=below:$v_{4}$] (fifth noise) at (0, 0) {};
    \node [dot, label=right:$v_{3}$] (fourth noise) at (1, 1) {};

    \node [dot, label=below:$u_{3}$] (internal 4) at (1, 0) {};

    \node [dot, label=below:$u_{2}$] (internal 3) at (2, 0) {};
    \node [dot, label=right:$v_{2}$] (third noise) at (2, 1) {};
    \node [circ, label=right:$w_{2}$] (third anoise) at (2, 1.5) {};

    \node [root, label=below:$u_{1}$] (internal 2) at (3, 0) {};
    \node [dot, label=right:$v_{1}$] (second noise) at (3, 1) {};
    \node [circ, label=right:$w_{1}$] (second anoise) at (3, 1.5) {};

    \draw[kernel] (internal 2) to (second noise);
    \draw[leaf] (second noise) to (second anoise);
    \draw[kernel] (internal 2) to (internal 3);

    \draw[kernel] (internal 3) to (third noise);
    \draw[leaf]   (third noise) to (third anoise);
    \draw[kernel] (internal 3) to (internal 4);

    \draw[kernel] (internal 4) to (fourth noise);
    \draw[kernel] (internal 4) to (fifth noise);

    \draw[leaf] (fourth noise) to (fourth anoise);
    \draw[leaf] (fifth noise) to (fifth anoise);
\end{tikzpicture}
\end{equ}
One has $N(T) = \{u_{i}\}_{i=1}^{3} \sqcup \{v_{i}\}_{i=1}^{4}$ with $\rho_{T} = u_{1}$ and $L(T) = \{v_{i}\}_{i=1}^{4}$. 

We now describe what the RHS of \eqref{eq: chaos decomp} looks like in this example.
Assuming that we are considering centred Gaussian approximations to space-time white noise, 
only second order cumulants are non-zero, so that
any summand corresponding to a pair $(\wickleaves,\pi)$ not appearing in the list below does vanish. 
(In the general case considered in \cite{HS15}, one has to keep a few more terms.)
\begin{enumerate} 
\item $\wickleaves = \{v_{1},v_{2},v_{3},v_{4}\}$, $\pi = \emptyset$\;.
\item $\wickleaves = \{v_{1},v_{2}\}$, $\pi = \{\{v_{3},v_{4}\}\}$\;.
\item $\wickleaves = \{v_{3},v_{4}\}$, $\pi = \{\{v_{1},v_{2}\}\}$\;.
\item $\wickleaves = \{v_{1},v_{4}\}$, $\pi = \{\{v_{2},v_{3}\}\}$ + $v_{3} \leftrightarrow v_{4}$\;.
\item $\wickleaves = \{v_{2},v_{3}\}$, $\pi = \{\{v_{1},v_{4}\}\}$ + $v_{3} \leftrightarrow v_{4}$\;.
\item $\wickleaves = \emptyset$, $\pi = \{\{v_{1},v_{4}\},\{v_{2},v_{3}\}\}$ + $v_{3} \leftrightarrow v_{4}$\;.
\item $\wickleaves = \emptyset$, $\pi = \{\{v_{1},v_{2}\},\{v_{3},v_{4}\}\}$\;.
\end{enumerate}
Here, we write $v_{3} \leftrightarrow v_{4}$ to denote the same term with
$v_3$ and $v_4$ exchanged.

Concerning positive renormalizations, the set $\cut$ contains only a single edge: $\cut = \{(u_{1},u_{2})\}$.
Using the same numbering as above (and ignoring the permuted situations we didn't explicitly write out), we have the following possibilities for $\mathbb{F}_{\pi}$ 
\begin{enumerate} 
\item $\{\emptyset\}$\;.
\item $\{\emptyset, \{S_{1}\}\}$\;.
\item $ \{\emptyset,\{S_{2}\}\}$\;.
\item $\{\emptyset,\{S_{3}\}\}$.
\item $\{\emptyset\}$\;.
\item $ \{\emptyset,\{S_{3}\},\{T\},\{S_{3},T\}\}$
\item $\{\emptyset, \{S_{1}\}, \{S_{2}\}, \{S_{1},S_{2}\}, \{T,S_{1},S_{2}\}, \{S_{1},T\}, \{S_{2},T\}, \{T\}\}$.
\end{enumerate}
Here the trees $S_{1}$, $S_{2}$, and $S_{3}$ are respectively given by
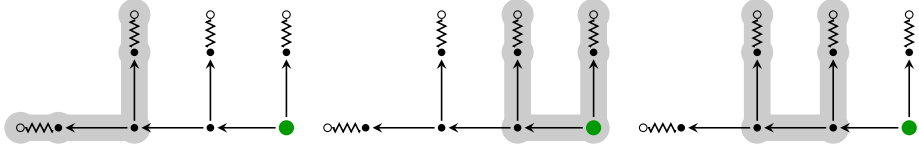
\[
\begin{tikzpicture}
    \node [subtreenode] (s fifth anoise) at (-.5, 0) {};
    \node [subtreenode] (s fourth anoise) at (1, 1.5) {};

    \node [subtreenode] (s fifth noise) at (0, 0) {};
    \node [subtreenode] (s fourth noise) at (1, 1) {};

    \node [subtreenode] (s internal 4) at (1, 0) {};

    \draw[subtreeedge] (s internal 4) to (s fourth noise);
    \draw[subtreeedge] (s internal 4) to (s fifth noise);

    \draw[subtreeedge] (s fourth noise) to (s fourth anoise);
    \draw[subtreeedge] (s fifth noise) to (s fifth anoise);

    \node [circ] (fifth anoise) at (-.5, 0) {};
    \node [circ] (fourth anoise) at (1, 1.5) {};

    \node [dot] (fifth noise) at (0, 0) {};
    \node [dot] (fourth noise) at (1, 1) {};

    \node [dot] (internal 4) at (1, 0) {};

    \node [dot] (internal 3) at (2, 0) {};
    \node [dot]  (third noise) at (2, 1) {};
    \node [circ] (third anoise) at (2, 1.5) {};

    \node [root] (internal 2) at (3, 0) {};
    \node [dot] (second noise) at (3, 1) {};
    \node [circ] (second anoise) at (3, 1.5) {};

    \draw[kernel] (internal 2) to (second noise);
    \draw[leaf] (second noise) to (second anoise);
    \draw[kernel] (internal 2) to (internal 3);

    \draw[kernel] (internal 3) to (third noise);
    \draw[leaf]   (third noise) to (third anoise);
    \draw[kernel] (internal 3) to (internal 4);

    \draw[kernel] (internal 4) to (fourth noise);
    \draw[kernel] (internal 4) to (fifth noise);

    \draw[leaf] (fourth noise) to (fourth anoise);
    \draw[leaf] (fifth noise) to (fifth anoise);
\end{tikzpicture}
\quad
\begin{tikzpicture}
    \node [subtreenode] (s internal 3) at (2, 0) {};
    \node [subtreenode]   (s third noise) at (2, 1) {};
    \node [subtreenode]  (s third anoise) at (2, 1.5) {};

    \node [subtreenode]  (s internal 2) at (3, 0) {};
    \node [subtreenode]  (s second noise) at (3, 1) {};
    \node [subtreenode]  (s second anoise) at (3, 1.5) {};

    \draw[subtreeedge] (s internal 2) to (s second noise);
    \draw[subtreeedge] (s second noise) to (s second anoise);
    \draw[subtreeedge] (s internal 2) to (s internal 3);

    \draw[subtreeedge] (s internal 3) to (s third noise);
    \draw[subtreeedge] (s third noise) to (s third anoise);

    \node [circ] (fifth anoise) at (-.5, 0) {};
    \node [circ] (fourth anoise) at (1, 1.5) {};

    \node [dot] (fifth noise) at (0, 0) {};
    \node [dot] (fourth noise) at (1, 1) {};

    \node [dot] (internal 4) at (1, 0) {};

    \node [dot] (internal 3) at (2, 0) {};
    \node [dot]  (third noise) at (2, 1) {};
    \node [circ] (third anoise) at (2, 1.5) {};

    \node [root] (internal 2) at (3, 0) {};
    \node [dot] (second noise) at (3, 1) {};
    \node [circ] (second anoise) at (3, 1.5) {};

    \draw[kernel] (internal 2) to (second noise);
    \draw[leaf] (second noise) to (second anoise);
    \draw[kernel] (internal 2) to (internal 3);

    \draw[kernel] (internal 3) to (third noise);
    \draw[leaf]   (third noise) to (third anoise);
    \draw[kernel] (internal 3) to (internal 4);

    \draw[kernel] (internal 4) to (fourth noise);
    \draw[kernel] (internal 4) to (fifth noise);

    \draw[leaf] (fourth noise) to (fourth anoise);
    \draw[leaf] (fifth noise) to (fifth anoise);
\end{tikzpicture}
\quad
\begin{tikzpicture}
    \node [subtreenode] (s fourth anoise) at (1, 1.5) {};
    \node [subtreenode] (s fourth noise) at (1, 1) {};

    \node [subtreenode] (s internal 4) at (1, 0) {};

    \node [subtreenode] (s internal 3) at (2, 0) {};
    \node [subtreenode]  (s third noise) at (2, 1) {};
    \node [subtreenode] (s third anoise) at (2, 1.5) {};

    \draw[subtreeedge] (s internal 3) to (s third noise);
    \draw[subtreeedge]   (s third noise) to (s third anoise);
    \draw[subtreeedge] (s internal 3) to (s internal 4);

    \draw[subtreeedge] (s internal 4) to (s fourth noise);

    \draw[subtreeedge] (s fourth noise) to (s fourth anoise);

    \node [circ] (fifth anoise) at (-.5, 0) {};
    \node [circ] (fourth anoise) at (1, 1.5) {};

    \node [dot] (fifth noise) at (0, 0) {};
    \node [dot] (fourth noise) at (1, 1) {};

    \node [dot] (internal 4) at (1, 0) {};

    \node [dot] (internal 3) at (2, 0) {};
    \node [dot]  (third noise) at (2, 1) {};
    \node [circ] (third anoise) at (2, 1.5) {};

    \node [root] (internal 2) at (3, 0) {};
    \node [dot] (second noise) at (3, 1) {};
    \node [circ] (second anoise) at (3, 1.5) {};

    \draw[kernel] (internal 2) to (second noise);
    \draw[leaf] (second noise) to (second anoise);
    \draw[kernel] (internal 2) to (internal 3);

    \draw[kernel] (internal 3) to (third noise);
    \draw[leaf]   (third noise) to (third anoise);
    \draw[kernel] (internal 3) to (internal 4);

    \draw[kernel] (internal 4) to (fourth noise);
    \draw[kernel] (internal 4) to (fifth noise);

    \draw[leaf] (fourth noise) to (fourth anoise);
    \draw[leaf] (fifth noise) to (fifth anoise);
\end{tikzpicture}
\]
The sum corresponding to scenario 1 gives the first term in \eqref{eq: from KPZJeremy}. 

The sum for scenario 2 vanishes for either of two reasons: one is that the renormalization of $S_{1}$ is a ``Wick''-type renormalization (so the term with $S_{1}$ renormalized precisely kills the one without) and a second reason is that the kernel $K'$ in \cite{KPZJeremy} is chosen to annihilates constants.

The sum for scenario 3 gives the second term and third term in \eqref{eq: from KPZJeremy}, the third term comes when one chooses $\mathscr{C} = \{(u_{1},u_{2})\}$ which obstructs the renormalization.

The sums for scenarios 4 and 5 give, respectively, the fourth and fifth terms in \eqref{eq: from KPZJeremy} with the factors $2$ coming from the permutation we mentioned.

For the sum in scenario 6 we first note that one has 
\begin{equs}
\sum_{\mcF \in \mathbb{F}_{\pi}} 
\CW^{\pi}_{\emptyset}[\mcF,\emptyset]
=&
\sum_{
\substack{
\mcF \in \mathbb{F}_{\pi}\\
\mcF \not \ni T}
} 
\big(
\CW^{\pi}_{\emptyset}[\mcF,\emptyset] 
+
\CW^{\pi}_{\emptyset}[\mcF \cup \{T\},\emptyset]
\big)\\
=&
\sum_{
\substack{
\mcF \in \mathbb{F}_{\pi}\\
\mcF \not \ni T}
}
\big( 
\CW^{\pi}_{\emptyset}[\mcF,\emptyset] 
-
\CW^{\pi}_{\emptyset}[\mcF,\emptyset]
\big)
=
0\;.
\end{equs}
This sort of cancellation is quite common when $\wickleaves = \emptyset$. The rest of the sum in scenario 6 (given by $\CW^{\pi}_{\emptyset}[\emptyset,\{(u_{1},u_{2})\}] + \CW^{\pi}_{\emptyset}[\{S_{3}\},\{(u_{1},u_{2})\}]$) gives the sixth term in \eqref{eq: from KPZJeremy}. 

Finally, the same arguments used for scenario 2 also lead to the vanishing of the sum in scenario 7. 
\section{Reorganizing sums}\label{Sec: reorganizing sums}

Our main goal moving forward is to obtain, for $\xi \in \mcM(\Omega_{\infty},\mfL_{\CCum})$, uniform in $\lambda \in (0,1]$, $L^{2p}$ estimates for the random variable $\hat{\Upsilon}^{\xi}[\sT^{\sn}_{\se}](\psi^{\lambda}_{z})$ for some continuous function $\psi$ as in Theorem~\ref{upgraded thm - main theorem} and $z \in \R^{d}$. 
Our main theorem promises estimates uniform in $\psi$ but this comes automatically\footnote{See \cite[Remark~10.8 and Sec.~3.1]{Regularity}.}
by choosing $\psi$ to be the generator of a wavelet basis, so from now on we treat $\psi$ as fixed.\martin{Still need to make sure that bounds only depend on finitely many
derivatives of $\psi$.} \ajay{The only time the root $\mainroot$ gets hit with a derivative is with fictitous renormalization when $\sT^{\sn}_{\se} = \Xi \tau$ - also, we never see derivatives in $\logof$}
The theorem also gives uniformity in $z$ but this will be automatic from stationarity.

If we define, for $\wickleaves \subset L(\sT)$ and $\pi \in \fillingpon{L(T) \setminus \wickleaves}$, the kernels 
$\mathcal{W}^{\pi}_{\wickleaves,\lambda} \in \mcb{C}_{\wickleaves \cup \{\logof,\mainroot\}}$ by setting, for each $x \in (\R^{d})^{\wickleaves \cup \{\logof,\mainroot\}}$, 
\begin{equation}\label{eq: chaos kernels}
\mathcal{W}^{\pi}_{\wickleaves,\lambda}(x)
\eqdef 
\sum_{
\substack{
\mcF \in \mathbb{F}_{\pi}\\
\cC \subset \cut_{\mcF}
}
} 
\mathcal{W}_{\wickleaves}^{\pi}[\mcF, \cC](x)
\cdot 
\psi^{\lambda}_{x_{\logof}}(x_{\mainroot})
\end{equation}
then the functions 
$\big\{
\sum_{\pi \in \fillingpon{L(\sT) \setminus \wickleaves}} \mathcal{W}^{\pi}_{\wickleaves,\lambda}
\big\}_{\wickleaves \subset L_{\sT}}$ 
play the role of chaos kernels in a non-Gaussian analogue to the Wiener chaos decomposition for the random variable\footnote{Being pedantic, this is true upon fixing $x_{\logof} = \bar{z}$ and then integrating $x_{\mainroot}$ if $\mainroot \not \in \wickleaves$.} $\hat{\Upsilon}^{\xi}_{\bar{z}}[\sT^{\sn}_{\se}](\psi^{\lambda}_{\bar{z}})$. 
In order to obtain the mentioned $L^{2p}$ estimates we try to get good control over the behavior of the kernels $\mathcal{W}^{\pi}_{\wickleaves}$ for fixed\footnote{Applying the triangle inequality to deal with the sum over $\wickleaves$ and $\pi$ introduces combinatorial factors that are unimportant in our context.} $\wickleaves \subset L(\sT)$ and $\pi \in \fillingpon{L(T) \setminus \wickleaves}$ and throughout Sections~\ref{Sec: reorganizing sums} to \ref{sec: est on moments of trees} we treat both of these parameters as fixed. 

In what follows we enforce the condition $\xi \in \mcM(\Omega_{\infty},\mfL_{\CCum})$, we then have that $\mathcal{W}^{\pi}_{\wickleaves}$ vanishes unless $\pi$ is a partition on $L(T) \setminus \wickleaves$ which satisfies the property that for every $B \in \pi$ one has $(\mft,B) \in \mfL_{\CCum}$. Therefore we assume this of $\pi$, note that this prevents $\pi$ from containing any singletons.

We also write $\mathcal{W}_{\lambda}[\mcF, \cC] \in \mcb{C}_{\wickleaves \cup \{\mainroot,\logof\}}$ for summand of \eqref{eq: chaos decomp} for fixed $\mcF \in \mbbF_{\pi}$ and $\cC \subset \cut_{\mcF}$. 

One does not expect to get good control over the LHS of \eqref{eq: chaos kernels} by inserting absolute values and controlling each summand on the right-hand side separately -- doing so would prevent us from harvesting any of the numerous cancellations created by our renormalization procedures. 
Exploiting a renormalization cancellation for an edge $e \in \cut$ or divergent subtree $S \in \Div$ requires us to estimate together
\begin{equation*}
\begin{split}
\mathcal{W}_{\lambda}[\mcF, \cC] &\textnormal{ and } \mathcal{W}_{\lambda}[\mcF, \cC \sqcup \{e\}]\\ 
\textnormal{ or }
\mathcal{W}_{\lambda}[\mcF, \cC] &\textnormal{ and }  \mathcal{W}_{\lambda}[\mcF \sqcup \{S\}, \cC].
\end{split}
\end{equation*}   
In many examples it is impossible to harvest all the renormalizations of $\cut$ and $\Div$ simultaneously, this is, in our context, the problem of ``overlapping divergences''. 
The solution to this problem is the observation that ``overlapping divergences don't overlap in phase space''. 
More concretely we perform a multiscale expansion of the quantities we wish to control and then design an algorithm which tells us which renormalizations to harvest for a fixed \emph{scale assignment} $\mbn$ as in \cite{KPZJeremy}. 
This algorithm requires some basic manipulations with partially ordered sets (posets).
To make notions clear, we introduce these concepts prematurely, 
that is before we introduce the multiscale expansion itself. 
This means that many of the notions of this section will be used in an $\mbn$-dependent way in the sequel.
 
All posets have a natural notion of ``intervals'', that is subsets of the poset which are either empty or contain unique minimal and maximal elements along with all elements of the poset that fall between these two extremal elements. 
\begin{definition}
Given a poset $\mathbb{A}$, $\mbbM \subset \mathbb{A}$ is said to be an interval of $\mathbb{A}$ if $\mbbM$ is either empty or there exist elements $s(\mbbM)$ and $b(\mbbM)$ such that
\begin{equ}
\mbbM = \{a \in \mathbb{A}\,:\, s(\mbbM) \le a \le b(\mbbM)\}\;.
\end{equ}
\end{definition}
For a non-empty interval $\mathbb{M}$ we also use the notation $\mbbM = \left[ s(\mbbM), b(\mbbM) \right]$. 
If elements of $\mathbb{A}$ happen to be sets themselves, equipped with the partial
order given by inclusion (which will always be the case for the intervals
we consider), we write $\delta(\mathbb{M}) \eqdef b(\mathbb{M}) \setminus s(\mathbb{M})$.
As a mnemonic one should think of $s(\cdot)$ and $b(\cdot)$ as standing for ``small'' and ``big'', respectively.  
In what follows we view $\mathbb{F}_{\pi}$ as a poset by equipping it with the inclusion partial order.
\begin{remark} If $\mathbb{M}$ is an interval of $\mathbb{F}_{\pi}$ then for any $\cC \subset \cut$ one has that $\mbbM \cap \mathbb{F}_{\cC,\pi}$ is an interval of $\mathbb{F}_{\pi}$.

This is immediate when $\mathbb{M} = \emptyset$ or $\mathbb{M}$ is non-empty and $s(\mbbM) \not \in \mathbb{F}_{\cC,\pi}$, in both situations $\mathbb{M} \cap \mathbb{F}_{\cC,\pi} = \emptyset$. 
If $\mbbM$ is non-empty with $s(\mbbM) \in \mathbb{F}_{\cC,\pi}$, then it is straightforward to check that
$\mbbM  
\cap 
\mathbb{F}_{\cC,\pi}
=
\left[ s(\mbbM),
b(\mbbM) \cap \Div_{\cC}
\right]$.
\end{remark}
For the rest of the paper an \emph{interval of forests} refers to an interval of $\mathbb{F}_{\pi}$.
The intervals of forests of interest to us are specified as pullbacks of certain 
maps from $\mathbb{F}_{\pi}$ to itself called forest projections.
\begin{definition}\label{def: forest projection}
A map $P: \mathbb{F}_{\pi} \rightarrow \mathbb{F}_{\pi}$ is said to be a forest projection if for any $\mcS \in \fullf$, $P^{-1}[\mcS]$ is either empty or an interval of forests $\mbbM$ with $s(\mbbM) = \mcS$.
\end{definition}
Given $\cC \subset \cut$ and a forest projection $P$, we sometimes use the shorthand $P_{\cC}^{-1}[\cdot] = P^{-1}[\cdot] \cap \mathbb{F}_{\cC,\pi}$. 
A simple observation that we use frequently is that 
\begin{equation}\label{cut observation}
\cC_{1} 
\subset \cC_{2} \subset \cut \quad\Rightarrow\quad P_{\cC_{2}}^{-1}[\mcS] \subset P_{\cC_{1}}^{-1}[\mcS] \quad \forall\,\mcS \in \mathbb{F}_{\pi}\;.
\end{equation}
For the rest of this subsection assume that we have fixed a forest projection $P$. 
Given any forest projection $P$ and $\cC \subset \fullcuts$ we define $\mfM^{P}(\cC) \subset 2^{\fullf}$ via setting
\begin{equation*}\label{equ: intervals arising from cut set}
\mfM^{P}(\cC) \eqdef\ 
\left\{ 
\mbbM \neq \emptyset \,:\, \exists \CF \in \mathbb{F}_\pi\,\text{with}\,
P^{-1}_{\cC}[\CF] = \mathbb{M}
\right\}\;.
\end{equation*}
Conversely, for any forest projection $P$ and any $\mbbM \subset \fullf$ we define 
\begin{equation*}\label{equ: cuts giving interval}
\fullcuts^{P}(\mbbM) \eqdef 
\{ \cC \subset \fullcuts:\ P^{-1}_{\cC}[s(\mbbM)] = \mbbM \}\;.
\end{equation*}
Clearly $\fullcuts^{P}(\mbbM)$ is empty unless $\mbbM$ is an interval of forests. 
Then, for any forest projection $P$ one can rewrite the set $\{(\CF,\cC)\,:\, \cC \subset \fullcuts\;\&\; \mcF \in \fmod{\cC}\}$ as 
\begin{equ}[e:cutsandforests]
\bigsqcup_{\cC \subset \fullcuts} \bigsqcup_{\mbbM \in  \mfM^{P}(\cC)}
\{(\CF,\cC)\,:\, \mcF \in \mbbM\}
\enskip
=
\enskip
\bigsqcup_{\mbbM \subset \fullf} \bigsqcup_{\cC \in  \fullcuts^{P}(\mbbM)}
\{(\CF,\cC)\,:\, \mcF \in \mbbM\}\;.
\end{equ}
We use the shorthand
$\mfM^{P} \eqdef \bigcup_{\cC \subset \fullcuts} \mfM^{P}(\cC)$.
Equipping $2^{\fullcuts}$ with the inclusion partial order,
we use the term \emph{interval 
of cuttings} to refer to an interval of $2^{\fullcuts}$.
\begin{remark}
In this paper the use of the term ``interval'' will always refer to an interval of the poset $\mathbb{F}_{\pi}$ or $2^{\cut}$. However, the reader should be cautious since these two posets are themselves collections of subsets of other posets -- $\Div$ and $\cut$, respectively.

Our maximum and minimum operations are always applied to \emph{collections of elements} of a poset. For example, given an interval of forests $\mathbb{M}$, $\overline{s(\mbbM)}$ denotes the collection of maximal subtrees of the forest $s(\mbbM)$ -- the selection of maximal elements is taking place with respect to the poset structure of $\Div$
since $s(\mbbM) \subset \Div$.
\end{remark} 
Before stating our next level we define, for any interval of forests $\mathbb{M}$, 
\begin{equation*}\label{def: cuts away from interval}
\cC_{\mathbb{M}} 
\eqdef
\fullcuts 
\setminus 
\Big(
\bigcup_{S \in \overline{b(\mathbb{M})}}
K(S)
\Big).
\end{equation*}
\begin{lemma} 
For any forest projection $P$ and $\mathbb{M} \in \mfM^{P}$ one has $\overline{\fullcuts^{P}(\mbbM)} = \{ \cC_{\mbbM} \}$.
\end{lemma} 
\begin{proof}  
First we observe that for any $\cC \in \fullcuts^{P}(\mbbM)$ one must have $\cC \subset \cC_{\mbbM}$ since $\mathbb{M} \subset \mathbb{F}_{\cC,\pi}$.
Next we show $\cC_{\mbbM} \in \fullcuts^{P}(\mbbM)$. Fixing some $\cC \in \fullcuts^{P}(\mbbM)$ (the latter set is non-empty by assumption) we have 
\[
P^{-1}_{\cC}[s(\mbbM)]
=
\mathbb{M}
\subset
P^{-1}_{\cC_{\mbbM}}[s(\mbbM)].
\] 
The subset relation is a consequence of the facts that $\mbbM \subset P^{-1}[s(\mbbM)]$ and $\mbbM \subset \mbbF_{\cC_{\mbbM},\pi}$. 
We also have $P^{-1}_{\cC_{\mbbM}}[s(\mbbM)] 
\subset 
P^{-1}_{\cC}[s(\mbbM)] 
= \mathbb{M}
$ by virtue of \eqref{cut observation}.
\end{proof}

\begin{definition}
A cut rule is a map $\cG: \mathbb{F}_{\pi} \rightarrow \fullcuts$ such that for each $\mcF \in \mathbb{F}_{\pi}$, $\cG(\mcF) \subset \fullcuts_{\mcF}$. 
\end{definition}

When we perform a multiscale expansion, each scale assignment $\mbn$ will determine a forest projection $P^{\mbn}$ which will be used to organize the sum over forests of divergent subtrees and a cut rule $\cG^{\mbn}$ for organizing the sum over cut sets.

We now formulate a criterion which will guarantee that conflicts do not arise between negative renormalizations and positive cuts we wish to harvest, that is $P^{\mbn}$ and $\cG^{\mbn} $ do not get in each other's way. 
\begin{definition}\label{compatforests}
Given a forest projection $P$ and a cut rule $\cG$ we say that $\cG$ and $P$ are \emph{compatible} if, for all  $\mbbM \in \mfM^{P}$, $e \in \cG(b(\mbbM))$, and $\cC \subset \fullcuts$, one has 
\[
\cC \cup \{e\} \in \fullcuts^{P}(\mbbM)
\quad\iff\quad
\cC \setminus \{e\} \in \fullcuts^{P}(\mbbM).
\]
\end{definition}
For any compatible $\cG$ and $P$ we define, for each $\mbbM \in \mfM^{P}$, a collection of intervals of cut sets
\[
\mathfrak{G}^{P}_{\cG}(\mbbM)
\eqdef
\left\{ 
\left[
\cC , \cC \sqcup \cG(b(\mbbM)
\right]:\ 
\cC \in \fullcuts^{P}(\mbbM),\ \cC \cap \cG(b(\mbbM) = \emptyset
\right\}\;.
\]
Note that, for each $\mbbM \in \mfM^{P}$, by compatibility one has that $\mathfrak{G}^{P}_{\cG}(\mbbM)$ is a partition  of $\fullcuts^{P}(\mbbM)$. 
The following proposition is essentially immediate.
\begin{proposition}\label{prop:sumcuts}
Given a forest projection $P$ and a compatible $\cG \subset \fullcuts$, one has 
\[
\bigsqcup_{\cC \subset \fullcuts}
\bigsqcup_{\mcF \in \fmod{\cC}}
\{ (\mcF, \cC)  \}
=
\bigsqcup_{\mbbM \in \mfM^{P}}
\bigsqcup_{\mathbb{G} \in \mathfrak{G}^{P}_{\cG}(\mbbM)}
\bigsqcup_{
\substack{
\mcF \in \mbbM\\
\cC \in \mathbb{G}
}
}
\{ (\mcF , \cC)  \}\;.
\]
\end{proposition}
\begin{proof}
It follows from \eqref{e:cutsandforests} that the left hand side
is equal to
\begin{equ}
\bigsqcup_{\mbbM \in \mfM^{P}} \bigsqcup_{\cC \in  \fullcuts^{P}(\mbbM)}\bigsqcup_{
\mcF \in \mbbM
}
\{(\CF,\cC)\}\;,
\end{equ}
since $\fullcuts^{P}(\mbbM)$ is empty unless $\mbbM \in \mfM^{P}$.
The claim then follows from the fact that all subsets in 
$\mathfrak{G}^{P}_{\cG}(\mbbM)$ are disjoint by definition and
every $\cC \in \fullcuts^{P}(\mbbM)$ has a unique decomposition 
$\cC = \cD \sqcup \cA$ with 
$\cD \in \fullcuts^{P}(\mbbM)$ and $ \cD \cap \cG(b(\mbbM) = \emptyset$,
as well as $\cA \subset \cG(b(\mbbM)$.
Indeed, simply set $\cD = \cC \setminus \cG(b(\mbbM)$ and $\cA = \cC \cap \cG(b(\mbbM)$: one then 
has $\cD \in \fullcuts^{P}(\mbbM)$ by the compatibility of $\cG$ and $P$.
\end{proof}
\section{Multiscale expansion} \label{Sec: multiscale expansion}
Given any set $A$ we denote by $A^{(2)}$ the collection of all two element subsets of $A$.

We also define
\[
\CE_{\pi} \eqdef
\bigsqcup_{B \in \pi} B^{(2)}\;.
\] 
When viewed as a set of edges, $\CE_{\pi}$ forms a union of complete graphs with vertex 
sets given by the true nodes of the blocks of $\pi$. 


We also define $\CE_{\logof} \eqdef \{ \{\logof,u\}: u \in N(\sT)\}$ and
\begin{equ}\label{def: set of edges for scale expansion}
\mcE \eqdef
K(\sT)
\sqcup
\CE_{\pi}
\sqcup
\CE_{\logof},
\end{equ}
For any subset $A \subset N(\sT)$ we write $\CE_{\logof}(A) \eqdef \{\{\logof,u\}:\ u \in A\} \subset \CE_{\logof}$.
\begin{remark}
Note that we often view the elements of $K(\sT)$ as sets of two element subsets of $N(\sT)$ in \eqref{def: set of edges for scale expansion}. 
When we are doing this there will be pairs $\{a,b\}$ that may appear ``twice'', once as an element of $K(\sT)$ and another as an element of $\CE_{\pi}$, however we see these occurrences as distinct and distinguishable. 
This should also be kept in mind when we present some subsets of $\CE$ as disjoint unions of subsets of $K(\sT)$, $\CE_{\pi}$, and $\CE_{\logof}$.
\end{remark}

We define a \emph{global scale assigment} for $\sT$ to be a tuple 
\[
\mbn = (n_{e})_{e \in \mcE} \in \N^{\mcE}\;.
\]
We fix $\psi:\R \rightarrow [0,1]$ to be a smooth function supported on $[3/8,1]$ and with the property that $\sum_{n \in \Z}\psi(2^{n}x) = 1$ for $x \not = 0$.
We then define another family of functions $\{\Psi^{(k)}\}_{k \in \N}$, $\Psi^{(k)}:\R^{d} \rightarrow [0,1]$ by 
setting $\Psi^{(k)}(0) = \delta_{k,0}$ and, for $x\neq 0$,
\[
\Psi^{(k)}(x) 
\eqdef 
\begin{cases}
\sum_{n \le 0} \psi(2^{n}|x|) & \textnormal{ if } k=0\\
\psi(2^{k}|x|) & \textnormal{ if } k \not = 0.
\end{cases} 
\]
Given $e \in \mcE$ with $e = \{a,b\}$ and a global scale assignment $\mbn$ we define $\Psi_{\mbn}^{e} \in \allf$ via $\Psi_{\mbn}^{e}(x_{a} - x_{b})$. 
Given $E \subset \mcE$  we define $\Psi_{\mbn}^{E} \in \allf$ via
\[
\Psi_{\mbn}^{E} 
\eqdef
\prod_{e \in E} \Psi^{e}_{\mbn}\;.
\]

We now define single scale analogues of some of the functions we introduced earlier,
as well as for the functions
\begin{equ}[e:kerHat]
\KerHat_{\mfe}^{\cD}
\eqdef
\prod_{e \in \cD}
\left( 
\ke{\{e\}}{\mfe}
+
\RKer^{\{e\}}_{\mfe}
\right).
\end{equ}
For any $\mbn \in \N^{\mcE}$ we set
\begin{equs}[2]
\ke{E}{\mfe,\mbn} 
&\eqdef
\ke{E}{\mfe}
\cdot
\Psi_{\mbn}^{E},&\qquad
\KerHat_{\mfe,\mbn}^{\cD}
&\eqdef
\KerHat_{\mfe}^{\cD}
\cdot
\Psi_{\mbn}^{\cD},\\
\RKer_{\mfe,\mbn}^{\cD}
&\eqdef
\RKer_{\mfe}^{\cD}
\cdot
\Psi_{\mbn}^{\cD},
&\qquad
\powroot{N}{\mfn}{v,\mbn} 
&\eqdef
\powroot{N}{\mfn}{v}
\cdot
\prod_{u \in N}
\Psi_{\mbn}^{\{\{\logof,u\}\}},
\end{equs}
as well as $\psi^{\lambda}_{\mbn} \in \mcb{C}_{\{\logof,\mainroot\}}$ via $\psi^{\lambda}_{\mbn}(x)
\eqdef\ 
\psi^{\lambda}_{x_{\logof}}(x_{\mainroot}) 
\cdot
\Psi_{\mbn}^{\{\logof,\mainroot\}}(x)$.

We now describe our multiscale expansion for cumulants.
Our expansion will be more involved for second cumulants.  Since we allow $\mft \in \Le$ to have $|\mft|_{\s} \in (-\frac{2}{3}|\s|, - |\s|)$ this means that if $\|\xi\|_{N,\c}$ only contained information about cumulants away from the diagonal then the quantity $\|\xi\|_{N,\c}$ could not be used to control, for arbitrary $\xi \in \mcM(\Omega_{\infty})$, $\delta > 0$, and $\mft, \mft' \in \Le$ with $|\mft|_{\s} + |\mft'|_{\s} \le - |\s|$, the quantity 
\[
\int_{
\substack{
y \in \R^{d}\\
|y|_{\s} < \delta
}
}
dx\ 
\Cum[\{\xi_{\mft}(0),\xi_{\mft'}(x)\}].
\]
In particular, if we use a brutal multiscale bounds on the above quantity and only use the information on $\|\xi\|_{N,\c}$ given away from the diagonal then by power-counting the above integral will look divergent.
However our norm $\|\xi\|_{N,\c}$ includes data on the behavior of the diagonal through \eqref{def: control over the diagonal} so there is no real problem here. 
In order to overcome this in our power-counting analysis, we perform a built-in ``fictitious''-renormalization of second cumulants which uses the data from \eqref{def: control over the diagonal}. 
For each $B \subset L(\sT)$ with $B = \{u,v\}$ and $|\mft(B)|_{\s} \le -|\s|$ we define, by applying the same trick as used in \cite[Lem.~A.4]{KPZJeremy}, a family of functions $\{\widetilde{\mathrm{Cu}}_{B,j}(x_{u},x_{v})\}_{j=0}^{\infty}$ with the property that each $\widetilde{\mathrm{Cu}}_{B,j}(x_{u},x_{v})$ is translation invariant (i.e., expressable as a function of $x_{u}- x_{v}$) and furthermore
\begin{enumerate}
\item The identity
\[
\sum_{j=0}^{\infty}
\widetilde{\mathrm{Cu}}_{B,j}
=
\Cum[\{\xi_{\mft(u)}, \xi_{\mft(v)}\}]
\]
holds in the sense of distributions on $(\R^d)^{B}$.
\item $\widetilde{\mathrm{Cu}}_{B,j}$ is supported on $(x_{u},x_{v})$ with $2^{-n-2} \le |x_{u} - x_{v}| \le 2^{-n}$.
\item One has, uniform in $j \ge 0$ and $x_{B} \in \R^{2d}$, 
\[
|\widetilde{\mathrm{Cu}}_{B,j}(x_{B})|
\lesssim
\| \xi \|_{2,\c}
\cdot
|x_{u} - x_{v}|^{|\mft(B)|_{\s}}\;.
\]
\item For any polynomial $Q$ on $\R^{d}$ of $\s$-degree strictly less than $- |\s| - |\mft(B)|_{\s}$ and every $j > 0$, one has
\begin{equ}\label{eq: fict renorm int vanishes}
\int dx\, \widetilde{\mathrm{Cu}}_{B,j}(0,x)
Q(x)
=
0.
\end{equ}
\end{enumerate}
\begin{remark}
Note that we do \textit{not} impose \eqref{eq: fict renorm int vanishes}
for $j = 0$. 
Indeed, this would in general be in contradiction with the first item. 
\end{remark}
For any $B \subset L(\sT)$ and any global scale assignment $\mbn$ we set
\[
\mathrm{Cu}_{B,\mbn}(x)
\eqdef
\begin{cases}
\widetilde{\mathrm{Cu}}_{B,n_{B}}(x)
&\ 
\textnormal{if }|B| = 2 \textnormal{ and }|\mft(B)|_{\s} \le -|\s|\\[1.5ex]
\Psi_{\mbn}^{B^{(2)}}(x) \cdot \Cum[\{ \xi_{\mft(u)}(x_{u}) \}_{u \in B}]
&\ 
\textnormal{otherwise}
\end{cases}
\]
Finally, for any $L \subset L(\sT)$ and global scale assignment $\mbn$ we set 
\[
\mathrm{Cu}^{L}_{\pi,\mbn} 
\eqdef
\prod_{
\substack{
B \in \pi,\\ 
B \subset L
}}
\mathrm{Cu}_{B,\mbn}
.
\]
We now define operators $\singleslicegenvert{\pi,\mbbM,S}{\mbn}: \allf \rightarrow \mcb{C}_{\tilde{N}(S)^{c}}$, where $\mbbM$ is an interval of forests, $S \in b(\mbbM)$, and $\mbn \in \N^{\mcE}$.
This definition is again recursive and for $\phi \in \allf$ we set, for any $x \in (\R^{d})^{\tilde{N}(S)^{c}}$, 
\begin{equation*}
\begin{split}
[\singleslicegenvert{\pi,\mbbM,S}{\mbn}\phi](x)\ 
\eqdef&\ 
\int_{\tilde{N}_{b(\mbbM)}(S)} dy\
\mathrm{Cu}^{L_{b(\mbbM)}(S)}_{\pi,\mbn}(y) 
\cdot
\ke{\mathring{K}_{b(\mbbM)}(S)}{0,\mbn}(x_{\rho_{S}} \sqcup y)\\
&\quad
\cdot
\singleslicegenvert{\pi,\mbbM,C_{b(\mbbM)}(S)}{\mbn}
\left[
\ke{K^{\partial}_{b(\mbbM)}(S)}{0,\mbn}
\cdot
[\mathscr{Y}_{S,\mbbM}^{\#}\phi]
\right](x \sqcup y)
\end{split}
\end{equation*}
where 
\begin{equation*}
\mathscr{Y}_{S,\mbbM}^{\#}\phi
=
\begin{cases}
- \mathscr{Y}_{S} \phi &\textnormal{ if } S \in s(\mbbM)\\
(\mathrm{Id} - \mathscr{Y}_{S})\phi &\textnormal{ if } S \in \delta(\mbbM).
\end{cases}
\end{equation*}
Recall that the operators $\mathscr{Y}_{S}$ were introduced in Definition~\ref{set of derivatives}.
The notation $\singleslicegenvert{\pi,\mbbM,C_{b(\mbbM)}(S)}{\mbn}$ denotes the composition of the operators $\{\singleslicegenvert{\pi,\mbbM,T}{\mbn}\}_{T \in C_{b(\mbbM)}(S)}$. 
Again, no order needs to be prescribed for this composition because these operators commute -- the justification of this claim is  essentially the same as that given for the commutation proved in Remark~\ref{remark: explanation of negative renormalization}.
The difference here is that some instances of $\mathscr{Y}_{S}$ are replaced by $(\mathrm{Id} - \mathscr{Y}_{S})$ and all the cumulants and kernels are replaced by single slices of their multiscale expansions but these changes make no real difference when checking that these operators commute. 

We can now define single slice, partially resummed chaos kernels. 
For any $\mbn \in \N^{\mcE}$, any interval of forests $\mathbb{M}$, and any interval of cuttings $\mathbb{G}$ with $b(\mbbG) \subset \cut_{b(\mbbM)}$, we define $\mathcal{W}_{\lambda}^{\mbn}[\mbbM, \mathbb{G}] \in \mcb{C}_{\wickleaves \cup \{\mainroot,\logof\}}$ via setting, for each $x \in (\R^{d})^{\wickleaves \cup \{\logof,\mainroot\}}$, 
\begin{equs}\label{multiscale - integrand}
\mathcal{W}^{\mbn}_{\lambda}&[\mbbM, \mathbb{G}](x)
\eqdef 
\int_{\nodesleft{b(\mbbM)}{\sT} \setminus 
(\wickleaves \cup \{\logof\})}
dy\ 
\psi^{\lambda}_{\mbn}(x \sqcup y)
\cdot
\powroot{\nodesleft{b(\mbbM)}{\sT}}{\sn}{\logof,\mbn}(x \sqcup y)\\
&
\cdot
\mathrm{Cu}^{\leavesleft{b(\mbbM)}{\sT}}_{\pi,\mbn}(x \sqcup y)
\,\ke{
\kernelsleft{b(\mbbM)}{\sT} \setminus b(\mathbb{G})
}{0,\mbn}(x \sqcup y)\\
&
\cdot
\,\RKer_{0,\mbn}^{s(\mbbG) \setminus K^{\downarrow}(\overline{b(\mbbM)})}(x \sqcup y)
\,\KerHat_{0,\mbn}^{\delta(\mbbG) \setminus K^{\downarrow}(\overline{b(\mbbM)})}(x \sqcup y)
\\
\cdot
\Big( 
& 
\prod_{S \in \mmax{b(\mbbM)}}
H_{\pi,\mbbM,S}^{\mbn} \left[ 
\RKer_{0,\mbn}^{s(\mbbG) \cap K^{\downarrow}(S)}
\,\KerHat_{0,\mbn}^{\delta(\mbbG) \cap K^{\downarrow}(S)}
\,\ke{K^{\downarrow}(S) \setminus b(\mathbb{G})}{0,\mbn}
\powroot{\tilde{N}(S)}{\sn,\logof}{\mbn}
\right](x \sqcup y)
\Big)\;.
\end{equs}
With our definitions the following lemma is straightforward.
\begin{lemma}\label{lem:reorderSum}\martin{Double-check that.}
For any $\cC \subset \cut$ and $\mcF \in \mathbb{F}_{\cC,\pi}$ one has
\begin{equ}\label{eq: full multiscale expansion}
\sum_{\mbn \in \N^{\mcE}}
\mathcal{W}^{\mbn}_{\lambda}\left[\{\mcF\}, \{\cC\} \right]
=
\mathcal{W}_{\lambda}[\mcF, \cC]\;,
\end{equ}
where for the sum on the the LHS we have absolute convergence pointwise in $x_{\wickleaves}$. 

On the other hand, for fixed $\mbn \in \N^{\mcE}$, interval of forests $\mathbb{M}$, and any interval of cut sets $\mathbb{G}$ with $b(\mbbG) \subset \cut_{b(\mbbM)}$, one has
\begin{equ}\label{eq: interval sum}
\sum_{
\substack{
\mcF \in \mathbb{M}\\
\cC \in \mathbb{G}}
}
\mathcal{W}_{\lambda}^{\mbn}\left[\{\mcF\}, \{\cC\} \right]
=
\mathcal{W}_{\lambda}^{\mbn}\left[\mathbb{M}, \mathbb{G} \right].
\end{equ}
\end{lemma}
\begin{proof}
Both statements are straightforward but we sketch how to show \eqref{eq: interval sum} since the notation there is somewhat heavy. 
Note that the fact that we stated this identity for a ``single scale'' slice (i.e. with an  $\mbn$) plays no role (a similar identity holds before the multiscale expansion) but we have stated things this way for later use.

The heuristic for \eqref{eq: interval sum} is as follows: for any interval $\mbbA$ of sets where the ordering is given by inclusion one has  
\[
\sum_{\mcA \in \mbbA} \prod_{a \in \mcA} (-  y_{a})
=
\Big(\prod_{a \in s(\mbbA)} (-y_{a}) \Big) 
\Big(\prod_{a \in \delta(\mbbA)} ( 1 -y_{a}) \Big) 
\]
where the $\{y_{a}\}_{a \in b(\mbbA)}$ are indeterminates which need not be commuting as long as one interprets the products on either side consistently. 

Fix $\mbn \in \N^{\CE}$. 
We first show that for any interval of forests $\mbbM$, 
\begin{equ}\label{eq: interval sum work}
\sum_{\mcF \in \mathbb{M}}
\mathcal{W}_{\lambda}^{\mbn}\left[\{\mcF\}, \emptyset \right]
=
\mathcal{W}_{\lambda}^{\mbn}\left[\mbbM, \emptyset \right]\;.
\end{equ}
We prove the above identity via induction on $|\delta(\mbbM)|$. 
The base case, which occurs when $|\mbbM| = 1$ and $|\delta(\mbbM)| = 0$, is immediate. 
For the inductive step fix $l > 0$, assume the claim has been proven whenever $|\delta(\mbbM)| < l$, and fix $\mbbM$ with $|\delta(\mbbM)| = l$.

Fix $T \in \delta(\mbbM)$, we prove the claim in the case where there exists $\tilde{T} \in b(\mbbM)$ with $T \in C_{b(\mbbM)}(\tilde{T})$. 
The case when there is no such $\tilde{T}$ is easier. 

Let $\mbbM_{1} \eqdef \{ \mcF \in \mbbM: \mcF \not \ni T\}$ and $\mbbM_{2} \eqdef \{ \mcF \sqcup \{T\}: \mcF \in \mbbM_{1} \}$. 
Note that $\mbbM_{1}$ and $\mbbM_{2}$ are intervals that partition $\mbbM$. Therefore, by our inductive hypothesis, 
\begin{equ}\label{eq: interval work 3}
\mathcal{W}_{\lambda}^{\mbn}[\mbbM, \emptyset ] = \mathcal{W}_{\lambda}^{\mbn}[\mbbM_{1}, \emptyset ] + \mathcal{W}_{\lambda}^{\mbn}[\mbbM_{2}, \emptyset ]\;.
\end{equ} 
Clearly one has, for all $S \in b(\mbbM)$ with $T \not \le S$, 
\begin{equ}\label{interval work 3.5}
H_{\pi,\mbbM_{1},S}^{\mbn} = H_{\pi,\mbbM_{2},S}^{\mbn}\;.
\end{equ}
Next we claim that for every $S \in b(\mbbM)$ with $S > T$ one has
\begin{equ}\label{eq: interval work 4}
H_{\pi,\mbbM_{1},S}^{\mbn} + H_{\pi,\mbbM_{2},S}^{\mbn} = H_{\pi,\mbbM,S}^{\mbn}\;. 
\end{equ}
Proving this claim finishes our proof since the combination of \eqref{interval work 3.5} and \eqref{eq: interval work 4} yields \eqref{eq: interval sum work}. 
We prove the claim using an auxilliary induction in 
\[
\mathrm{depth}(\{ S' \in b(\mbbM): S > S' > T\})\;.
\] 
The inductive step for this induction is immediate upon writing out both sides of \eqref{eq: interval work 4} and remembering that $\mathscr{Y}_{S,\mbbM_{1}}^{\#} = \mathscr{Y}_{S,\mbbM_{2}}^{\#} = \mathscr{Y}_{S,\mbbM}^{\#}$. 
What remains is to check base case of this induction which occurs when $S = \tilde{T}$. 

To obtain \eqref{eq: interval work 4} when $S = \tilde{T}$ we first observe that $C_{b(\mbbM)}(\tilde{T}) = C_{b(\mbbM_{2})}(\tilde{T}) =  C_{b(\mbbM_{1})}(\tilde{T})~\sqcup~\{T\}$ and then rewrite, for $i = 1,2$,  $H_{\pi,\mbbM_{i},\tilde{T}}^{\mbn}[\phi]$ as 
\begin{equs}\label{eq: final work for interval sum}
{}&
\int_{\tilde{N}_{b(\mbbM)}(\tilde{T})} dy\
\mathrm{Cu}^{L_{b(\mbbM)}(\tilde{T})}_{\pi,\mbn}(y)
\cdot
\ke{\mathring{K}_{b(\mbbM)}(\tilde{T})}{0,\mbn}(x_{\rho_{\tilde{T}}} \sqcup y)
\int_{\tilde{N}_{b(\mbbM)}(T)}dz\ 
\mathrm{Cu}^{L_{b(\mbbM)}(T)}_{\pi,\mbn}(z)\\
& \quad
\cdot
\ke{\mathring{K}_{b(\mbbM)}(\tilde{T})}{0,\mbn}(w)
\big(
\mathscr{Y}^{(i)}
\circ
H^{\mbn}_{\pi,C_{b(\mbbM)}(\tilde{T}) \setminus \{T\}}
[
\ke{K^{\partial}_{b(\mbbM)}(S)}{0,\mbn}
\mathscr{Y}_{\tilde{T},\mbbM}^{\#}\phi
]
\big)(w),
\end{equs}
where $w = x_{\rho_{\tilde{T}}} \sqcup y \sqcup z$,  $\mathscr{Y}^{(1)} = \Id$, and $\mathscr{Y}^{(2)} = (-\mathscr{Y}_{T})$. 
In obtaining \eqref{eq: final work for interval sum} for $i=1$ we used \eqref{eq: factorization}. 

The corresponding identity for summing over $\cC \in \mathbb{G}$ is easier to check, one just expands, for each $e \in \delta(\mbbG)$, the $\KerHat_{0,\mbn}^{\{e\}}$ appearing in the RHS of \eqref{multiscale - integrand} as $\ke{\{e\}
}{0,\mbn} + \RKer_{0,\mbn}^{\{e\}}$.  
\end{proof}
\begin{remark}
In what follows we will perform various interchanges between integrals and infinite sums over scale assignments with an implicit assumption of equality. 
Once we fix a final method of organizing these sums and integrals we will establish their absolute convergence which will a posteriori justify such interchanges. 
\end{remark}
\section{The projection onto safe forests and choosing positive cuts} \label{sec: Safe forest projection}
We now introduce the family of forest projections $\{P^{\mbn}\}_{\mbn \in \N^{\mcE}}$ used to organize the sum over forests, these are the ``projections onto safe forests'' of \cite{FMRS85}. 
For any subtree $S$ of $\sT$, we define the edge sets
\begin{equation*}
\begin{split}
\mcE_{\emptyset}^{\inte}(S) &\eqdef
K(S) \sqcup
\left\{ e \in \CE_{\pi}:\ 
e \subset N(S)
\right\}\\
\CE_{\emptyset}^{\exte}(S) &\eqdef
\{e \in \CE \setminus \mcE^{\inte}(S):\ e \cap N(S) \not = \emptyset \}.
\end{split}
\end{equation*}
Heuristically, one determines whether the divergent subtree $S$ needs to be renormalized under slice $\mbn$ based on the relative values of the quantities
\begin{equation}\label{eq: inner outer scales}
\mathrm{int}_{\emptyset}^{\mbn}(S) 
\eqdef
\min
\{n_{e}:\ e \in \mcE_{\emptyset}^{\inte}(S) \}
\enskip
\textnormal{and}
\enskip 
\mathrm{ext}_{\emptyset}^{\mbn}(S) 
\eqdef
\max
\{ n_{e}:\ e \in \CE_{\emptyset}^{\exte}(S)  \}\;.
\end{equation}
We describe the general idea in the scenario where $S$ is the only divergent structure.
In this scenario one would sum over the scales assignments in the following way: we freeze the value of the scales external to $\mathrm{ext}_{\emptyset}^{\mbn}(S)$ and perform a sum on scales internal to $S$, after this is done we sum over the scales external to $S$.

It is the sum over the scales internal to $S$ where our bounds blow up due to the divergence of $S$. 

In the parlance of \cite{FMRS85}, $S$ is \emph{dangerous} for the scale assignment $\mbn$ if
\begin{equ}\label{eq: inner > outer scales}
\mathrm{int}_{\emptyset}^{\mbn}(S) > \mathrm{ext}_{\emptyset}^{\mbn}(S).
\end{equ}
The contribution of $S$ for a given scale will be of the order $2^{\omega(S)\mathrm{int}_{\emptyset}^{\mbn}(S)}$, 
which means that the our estimates on the sum over internal scales, subject to the constraint 
\eqref{eq: inner > outer scales}, will diverge. These are the scale assignments for which we need to 
exploit the cancellation that occurs between the original integral and the counter-term we subtract for $S$ -- we need to 
harvest this negative renormalization.

On the other hand, the sum over internal scales subject to the constraint
\begin{equ}\label{eq: inner < outer scales}
\mathrm{int}_{\emptyset}^{\mbn}(S) \le \mathrm{ext}_{\emptyset}^{\mbn}(S)
\end{equ}
is not a problem, since we have an upper bound on our sum. 
When \eqref{eq: inner < outer scales} holds \cite{FMRS85} calls $S$ \emph{safe} for the scale assignment $\mbn$. We get a bound of order $2^{\omega(S)\mathrm{ext}_{\emptyset}^{\mbn}(S)}$ after performing the sum over scales internal to $S$. If $S$ is the only divergence, then our estimate on the entire integral, written as a product of exponential factors, will have the scale $\mathrm{ext}_{\emptyset}^{\mbn}(S)$ multiplied by a negative constant in the exponent -- this comes from the fact that any structure larger than $S$ would be power-counting convergent in this scenario. Therefore the sum over $\mathrm{ext}_{\emptyset}^{\mbn}(S)$ also creates no problem. 

Loosely speaking, the issue of overlapping negative renormalizations is clarified by working in this multiscale approach since, for a given scale assignment $\mbn$, the set of divergent structures which are dangerous for $\mbn$ form a forest. 

We now generalize the discussion above, as well as make it more concrete.
When working with forests $\mcF$ of divergences, the quantities appearing in 
\eqref{eq: inner outer scales} should be replaced by analogues which are computed ``mod $\mcF$'', that is 
they take the renormalizations of $\mcF$ into consideration with regards to the notions of internal and external. 
The forest projections $P^{\mbn}$ we use will map a forest $\mcF$ to the subset $\mcS \subset \mcF$ which consists of all the elements of $\mcF$ which are safe for $\mbn$ when the notions of internal and external are taken ``mod $\mcF$''. 
Defining the projection $P^{\mbn}$ requires us to introduce some notions. 

Fix a subtree $S$ and let $\mcF \in \mathbb{F}_{\pi}$. 
We now describe some useful notation.
We define the immediate ancestor of $S$ in $\mcF$ by
\[
A_{\mcF}(S)
\eqdef 
\begin{cases}
T & \textnormal{ if }\mathrm{Min}(\{\tilde{T} \in \mcF:\ \tilde{T} > S\}) = \{T\},\\
\sT^{\ast} & \textnormal{ if }\{\tilde{T} \in \mcF:\ \tilde{T} > S\} = \emptyset
\end{cases}
\]
where we view $\sT^{\ast}$ as an undirected multigraph with node set $\allnodes$ and edge set $\mcE^{\inte}(\sT^{\ast}) \eqdef \mcE$. 

We also define the edge sets
\begin{equs}
\mcE^{\inte}(C_{\mcF}(S)) 
&\eqdef \bigsqcup_{T \in C_{\mcF}(S)} \mcE^{\inte}(T),\\
\mcE^{\inte}_{\mcF}S) 
&\eqdef
\mcE^{\inte}(S)\setminus \mcE^{\inte}(C_{\mcF}(S)),\\
\textnormal{and}\enskip
\mcE^{\exte}_{\mcF}(S)
&\eqdef 
\mcE^{\exte}(S) 
\cap
\mcE^{\inte}(A_{\mcF}(S)).
\end{equs}
The definitions generalizing \eqref{eq: inner outer scales} are then
\begin{equs}[inner and outer scales def]
\mathrm{int}_{\mcF}^{\mbn}(S) 
&
\eqdef 
\min 
\{ n_{e}:\ 
e \in \mcE^{\inte}_{\mcF}S) 
\}\;,
\\
\mathrm{ext}_{\mcF}^{\mbn}(S)
&\eqdef
\max
\{
n_{e}:\  
e \in \mcE^{\exte}_{\mcF}(S) 
\}\;.
\end{equs}
We now introduce the forest projections we will use. For each $\mbn \in \N^{\CE}$, we 
define $P^{\mbn}: \mathbb{F}_{\pi} \rightarrow \mathbb{F}_{\pi}$ by
\begin{equation*}\label{def: projection onto safe forests}
P^{\mbn}[\mcF] \eqdef
\left\{ 
S \in \mcF:\  
\inte_{\mcF}^{\mbn}(S) \le \exte_{\mcF}^{\mbn}(S)
\right\}\;.
\end{equation*}
We now present the main arguments (Lemma~\ref{main classification lemma} and Proposition~\ref{prop: forest proj}) 
showing that $P^{\mbn}$ is a forest projection -- here we are adapting \cite[Lem.~2.2,~2.3]{FMRS85}. 
\begin{lemma}\label{main classification lemma}
For any $\mbn \in \N^{\CE}$, $\mcF \in \mathbb{F}_{\pi}$ and $S \in \mcF$, one has
\begin{equ}
\inte_{\mcF}^{\mbn}(S) =
\inte_{P^{\mbn}[\mcF]}^{\mbn}(S),\qquad
\exte_{\mcF}^{\mbn}(S) =
\exte_{P^{\mbn}[\mcF]}^{\mbn}(S).
\end{equ} 
\end{lemma}
\begin{proof}
Since $P^{\mbn}[\mcF] \subset \mcF$ and since both $\mcE^{\inte}_{\mcF}$ and $\mcE^{\exte}_{\mcF}$
are decreasing in $\mcF$, it is immediate that one has
\begin{equ}\label{easy inequality - classification}
\inte_{\mcF}^{\mbn}(S) \ge
\inte_{P^{\mbn}[\mcF]}^{\mbn}(S),\qquad
\exte_{\mcF}^{\mbn}(S) \le
\exte_{P^{\mbn}[\mcF]}^{\mbn}(S).
\end{equ}
We start by turning the first inequality of \eqref{easy inequality - classification} into an equality. Fix $S$ and $\mcF$ and define
\begin{equ}
K \eqdef \mcE^{\inte}(C_{P^{\mbn}[\mcF]}(S))\;,\qquad 
\tilde{K} \eqdef \mcE^{\inte}(C_{\mcF}(S))\;.
\end{equ}
The claim is non-trivial only if $\tilde K \neq K$.
It then suffices to show that for any $e \in \tilde{K} \setminus K$ there 
exists $\tilde{e} \in \mcE^{\inte}(S) \setminus \tilde{K}$ with $n_{\tilde{e}} \le n_{e}$. 

Fix such an $e$. Observe that if $T \in \mcF$ with $T < S$ and $\mcE^{\inte}(T) \ni e$ then 
it must be the case that $T \in \mcF \setminus P^{\mbn}[\mcF]$, moreover there
is at least one such tree $T$. 
In particular, there is a (unique) sequence of trees $T_{1},\dots,T_{k+1} \in \mcF \setminus P^{\mbn}[\mcF]$, 
with $k \ge 0$, such that
\begin{itemize}
\item $T_{1}$ is the minimal element of $\mcF$ with $e \in \mcE^{\inte}(T_{1})$, 
\item for $1 \le j \le k$ one has $T_{j+1} = A_{\mcF}(T_{j})$, and $T_{k+1} = S$.
\end{itemize}
By our assumptions, for $1 \le j \le k$ we have 
\begin{equation}\label{intermedwork classification}
\inte_{\mcF}^{\mbn}(T_{j}) > \exte_{\mcF}^{\mbn}(T_{j})\;.
\end{equation} 
For each such $j$, fix a choice of $e_{j} \in \CE^{\exte}(T_{j}) \cap K(T_{j+1})$. Then our desired claim follows by setting $\tilde{e} = e_{k}$ since this edge belongs to $\mcE^{\inte}(S) \setminus \tilde{K}$ and \eqref{intermedwork classification} implies
\[
n_{e} > n_{e_{1}} > n_{e_{2}} > \cdots > n_{e_{k}}\;.
\]

We now turn the second inequality of \eqref{easy inequality - classification} into an equality. 
For this, we start by introducing the shorthand
\begin{equ}
T \eqdef A_{\mcF}(S)\; \quad 
\textnormal{and}
\quad
\widetilde{T} \eqdef 
A_{P^{\mbn}[\mcF]}(S)\;.
\end{equ}
With this notation, the claim is trivial if $\mcE^{\exte}(S) \cap \big( \mcE^{\inte}(\tilde{T}) \setminus \mcE^{\inte}(T) \big)$ 
is empty so we suppose this is not the case. (In particular, this means that $A_{\mcF}(S) \not = \sT^{\ast}$.)

In a way anaolgous to above, it remains to show that for any 
$e \in \mcE^{\exte}(S) \cap \big( \mcE^{\inte}(\tilde{T}) \setminus \mcE^{\inte}(T) \big)$, one can find 
$\tilde{e} \in \mcE^{\exte}(S) \cap \CE^{\inte}(T)$ with $n_{\tilde{e}} \ge n_{e}$. 
Fixing such an $e$, we define similarly to before the unique sequence of subtrees 
$T_{1},\dots,T_{k}$ such that
\begin{itemize}
\item $T_{1} = S$ and, for $1 \le j \le k$, $T_{j+1} = A_\CF(T_j)$.
\item $T_{k}$ is the minimum element of $\mcF$ with $S \le T_{k}$ and $e \in \CE^{\exte}(T_{k}) \cap \CE^{\inte}(A_{\mcF}(T_{k}))$.
\end{itemize}
One then has $T_j \in \mcF \setminus P^{\mbn}[\mcF]$ for $j=1,\ldots,k$
so that, in particular, \eqref{intermedwork classification} holds.
We then set $e_k = e$ and, for each for $1 \le j < k$, we pick some $e_{j} \in  \CE^{\exte}(T_{j}) \cap \CE^{\inte}(T_{j+1})$ arbitrarily. It then follows again by \eqref{intermedwork classification} that $n_{e_{j}} > n_{e_{j+1}}$ so that,
by setting $\tilde{e} = e_{1}$ our claim is proved.
\end{proof} 
\begin{corollary} 
For any $\mbn \in \N^{\CE}$, $P^{\mbn} \circ P^{\mbn} = P^{\mbn}$.
\end{corollary}
We can now show that $P^{\mbn}$ is a forest projection.
\begin{proposition}\label{prop: forest proj}
For any $\mbn \in \N^{\CE}$, $P^{\mbn}$ is a forest projection.
\end{proposition}
\begin{proof}
Fix $\mbn$. 
By the previous corollary it follows that if for some $\mcS \in \mathbb{F}_{\pi}$ the set $(P^{\mbn})^{-1}[\mcS]$ is non-empty then $P^{\mbn}[\mcS] = \mcS$. Fix such an $\mcS$.

Since $P^{\mbn}[\mcF] \subset \mcF$ for every $\mcF \in \mathbb{F}_{\pi}$ it follows 
that $\mcS$ is the unique minimal element of $(P^{\mbn})^{-1}[\mcS]$, so all that is left 
to show is that this pullback is an interval. Define 
\begin{equation}\label{maximal dangerous extension}
\mcG 
\eqdef 
\left\{ T \in \Div:\ 
\{T\} \cup \mcS \in \mathbb{F}_{\pi} \textnormal{  and  }
\inte_{\mcS}^{\mbn}(T) > \exte_{\mcS}^{\mbn}(T)
\right\}\;.
\end{equation}
Since $P^{\mbn}[\mcS] = \mcS$, it follows from the definitions of $P^{\mbn}$
and $\mcG$ that $\mcG \cap \mcS = \emptyset$.
We claim that one also had $\mcG \sqcup \mcS \in \mathbb{F}_{\pi}$. 
For this, it suffices to show that for any $S_{1}, S_{2} \in \mcG$ one has $N(S_{1})$ and $N(S_{2})$ 
either nested or disjoint. 

Suppose on the contrary that these two sets are neither disjoint nor nested, it follows 
that $N(S_{1}) \setminus N(S_{2})$, $N(S_{2}) \setminus N(S_{1})$, and $N(S_{1}) \cap N(S_{2})$ 
are all non-empty. 
In particular, there must be two edges $e_{1} = (u_{1},v_{1})$ and $e_{2} = (u_{2},v_{2})$ 
with $e_{1} \in \CE^{\inte}(S_{1}) \cap  \CE^{\exte}(S_{2})$ and $e_{2} \in \CE^{\inte}(S_{2}) \cap \CE^{\exte}(S_{1})$.

Moreover, the condition that $S_{1},S_{2}$ be compatible with the forest $\mcS$ implies that 
$A_{\mcS}(S_{1}) = A_{\mcS}(S_{2})$ contains both $S_1$ and $S_2$ and that 
the node set given by the trees in $C_{\mcS}(S_{1})$ does not intersect the node set of
$S_2$, and vice-versa. As a consequence, 
for $i,j \in \{1,2\}$, $i\not = j$, one has $e_{i} \in \CE^{\inte}(A_{\mcS}(S_{j})) \setminus \mcE^{\inte}(C_{\mcS}(S_{i}))$.

Since $\exte_{\mcS}^{\mbn}(S_{1}) < \inte_{\mcS}^{\mbn}(S_{1})$ we must have $n_{e_{2}} < n_{e_{1}} $, on the other hand $\exte_{\mcS}^{\mbn}(S_{2}) < \inte_{\mcS}^{\mbn}(S_{2})$ forces the reverse inequality which gives us a contradiction, thus proving 
our claim that $\mcG \sqcup \mcS \in \mathbb{F}_{\pi}$ 

We complete the proof by showing that $(P^{\mbn})^{-1}[\mcS] = [ \mcS, \mcG \sqcup \mcS]$.
For this, we first show that for $\mcF$ such that $\mcS \subset \mcF \subset (\mcG \sqcup \mcS)$, 
one has $P^{\mbn}[\mcF] = \mcS$. 
Observe that $P^{\mbn}[\mcF] \subset \mcS$ since for any $T \in (\mcF \setminus \mcS) \subset \mcG$ one has
\begin{equation*}
\exte_{\mcF}^{\mbn}(T) 
\le 
\exte_{\mcS}^{\mbn}(T) 
< 
\inte_{\mcS}^{\mbn}(T) 
\le 
\inte_{\mcF}^{\mbn}(T).
\end{equation*}
The middle inequality is a consequence of $T \in \mcG$ and the outer 
inequalities are a consequence of $\mcS \subset \mcF$.

The fact that one also has $\mcS \subset P^{\mbn}[\mcF]$ follows from the chain of inequalities 
\[
\inte_{\mcF}^{\mbn}(T)
=
\inte_{P^{\mbn}[\mcF]}^{\mbn}[T]
\le 
\inte_{\mcS}^{\mbn}[T]
\le
\exte_{\mcS}^{\mbn}[T] 
\le 
\exte_{P^{\mbn}[\mcF]}^{\mbn}[T]
=
\exte_{\mcF}^{\mbn}[T]
\]
which hold for arbitrary $T \in \mcS$. Lemma~\ref{main classification lemma} gives the outermost equalities. 
Working inward, the next two inequalities follow from $P^{\mbn}[\mcF] \subset \mcS$ and the center inequality is a consequence of $P^{\mbn}[\mcS] = \mcS$.

All that is left is proving that $\mcG \sqcup \mcS$ is the unique maximal element of $(P^{\mbn})^{-1}[\mcS]$, 
namely that if $P^{\mbn}[\mcF] = \mcS$ then $(\mcF \setminus \mcS) \subset \mcG$.
For any $T \in \mcF \setminus \mcS$ we have indeed
\[
\exte_{\mcS}^{\mbn}[T]
=
\exte_{\mcF}^{\mbn}[T] 
< 
\inte_{\mcF}^{\mbn}[T] 
=
\inte_{\mcS}^{\mbn}[T]
\] 
where the outer equalities are given by Lemma~\ref{main classification lemma} and the middle inequality follows from assuming $T \in \mcF \setminus P^{\mbn}[\mcF]$.
\end{proof}  
We say a subset $\tilde{\CE} \subset \CE$ connects $u,v \in \allnodes$ if one can find a sequence of $e_{1},\dots,e_{k} \in \CE$ with $ u \in e_{1}$, $v \in e_{k}$ and $e_{j} \cap e_{j+1} \not = \emptyset$ for $1 \le j \le k-1$. 

Then given $u,v \in \allnodes$, $\mbn \in \N^{\CE}$, and $\mcF \in \mathbb{F}_{\pi}$, we define 
\begin{equ}\label{pathfinder scale}
\mbn_{\mcF}(u,v) 
\eqdef
\max
\Big\{ 
\min\{ n_{e}: e \in \CE' \setminus \CE^{\inte}(\mcF) \}:\ 
\CE' \subset \CE \textnormal{ connects }u,v
\Big\}\;.
\end{equ}
We now define, for each $\mbn \in \N^{\CE}$, by setting, for each $\mcF \in \mathbb{F}_{\pi}$,
\begin{equation}\label{def: cuts to harvest}
\cG^{\mbn}(\mcF) \eqdef
\{ e \in \fullcuts:\ 
\mbn_{\mcF}(\logof,e_{\p}) > \mbn_{\mcF}(e_{\p},e_{\ch})
\}\;.
\end{equation}
These represent precisely those edges for which there is a cancellation between
the term $\Ker^{\{e\}}_\mfe$ and its Taylor expansion $\RKer^{\{e\}}_\mfe$ appearing in \eqref{e:kerHat}.
We also introduce the shorthand $\mfM^{\mbn} \eqdef \mfM^{P^{\mbn}}$, $\mathfrak{G}^{\mbn}(\cdot) \eqdef \mathfrak{G}^{P^{\mbn}}_{{\cG}^{\mbn}}(\cdot)$, and $\cut^{\mbn}(\cdot) \eqdef \cut^{P^{\mbn}}(\cdot)$. 
We can now state our final combinatorial result.
\begin{proposition} For any $\mbn \in \N^{\CE}$ the forest projection $P^{\mbn}$ is compatible 
with the cut rule $\cG^{\mbn}$ in the sense of Definition~\ref{compatforests}.
\end{proposition}
\begin{proof} 
Fix $\mbn \in \N^{\CE}$ and $\mbbM \in \mfM^{\mbn}$. To keep notation light, for arbitrary $\cD \subset \fullcuts$ we write similarly to before $P^{-1}_{\cD}[\cdot]$ to denote $(P^{\mbn}_{\cD})^{-1}[\cdot] \cap \mathbb{F}_{\cD,\pi}$.  

First we prove that for any $e \in \cC_{\mbbM}$ and any $\cC \subset \fullcuts$
\[
\cC \setminus \{e\} \in \fullcuts^{\mbn}(\mbbM)
\quad\Rightarrow\quad
\cC \cup \{e\} \in \fullcuts^{\mbn}(\mbbM)\;.
\]
Fix $e \in \cC_{\mbbM}$ and let $\cC \in \fullcuts^{\mbn}(\mbbM)$ with $\cC \not \ni e$. 
The fact that $\cC \sqcup \{e \} \in \fullcuts^{\mbn}(\mbbM)$ follows from the observation that,
since $\cC \subset \cC_{\mbbM}$, one has the inclusions
\[
P^{-1}_{\cC_{\mbbM}}[s(\mbbM)] 
\subset 
P^{-1}_{\cC \sqcup \{e\}}[s(\mbbM)]
\subset
P^{-1}_{\cC}[s(\mbbM)]\;,
\]
where the outermost sets are both equal to $\mbbM$.

We now turn to the proof of the converse statement.  
Fix $e \in \cC_{\mbbM}$ and suppose that $\widetilde{\cC} \subset \cC_{\mbbM}$, $\widetilde{\cC} \ni e$, and $\widetilde{\cC} \in \fullcuts^{\mbn}(\mbbM)$.
We now additionally assume that $\widetilde{\cC} \setminus  \{e\} \not \in \fullcuts^{\mbn}(\mbbM)$ and show that this forces $e \not \in \cG^{\mbn}(b(\mbbM))$, thus establishing the claim. 
Our assumptions imply
\[
b\big(P^{-1}_{\widetilde{\cC} \setminus \{e\}}[s(\mbbM)]\big) \supsetneq
b\big(P^{-1}_{\widetilde{\cC} } [s(\mbbM)]\big).
\] 
For any element $T$ of the left hand side which is not an element of the right
hand side, it must be the case that $e \in \CE^{\inte}(T)$ and $T \not \le S$ for any $S \in b(\mbbM)$ (in particular, $T \not \in s(\mbbM)$). 
It follows that 
\[
\mbn_{b(\mbbM)}(\logof,e_{\p}) \le \exte^{\mbn}_{s(\mbbM)}(T) < \inte^{\mbn}_{s(\mbbM)}(T) \le \mbn_{b(\mbbM)}(e_{\ch},e_{\p}),
\] 
as a consequence of the fact that $T \not \in P^{\mbn}[s(\mbbM)]$ and $e \not \in \CE^{\inte}(S)$ for all $S \in s(\mbbM)$.
It follows that indeed $e \not \in \cG^{\mbn}(b(\mbbM))$ as announced.
\end{proof}
\section{Summing over scales}\label{Sec: summing over scales}
\subsection{Interchanging the sum over scales and intervals}
In Lemma~\ref{lem:reorderSum} we showed, in an abstract sense, how a sum over forests and cuts could be reorganized into a smaller sum over 
intervals of forests $\mbbM$ and intervals of cuts $\mbbG$. 
In the previous section this was made more concrete, we specified an algorithm which, given a scale assignment $\mbn$, 
specifies a particular forest projection $P^{\mbn}$ and compatible cut rule $\mathscr{G}^{\mbn}(\cdot)$ for which we can apply Lemma~\ref{lem:reorderSum}.  

We now reorganize our sums again, first summing over pairs $(\mbbM,\mbbG)$ of intervals of forests and cuttings and then summing over the scale assignments $\mbn$ which allow a given pair of intervals to arise from from $P^{\mbn}$ and $\cG^{\mbn}$.
To this end we define 
\[
\mathfrak{R} \eqdef 
\left\{
(\mbbM, \mathbb{G}) 
\in 
2^{\mathbb{F}_{\pi}} 
\times
2^{2^{\fullcuts}}:\ 
\exists \mbn \in \N^{\CE} 
\textnormal{ such that } 
\mathbb{M} \in \mfM^{\mbn} \textnormal{ and }
\mathbb{G} \in \mathfrak{G}^{\mbn}(\mbbM)
\right\}\;
\] 
and, for any $(\mbbM, \mathbb{G}) \in \mathfrak{R}$ and $\lambda \in (0,1]$, we set
\[
\mcN_{\mbbM,\mathbb{G},\lambda}
\eqdef 
\left\{ 
\mbn \in \N^{\CE}:\ 
\mbbM \in \mfM^{\mbn},\  
\mathbb{G} \in \mathfrak{G}^{\mbn}(\mbbM),\ 
\textnormal{ and }
n_{\{\logof,\mainroot\}} \ge \lfloor - \log_{2}(\lambda) \rfloor
\right\}\;.
\]
We then have the following lemma.
\begin{lemma}\label{lemma: interval, scale expansion}
For any $\lambda \in (0,1]$, one has
\[
\mcW^{\pi}_{\wickleaves}
=
\sum_{
(\mathbb{M},\mathbb{G}) \in \mathfrak{R}
}
\sum_{\mbn \in \mcN_{\mbbM,\mathbb{G},\lambda}}
\mathcal{W}_{\lambda}^{\mbn}\left[\mathbb{M},\mathbb{G} \right]
\]
where the LHS is defined as in \eqref{eq: chaos kernels}. 
\end{lemma}
\begin{proof}
By Lemma~\ref{lem:reorderSum} and Proposition~\ref{prop:sumcuts}, 
and furthermore freely interchanging finite sums with infinite ones,  
\begin{equs}
\mcW^{\pi}_{\wickleaves}
=&
\sum_{\mfn \in \N^{\CE}}
\sum_{
\substack{
\mcF \in \mbbF_{\pi}\\
\mathscr{C} \subset \cut_{\mcF}
}
}
\mathcal{W}^{\mbn}_{\lambda}\left[\{\mcF\}, \{\cC\} \right]
= 
\sum_{\mfn \in \N^{\CE}}
\sum_{
\substack{
\mbbM \in \mfM^{\mbn}\\
\mbbG \in \mathfrak{G}^{\mbn}(\mbbM)
}
}
\mathcal{W}^{\mbn}_{\lambda}\left[\mbbM, \mbbG \right]\\
=& 
\sum_{\mfn \in \N^{\CE}}
\sum_{
(\mbbM,\mbbG) \in \mfR
}
\mathbbm{1}
\left\{
\begin{array}{c}
\mbbM \in \mfM^{\mbn},\\ 
\mbbG \in \mathfrak{G}^{\mbn}(\mbbM)
\end{array}
\right\}
\mathcal{W}^{\mbn}_{\lambda}\left[\mbbM, \mbbG \right]\\
=&
\sum_{
(\mbbM,\mbbG) \in \mfR
}
\sum_{\mfn \in \N^{\CE}}
\mathbbm{1}
\left\{
\begin{array}{c}
\mbbM \in \mfM^{\mbn},\\ 
\mbbG \in \mathfrak{G}^{\mbn}(\mbbM)
\end{array}
\right\}
\mathcal{W}^{\mbn}_{\lambda}\left[\mbbM, \mbbG \right]
=
\sum_{
(\mbbM,\mbbG) \in \mfR
}
\sum_{\mfn \in \mcN_{\mbbM,\mathbb{G},\lambda}}
\mathcal{W}^{\mbn}_{\lambda}\left[\mbbM, \mbbG \right].
\end{equs}
Note that the constraint that $n_{\{\logof,\mainroot\}} \ge \lfloor - \log_{2}(\lambda) \rfloor$ comes for free since the presence of $\psi^{\lambda}$ makes $\mathcal{W}^{\mbn}_{\lambda}\left[\cdot, \cdot \right]$ vanishes if this is not the case.
\end{proof}
The goal of the subsequent sections is to prove the following theorem.
\begin{proposition}\label{prop: main estimate}
For any $(\mathbb{M},\mathbb{G}) \in \mathfrak{R}$, and any $p \in \N$ there exists $C_{p}$ such that, uniform in $\xi \in \mcM(\Omega_{\infty})$ and $x_{\logof} \in \R^{d}$, one has
\begin{equs}\label{eq: main est}
{}
&
\E
\Big[
\Big(
\sum_{\mbn \in \mcN_{\mbbM,\mathbb{G},\lambda}}
\int_{\wickleaves \cup \{\mainroot\}} \!\! dx\ 
\mathcal{W}_{\lambda}^{\mbn}[\mathbb{M},\mathbb{G}](x)
\,
\wwick{\{\xi_{\mft(u)}(x_{u})\}_{u \in \wickleaves}}
\Big)^{2p}\Big]
\\
&
\le
\ 
C_{p}
\Big(
\prod_{e \in K(\sT)}
\|K_{\mft(e)}\|_{|\mft(e)|_{\s},m}
\Big)^{2p}
\|\xi\|_{j,\c}
\lambda^{2p\alpha}
\end{equs}
uniform in $\lambda \in (0,1]$.
Here $m \eqdef 2|\s|\cdot |N(\sT)|$, $j \eqdef 2p |L(\sT)|$, and $\alpha = |\sT^{\sn}_{\se}|_{\s}$.
\end{proposition}
We pick an arbitrary $(\mathbb{M},\mathbb{G}) \in \mathfrak{R}$ and treat it as fixed for the remainder of the paper, writing $\mathcal{W}^{\mbn}_{\lambda}$ instead of $\mathcal{W}^{\mbn}_{\lambda}\left[\mbbM, \mbbG \right]$.
We also fix the shorthands $\mcS \eqdef s(\mbbM)$, $\mcB \eqdef b(\mbbM)$, $\mcD \eqdef \delta(\mbbM)$, $\cS \eqdef s(\mathbb{G})$, $\cB \eqdef b(\mathbb{G})$, and $\cD \eqdef \delta(\mathbb{G})$. 
\subsection{Summing over scales inductively}\label{subsec: inductively summing scales}
In order to efficiently prove a theorem that applies for any value of $\mathrm{depth}(\mcB)$ it is natural to
formulate our analytic scheme in a way that operates inductively with respect to this quantity. 

To that end we will decompose the single sum over global scale assignments $\mcN_{\mbbM,\mathbb{G},\lambda}$ into a family of sums which facilitate summing the scales internal to single $T \in \mcB$ conditioned on the values of relevant external scales (these two sets of quantities being dependent through the requirement that $T \in \mcS$ or $T \in \mcD$). 
In what follows, for any $\tilde{\CE} \subset \CE$ and $\mbj \in \N^{\CE}$ we write $\mbj \restr {\CE'} \in \N^{\CE'}$ for the restriction of $\mbj$ to the edges in $\CE'$. 

The most external set of edges is given by
\begin{equ}\label{external edges for a single tree}
\CE^{\inte}_{\mcB}(\sT^{\ast})
\eqdef
(K^{\downarrow}(\overline{\mcB}) \cup \kernelsleft{\mcB}{\sT})
\sqcup 
\{e \in \CE_{\pi}:\ e \subset \leavesleft{\mcB}{\sT} \}
\sqcup 
\CE_{\logof}\;.
\end{equ}
We now introduce some sets of ``partial'' scale assignments, writing  
\[
\d\mcN_{\mcB,\lambda} \eqdef
\{
\mbk \in \N^{\CE^{\inte}_{\mcB}(\sT^{\ast})}:\ 
\exists \mbj \in \mcN_{\mbbM,\mathbb{G},\lambda} \textnormal{ with } \mbj \restr {\CE^{\inte}_{\mcB}(\sT^{\ast})} = \mbk
\}
\]
and, for any $S \in \mcB$, $\CE' \subset \CE$ with $\CE' \supset \CE^{\exte}_{\mcB}(S)$, and $\mbj \in \N^{\CE'}$, setting
\begin{equ}[e:defNring]
\mathring{\mcN}_{S}(\mbj) 
\eqdef
\left\{ \mbk \in  \N^{\CE^{\inte}_{\mcB}(S)}:\ 
\exists \tilde{\mbj} \in \mcN_{\mbbM,\mathbb{G},\lambda} 
\textnormal{ with }
\tilde{\mbj} \restr {\CE'} = \mbj 
\textnormal{ and }
\tilde{\mbj}
\restr {\CE^{\inte}_{\mcB}(S)} 
=
\mbk
\right\}\;.
\end{equ}
Note that for every $\mbk \in \mathring{\mcN}_{S}(\mbj)$ one then has 
\begin{equs}[2]
\inte^{\mbk}_{\mcB}(S) &\le \exte^{\mbj}_{\mcB}(S) &\qquad&\textnormal{ if }S \in \mcS,\\
\inte^{\mbk}_{\mcB}(S) &> \exte^{\mbj}_{\mcB}(S) &\qquad&\textnormal{ if }S \in \mcD.
\end{equs}
The following lemma verifies that the set of scale assignments $\mathring{\mcN}_{S}(\mbj)$ 
really only depends on $\mbj$'s values on the edges of of $\CE^{\exte}_{\mcB}(S)$. 
\begin{lemma}\label{lem: independence of scale sums}
For any $S \in \mcB$ and $\mbj \in \mcN_{\mbbM,\mathbb{G},\lambda}$ one has 
\begin{equ}\label{eq: scale independence}
\mathring{\mcN}_{S}(\mbj \restr {\CE^{\exte}_{\mcB}(S)})
=
\mathring{\mcN}_{S}(\mbj).
\end{equ}
\end{lemma}
\begin{proof}
Clearly the RHS of \eqref{eq: scale independence} is contained on the LHS, we now prove the reverse inclusion.
 Fix $ \mbk \in \mathring{\mcN}_{S}(\mbj \restr {\CE^{\exte}_{\mcB}(S)})$.
We are guaranteed the existence of corresponding $\tilde{\mbj} \in \mcN_{\mbbM,\mathbb{G},\lambda}$ with $\tilde{\mbj} \restr {\CE^{\exte}_{\mcB}(S)} = \mbj \restr {\CE^{\exte}_{\mcB}(S)}$ and $ \tilde{\mbj} \restr {\CE^{\inte}_{\mcB}(S)} = \mbk$. 
We then define $\hat{\mbj} \in \N^{\CE}$ by setting $\hat{\mbj} = ( \tilde{\mbj} \restr {\mcE^{\inte}(S)}) \sqcup ( \mbj \restr {\CE \setminus \mcE^{\inte}(S)})$.
To finish our proof we need to show that $\hat{\mbj} \in \mcN_{\mbbM,\mbbG,\lambda}$. 

To prove that $\mbbM \in \mfM^{\hat{\mbj}}$ it suffices to show that
$P^{\hat{\mbj}}[\mcB] = \mcS$ and that for any $\mathscr{C} \in \mbbG$ and any $T \in \Div_{\pi} \setminus \mcB$, which is both compatible with $\mcB$ and does not contain any edge of $\mathscr{C}$, one has 
$\inte_{\mcB}^{\hat{\mbj}}(T)
\le
\exte_{\mcB}^{\hat{\mbj}}(T)$. 
These two statements can together be rewritten as the claim that for every $T \in \Div_{\pi}$ compatible with $\mcB$ and disjoint from $\mathscr{C}$, one has
\[
\begin{cases}
\inte_{\mcB}^{\hat{\mbj}}(T) \le \exte_{\mcB}^{\hat{\mbj}}(T)
&
\textnormal{ if } T \not \in \mcD \\[1.5ex]
\inte_{\mcB}^{\hat{\mbj}}(T) > \exte_{\mcB}^{\hat{\mbj}}(T)
&
\textnormal{ if } T \in \mcD.
\end{cases}
\] 
This claim can be checked in the following four cases: (i) $T$ is disjoint from all elements of $\mcF$, (ii) $T$ 
properly contains at least one element of $\mcF$, (iii) $T$ is properly contained in an element of $\mcF$, or (iv) $T$ is an element of $\mcF$.

In the first two cases our claim follows from the fact that both $\CE^{\inte}_{\mcB}(T)$ and $\CE^{\exte}_{\mcB}(T)$ consist of edges where $\hat{\mbj}$ is determined by $\mbj \in \mcN_{\mbbM,\mbbG,\lambda}$, while for the last two cases it is because $\hat{\mbj}$ is determined by $\tilde{\mbj} \in \mcN_{\mbbM,\mbbG,\lambda}$ on these edges. 

Clearly all edges whose scale assignments are involved determining $\mathscr{G}^{\hat{\mbj}}(\overline{\mcB})$ are edges where $\hat{\mbj}$ is determined by $\mbj$, so it follows that $\mbbG \in \mathfrak{G}^{\hat{\mbj}}(\mbbM)$. 
Finally, we also have $\hat{j}_{\{\logof,\mainroot\}} = j_{\{\logof,\mainroot\}} \ge \lfloor - \log_{2}(\lambda) \rfloor$. 
This shows that $\hat{\mbj} \in \mcN_{\mbbM,\mbbG,\lambda}$ as desired.
\end{proof} 
We immediately have the following corollary. 
\begin{corollary}\label{cor: tower property}
For any $F: \mcN_{\mbbM,\mbbG,\lambda}: \rightarrow \R$ one has
\[
\sum_{\mbn \in \mcN_{\mbbM,\mbbG,\lambda}}F(\mbn)
=
\sum_{\mbn^{(0)} \in A_{0}}
\sum_{\mbn^{(1)} \in A_{1}(\mbn^{(0)})}
\cdots
\sum_{\mbn^{(j)} \in A_{j}(\mbn^{(j-1)})}
F( \mbn^{(0)} \sqcup \mbn^{(1)} \sqcup \cdots \sqcup \mbn^{(j)})
\]
where $j \eqdef \mathrm{depth}(\mcB)$, $A_{0} \eqdef \d\mcN_{\mcB,\lambda}$, and for each $1 \le i \le j$ we inductively set, for each $\mbn^{(i-1)} \in A_{i-1}$,  
\[
A_{i} \eqdef \bigtimes_{S \in D_{i}(\mcB)} \mcN_{S}(\mbn^{(i-1)}).
\]
\end{corollary}
Here, we implicitly make the identifications $(\mbk_1,\ldots,\mbk_j) \sim \mbk_1 \sqcup\ldots \sqcup \mbk_j$
for scale assignments $\mbk_i$ that involve disjoint sets of edges.

We now inductively define a family of operators $\hat{H}^{\mbj}_{S}:\allf \rightarrow \allf$, where $S \in \mcB$ and $\mbj \in \N^{\CE'}$ with $\CE' \supset \CE^{\exte}_{\mcB}(S)$ by setting
\begin{equation}\label{genvert def}
\begin{split}
[\hat{H}^{\mbj}_{S}\phi](x)\ 
\eqdef\  
&
\sum_{ \mbk \in \mathring{\mcN}_{S}(\mbj)}
\int_{\tilde{N}_{\mcB}(S)}\back dy \
\mathrm{Cu}^{L_{\mcB}(S)}_{\pi,\mbk}(y) 
\,
\ke{\mathring{K}_{\mcB}(S)}{0,\mbk}(y \sqcup x_{\rho_{S}})\\
&\quad
\times
\hat{H}^{\mbk}_{C_{\mcB}(S)}
\left[
\ke{K^{\partial}_{\mcB}(S)}{0,\mbk}
\,
[\mathscr{Y}_{S,\mbbM}^{\#}\phi]
\right](x_{\tilde{N}(S)^c} \sqcup 
y)\;,
\end{split}
\end{equation}
with the base case of the induction given by setting $\hat{H}^{\mbj}_{\emptyset}$ to be the identity operator. The operators $\hat{H}^{\mbj}_{S}$ are partially summed analogues of the operators $H^{\cdot}_{\mbbM,S}$ which perform the summation of scales inside $S$ but outside $C_{\mcB}(S)$ (note that this operator contains in its definition all expressions within $\mathcal{W}_{\lambda}^{\bullet}$ which depend on these scales). 
We define, for each $\mbj \in \partial \mcN_{\mcB,\lambda}$, a function $\mathring{\CW}_{\lambda}^{\mbj} \in \mcb{C}_{\tilde{N}(\overline{\mcB})^{c}}$ via
\begin{equation}\label{def of inductively summed kernels}
\begin{split}
\mathring{\CW}_{\lambda}^{\mbj}
\eqdef&\; 
\psi^{\lambda}
\cdot
\ke{
\kernelsleft{\mcB}{\sT}\setminus \cB}{0,\mbj}
\mathrm{Cu}^{\leavesleft{\mcB}{\sT}}_{\pi,\mbj}
\KerTilde_{0,\mbj}^{\cB \setminus K^{\downarrow}(\overline{\mcB})}
\powroot{\nodesleft{\mcB}{\sT}}{\sn,\logof}{\mbj}\\
&
\quad \cdot
\prod_{S \in \overline{\mcB}}
\hat{H}_{S}^{\mbj} \left[ 
\KerTilde_{0,\mbj}^{\cB \cap K^{\downarrow}(S)} 
\ke{K^{\downarrow}(S) \setminus \cB}{0,\mbj}
\powroot{\tilde{N}(S)}{\sn,\logof}{\mbj}
\right]\;,
\end{split}
\end{equation}
where for any $\cC \subset \cut$ and $\mbj \in \partial \mcN_{\mcB,\lambda}$ we set 
\[
\KerTilde_{0,\mbj}^{\cC} 
\eqdef 
\RKer_{0,\mbj}^{\cC \cap \cS}
\cdot
\KerHat_{0,\mbj}^{\cC \cap \cD}\;.
\]
Lemma~\ref{lem: independence of scale sums} and Corollary~\ref{cor: tower property} together give the following lemma.
\begin{lemma}\label{inductively summed integrand}
\begin{equation}\label{inductive sum identity}
\sum_{\mbj \in \partial \mcN_{\mcB,\lambda}}
\int_{\nodesleft{\mcB}{\sT} \setminus \wickleaves}dy\ 
\mathring{\mathcal{W}}_{\lambda}^{\mbj}(x \sqcup y)
= 
\sum_{\mbn \in \mcN_{\mbbM,\mathbb{G},\lambda}}
\mathcal{W}_{\lambda}^{\mbn}(x )\;.
\end{equation}
\end{lemma}
\subsection{Estimates on renormalization}\label{sec: inductive bound}
\subsubsection{More notational preliminaries}
In what follows we will frequently use the generalized Taylor remainder estimate of \cite[Prop.~A.1]{Regularity}. 
In view of this, it is natural to define, for any set of multi-indices $A$ and subset $N \subset \allnodes$, the set of multi-indices
\[
\partial_{N} A 
\eqdef
\left\{ 
k \in (\N^{d})^{\allnodes} \setminus A: \exists j \in (\N^{d})^{N} \textnormal{ with } 
\|j\| = 1,\ k - j \in A \right\}\;.
\]
Note that if $N = \{u\}$ we sometimes write $\partial_{u}A$ instead of $\partial_{\{u\}}A$.

We now introduce some notation for the renormalization of second cumulants. First we set $
R(\pi)
\eqdef
\{ B \in \pi:\ \fict(B) > 0\}$ where $\fict(\cdot)$ is defined in \eqref{def: fict}.

The second cumulants corresponding to $B \in R(\pi)$ will be renormalized in a manner similar to the divergent subtrees -- in particular one needs to choose a distinguished vertex for each such $B$ which serves the same rule that $\rho_{S}$ does for $S \in \Div$.
To that end for each $B \in R(\pi)$ we label the elements of $B$ with either a $+$ or $-$, so for each such $B$ we can write $B = \{ (+,B), (-,B) \}$ -- this labeling is arbitrary except for one constraint: we require that for every $S \in \mcB$ and every $B \in L_{\mcB}(S)$ one has $(+,B) \not = \rho_{S}$. 
One should think of $(-,B)$ as serving as the ``root'' of $B$ for renormalization. 

Moving forward we take these labelling as fixed. 
We then define, for every $B \in R(\pi)$, an operator $\mathscr{Y}_{B}: \allf \rightarrow \allf$ via setting, for each $\phi \in \allf$, 
\begin{equ}\label{def: fict renorm}
[
\mathscr{Y}_{B}
\phi
]
(z)
\eqdef
\sum_{
b \in \mathrm{Der}(B)
}
(
D^{b}
\phi
)
(\Coll_{B}(z)
)
(z_{(+,B)} - z_{(-,B)})^{b}\;,
\end{equ}
where 
\begin{equ}\label{eq: overload collapse}
\Coll_{B}(z)_u \eqdef 
\left\{\begin{array}{cl}
  z_{(-,B)} & \text{if $u \in B$,} \\
  z_u & \text{otherwise,}
\end{array}\right.
\end{equ}
and $\mathrm{Der}(B)$ is the set of all multi-indices $b$ supported on $(+,B)$ and satisfying $|b|_{\s} < \fict(B)$. We also set $\overline{\mathrm{Der}}(B) \eqdef \mathrm{Der}(B) \sqcup \partial_{(+,B)} \mathrm{Der}(B)$ and for any subset $A \subset R(\pi)$ we set 
\begin{equ}\label{sum of second cumulant derivatives}
\overline{\mathrm{Der}}(A)
\eqdef
\Big\{ \sum_{B \in A} m_{B} :\ m_{B} \in \overline{\mathrm{Der}}(B) \Big\}\;.
\end{equ}
Finally, for any $S \in \mcB$ we define 
\begin{equ}
\overline{\mathrm{Der}}(\pi,S) \eqdef
\Big\{ \sum_{
\substack{
B \in R(\pi)\\
B \subset L_{\mcB}(S)}}
k_{B}:\ 
k_{B} \in \overline{\mathrm{Der}}(B) 
\Big\}\;.
\end{equ}
This finishes our additional notation for renormalizing second cumulants, we now introduce other useful shorthands.  
We define a map $\mfh:K(\sT) \rightarrow \R$ by setting 
\begin{equ}\label{equ: homogeneity map}
\mfh(e) \eqdef |\s| - |\mft(e)|_{\s} + |\se(e)|_{\s}\;.
\end{equ} 
The mnemonic here is that $\mfh$ stands for homogeneity, the kernel $\ke{\{e\}}{0}$ blows up like $|x_{e_{\p}} - x_{e_{\ch}}|^{-\mfh(e)}$ as $|x_{e_{\p}} - x_{e_{\ch}}| \rightarrow 0$. 

We also introduce notation for various domain constraints. 
For $z,w \in \R^{d}$ and $t \in \R$ we write $\link{z}{w}{t}$ for the condition
\begin{equ}\label{defining annular region}
C^{-1} 2^{-t} \le |z-w| \le C2^{-t}
\end{equ}
write $\ulink{z}{w}{t}$ for the condition
\begin{equ}\label{defining circular region}
|z-w| \le C2^{-t}\;.
\end{equ}
In both \eqref{defining annular region} and \eqref{defining circular region} one chooses a
fixed value $C > 0$ (not dependent on $t$). 
\begin{remark}
Note that from line to line the constant $C$ implicit in the notations \eqref{defining annular region} and \eqref{defining circular region} may change but remains suppressed from the notation. 

In particular, when the notations \eqref{defining annular region} and/or \eqref{defining circular region} 
appear in the assumptions of a lemma or proposition one is allowed to choose any value(s) of $C$.
Any proportionality constant hidden in the notation $\lesssim$ appearing in the conclusion 
may then depend on the choices of constants $C$ made in the assumption. 
Moreover, if there is another use of the notations \eqref{defining annular region} and/or 
\eqref{defining circular region} in the conclusion of the lemma or proposition, then implicit constants 
for these conditions cannot be chosen arbitrarily, but have to be taken sufficiently large in
a way that may depend on the constants chosen for the analogous expression in the assumptions. 

In the end, all of these constants influence the overall factor $C_{\tau,p}$ in \eqref{upgraded thm - main theorem}. 
\end{remark}
\subsubsection{Inductive estimates for negative renormalizations}
After renormalizing everything in $\mcF \subset \mcB$, the sets $\tilde{N}(\mcF)^{c}$ and $\tilde{N}(\mcF)$ are the sets of free and integrated out variables of our integrand, respectively. 
Since the forest $\CB$ is now fixed once and for all, we also write 
$\inte^{\bullet}$ and $\exte^{\bullet}$ instead of $\inte^{\bullet}_{\mcB}$ and 
$\exte^{\bullet}_{\mcB}$ and, for any $T \in \Div$, we define $\bar{\omega}(T) \eqdef \lfloor \omega(T) \rfloor < \omega(T)$. 

The goal in this subsection is to estimate, for any forest $\mcF \subset \mcB$ of depth $1$ and 
any scale assignment $\mbj \in \N^{\CE'}$ with $\CE' \supset \CE^{\exte}_{\mcB}(\mcF)$, quantities of the form $\hat{H}^{\mbj}_{\mcF}(\phi)$. 
In order to facilitate this, we define a family of seminorms $\| \cdot \|_{\mcF,\mbj}$ on the functions of $\allf$.
These seminorms control the derivatives of $\phi$ that are generated when one renormalizes the trees $\mcF$, as well as the derivatives generated by the renormalization of the sub-divergences within these trees. 
To that end, for any forest $\mcF \subset \mcB$ with $\mathrm{depth}(\mcF) \le 1$, we define
\begin{equ}\label{forest of derivatives}
\widetilde{\mathrm{Der}}(\mcF) \eqdef
\Big\{
\sum_{S \in \mcG} 
k_{S}
+
\sum_{
\substack{
B \in R(\pi),\ B \subset L(\mcF)\\
B \not \subset L(\mcG)}}
k_{B}:\ 
\begin{array}{c}
\mcG \subset \mcB \textnormal{ with } 
\mcG \in \mathbb{F}_{\le}[\mcF]\\[1ex] 
k_{S} \in 
\overline{\mathrm{Der}}(S),\ 
k_{B} \in \overline{\mathrm{Der}}(B)
\end{array}
\Big\}\;,
\end{equ}
where $\mathbb{F}_{\le}[\mcF]$ is defined as in \eqref{e:maxForest} and for any tree in $T \in \Div$, we set
\begin{equ}
\overline{\mathrm{Der}}(T)
\eqdef
\mathrm{Der}(T)
\cup
\partial_{\tilde{N}(T)} \mathrm{Der}(T)
\end{equ}
where the notation $\mathrm{Der}(T)$ was introduced in Definition~\ref{set of derivatives}.

The control over derivatives will be modulated by a certain scaling. 
For $\mcF$ and $\mbj$ as before and $b \in \widetilde{\mathrm{Der}}(\mcF)$ we define
a  differential operator with constant coefficients on $\allf$ by writing
\[
\mbbD_{\mcF}^{b,\mbj} 
\eqdef
\prod_{S \in \mcF}
\frac{D^{b_{S}}}{2^{\exte^{\mbj}(S)|b_{S}|_{\s}}},
\]
where we use the fact that each such $b$ admits a unique decomposition $b = \sum_{S \in \mcF}b_{S}$ with $b_{S}$ supported on $\tilde{N}(S)$.
Additionally, we build a seminorm which controls the result of acting on a test function $\phi$ with the 
renormalisation operator $\hat{H}_{\mcF}^{\mbj}$. This is achieved by taking suitable suprema of $\phi$
and some of its partial derivatives over the variables to be integrated. 
Recall that $\rho(\mcF) \subset \tilde N(\mcF)^{c}$ denotes the set of 
roots of $\CF$.
Then, for any $V \supset \rho(\mcF)$ fixed $x \in (\R^{d})^{V}$, $\CE' \supset \CE^{\exte}_{\mcB}(\mcF)$, and $\mbj \in \N^{\CE'}$, we define $\Dom(\mcF,\mbj,x)$ to be the set of all $y \in (\mathbf{R}^{d})^{\tilde{N}(\mcF)}$ such that the following conditions hold
\begin{itemize}
\item For each $S \in \mcF \cap \mcS$, one has $y_v = x_{\rho_{S}}$
for all nodes $v \in \tilde N(S)$.
\item For each $S \in \mcF \cap \mcD$, one has
$\ulink{y_{v}}{x_{\rho_S}}{\exte^{\mbj}(S)}$ for all nodes $v \in \tilde N(S)$.
\end{itemize}
In other words, we restrict ourselves to coordinates $y$ for which ``safe'' trees are collapsed to
a point, while ``unsafe'' trees are restricted to be of diameter of order $2^{-\exte^{\mbj}(S)}$.
We can now define the previously mentioned seminorms. 

\begin{definition}\label{seminorm def}
Let $\mcF \subset \mcB$  with $\mathrm{depth}(\mcF) \le 1$. 
Let $\mbj = \N^{\CE'}$ with $\CE' \supset \CE_{\mcB}(\mcF)$. 
For $x \in (\R^d)^{\tilde{N}(\mcF)^{c}}$, we define a family of seminorms $\|\cdot\|_{\mcF,\mbj}(x)$ on $\allf$ by
\begin{equation*}
\|\phi\|_{\mcF,\mbj}(x)
\eqdef 
\sup
\left\{
\left|
\left(
\mbbD^{b,\mbj}_{\mcF}\phi
\right)
(x \sqcup y)
\right|:\ 
\begin{array}{c}
b \in \widetilde{\mathrm{Der}}(\mcF)
\textnormal{ and }\\
y \in \Dom(\mcF,\mbj,x) 
\end{array}
\right\}\;.
\end{equation*}
In the particular case when $\CF = \{S\}$ (i.e.\ it only consists of a single tree),
we also write $\|\phi\|_{S,\mbj}(x)$ instead of $\|\phi\|_{\{S\},\mbj}(x)$.
\end{definition}
\begin{remark} 
We make a few simple observations about the seminorms defined in Definition~\ref{seminorm def}. 
The first is that $\| \cdot \|_{\mcF,\mbj}(x)$ clearly only depends on $\mbj$'s values on $\CE_{\mcB}(\mcF)$. 
The second is that the case $\mcF = \emptyset$ is somewhat degenerate, here the ``seminorm'' is really just a point-wise absolute value.
\end{remark}

We then have the following lemmas.
\begin{lemma}\label{lem: productrule}
Let $\mcF \subset \mcB$ with $\mathrm{depth}(\mcF) \le 1$. 
Then uniform in $\mbj \in \N^{\CE'}$ with $\CE' \supset \CE_{\mcB}(\mcF)$, $x \in (\R^d)^{\tilde{N}(\mcF)^{c}}$,  and $f,g \in \allf$, one has 
\begin{equation*}
\|f g\|_{\mcF, \mbj}(x) 
\lesssim
\|f\|_{\mcF, \mbj}(x) 
\, 
\|g \|_{\mcF, \mbj}(x)
\end{equation*}
Furthermore, if $g \in \mcb{C}_{\tilde{N}(\mcF)^{c}}$ then 
$\|f g \|_{\mcF, \mbj}(x)
=
\|f \|_{\mcF, \mbj}(x)\,g(x)$.
\end{lemma}
\begin{proof} The first claim follows from Leibniz's Rule and the second is immediate
from the definitions.
\end{proof}
 
We use a shorthand for the kernel norms of \eqref{def: singular kernel norm} by writing, 
for any subset of edges $E \subset K(\sT)$,
\[
\|E\|_{\mft}
\eqdef
\prod_{e \in E}
\|K_{\mft(e)}\|_{|\mft(e)|_{\s},2 |\s| \times  |N(\sT)|}\;.
\]
\begin{lemma}\label{lem: kernel bound and support}
Let $S \in \mcB$ and $ \mcF \eqdef C_{\mcB}(S)$. 
Then uniform in $\mbk \in \N^{\CE'}$ with $\CE' \supset \CE^{\inte}_{\mcB}(S)$, $x \in (\R^d)^{\tilde{N}(\mcF)^{c}}$, and any multi-index $p$ supported on $e_{\ch}[K^{\partial}_{\mcB}(S)]$
with $|p(u)|_{\s} \le |\s|$ for every $u$ in that node set, one has
\begin{equ}[e:boundKernel]
\Big\|
D^{p}
\ke{K^{\partial}_{\mcB}(S)}{0,\mbk}
\Big\|_{\mcF,\mbk}(x)
\lesssim 
\|K^{\partial}_{\mcB}(S)\|_{\mft}
\cdot
\prod_{e \in K^{\partial}_{\mcB}(S)}
2^{
(\mfh(e) + |p(e_{\ch})|_{\s}) k_{e}
}
\;.
\end{equ}
Furthermore, the left hand side vanishes unless, for all $T \in \mcF$ and 
$e \in K^{\partial}_{T}(S)$, one has $\link{x_{e_{\ch}}}{x_{\rho_{T}}}{k_{e}}$.
\end{lemma}
\begin{proof}
When $\mcF = \emptyset$ the statement is an immediate consequence of the definition of \eqref{def: singular kernel norm} so we turn to the case of $\mcF \not = \emptyset$
Using Lemma~\ref{lem: productrule} it suffices to show that for each $T \in \mcF$, $e \in K^{\partial}_{T}(S)$
\begin{equ}[e:wantedBound]
\big\|
D^{p'}
\ke{\{e\}}{\mfe,\mbk}
\big\|_{T,\mbk}(x)
\lesssim
\mathbbm{1}\{
\link{x_{e_{\ch}}}{x_{\rho_{T}}}{k_{e}}
\}
\|K_{\mft(e)}\|_{\mft(e)}
2^{(\mfh(e) + |p'|_{\s} + |j(e_{\ch})|_{\s}) k_{e}}\;,
\end{equ}
where the multi-index $p'$ is supported on $e_{\ch}$ and is bounded as in the assumption. 
By definition, we have 
\begin{equ}[e:boundSingle]
\big\|
D^{p'}
\ke{\{e\}}{0,\mbk}
\big\|_{{T},\mbk}(x)
= 
\sup_{
\substack{
b \in \widetilde{\mathrm{Der}}(\set{T})\\
y \in \Dom(\set{T},\mbk,x)
}
}
2^{-|b|_{\s}\exte^{\mbk}(T)}
\big|
\big(
D^{b + p'}\ke{\{e\}}{0,\mbk}
\big)
(x_{e_{\ch}} \sqcup y_{e_{\p}})
\big|\;.
\end{equ}
Note now that the bound $\exte^{\mbk}(T) \ge k_{e}$ holds by the definition
of $\exte^{\mbk}(T)$ and the fact that $e \in K^{\partial}_{T}(S) \subset K_\CB^{\exte}(T)$, so that 
\begin{equs}
\big|
D^{b + p'}\ke{\{e\}}{0,\mbk}
(\cdot)
\big|
&\lesssim
2^{(|b|_{\s} + |p'|_{\s} + \mfh(e))k_{e}}\|K_{\mft(e)}\|_{\mft(e)}\\
&\le
2^{ (|p'|_{\s} + \mfh(e))k_{e} +|b|_{\s}\exte^{\mbk}(T)}\|K_{\mft(e)}\|_{\mft(e)}
\end{equs} 
yielding
\begin{equ}
\big\|
\ke{\{e\}}{0,\mbk}
\big\|_{{T},\mbk}(x)
\lesssim 
2^{ (\mfh(e) + |p'|_{\s} )k_{e}}\|K_{\mft(e)}\|_{\mft(e)}\;.
\end{equ}
It remains to show that all points $x$ in the support of the left hand side of \eqref{e:boundKernel} 
satisfy $\link{x_{e_{\ch}}}{x_{\rho_{T}}}{k_{e}}$.
The support property of $\ke{\{e\}}{0,\mbk}$ enforces $\link{x_{e_{\ch}}}{y_{e_{\p}}}{k_{e}}$
for all $y$'s over which the supremum in our seminorm is taken.
On the other hand the condition $y \in \Dom(\set{T},\mbk,x)$ forces
$\ulink{y_{e_{\p}}}{x_{\rho_{T}}}{\exte^{\mbk}(T)}$, so that the required relation 
follows from the triangle inequality, thus completing the proof.
\end{proof}
\begin{lemma}\label{lem: renorm factors}
Let $S \in \mcB$ and $ \mcF \eqdef C_{\mcB}(S)$. 
Then, uniform in $\mbj \in \N^{\CE'}$ with $\CE' \supset \CE^{\exte}_{\mcB}(S)$, $\mbk \in \mathring{\mathcal{N}}_{S}(\mbj)$, $m \in \mathrm{Der}(\pi,S)$, and $x \in (\R^{d})^{\tilde{N}(\mcF)^{c}}$ satisfying the constraints 
\begin{equs}\label{domainconstraint}
\ulink{x_{e_{\p}}}{x_{e_{\ch}}}{k_{e}} 
&\,\textnormal{for all $e \in \mathring{K}_{\mcB}(S)$,}\\
\ulink{x_{e_{\ch}}}{x_{\rho_{T}}}{k_{e}}
&\,\textnormal{for all $T \in \mcF$ and $e \in K^{\partial}_{T}(S)$,}
\end{equs}
one has the bound
\begin{equ}\label{eq: renorm factor}
\big\| D^{m} \mathscr{Y}^{\#}_{\mbbM,S} \phi 
\big\|_{\mcF, \mbk}( x)
\lesssim 
\|\phi\|_{S,\mbj}(x_{\tilde{N}(S)^{c}})
\cdot
2^{\bar{\omega}^{\#}(S)[\exte^{\mbj}(S) - \inte^{\mbk}(S)]
+
|m|_{\s}\inte^{\mbk}(S)}\;,
\end{equ}
where
\begin{equ}\label{def: barred omega}
\bar{\omega}^{\#}(S)
\eqdef
\begin{cases}
\bar{\omega}(S) 
& 
\textnormal{ if } 
S \in \mcS,\\
\bar{\omega}(S) + 1 
& \textnormal{ if }S \in \mcD\;.
\end{cases}
\end{equ}
\end{lemma}
\begin{proof}

We first treat the case where $\mcF \not = \emptyset$. 
We fix $\mbj$ as above and $\mbk \in \mathring{\mathcal{N}}_{S}(\mbj)$.

Throughout the proof we implicitly assume that $x$ satisfies the constraints \eqref{domainconstraint}.

We first establish the desired bound when $S \in \mcS$. 
A term by term estimate gives
\begin{equ}[e:firstBoundYsphi]
\left\| 
D^{m}
\mathscr{Y}_{S} \phi
\right\|_{\mcF, \mbk}(x)
\lesssim
\max_{
p + m \in \mathrm{Der}(S)
}
\Big[
\big\|\powroot{\tilde{N}(S)}{p}{\rho_{S}} \big\|_{\mcF, \mbk}(x)
\cdot
\big|
(D^{p + m}\phi)  
(\Coll_{S}(x))
\big|
\Big]\;.
\end{equ}
Here, the factor $(D^{p + m}\phi)  
(\Coll_{S}(x))$ could be pulled out of the seminorm 
because it does not depend on any of the variables in $\tilde{N}(\mcF)$.

It is easy to see that we have the bound
\begin{equation*}
\begin{split}
\big\| 
\powroot{\tilde{N}(S)}{p}{\rho_{S}}
 \big\|_{\mcF, \mbk}
( x) 
=&
\sup
\left\{ 
\big|
\mbbD^{b,\mbk}_{\mcF}
\powroot{\tilde{N}(S)}{p}{\rho_{S}}
\big|
(x \sqcup y):
\begin{array}{l}
b \in \widetilde{\mathrm{Der}}(\mcF),\ b \le p\\
y \in 
\Dom(\mcF,\mbk,x)
\end{array}
\right\}\\
\lesssim&\ 
2^{-|p|_{\s} \inte^{\mbk}(S)}\;.
\end{split}
\end{equation*}
Here we used that our condition on $x$ forces the distances between coordinates in $N_{\mcF}(S)$ to be at most of 
order $2^{-\inte^{\mbk}(S)}$, while the constraint $y \in \Dom(\mcF,\mbk,x)$ forces the distances between 
coordinates in $N_{T}$ for 
each $T \in \mcF$ to be of order at most $2^{-\exte^{\mbk}(T)}$. Furthermore, for any $T \in \mcF$ we have
\begin{equ}\label{eq: basic fact about scale indices}
\exte^{\mbk}(T) \ge \inte^{\mbk}(S).
\end{equ} 

Inserting this into \eqref{e:firstBoundYsphi} we have, for $x \in (\R^d)^{\tilde{N}(S)^{c}}$,
\begin{equation*}
\begin{split}
\big\| 
D^{m}&
\mathscr{Y}_{S} \phi
\big\|_{\mcF, \mbk}(x)
\lesssim
\max 
\left\{ 
2^{-|p|_{\s} \inte^{\mbk}(S)}
\big|
(D^{p + m}\phi)  
(\Coll_{S}(x))  
\big|:\ 
 p+m \in \mathrm{Der}(S)
\right\}\\
&=
\max 
\left\{ 
2^{|p|_{\s}(\exte^{\mbj}(S) - \inte^{\mbk}(S)) + |m|_{\s}\exte^{\mbj}(S)}:\ 
 p+m \in \mathrm{Der}(S)
\right\}
\cdot
\|\phi\|_{S,\mbj}(x)\\
&= 2^{(\bar{\omega}(S) - |m|_{\s})(\exte^{\mbj}(S) - \inte^{\mbk}(S)) + |m|_{\s}\exte^{\mbj}(S)}
\cdot
\|\phi\|_{S,\mbj}(x)\;,
\end{split}
\end{equation*}
where in the last line we used the fact that $S \in \mcS$ implies 
that $\exte^{\mbj}(S) \ge \inte^{\mbk}(S)$.

We turn to case $S \in \mcD$ and start with the estimate
\begin{equation*}
\begin{split}
&
\left\| 
D^{m}
(1 - \mathscr{Y}_{S}) \phi 
\right\|_{\mcF,\mbk}(x)\\ 
&\le
\sup
\left\{ 
2^{-|b|_{\s} \inte^{\mbk}(S)}
\left|
D^{b + m}
\left[
(1 - \mathscr{Y}_{S}) \phi
\right]
(x \sqcup y)
\right|:\ 
\begin{array}{c}
b \in \widetilde{\mathrm{Der}}(\mcF)\\
y \in \Dom(\mcF,\mbk,x)
\end{array}
\right\}\;,
\end{split}
\end{equation*}
where we have an inequality because we used \eqref{eq: basic fact about scale indices}. 
We treat the cases 
$|b + m|_{\s} > \bar{\omega}(S)$ and $|b + m|_{\s} \le \bar{\omega}(S)$ separately. 
In the former case one has $D^{b + m}\mathscr{Y}_{S} \phi = 0$ and we arrive at the bound 
\begin{equs}
{}&\sup
\left\{ 
2^{-|b|_{\s} \inte^{\mbk}(S)}
\big|\big(
D^{b + m} 
\phi
\big)
(y \sqcup x)\big|:
\begin{array}{c}
b \in \widetilde{\mathrm{Der}}(\mcF),\ |b + m|_{\s} > \bar{\omega}(S)\\
y \in \Dom(\mcF,\mbk,x)
\end{array}
\right\}\\
&\le 
\max 
\left\{ 
2^{ |b|_{\s} [\exte^{\mbj}(S) - \inte^{\mbk}(S)]
+
|m|_{\s} \exte^{\mbj}(S) }:\ 
\begin{array}{c}
b \in \widetilde{\mathrm{Der}}(\mcF)\;,\\ 
|b + m|_{\s} > \bar{\omega}(S)
\end{array}
\right\}
\cdot
\|\phi\|_{S,\mbj}(x)\\
&\le 
2^{ (\bar{\omega}(S) + 1 - |m|_{\s}) [\exte^{\mbj}(S) - \inte^{\mbk}(S)] + |m|_{\s}\exte^{\mbj}(S)}
\cdot
\|\phi\|_{S,\mbj}(x)\;. \label{intermedwork powercounting factors}
\end{equs} 
To obtain the first inequality of \eqref{intermedwork powercounting factors} observe 
that the condition $y \in \Dom(\mcF,\mbk,x)$ implies that one also has
\begin{equation}\label{eq: domain inclusion}
x_{\tilde{N}_{\mcF}(S)} \sqcup y
\in
\Dom(\{S \},\mbj,x).
\end{equation}
For the second equality we used $\exte^{\mbj}(S) < \inte^{\mbk}(S)$.

We now treat the case of $b \in \widetilde{\mathrm{Der}}(\mcF)$ with $|b + m|_{\s} \le \bar{\omega}(S)$ and write $D^{b + m}(1 - \mathscr{Y}_{S})\phi(x)$ as
\begin{equation}\label{tayremaind}
D^{b + m}\phi(x) - 
\sum_{
\substack{p,\\
p + m \in \mathrm{Der}(S)
}
}
\frac{(x-\Coll_S(x))^p}{p!}
\big(
D^{b+p + m}\phi\big)(\Coll_S(x))\;.
\end{equation}
Viewing \eqref{tayremaind} as a Taylor remainder of order $\bar{\omega}(S) - |b|_{\s} - |m|_{\s}$ 
for $D^{b + m}\phi$, we can apply \cite[Prop.~A.1]{Regularity} followed by \eqref{eq: domain inclusion}
to get the estimate
\begin{equation*}
\begin{split}
\sup&
\left\{ 
2^{-|b|_{\s} \inte^{\mbk}(S)}
\left|
D^{b}
\left[
(1 - \mathscr{Y}_{S})\phi
\right]
(y \sqcup x)
\right|:\ 
\begin{array}{c}
b \in \widetilde{\mathrm{Der}}(\mcF),\ |b + m|_{\s} \le \bar{\omega}(S)\\
y \in \Dom(\mcF,\mbk,x)
\end{array}
\right\}\\
&\le
2^{(\bar{\omega}(S) + 1 - |m|_{\s}) [ \exte^{\mbj}(S) - \inte^{\mbk}(S)]
+
|m|_{\s} \exte^{\mbj}(S)} 
\cdot
\|\phi\|_{S,\mbj}(x)\;.
\end{split}
\end{equation*}
Combining this with \eqref{intermedwork powercounting factors} yields the required
bound in the case $S \in \CD$ and thus concludes the proof for when $C_{\mcB}(S) \not = \emptyset$. 

The case $C_{\mcB}(S) = \emptyset$ follows by the same argument (a Taylor remainder estimate when $S \in \mcD$ and a term by term estimate when $S \in \mcS$) but is strictly easier.
\end{proof}
We now state and prove the advertised bound on the operators $\hat{H}_{\mcF}^{\mbj}$ which was the motivation for the
introduction of these seminorms.
\begin{lemma}\label{lem: genvertbd}
Let $\mcF \subset \mcB$ with $\mathrm{depth}(\mcF) \le 1$. 
Then, uniform in $x \in (\R^d)^{\tilde{N}(\mcF)^{c}}$, $\mbj \in \N^{\CE'}$ with $\CE' \supset \CE^{\exte}_{\mcB}(\mcF)$, and $\phi \in \allf$, one has the bound
\begin{equ}\label{genvertbd}
\left|
\hat{H}_{\mcF}^{\mbj}
[\phi](x)
\right|
\lesssim 
\Big(
\prod_{S \in \mcF}
2^{\omega(S) \exte^{\mbj}(S)}
\|K(S)\|_{\mft}
\cdot
\| \xi \|_{|L(S)|,\c}
\Big)
\|\phi\|_{\mcF, \mbj}(x)\;.
\end{equ} 
\end{lemma}
\begin{proof}
Our proof uses two nested inductions. The outer one, which is also the less trivial one, 
is an induction in the quantity 
\begin{equation}\label{full forest}
\mathrm{depth}_{\mcB}(\mcF)
\eqdef 
\mathrm{depth}
\left[ 
\mcF \sqcup \{ T \in \mcB:\ \exists S \in \mcF \textnormal{ with } T < S \}
\right]\;.
\end{equation} 
The second, simpler, induction step is then in the cardinality of $\mcF$
for a fixed value of $\mathrm{depth}_{\mcB}(\mcF)$.
For both inductions we focus on the inductive steps, the base case being strictly easier
to verify.

Fix $m \ge 1$ and assume that \eqref{genvertbd} has already been proven for all forests $\mcG$ of depth 
$1$ with $\mathrm{depth}_{\mcB}(\mcG) \le m$. 
Our aim is to then prove \eqref{genvertbd} for any $\mcF$ of cardinality one, i.e.\ for 
$\mcF =\{S\} \subset \mcB$ with $\mathrm{depth}_{\mcB}(\{S\}) = m+1$. 

Fix $\mbj \eqdef \mbj_{S} \in \N^{\CE^{\exte}_{\mcB}(S)}$, our concern to control the corresponding sum over $\mbk \in \mathring{\mcN}_{S}(\mbj)$ appearing in the definition of $\hat{H}^{\mbj}_{S}$ (see \eqref{genvert def}). 
We will use Lemmas~\ref{lem: productrule},~\ref{lem: kernel bound and support},~\ref{lem: renorm factors} and then appeal to Theorem~\ref{multicluster 1} to control the integral over $\tilde{N}_{\mcB}(S)$, the application of this theorem will occur with $x \in (\R^{d})^{\tilde{N}(S)^{c}}$ fixed but our our estimates will be uniform in $x$ (however dependence on $x$ may sometimes be suppressed from the notation).

The multigraph underlying our application of Theorem~\ref{multicluster 1} is given by a quotient of the multigraph $\CE^{\inte}_{\mcB}(S)$ where, for each $T \in C_{\mcB}(S)$, one performs a contraction and identifies the collection of vertices $N(T)$ as a single equivalence class of points which we identify with $\rho_{T}$. 

Then our set of vertices is given by $\CV \eqdef N_{\mcB}(S)$ with $\CV_{0} \eqdef \tilde N_{\mcB}(S)$ (so $\rho_{S}$ serves the role of the pinned vertex) and our multigraph $\go{G}$ is given by 
\begin{equation}\label{multigraph for inductive bound}
\go{G}
\eqdef\ 
\mathring{K}_{\mcB}(S) 
\sqcup 
\left\{ 
\{u,\rho_{T}\}:\ 
T \in C_{\mcB}(S),\ 
u \in e_{\ch}[K^{\partial}_{T}(S)]
\right\}
\sqcup
\go{C}\;,
\end{equation}
where we selectively view $\mathring{K}_{\mcB}(S) \subset \CV^{(2)}$ as a set of undirected edges and $\go{C}$ is the set of contractions given by $\go{C}
\eqdef
\left\{ e \in \CE_{\pi} 
:\ 
e \subset L_{\mcB}(S)
\right\}\;.$
We define a bijection $\mfq_{S}: \CE^{\inte}_{\mcB}(S) \rightarrow \go{G}$ in the natural way: by asking that $\mfq_{S}$ maps $\CE_{\pi} \cap \CE^{\inte}_{\mcB}(S)$ onto $\go{C}$, $\mfq_{\go{S}}$ maps $\mathring{K}_{\mcB}(S)$ onto $\mathring{K}_{\mcB}(S)$, and $\mfq_{S}$ maps $K^{\partial}_{\mcB}(S)$ onto the middle set of \eqref{multigraph for inductive bound} as follows: for $(e_{\p},e_{\ch}) \in K^{\partial}_{T}(S)$ for $T \in \mcC_{\mcB}(S)$ one sets $\mfq_{S}(\{e_{\ch},e_{\p}\}) = \{e_{\ch},\rho_{T}\}$.
The map $\mfq_{S}$ induces a corresponding bijection between $\N^{\CE^{\inte}_{\mcB}(S)}$ and $\N^{\go{G}}$. 
In what follows, we abuse notation and treat this bijection as an identification.
As a start we define $\mcN_{\go{G}} \subset \N^{\go{G}}$ by setting $\mcN_{\go{G}} \eqdef  \mathring{\mcN}_{S}(\mbj)$.

For each $\mbk \in \mcN_{\go{G}}$ we define a function $F^{\mbk} \in \mcb{C}_{\CV}$ by setting 
\[
F^{\mbk}(y)
\eqdef
\mathrm{Cu}^{L_{\mcB}(S)}_{\pi,\mbk}(y) 
\ke{\mathring{K}_{\mcB}(S)}{0,\mbk}(y)
\hat{H}^{\mbk}_{C_{\mcB}(S)}
\left[
\ke{K^{\partial}_{\mcB}(S)}{0,\mbk}
\,
[\mathscr{Y}_{S,\mbbM}^{\#}\phi]
\right](x \sqcup 
y)\;.
\]
So we then have
\[
\hat{H}^{\mbj}_{S}[\phi](x \sqcup y_{\rho_{S}})
=
\sum_{ \mbk \in \mathring{\mcN}_{S}(\mbj)}
\int_{\tilde{N}_{\mcB}(S)}\back dy\ F^{\mbk}(y)\;.
\] 
Define $\hat{\mfq}_{S}:N(S) \rightarrow \CV$ via setting $\hat{\mfq}(S)(u) \eqdef \rho_{T}$ if there exists $T \in C_{\mcB}(S)$ with $u \in N(T)$, we set $\hat{\mfq}_{S}(u) \eqdef u$ otherwise.
We now define a total homogeneity $\varsigma$ on the trees of $\widehat{\CU}_{\CV}$ 
(see Definitions~\ref{def: coalescence tree} and~\ref{def:homogeneity} below)
by setting
\begin{equ}\label{def of sigma - genvertbd}
\varsigma
\eqdef 
- \bar{\omega}^{\#}(S)\delta^{\uparrow}[\CV]
+
\sum_{e \in K_{\mcB}(S)}
\mfh(e) \delta^{\uparrow}[\{e_{\ch},\hat{\mfq}_{S}(e_{\p})\}]
+
\varsigma^{\go{C}}
+
\varsigma^{R}
+
\sum_{T \in C_{\mcB}(S)} 
\omega(T)
\delta^{\uparrow}[\rho_{T}]\;,
\end{equ}
where the total homogeneity $\varsigma^{\go{C}}$ is given by
\begin{equ}\label{def of sigma - cumulants - genvert bd - new}
\varsigma^{\go{C}}
\eqdef
\sum_{
\substack{
B \in \pi\\
B \subset L_{\mcB}(S)
}
}  
\c^{B}\;.
\end{equ}
The total homogeneity $\c^{B}$ on the coalescence trees of $\widehat{\CU}_{\CV}$ is defined in Definition~\ref{def: useful cumulant bound}. 
The total homogeneity $\varsigma^{R}$ is given by setting, for each $\go{T} \in \widehat{\CU}_{\CV}$,
\begin{equs}
\varsigma^{R}_{\go{T}}
&\eqdef
\sum_{B \in R(\go{T})}
\fict(B)
\left(
\delta_{\go{T}}^{\Uparrow}[B]
-
\delta_{\go{T}}^{\uparrow}[B]
\right), \textnormal{ where}\\
R(\bo{T}) 
&\eqdef
\left\{ B \in \pi:\ 
B \subset L_{\mcB}(S),\ 
\go{L}_{B^{\uparrow}},\
= 
B,
\textnormal{ and }
\fict(B) > 0
\right\}\;.
\end{equs}
The notation $\fict(B)$ was defined in \eqref{def: fict}.
and, for $a \in \mathring{\go{T}}$ we write $\go{L}_{a}$ for the set of leaves of $\go{T}$ which are descendants of $a$.
In Lemma~\ref{lem: proof of subdivfree} we establish that $\varsigma$ is subdivergence free on $\CV$ for the set of scales $\mcN_{\go{G}}$. Taking this for granted for the moment, we check the other conditions of Theorem~\ref{multicluster 1}. 
Observe that the total homogeneity $\varsigma$ is of order $\omega(S) - \bar{\omega}(S)^{\#}$ which is negative when $S \in \mcD$ and positive when $S \in \mcS$. Additionally, one has 
\[
\mcN_{\go{G}} 
\eqdef 
\begin{cases}
\mcN_{\go{G}, > \exte^{\mbj}(S)}
& 
\textnormal{ if } S \in \mcD\;,\\
 \mcN_{\go{G}, \le \exte^{\mbj}(S)}
& 
\textnormal{ if } S \in \mcS\;.
\end{cases}
\] 
We are then done if we can exhibit a modification of $\tilde{F} = (\tilde{F}^{\mbk})_{\mbk \in \mcN_{\go{G}}} \in \mathrm{Mod}(F)$ such that $\tilde{F}$ is is bounded by $\varsigma$ in the sense of Definition~\ref{def:boundedBy} with
\begin{equ}\label{eq: inductive bound, total homogeneity bound}
\|\tilde{F}\|_{\varsigma, \mcN_{\go{G}}} \lesssim 
\left\|
\phi
\right\|_{S,\mbj}(x)
\cdot
2^{\bar{\omega}^{\#}(S)\exte^{\mbj}(S)}
\|K(S)\|_{\mft} \cdot \|\xi\|_{|L(S)|,\c}\;.
\end{equ}
For any $\mbk \in \mcN_{\go{G}}$ we will set
\[
\tilde{F}^{\mbk}
\eqdef
\begin{cases}
F^{\mbk} & \textnormal{ if } R(\mcT(\mbk)) = \emptyset\;,\\
\mathring{F}^{\mbk} & \textnormal{ if } R(\mcT(\mbk)) \not = \emptyset\;.
\end{cases}
\]
We will define $\mathring{F}^{\mbk}$ on a tree by tree basis later. It suffices to check the domain condition and desired supremum bound in each of the two cases separately.

We treat the first case. The domain condition \eqref{domain condition} is easy to check so we turn to the desired supremum bounds needed for \eqref{eq: inductive bound, total homogeneity bound}.

Fix $\go{T} \in \CU_{\CV}$ with $R(\go{T}) = \emptyset$.
We now obtain the desired uniform in $\go{s} \in \go{\mathrm{Lab}}_{\go{T}}$ and $\mbk \in \mcN_{\go{G},\tr}(\go{T},\go{s})$ estimates.
Fix appropriate $\go{s}$ and $\mbk$, we start by applying the inductive hypothesis to $\hat{H}_{C_{\mcB}(S)}^{\bullet}$ along with Lemmas~\ref{lem: productrule},~\ref{lem: kernel bound and support}, and~\ref{lem: renorm factors} which yields
\begin{equs}
\big|
F^{\mbk}
(y)
\big|
&\lesssim
\Big(
\prod_{T \in C_{\mcB}(S)}
2^{\omega(T)\exte^{\mbk}(T)}
\|K(T)\|_{\mft}
\cdot
\| \xi \|_{|L(T)|,\c}
\Big)
\cdot
\left\|
\mathscr{Y}_{S}^{\#}\phi
\right\|_{C_{\mcB}(S),\mbk}(x \sqcup y)\\
&
\cdot
\left|
\ke{\mathring{K}_{\mcB}(S)}{0,\mbk}(y \sqcup x_{\rho_{S}})
\mathrm{Cu}^{L_{\mcB}(S)}_{\pi,\mbk}(y \sqcup x_{\rho_{S}})
\right|
\cdot
\left\|
\ke{K^{\partial}_{\mcB}(S)}{0,\mbk}
\right\|_{C_{\mcB}(S),\mbk}(x \sqcup y)\\
&\lesssim 
\left\|
\phi
\right\|_{S,\mbj}(x)
\cdot
2^{\bar{\omega}^{\#}(S)\exte^{\mbj}(S)}\label{inductive bound - starting point}
\\
&
\cdot
2^{-\bar{\omega}^{\#}(S)\inte^{\mbk}(S)}
\Big(
\prod_{T \in C_{\mcB}(S)}
2^{\omega(T)\exte^{\mbk}(T)}
\|K(T)\|_{\mft}
\cdot
\| \xi \|_{|L(T)|,\c}
\Big)\\
&
\cdot
\left|
\ke{\mathring{K}_{\mcB}(S)}{0,\mbk}(y \sqcup x_{\rho_{S}})
\mathrm{Cu}^{L_{\mcB}(S)}_{\pi,\mbk}(y \sqcup x_{\rho_{S}})
\right|
\cdot
\left\|
\ke{K^{\partial}_{\mcB}(S)}{0,\mbk}
\right\|_{C_{\mcB}(S),\mbk}(x \sqcup y)\;.
\end{equs}
The desired bound follows using the supremum estimates
\begin{equs}
\left|
\ke{\mathring{K}_{\mcB}(S)}{0,\mbk}(y \sqcup x_{\rho_{S}})
\mathrm{Cu}^{L_{\mcB}(S)}_{\pi,\mbk}(y \sqcup x_{\rho_{S}})
\right|
&\lesssim
\|\mathring{K}_{\mcB}(S)\|_{\mft} 
\|\xi\|_{L_{\mcB}(S),\c}
\cdot
2^{
\langle
\varsigma^{C}_{\go{T}}
+
\varsigma_{\go{T}}[\mathring{K}_{\mcB}(S)],
\go{s}
\rangle
}
\\
\textnormal{and }
\left\|
\ke{K^{\partial}_{\mcB}(S)}{0,\mbk}
\right\|_{C_{\mcB}(S),\mbk}(x \sqcup y)
&\lesssim
\|K^{\partial}_{\mcB}(S)\|_{\mft} 
\cdot
2^{
\langle
\varsigma_{\go{T}}[\mathring{K}_{\mcB}(S)],\go{s}
\rangle
}\;.
\end{equs}
Now fix $\go{T} \in \CU_{\CV}$ with $R(\go{T}) \not = \emptyset$. 
For any $\mbk \in \mcN_{\go{G}}$ with $\mcT(\mbk) = \go{T}$ we set, for $y \in (\R^{d})^{\CV_{0}}$, 
\begin{equs}
\mathring{F}^{\mbk}(y)
\eqdef&\;
\mathrm{Cu}^{L_{\mcB}(S)}_{\pi,\mbk}(y \sqcup x_{\rho_{S}}) 
\cdot
\Big(
\prod_{B \in R(\bo{T})}
(1-\mathscr{Y}_{B})
\Big)
[G^{\mbk}]
(x \sqcup y), \textnormal{ where}\\
G^{\mbk}(x \sqcup y)
\eqdef&\;
\ke{\mathring{K}_{\mcB}(S)}{0,\mbk}(y \sqcup x_{\rho_{S}})
\hat{H}^{\mbk}_{C_{\mcB}(S)}
\Big[
\ke{K^{\partial}_{\mcB}(S)}{0,\mbk}
\,
(\mathscr{Y}_{S}^{\#}\phi)
\Big](x \sqcup y)\;,
\end{equs}
and the notation $\mathscr{Y}_{B}$ was defined in \eqref{def: fict renorm}. 
Clearly 
\[
F^{\mbk}(y)
-
\mathring{F}^{\mbk}(y)
\eqdef
\mathrm{Cu}^{L_{\mcB}(S)}_{\pi,\mbk}(y \sqcup x_{\rho_{S}})
\cdot
\sum_{ 
\substack{
\mathcal{A}
\subset
R(\go{T})\\
\mathcal{A} \not = \emptyset
}
}
\Big(
\prod_{B \in \mathcal{A}}
( - \mathscr{Y}_{B})
\Big)
[G^{\mbk}]
(x \sqcup y)\;.
\]
We claim that for each fixed $x$ and for each $\mathcal{A}$ in the above sum the integral over $y_{\CV_{0}}$ of the corresponding summand vanishes.
Fixing $\mathcal{A}$ and expanding the summand completely yields a linear combination of terms -- in each individual term all the dependence on $y_{(+,B)}$ can extracted as a factor of the form 
\[
(y_{(+,B)} - z_{(-,B)})^{n} \widetilde{\mathrm{Cu}}_{B,k_{B}}(z_{B})\;,
\]
where $z = y \sqcup x_{\mainroot}$ and $n \in \N^{d}$ with $|n|_{\s} \le \fict(B)$. 
Since $B \in R(\go{T})$ on must have $k_{B} > 0$ and therefore the above quantity vanishes when integrating $y_{(+,B)}$ by \eqref{eq: fict renorm int vanishes}. 

The required domain constraint for $\tilde{F}^{\mbk}$ is straightforward to check and for the supremum bound we observe that by the estimate of \cite[Prop.~A.1]{Regularity} one has, uniform in $\go{s}$, $\mbk \in \mcN_{\tr}(\go{T},\go{s})$, and in $y$ with $\ulink{y_{(+,B)}}{y_{(-,B)}}{k_{B}}$ for every $B \in R(\bo{T})$ -- which is required to be in the support of $\mathrm{Cu}^{L_{\mcB}(S)}_{\pi,\mbk}$ -- the bound 
\begin{equation}\label{mistake 1}
\begin{split}
&
\sup_{y}
\Big|
\Big(
\prod_{B \in R(\go{T})}
(1 - \mathscr{Y}_{B})
\Big)
[G^{\mbk}]
(x \sqcup y) 
\Big|\\
&\lesssim
\sup_{y}
\max_{
\substack{
f \in \overline{\mathrm{Der}}(R(\go{T}))\\
a + b + c = f}
}
\Big(
\prod_{B \in R(\go{T})}
2^{|f_{(+,B)}|_{\s} \go{s}(B^{\uparrow})}
\Big)
\big|
D^{a}\ke{\mathring{K}_{\mcB}(S)}{0,\mbk}(y)
\big|\\
&
\qquad
\cdot
\big|
\hat{H}^{\mbk}_{C_{\mcB}(S)}
\Big[
\big(
D^{b}
\ke{K^{\partial}_{\mcB}(S)}{0,\mbk}
\big)
[D^{c}\mathscr{Y}_{S,\mbbM}^{\#}\phi]
\Big](x \sqcup y)
\big|\;,
\end{split}
\end{equation}
where $\overline{\mathrm{Der}}(R(\go{T}))$ was defined in \eqref{sum of second cumulant derivatives} and for the multi-indices $m=  a,b,c,f$ and any $B \in \go{R}(T)$ we use the notation $m_{(+,B)}$ for the multi-index which is given by $m$ on the site $(+,B)$ but vanishes everywhere else. 
We proceed as earlier and  use our inductive hypothesis along with Lemmas~\ref{lem: kernel bound and support} and~\ref{lem: renorm factors} to get
\begin{equs}\label{mistake 2}
{}
&
\big|
\hat{H}^{\mbk}_{C_{\mcB}(S)}
\Big[
\big(
D^{b}
\ke{K^{\partial}_{\mcB}(S)}{0,\mbk}
\big)
[D^{c}\mathscr{Y}_{S,\mbbM}^{\#}\phi]
\Big](x \sqcup y)
\big|\\
&
\lesssim
\Big(
\prod_{T \in C_{\mcB}(S)}
2^{\omega(T)\exte^{\mbk}(T)}
\|K(T)\|_{\mft}
\cdot
\| \xi \|_{|L(T)|,\c}
\Big)
\cdot
\left\|
D^{c}
\mathscr{Y}_{S}^{\#}\phi
\right\|_{C_{\mcB}(S),\mbk}(x \sqcup y)\\
&
\quad
\cdot
\left\|
D^{b}
\ke{K^{\partial}_{\mcB}(S)}{0,\mbk}
\right\|_{C_{\mcB}(S),\mbk}(y)
\\
&
\lesssim
\Big(
\prod_{T \in C_{\mcB}(S)}
2^{\omega(T)\go{s}(\rho_{T}^{\uparrow})}
\|K(T)\|_{\mft}
\cdot
\| \xi \|_{|L(T)|,\c}
\Big)
\cdot
2^{\bar{\omega}^{\#}(S)\exte^{\mbj}(S)}
2^{-\bar{\omega}^{\#}(S)\go{s}(\rho_{\go{T}})}
\left\|
\phi
\right\|_{S,\mbj}(x)\\
&
\quad
\cdot
\left\|
K^{\partial}_{\mcB}(S)
\right\|_{\mft}
\cdot
2^{\big\langle \varsigma_{\go{T}}[K^{\partial}_{\mcB}(S)], \go{s} \big\rangle}
\Big(
\prod_{B \in \go{R}(\go{T})}
2^{|b_{(+,B)} + c_{(+,B)}|_{\s} \go{s}(B^{\Uparrow})}
\Big)\;.
\end{equs}
In the last inequality above the key observation is that for each $B \in R(\go{T})$, the scale $\go{s}(B^{\Uparrow})$ dominates the scale $\go{s}(\cdot)$ at which we would pay the cost of $|b_{(+,B)}|_{\s}$ or $|c_{(+,B)}|_{\s}$ with regards to Lemmas~\ref{lem: kernel bound and support} and~\ref{lem: renorm factors}. Similarly, one has the bound
\begin{equ}\label{mistake 3}
\big|
D^{a}\ke{\mathring{K}_{\mcB}(S)}{0,\mbk}(y)
\big|
\lesssim
\|\mathring{K}_{\mcB}(S)\|_{\mft}
2^{\langle \varsigma_{\go{T}}[\mathring{K}_{\mcB}(S)],\go{s} \rangle}
\prod_{B \in \go{R}(\go{T})}
2^{|a_{(+,B)}|_{\s} \go{s}(B^{\Uparrow})}
\Big)\;.
\end{equ}
Inserting \eqref{mistake 2} and \eqref{mistake 3} into \eqref{mistake 1} we get
\begin{equs}\label{mistake 4}
{}&
\sup_{y}
\Big|
\Big(
\prod_{B \in R(\go{T})}
(1 - \mathscr{Y}_{B})
\Big)
[\tilde{G}^{\mbk}]
(x \sqcup y) 
\Big|\\
&\lesssim
\|K_{\mcB}(S)\|_{\mft}
2^{\langle \varsigma_{\go{T}}[K_{\mcB}(S)],\go{s} \rangle}
\Big(
\prod_{T \in C_{\mcB}(S)}
2^{\omega(T)\go{s}(\rho_{T}^{\uparrow})}
\| \xi \|_{|L(T)|,\c}
\Big)
\cdot 2^{\bar{\omega}^{\#}(S)\exte^{\mbj}(S)}\\
&
\quad
\cdot
2^{-\bar{\omega}^{\#}(S)\go{s}(\rho_{\go{T}})}
\left\|
\phi
\right\|_{S,\mbj}(x)
\cdot
\max_{
\substack{
f \in \overline{\mathrm{Der}}(R(\go{T}))\\
a + b + c = f}
}
\Big(
\prod_{B \in R(\go{T})}
2^{|f_{(+,B)}|_{\s} [\go{s}(B^{\Uparrow})-\go{s}(B^{\uparrow})]}
\Big)
\end{equs}
The desired supremum bound is then obtained by noticing that the maximum of the last factor is obtained when $|f_{(+,B)}|_{\s} = \fict(B)$ for each $B \in R(\go{T})$ and by combining this with the earlier used supremum bound on $\mathrm{Cu}^{L_{\mcB}(S)}_{\pi,\mbk}$. 
This concludes the proof of \eqref{eq: inductive bound, total homogeneity bound}. 

We now prove the inductive step for the easier induction: let $n,m \ge 1$ and suppose that \eqref{genvertbd} has been proven for every forest $\mcG \subset \mcB$ of depth $1$ with either $\mathrm{depth}_{\mcB}(\mcG) \le n$ or $\mathrm{depth}_{\mcB}(\mcG) = n + 1$ and $|\mcG| \le m$. 

Now suppose $\mcF \subset \mcB$ is of depth $1$,  
$\mathrm{depth}_{\mcB}(\mcF) = n + 1$, and $|\mcF| = m + 1$. We prove \eqref{genvertbd} in this case. 
Fix $S \in \mcF$ and write $\mcG \eqdef \mcF \setminus \{S\}$, then for $x \in \R^{\tilde{N}(\mcF)^{c}}$,
\[
\hat{H}_{\mcF}^{\mbj}[\phi](x)
=
\hat{H}_{\mcG}^{\mbj}\left[ 
\hat{H}_{S}^{\mbj}[ 
\phi]
\right](x)\;.
\]
By applying the inductive hypothesis for $\mcG$ we have the bound
\begin{equ}
\left|
\hat{H}_{\mcF}^{\mbj}[\phi](x)
\right|
\lesssim
\Big(
\prod_{T \in \mcG}
2^{\omega(T) \exte^{\mbj}(T)}
\|K(T)\|_{\mft}
\cdot
\| \xi \|_{|L(T)|,\c}
\Big) 
\Big\|
\hat{H}_{S}^{\mbj}[ 
\phi]
\Big\|_{\mcG,\mbj}(x_{\tilde{N}(\mcF)^{c}})\;.
\end{equ}
It suffices to prove
\begin{equ}\label{claim for second induction}
\Big\|
\hat{H}_{S}^{\mbj}[ 
\phi]
\Big\|_{\mcG,\mbj}(x)
\lesssim
2^{\omega(S) \exte^{\mbj}(S)}
\|K(S)\|_{\mft}
\cdot
\| \xi \|_{|L(S)|,\c}
\cdot
\| \phi \|_{\mcF,\mbj}(x)\;.
\end{equ}
This is not hard to see since
\begin{equs}
{}&
\|
\hat{H}_{S}^{\mbj}[ 
\phi]
\|_{\mcG,\mbj}(x)\\
=\;& 
\sup
\left\{
\left|
\left(
\mbbD^{b,\mbj}_{\mcG}
\hat{H}_{S}^{\mbj}[ 
\phi]
\right)
(x,y)
\right|:\ 
\begin{array}{c}
b \in \mathrm{Der}(\mcG)
\textnormal{ and }\\
y \in \Dom(\mcG,\mbj,x) 
\end{array}
\right\}\\
=\;& 
\sup
\left\{
\left|
\hat{H}_{S}^{\mbj}[
\mbbD^{b,\mbj}_{\mcG} 
\phi]
(x,y)
\right|:\ 
\begin{array}{c}
b \in \mathrm{Der}(\mcG)
\textnormal{ and }\\
y \in \Dom(\mcG,\mbj,x) 
\end{array}
\right\}\\[1ex]
\lesssim\;
&
\sup
\left\{
\left|
\left(
\mbbD^{\bar{b},\mbj}_{\set{S}}
\mbbD^{b,\mbj}_{\mcG} 
\phi
\right)
(x,y,z)
\right|
:\ 
\begin{aligned}
&\bar{b} \in \mathrm{Der}(\set{S}),\ b \in \mathrm{Der}(\mcG),\\[-1ex]
&y_{\tilde{N}_{\mcG}} \in \Dom(\mcG,\mbj,x),\\[-1ex]
\textnormal{and}&\ z_{\tilde{N}(S)} \in \Dom(\set{S},\mbj,x)
\end{aligned}
\right\}\\
&
\qquad
\cdot
2^{\omega(S) \exte^{\mbj}(S)}
\|K(S)\|_{\mft}
\cdot
\| \xi \|_{|L(S)|,\c}\;,
\end{equs}
where we used the inductive hypothesis for $\{S\}$ in the final inequality.
\end{proof}
In the following lemma and later sections we use the following notation: for any $K(\sT)$ connected subset $A \subset N(\sT)$ we write $T(A)$ for the maximal subtree of $\sT$ with true node set $A$ and also write $\rho_{A} \eqdef \rho_{T(A)}$. 
Observe that if $e \in L(\sT)$ with $e_{\p} \in A$ then $e \in L(T(A))$.
\begin{lemma}\label{lem: proof of subdivfree} In the context of Lemma~\ref{lem: genvertbd}, the total homogeneity $\varsigma$ given by \eqref{def of sigma - genvertbd} is subdivergence free in $N_{\mcB}(S)$ for the set of scales $\mcN_{\go{G}}$.
\end{lemma}
\begin{proof}
Fix an arbitrary $\bo{T} \in \mcU_{\mcV}$.
For any $a \in \mathring{\bo{T}}$ we define 
\begin{equ}\label{def: rexpanded set}
\go{N}(a)
\eqdef 
\bo{L}_{a}
\sqcup 
\Big( 
\bigsqcup_{
\substack{
T \in C_{\mcB}(S)\\
\rho_{T} \in \bo{L}_{a}
}
}
\tilde{N}(T)
\Big)\;.
\end{equ}
We also set 
\begin{equ}\label{def: family of re-expanded sets}
\mcQ \eqdef 
\left\{
\go{N}(a):\ 
a \in \mathring{\bo{T}} \setminus \{\rho_{\bo{T}}\}
\right\}\;.
\end{equ}
Observe that  $N(T) \not \in \mcQ$ for any $T \in C_{\mcB}(S)$ and that any node-set $M \in \mcQ$ is edge connected by $\mcE^{\inte}(S)$. 

We define a map $\tilde{\varsigma}: 2^{N(S)} \setminus \{\emptyset\} \rightarrow \R$ as follows: for $M \subset N(S)$ we set
\begin{equs}\label{eq: summed homogeneity for genvertbd}
\tilde{\varsigma}(M) 
\eqdef
& 
\sum_{
\substack{
e \in K(S)\\
e \subset M}
}
\mfh(e)
-
\Big(
\sum_{
\substack{
B \in \pi,\\
B \not \subset M
}
}
|\mft( B \cap M)|_{\s,\c,L(\sT)}
\Big)
\\
&
\enskip
-
(|M| - 1) |\s|
-
\sum_{
\substack{
B \in \pi\\
B \subset M}}
\Big[
|\mft(B)|_{\s}
+ 
\fict(B)
\mathbbm{1}
\left\{
M=B
\right\}
\Big]\;
.
\end{equs}
Observe that if $M \subset N(S)$ is $K(S)$-connected then $\tilde{\varsigma}(M) \le - |T(M)^{0}_{\se}|_{\s}$. Furthermore, it is straightforward to check that for any $a \in \mathring{\bo{T}} \setminus \{ \rho_{\bo{T}}\}$, 
\begin{equ}\label{eq: summed homogeneity equality - genvertbd}
\sum_{
b \in \bo{T}_{\ge a}
}
\varsigma_{\bo{T}}(b) 
-
(|\bo{L}_{a}| - 1)|\s| 
= 
\tilde{\varsigma}(\go{N}(a))\;.
\end{equ} 
Our desired result will follow if we show $\tilde{\varsigma}(M) < 0$ for all $M \in \mcQ$. 
We do this by splitting into various cases. 
While we are not always explicit about it, at each point in this proof we are always restricting ourselves to the complement of all the cases treated previously. 

First we assume $M \in \mcQ$ contains no $K(S)$-connected components of cardinality more than $1$. In this subcase the first sum on the RHS of \eqref{eq: summed homogeneity equality - genvertbd} vanishes and by $\mcE^{\inte}(S)$-connectivity this forces $M \subset L(S)$. 
Moreover, there must be a unique block $\bar{B} \in \pi$ with $M \subset \bar{B}$. 
Therefore $(\mft,M) \in \mfL^{\all}_{\CCum}$ and $\tilde{\varsigma}(M)$ is bounded above by the RHS of \eqref{power counting cumulant bound} which is negative by Lemma~\ref{lemma: higher cumulants are fine}.
  
Next we treat the case where $M$ is not $K(S)$ connected but has at least one component of cardinality at least $2$. 
Then, for some $k \ge 1$, 
\begin{equ}\label{connectivity decomp}
M 
= 
\Big(
\bigsqcup_{j=1}^{k} M_{j} 
\Big) \sqcup \tilde{M}\;,
\end{equ} 
where the $M_{j}$ are the $K(S)$-components of $M$ of cardinality at least $2$ and $\tilde{M} \subset L(S)$ is given by the union of the $K(S)$-components of $M$ of cardinality $1$. 

Now we treat the subcase where $k=1$ and $\tilde{M}$ is non-empty. 
By $\mcE^{\inte}(S)$-connectivity there exist $B_{1},\dots,B_{n} \in \pi$ such that $\cup_{m=1}^{n} B_{m} \supset [\tilde{M} \sqcup L(T(M_{1})]$ and, for every $m \in [n]$, one has $B_{m} \cap L(T(M_{1})) \not = \emptyset$.
Now if it is the case that for every $m$ one has $B_{m} \cap M \in \mfL_{\CCum}$ then we are done since one has
\[
\tilde{\varsigma}(M)
\le
- |T(M_{1})^{0}_{\se}|_{\s}
- |\mft(\tilde{M})|_{\s} - |\s| \cdot |\tilde{M}|
< 0
\]
where the second inequality comes from \eqref{super-reg 0}. 
On the other hand, if there exists $m \in [n]$ for which $B_{m} \cap M \not \in \mfL_{\CCum} $ then it follows that $B_{m} \cap M = B_{m} \cap L(T(M_{1})) = \{u\}$ so $|\mft(B_{m} \cap M)|_{\s,\c,L(\sT)} = 0$. Thus by adding and subtracting $|\mft(u)|_{\s}$ we get
\[
\tilde{\varsigma}(M)
\le
- \Big[ |T(M_{1})^{0}_{\se}|_{\s} 
-|\mft(\mfu)|_{\s} \Big]
- |\mft(\tilde{M})|_{\s} - |\s| \cdot |\tilde{M}|
< 0\;,
\]
where in the second inequality we used \eqref{super-reg 0} to see that the bracketed quantity is strictly positive. 
This finishes the subcase where $k =1$ and $|\tilde{M}| \not = 0$. 
We next treat the subcase where $k \ge 2$ and $|\tilde{M}|$ is arbitrary. 
Then one has 
\begin{equs}\label{eq: subdivergence - components}
\tilde{\varsigma}(M)
\le&
-(k-1) |\s| - \sum_{j=1}^{k}
|T(M_{j})^{0}_{\se}|_{\s}
- |\mft(\tilde{M})|_{\s} - |\s| \cdot |\tilde{M}|\\
\le&
-(k-1) |\s| - \sum_{j=1}^{k}
|T(M_{j})^{0}_{\se}|_{\s}
<
-(k-1)|\s| + \frac{|\s|}{2}k \le 0 \;,
\end{equs}
where in last inequality we used \eqref{super-reg 0}.

We now treat the case where $M \in \mcQ$, is $K(S)$-connected, and $T(M)$ is not compatible with $\pi$. 
It follows that there exists some $B \in \pi$ with $B \not \subset L(T(M))$ and $B \cap L(T(M)) \not = \emptyset$. 
Then one has 
\[
\tilde{\varsigma}(M)
=
-|T(M)^{0}_{\se}|_{\s} 
+
\big(
|\mft(B \cap L(T(M)))|_{\s}
-
|\mft(B \cap L(T(M)))|_{\s,\c,L(\sT)}
\big)
< 0\;,
\]
where the final inequality comes Definition~\ref{def: cumulant strong subcriticality}.

We are now left with the case that $M \in \mcQ$, $M$ is $K(S)$-connected, and $T(M)$ is compatible with $\pi$.
In this case $\tilde{\varsigma}(M) = - |T(M)^{0}_{\bar{\mfe}}|_{\s}$.  
We claim that indeed $|T(M)^{0}_{\bar{\mfe}}|_{\s} > 0$ as required since one would otherwise have $T(M) \in \Div$ and, combining this with the fact that $T(M)$ is compatible with $\mcB$ and $\pi$, this would force $T \not \in \mcQ$ by the definition of the set of scale assignments $\mathring{\mcN}_{S}$.
\end{proof}
\subsubsection{Estimates with positive renormalizations}\label{subsec: positive renorm}
To make formulas more readable we introduce the notation $\exp_{2} [ t ]
\eqdef
2^{t}$ for $t \in \R$. 
Also, for $u,v \in \tilde{N}(\mcB)^{c}$ we define $\mbj_{\mcB}(u,v)$ as in \eqref{pathfinder scale}. 
\begin{lemma}\label{lem: cut edges estimate}
Let $e \in \mathscr{B}$ and $c > 0$. 
Uniform in $\mbj \in \partial \mcN_{\mcB,\lambda}$ satisfying  
\begin{equ}\label{triangle inequality for positive renormalized}
|\mbj_{\mcB}(e_{\ch},e_{\p}) - j_{e}|,\ |\mbj_{\mcB}(e_{\p},\logof) - j_{\{e_{p},\logof\}}| < c\;,
\end{equ} 
any multi-index $k$ supported on $\{e_{\ch},e_{\p}\}$, and $x = (x_{\logof},x_{e_{\ch}},x_{\e_{\p}})$ such that $\link{x_{e_{\ch}}}{x_{\logof}}{j_{\{\logof,e_{\ch}\}}}$ and $\link{x_{e_{\p}}}{x_{\logof}}{j_{\{\logof,e_{\p}\}}}$, one has the estimate
\begin{equ}
\left|
D^{k}
\KerTilde^{\{e\}}_{0,\mbj}(x)
\right|
\lesssim
\|K_{\mft(e)}\|_{\mft(e)}
2^{
\eta(j_{e},j_{\{\logof,e_{\p}\}},j_{\{\logof,e_{\ch}\}},k_{e_{\ch}},k_{e_{\p}},e)
}\;
\end{equ}
where $\eta(j_{e},j_{\{\logof,e_{\p}\}},j_{\{\logof,e_{\ch}\}},k_{e_{\ch}},k_{e_{\p}},e)$ is given by
\begin{equ}
\begin{cases}
\gamma(e)(j_{e} - j_{\{\logof,e_{\p}\}}) + (\mfh(e) + |k_{e_{\ch}}|_{\s})j_{e}
+|k_{e_{\p}}|_{s}j_{\{\logof,e_{\p}\}}
&
\textnormal{if }e \in \cD,\\
-(\gamma(e) - 1 + |k_{e_{\p}}|_{\s})j_{\{\logof,e_{\p}\}}
+
(|k_{e_{\ch}}|_{\s} + \mfh(e) + \gamma(e)-1)j_{\{\logof,e_{\ch}\}}
&
\textnormal{if }e \in \cS\;.
\end{cases}
\end{equ}
Furthermore, the LHS vanishes unless the condition $\link{x_{e_{\ch}}}{x_{e_{\p}}}{j_{e}}$ holds. 
\end{lemma}
\begin{proof}
We prove the lemma in the case when $k=0$, upon inspection of how each type of derivative modifies a Taylor expansion (or Taylor remainder) then the estimate for $k \not = 0$ follows easily. 

We first treat the case when $e \in \cD$ in which case $\KerTilde^{\{e\}}_{0,\mbj}(x)$ is given by
\begin{equ}\label{eq: positively renormalized edge}
\Big(
\Ker_{0}^{\{e\}}(x_{e_{\p}}, x_{e_{\ch}})
-
\sum_{
m \in \mathrm{Der}(e)
}
\frac{(x_{e_{\p}} - x_{\logof})^{m}}{m!}
(D^{m}\Ker_{0}^{\{e\}})(x_{\logof},x_{e_{\ch}})
\Big) \Psi^{j_{e}}(x_{e_{\p}} - x_{e_{\ch}}),
\end{equ}
where $\mathrm{Der}(e)$ denotes the set of multi-indices supported on $e_{\p}$ with $|m|_{\s} < \gamma(e)$. 
By using \cite[Prop.~A.1]{Regularity} it follows that, uniform in $x$ satisfying $\link{x_{e_{\p}}}{x_{e_{\ch}}}{j_{e}}$ (enforced by the second factor of \eqref{eq: positively renormalized edge}) and $\link{x_{e_{\p}}}{x_{\logof}}{j_{\{\logof,e_{\p}\}}}$ one has a bound
\begin{equation*}
\begin{split}
&
\Big|
\Ker_{0}^{\{e\}}(x_{e_{\p}}, x_{e_{\ch}})
-
\sum_{
m \in \mathrm{Der}(e)
}
\frac{(x_{e_{\p}} - x_{\logof})^{m}}{m!}
(D^{m}\Ker_{0}^{\{e\}})(x_{\logof},x_{e_{\ch}})
\Big|\\
\lesssim& 
\sup_{n \in \partial_{e_{\p}} \mathrm{Der}(e)}
|x_{e_{\p}} - x_{\logof}|^{|n|_{s}}
\cdot
\sup_{
\substack{
y,y' \in \R^{d}\\
\link{y}{y'}{j_{e}}
}
}
|D^{n}\Ker_{0}^{\{e\}}(y, y')|\\
\lesssim&
\max_{n \in \partial_{e_{\p}} \mathrm{Der}(e)}
2^{-j_{\{\logof,e_{\p}\}}|n|_{s}}
\cdot
\|K_{\mft(e)}\|_{\mft(e)} \cdot
2^{(\mfh(e) + |n|_{\s})j_{e}}\;.
\end{split}
\end{equation*}
Here the use of the scale $j_{e}$ vs $j_{\{\logof,e_{\ch}\}}$ is interchangeable -- the combinations of the conditions \eqref{triangle inequality for positive renormalized}, the fact that $\mbj(e_{\p},\logof) > \mbj(e_{\ch},e_{p})$, $\link{x_{e_{\p}}}{x_{e_{\ch}}}{j_{e}}$, $\link{x_{e_{\p}}}{x_{\logof}}{j_{\{\logof,e_{\p}\}}}$, and $\link{x_{e_{\ch}}}{x_{\logof}}{j_{\{\logof,e_{\ch}\}}}$ and the triangle inequality mean that there exists some combinatorial constant $C > 0$, depending on $c$, such that if $|j_{\{\logof,e_{\ch}\}} - j_{e}| \le C$ does not hold then our domain in $x$ is empty.

The desired bound now follows upon using the constraint $\link{x_{e_{\p}}}{x_{\logof}}{j_{\{\logof,e_{\p}\}}}$, the fact that $j_{\{\logof,e_{\p}\}} + 2c > j_{e}$, and that $\min_{n \in \partial_{e_{\p}} \mathrm{Der}(e)} |n|_{\s} = \gamma(e)$. 
Now we turn into the case where $e \in \cS$. By a similar triangle inequality argument we may assume that there exists some combinatorial constant $C > 0$, depending on $c$, such that 
\begin{equ}\label{triangle inequality}
j_{\{\logof,e_{\ch}\}} + C \ge j_{e} \ge j_{\{\logof,e_{\p}\}}\;.
\end{equ}
Under this assumption we do a term by term estimate which gives
\begin{equation*}
\begin{split}
&
\Big|
\sum_{
m \in \mathrm{Der}(e)
}
\frac{(x_{e_{\p}} - x_{\logof})^{m}}{m!}
D^{m}\Ker_{0}^{\{e\}}(x_{\logof},x_{e_{\ch}})
\Big|\\
\lesssim 
&
\max_{m \in \mathrm{Der}(e)}
|(x_{e_{\p}} - x_{\logof})^{m}| 
\cdot
\sup_{
\substack{
y,y' \in \R^{d}\\
\link{y}{y'}{j_{\{\logof,e_{\ch}\}}}
}
}
|D^{m}\Ker_{0}^{\{e\}})(y,y')|\\
\lesssim&
\|K_{\mft(e)}\|_{\mft(e)} \cdot 
\max_{m \in \mathrm{Der}(e)} 
2^{(\mfh(e) + |m|_{\s})j_{\{\logof,e_{\ch}\}} - |m|_{\s}j_{\{\logof,e_{\p}\}}}\\
\lesssim
&
\|K_{\mft(e)}\|_{\mft(e)}
\cdot
2^{
(\mfh(e) + \gamma(e) - 1)
j_{\{\logof,e_{\ch}\}}
-
(\gamma(e)-1)j_{\{\logof,e_{\p}\}}
}\;,
\end{split}
\end{equation*}  
where in the final inequality we used \eqref{triangle inequality} and that $\max_{m \in \mathrm{Der}(e)} |m|_{\s} = \gamma(e) - 1$.
\end{proof}
\begin{lemma}\label{lem: top genvert bound}
Let $S \in \overline{\mcB}$, $c > 0$, and $k \in (\N^{d})^{\allnodes}$ be supported on $N^{\downarrow}(S) \sqcup \{ \rho_{S} \}$. 
One has, uniform in $\mbj \in \CE^{\exte}(S)$ satisfying \eqref{triangle inequality for positive renormalized} for every $e \in \mfC_{\mcB}$, and $x \in \R^{\tilde{N}(S)^{c}}$ satisfying  \tabularnewline
\begin{equs}
\link{x_{\rho_{S}}}{x_{\logof}}{j_{\{\logof,\rho_{S}\}}}\;,  \quad
\link{x_{e_{\ch}}}{x_{\logof}}{j_{\{\logof,e_{\ch}\}}}
\quad 
\forall
e \in K^{\downarrow}(S)\;,
\end{equs}
the bound, 
\begin{equation}\label{e: toplevel genvert bound}
\begin{split}
&
\left|
D^{k}
\hat{H}_{S}^{\mbj} \left[ 
\KerTilde_{0,\mbj}^{\cB \cap K^{\downarrow}(S)} 
\ke{K^{\downarrow}(S) \setminus \cB}{0,\mbj}
\powroot{\tilde{N}(S)}{\sn,\mbj}{\logof}
\right](x)
\right|\\
\lesssim&\ 
\exp_{2} 
\Bigg[
\omega(S) \exte^{\mbj}(S) + 
j_{(\rho_{S},\logof)} |\sn(\tilde{N}(S))|_{\s}
+
\sum_{e \in K^{\downarrow}(S) \setminus \cB} (\mfh(e) + |k_{e_{\ch}}|_{\s}) 
j_{e}
\\
&
\enskip
\qquad
+
\Big(
\sum_{e \in K^{\downarrow}(S) \cap \cS}
-(\gamma(e) - 1)j_{\{\logof,\rho_{S}\}} + (\mfh(e) + |k_{e_{\ch}}|_{\s} + \gamma(e) - 1)j_{\{\logof,e_{\ch}\}}
\Big)\\
&
\enskip
\qquad
+
\Big(
\sum_{e \in K^{\downarrow}(S) \cap \cD} \gamma(e)(j_{e} - j_{(\logof,\rho_S)}) + (\mfh(e) + |k_{e_{\ch}}|_{\s}) 
j_{e} \Big)
\Bigg]\\ 
&
\qquad
\cdot
\|\overline{K}^{\downarrow}(S) \|_{\mft}
\cdot
\| \xi \|_{|L(S)|}\;.
\end{split}
\end{equation}
Furthermore, there exists a combinatorial constant $C > 0$ such that the LHS vanishes unless $\link{x_{\rho_{S}}}{x_{e_{\ch}}}{j_{e}}$ for all $e \in K^{\downarrow}(S)$, and $|j_{\{\logof,u\}} - j_{(\logof,\rho_S)}| \le C$ for each $u \in \tilde{N}(S)$.
\end{lemma}
\begin{proof}
First observe that the differential operator $D^{k}$ commutes with $\hat{H}_{S}^{\mbj}$.
We use Lemma~\ref{genvertbd}, the desired bound is obtained by estimating $\|D^{k}\Theta\|_{S,\mbj}(x_{\tilde{N}(S)^{c}})$ where $\Theta$ is given by either: $\powroot{\{u\}}{\sn,\mbj}{\logof}$ for $u \in \tilde{N}(S)$, $\ke{\{e\}}{0,\mbj}$ for $e \in K^{\downarrow}(S)$, or $\KerTilde_{0,\mbj}^{\{e\}}$ for $e \in K^{\downarrow}(S) \cap \cB$.

The proof of the first two scenarios mirror that of Lemma~\ref{lem: kernel bound and support} and one obtains 
\begin{equs}
\|D^{k}\Theta\|_{S,\mbj}(x_{\tilde{N}(S)^{c}}) 
\lesssim
\begin{cases}
2^{-|\sn(u)|_{\s} j_{\{\logof,u\}}} 
\mathbbm{1}\{ \link{x_{\rho_{S}}}{x_{\logof}}{j_{\{\logof,u\}}} \} & 
\textnormal{if } \Theta = \powroot{\{u\}}{\sn,\mbj}{\logof}\;,\\
2^{(\mfh(e) + |k_{e_{\ch}}|_{\s})j_{e}}\|K_{\mft(e)}\|_{\mft(e)}
\mathbbm{1}\{ \link{x_{\rho_{S}}}{x_{e_\ch}}{j_{e}}\} & 
\textnormal{if } \Theta = \ke{\{e\}}{0,\mbj}\;.\\
\end{cases}
\end{equs} 
Our bound for the factor $\powroot{\tilde{N}(S)}{\sn,\logof}{\mbj}$ then gives us the constraint that  $\link{x_{\rho_{S}}}{x_{\logof}}{j_{\{\logof,u\}}}$ for every $u \in \tilde{N}(S)$. 
Combining this with the inherited $\link{x_{\rho_{S}}}{x_{\logof}}{j_{\{\logof,\rho_{S}\}}}$ constraint it follows that there exists a combinatorial constant $C > 0$ such that the LHS of \eqref{e: toplevel genvert bound} vanishes unless 
\begin{equ}\label{scale collapse}
|j_{\{\logof,u\}} - j_{(\logof,\rho_S)}| \le   C
\textnormal{ for every }u \in \tilde{N}(S)\;.
\end{equ}
We now turn to the case where $\Theta = \KerTilde_{0,\mbj}^{\{e\}}$ for $e \in K^{\downarrow}(S) \cap \cB$ which further splits into two subcases depending on whether $e \in \cD$ or $e \in \cS$. 
We claim that one has
\begin{equation*}
\begin{split}
\|D^{k}\KerTilde^{\{e\}}_{0,\mbj}\|_{S,\mbj}(x_{\tilde{N}(S)^{c}})
\lesssim
\mathbbm{1}
\{
\link{x_{\rho_{S}}}
{x_{e_{\ch}}}
{j_{e}}
\}
\|K_{\mft(e)}\|_{\mft(e)}
2^{\eta(j_{e},j_{\{\logof,\rho_{S}\}},j_{\{\logof,e_{\ch}\}},k_{e_{\ch}},0,e)}\;,
\end{split}
\end{equation*}
where $\eta(\cdot)$ is as in the statement of Lemma~\ref{lem: cut edges estimate}.
We only sketch the proofs here since the details are very similar to those of Lemma~\ref{lem: cut edges estimate}.

Writing $\KerTilde^{\{e\}}_{0,\mbj}(z) = \KerTilde^{\{e\}}_{0,\mbj}(z_{e_{\p}},z_{e_{\ch}},z_{\logof})$ we claim that for any $b \in \widetilde{\mathrm{Der}}(\set{S})$ we have, using a Taylor remainder bound or a term-by-term Taylor series bound, the following estimate, uniform $\mbj \in \partial \mcN_{S}$ satisfying our assumptions, \eqref{scale collapse} and in $x_{\tilde{N}(S)^{c}}$ satisfying  both $\link{x_{\rho_{S}}}{x_{\logof}}{j_{\{\logof,\rho_{S}\}}}$ and $\link{x_{e_{\ch}}}{x_{\logof}}{j_{\{\logof,e_{\ch}\}}}$,  
\begin{equs}\label{double renormalized edge}
&\sup_{y_{\tilde{N}(S)} \in \Dom(\set{S},\mbj,x)}
\left|
D^{b+k}\KerTilde^{\{e\}}_{0,\mbj}(y_{e_{\p}},x_{e_{\ch}},x_{\logof})
\right|\\
&\lesssim\ 
\sup_{
\substack{
y_{\tilde{N}(S)} \in \Dom(\set{S},\mbj,x_{\rho_{S}})\\
n \in A_{e}
}
}
|y_{u(e)} - x_{\logof}|^{|n|_{s} - |b_{e_{\p}}|_{\s}}
\cdot
\sup_{
\substack{
w,w' \in \R^{d}\\
\link{w}{w'}{j_{e}}
}
}
|D^{n+k_{e_{\ch}}}\Ker_{0}^{\{e\}}(w - w')|\\
&
\lesssim
\|K_{\mft(e)}\|_{\mft(e)}
\cdot
2^{(\mfh(e) + |k_{e_{\ch}}|_{\s})\theta_{2}(\mbj,e) + |b_{e_{\p}}|_{\s}j_{\{\logof,\rho_{S}\}}}
\max_{n \in A_{e}}
2^{
|n|_{\s}
\theta_{1}(\mbj,e)}.
\end{equs}
Here, if $e \in \cD$ one has $u(e) \eqdef e_{\p}$, $A_{e} \eqdef \partial_{e_{\p}} \mathrm{Der}(e)$, $\theta_{1}(\mbj,e) \eqdef (j_{e} - j_{\{\logof,\rho_{S}\}})$, and $\theta_{2}(\mbj,e) \eqdef j_{e}$.
On the other hand, if $e \in \cS$ then $A_{e} \eqdef \mathrm{Der}(e)$, $\theta_{1}(\mbj,e) \eqdef (j_{\{\logof,e_{\ch}\}} - j_{\{\logof,\rho_{S}\}})$, and $\theta_{2}(\mbj,e) \eqdef j_{\{\logof,e_{\p}\}}$.

We also used $\link{y_{e_{\p}}}{x_{\logof}}{j_{\{\logof,\rho_{S}\}}}$ which comes from the inherited constraint $\link{x_{\rho_{S}}}{x_{\logof}}{j_{\{\logof,\rho_{S}\}}}$ along with the domain of the supremum for $y_{e_{\p}}$.  
We also used the fact that by \eqref{triangle inequality for positive renormalized} and \eqref{scale collapse} there is a combinatorial constant $C' > 0$, depending on $c$, such that $j_{\{\logof,e_{\p}\}} + C' > j_{e}$ if $e \in \cD$ and $j_{\{\logof,e_{\p}\}} \le j_{e} + C'$ if $e \in \cS$
which is good enough, the discrepancies $C'$ only contributes to the constant of proportionality in our bound. 

The desired bound on $\|\KerTilde_{0,\mbj}^{\{e\}}\|_{S,\mbj}(x_{\tilde{N}(S)^{c}})$ follows from \eqref{double renormalized edge} by observing from multiplying the last line of \eqref{double renormalized edge} by $2^{-\exte^{\mbj}(S)|b|_{e_{\p}}}$, $\exte^{\mbj}(S) \ge j_{\{\logof,\rho_{S}\}}$, and using the observation of the above paragraph when taking the maximum over $n$. 
For the indicator function on the RHS one observes that the LHS of \eqref{double renormalized edge} vanishes  unless $\link{x_{\rho_{S}}}{x_{e_{\ch}}}{j_{e}}$, this follows by similar reasoning as in Lemma~\ref{lem: kernel bound and support}.
\end{proof}
\subsection{Full control over a single tree}\label{subsec: full control} 
We start by taking a quotient of the multigraph $\CE^{\inte}(\sT^{\ast})$ (defined in \eqref{external edges for a single tree}) by contracting every $T \in \mcB$ and identifying it with $\rho_{T}$, doing this we obtain a multigraph $\go{H}$ on the vertex set $\CX \eqdef \tilde{N}(\overline{\mcB})^{c}$. 
In particular,
\begin{equation}\label{multigraph for single tree}
\go{H}
\eqdef\ 
K(\mcB,\sT) 
\sqcup 
\Big\{ 
\{e_{\ch}, \rho_{T}\}:\ 
\begin{array}{c}
T \in \overline{\mcB},\\
e \in K^{\downarrow}(T)
\end{array}
\Big\}
\sqcup
\{e \in \CE_{\pi}:\ e \subset L(\mcB,T)\}\;.
\end{equation}
We write $\mfq: \N^{\CE^{\inte}(\sT^{\ast})} \rightarrow \N^{\go{H}}$ for the natural bijection between these two sets and then set $\mcN_{\go{H},\lambda} \eqdef \mfq(\partial \mcN_{\mcB,\lambda})$ and $\mcN_{\go{H}} \eqdef \cup_{\lambda \in (0,1]} \mcN_{\go{H},\lambda}$. 

From now on we switch viewpoints replacing, for every $\lambda \in (0,1]$, the family $\mathring{\CW}_{\lambda} = (\mathring{\CW}_{\lambda}^{\mbj})_{\mbj \in \partial \mcN_{\mcB,\lambda}}$  with $\mathring{\CW}_{\lambda} = (\mathring{\CW}_{\lambda}^{\mbn})_{\mbn \in \mcN_{\go{H}}}$ where we set, for $\mbn \in \mcN_{\go{H}}$, $\mathring{\CW}_{\lambda}^{\mbn} = \mathring{\CW}_{\lambda}^{\mfq^{-1}(\mbn)}$ if $\mbn \in \mcN_{\go{H},\lambda}$ and $\mathring{\CW}_{\lambda}^{\mbn} = 0$ otherwise. 

We also define $\hat{\mfq}:N^{\ast} \rightarrow \CX$ via setting $\hat{\mfq}(u) \eqdef \rho_{T}$ if there exists $T \in \overline{\mcB}$ with $u \in N(T)$, we set $\hat{\mfq}(u) \eqdef u$ otherwise.

We now define a total homogeneity $\varsigma$ on the trees of $\widehat{\CU}_{\CX}$ by setting  
\begin{equation}\label{single wick homogeneity}
\begin{split}
\varsigma
\eqdef& 
-\sum_{u \in N(\sT)}
|\sn(u)|_{\s} 
\delta^{\uparrow}[ \{\hat{\mfq}(u),\logof\} ]
+
\sum_{
\substack{
B \in  \pi\\
B \subset \leavesleft{\mcB}{\sT} 
}
}
\c^{B}
+
\sum_{T \in \overline{\mcB}}
\omega(T)
\delta^{\uparrow}[\rho_{T}]
+
\varsigma^{R}
\\
&
+
\sum_{
\substack{ 
e \in \kernelsleft{\mcB}{\sT}\sqcup K^{\downarrow}(\overline{\mcB})\\
e \not \in \enS
}
}
\mfh(e)
\delta^{\uparrow}[\{e_{\ch},\hat{\mfq}(e_{\p})\}]
+
\sum_{e \in \cD} 
\gamma(e)
\left(
\delta^{\uparrow}[\{e_{\ch},\hat{\mfq}(e_{\p})\}]
-\delta^{\uparrow} 
[\{\logof,\hat{\mfq}(e_{\p})\}]
\right)\\[1.5ex]
&
+
\big[
\sum_{e \in \enS} 
(\gamma(e)-1) 
\left(
\delta^{\uparrow}[\{e_{\ch},\logof\}]
-
\delta^{\uparrow} 
[
\{\logof,\hat{\mfq}(e_{\p})\}]
\right)
+
\mfh(e)
\delta^{\uparrow}
[
\{e_{\ch},\logof\}]
\big],
\end{split}
\end{equation}
where the the total homogeneity $\varsigma^{R}$ is defined by setting, for each $\bo{T} \in \widehat{\CU}_{\CX}$,  
\[
\varsigma_{\go{T}}^{R}
\eqdef
\sum_{B \in R(\go{T})} 
\fict(B)
\left(
\delta^{\Uparrow}[B]
-
\delta^{\uparrow}[B]
\right)
\]
and
$R(\bo{T}) 
\eqdef 
\left\{ B \in \pi:\  
B \subset \leavesleft{\mcB}{\sT},\ 
\fict(B) > 0,\
\textnormal{ and }
\go{L}_{B^{\uparrow}}
= 
B
\right\}$. 
We can then state the final lemma of this section.
\begin{lemma}\label{lem: bound on full integrand}
If one defines $\varsigma$ as in \eqref{single wick homogeneity}, then uniform in $\lambda \in (0,1]$, there exists a $F_{\lambda} \in \mathrm{Mod}_{\genwickleaves}(\mathring{\CW}_{\lambda})$ such that 
\[
\|F_{\lambda}
\|_{\varsigma,\mcN_{\go{H}},\wickleaves}
\lesssim 
\lambda^{-|\s|} \|K(\sT)\|_{\mft}
\cdot
\|\xi\|_{|L(\sT) \setminus \wickleaves|}\;,
\] 
where we are using the notation of \eqref{def: a derivative norm}. 
\end{lemma}
When we say this condition holds uniform in $\lambda$ we are including uniformity in the combinatorial constant appearing in \eqref{domain condition for modification}.
\begin{proof} 
We set, for each $\mbn \in \mcN_{\go{H}}$, 
\[
F^{\mbn}_{\lambda}
\eqdef
\begin{cases}
\mathring{\CW}^{\mbn}_{\lambda} & \textnormal{if } \mbn \in \mcN_{\go{H},\lambda}
\textnormal{ and } R(\mcT(\mbn)) = \emptyset\;,\\
\mathring{F}_{\lambda}^{\mbn} & \textnormal{if } \mbn \in \mcN_{\go{H},\lambda}
\textnormal{ and } R(\mcT(\mbn)) \not = \emptyset\;,\\
0 & \textnormal{otherwise.}
\end{cases}
\]
We will define $\mathring{F}_{\lambda}^{\mbn}$ on a tree by tree basis later. 
It suffices to check the domain condition and desired supremum bounds in each of the three cases separately.
For the first case one this is done by combining all the lemmas of Section~\ref{sec: inductive bound}, and all the previous lemmas of this section, note that by the condition \eqref{triangle condition for scales} there is come combinatorial constant $c$ such that one has the condition \eqref{triangle inequality for positive renormalized} for every $e \in \mcB$ and every $\mbn$ appearing in our supremum . 

For the second case we fix $\go{T} \in \CU_{\CX}$ with $R(\go{T}) \not = \emptyset$ and $\mbn \in \mcN_{\go{H},\lambda}$ with $\mcT(\mbn) = \go{T}$. 
We then set
\[
\mathring{F}_{\lambda}^{\mbn}
\eqdef
\mathrm{Cu}^{\hat{R}(\go{T})}_{\pi,\mbn}
\cdot
\Big(
\prod_{B \in R(\go{T})}
(1 - \mathscr{Y}_{B})
\Big)
\big[
\hat{G}_{\lambda}^{\mbj}
\big]\;,
\]
where $\hat{R}(\go{T}) \eqdef \bigcup_{B \in R(\go{T})} B$, the operators $\mathscr{Y}_{B}$ were defined in \eqref{def: fict renorm}, and
\begin{equs}
\hat{G}_{\lambda}^{\mbj}
\eqdef&\;
\psi^{\lambda}_{\mbj}
\cdot
\ke{
\kernelsleft{\mcB}{\sT}\setminus \cB}{0,\mbj}
\mathrm{Cu}^{\leavesleft{\mcB}{\sT} \setminus \hat{R}(\go{T})}_{\pi,\mbn}
\KerTilde_{0,\mbj}^{\cB \setminus K^{\downarrow}(\overline{\mcB})}
\powroot{\nodesleft{\mcB}{\sT}}{\sn}{\logof,\mbj}\\
&
\quad \cdot
\prod_{S \in \overline{\mcB}}
\hat{H}_{S}^{\mbj} \left[ 
\KerTilde_{0,\mbj}^{\cB \cap K^{\downarrow}(S)} 
\ke{K^{\downarrow}(S) \setminus \cB}{0,\mbj}
\powroot{\tilde{N}(S)}{\sn,\mbj}{\logof}
\right].
\end{equs}
An important observation is that the sets $\wickleaves$ and $\bigcup_{B \in R(\go{T})} B$ must be disjoint.
Therefore one can use the same argument as in Lemma~\ref{genvertbd} to show that for any $x \in (\R^{d})^{\wickleaves}$ one has
\[
\int_{\CX_{0} \setminus \wickleaves} dy\ 
\mathring{F}_{\lambda}^{\mbn}(x \sqcup y)
=
\int_{\CX_{0} \setminus \wickleaves} dy\ 
\tilde{\CW}_{\lambda}^{\mbn}(x \sqcup y).
\]
The domain condition for $\mathring{F}_{\lambda}^{\mbn}$ is also straightforward to check. 
The desired supremum bound for $\mathring{F}_{\lambda}^{\mbn}$ follows by the same sort of Taylor remainder argument as in Lemma \ref{genvertbd}. 
For multi-indices $m$ supported on $\wickleaves$ one gets the desired supremum bounds for on the corresponding derivative by following the same argument -- note that the operator $D^{m}$ commutes with the $\mathscr{Y}_{B}$ and only acts on the $\hat{G}_{\lambda}^{\mbj}$. 
\end{proof} 
\section{Moments of a single tree}\label{sec: est on moments of trees}
We now state the propositions which combined with Theorem \ref{thm: half graph assumptions} establish the main theorem.
\begin{proposition}\label{final prop}
The total homogeneity $\varsigma$ given by \eqref{single wick homogeneity} satisfies the conditions of Theorem~\ref{thm: half graph assumptions} with $\CX$, $\CX_{0}$, $\go{H}$ and $\mcN_{\go{H}}$ as given, $\logof$ as the pinned vertex, $\go{H}_{\ast} \eqdef \{ \logof,\mainroot\}$, $\alpha \eqdef - |\bar{T}^{\bar \mfn,0}_{\bar \mfe}| - |\mft(\wickleaves)|_{\s} - |\s|$, and $\genwickleaves \eqdef \wickleaves$ with its type map $\mft$.  
\end{proposition}
\begin{proof}
It is straightforward to check $\varsigma$ is of order $- |\bar{T}^{\bar \mfn,0}_{\bar \mfe}| - |\mft(\wickleaves)|_{\s} - |\s|$. We split the verification of \eqref{integrability condition - big graph} and \eqref{decay at large scales} into Lemma~\ref{subdiv free big graph} and Lemma~\ref{large scale bound for big graph}.
\end{proof}
Before stating and proving Lemma~\ref{subdiv free big graph} and Lemma~\ref{large scale bound for big graph} we give introduce more notation and state an observation.

For any $u \in N(\sT)$ we write $\sT_{\ge}(u)$ for the maximal subtree of $\sT$ with true node set $\{v \in N(\sT): v \ge u \}$. 
Also, for any $K(\sT)$ connected $M \subset N(\sT)$ we write
\begin{equs}
M^{\cu}
&\eqdef
(L(\sT) \setminus \wickleaves) \cap M,\   
M^{\exte}
\eqdef
\wickleaves \cap M,\\ 
K(M)
&\eqdef
K(T(M)), 
\enskip
\textnormal{and}
\enskip
K^{\downarrow}(M)
\eqdef
K^{\downarrow}(T(M))\;.
\end{equs}
Finally, an important observation we will use is that for any $K(\sT)$-connected subset $A \subset N(\sT)$ one has
\begin{equation}\label{tree homogeneity}
|T(A)^{\sn}_{\se}|_{\s} 
+
\sum_{e \in K^{\downarrow}(A)}
|
\clearrootn[\sT_{\ge}(e)^{\sn}_{\se}]|_{\s}
=
|\sT_{\ge}(\rho_{A})^{\bar \mfn}_{\bar \mfe}|.
\end{equation} 
\begin{lemma}\label{subdiv free big graph}
The total homogeneity $\varsigma$, as defined by \eqref{single wick homogeneity}, satisfies condition \eqref{integrability condition - big graph} of Theorem~\ref{thm: half graph assumptions}.
\end{lemma}
\begin{proof}
We define a map $\tilde{\varsigma}: 2^{\allnodes} \rightarrow \R$ as follows, for $M \subset \allnodes$ we set
\begin{equs}\label{eq: subdiv free big graph homogeneity}
\tilde{\varsigma}(M) 
\eqdef
& 
-
(|M| - 1) |\s|
+
\sum_{
\substack{
e \in K(\sT) \setminus \enS\\
e \subset M}
}
\mfh(e)\\
&
-
\mathbbm{1}\{\logof \in M\}
\left[
|\sn(M \setminus \{\logof\})|_{\s}
+
\sum_{
e \in \cD}
\gamma(e) 
\mathbbm{1}
\left\{ 
\begin{array}{c}
e_{\ch} \not \in M\\
e_{\p} \in M
\end{array} 
\right\}
\right]\\
&
+
\mathbbm{1}\{\logof \in M\}
\sum_{e \in \enS}
\left[
\mfh(e)
+
(\gamma(e) - 1) \mathbbm{1}\{e_{\p} \not \in M\}
\right]
\mathbbm{1}\{ e_{\ch} \in M\}\\
&
\quad
-
|\mft(M^{\exte})|_{\s}
-
\sum_{
\substack{
B \in \pi\\
B \subset M^{\cu}
} 
} 
|\mft(B)|_{\s}
-
\sum_{
\substack{
B \in \pi\\
B \not \subset M^{\cu}
}
}
|\mft(B \cap M^{\cu})|_{\s,\c,L(\sT)}
- 
\sum_{
\substack{
B \in \pi\\
M = B
}
}
\fict(B)\;.
\end{equs}
We fix, for the remainder of this proof, $\bo{T} \in \CU_{\CX}$. 
We define
\begin{equ}\label{def: rexpanded set2}
\go{N}(a)
\eqdef 
\bo{L}_{a}
\sqcup 
\Big( 
\bigsqcup_{
\substack{
T \in C_{\mcB}(\sT)\\
\rho_{T} \in \bo{L}_{a}
}
}
\tilde{N}(T)
\Big),
\end{equ}
and also set $
\mcQ \eqdef 
\left\{
\go{N}(a):\ 
a \in \mathring{\bo{T}} \setminus \{\rho_{\bo{T}}\}
\right\}$.

We claim that for any $a \in \mathring{\bo{T}} \setminus \{\rho_{\bo{T}}\}$, 
\begin{equ}\label{eq: a resummed homogeneity}
\Big(
\sum_{
b \in \bo{T}_{\ge a}
}
\varsigma(b)
\Big) 
- 
(|\bo{L}_{a}| -1) |\s|
\le
\tilde{\varsigma}(\go{N}(a))\;.
\end{equ}
To obtain \eqref{eq: a resummed homogeneity} note that our choice of cut rule means that for each $M \in \mcQ$, 
\begin{enumerate} 
\item $e \in \enS$ and $e_{\p}, \logof \in M$ together imply $e_{\ch} \in M$.
\item $e \in \cD$ and $e_{\ch},e_{\p} \in M$ together imply $\logof \in M$.
\end{enumerate}
We have an inequality because of the third term on the last line of \eqref{eq: subdiv free big graph homogeneity}. 

To prove the lemma it suffices to show that for all $M \in \mcQ$ one has
\begin{equation}\label{goal0}
\tilde{\varsigma}(M) <  
\mathbbm{1}
\left\{
M \not \ni \logof
\right\}
\Big[
\mcb{h}_{\c,\ell}(M^{\exte})
\wedge 
\mcb{j}_{\ell}(M^{\exte})
\wedge
\frac{|\s|}{2}
\Big]\;,
\end{equation}
where $\mcb{h}_{\c,\ell}(\cdot)$ is defined as in \eqref{homogeneity gain from missed noise}.
When $M$ satisfies both $M \not \ni \logof$ and $M^{\exte} = \emptyset$ one can show $\tilde{\varsigma}(M) < 0 $ by copying the argument of Lemma~\ref{lem: proof of subdivfree} nearly verbatim so this case is done.
 
The proof of the lemma will be completed by proving \eqref{goal0} in the two remaining cases, either: (i) $M \in \mcQ$, $M \not \ni \logof$, and $M^{\exte} \not = \emptyset$, or (ii) $M \in \mcQ$ and $M \ni \logof$.

We first treat case (i), again by splitting into subcases. Note that in this case the RHS of \eqref{goal0} is non-negative so for some subcases we just show $\tilde{\varsigma}(M) < 0$. 

Note that in case (i) it is impossible for $M$ to consist of $K(\sT)$-components of cardinality $1$. 
By $\go{H}$-connectivity this would forces $M = M^{\cu}$ so $M^{\exte} = \emptyset$.
 
Now suppose we are in a subcase of case $(i)$ where $M$ is not $K(\sT)$-connected and has at least one $K(\sT)$-component of cardinality greater than or equal to $2$.
We then have a decomposition of $M$ as in \eqref{connectivity decomp} and same arguments used in Lemma~\ref{lem: proof of subdivfree} can be copied verbatim to show that one always has $\tilde{\varsigma}(M) < 0$ in this subcase.

What is left in case (i) is to establish \eqref{goal0} for any $K(\sT)$-connected $M \in \mcQ$ with $M^{\exte} \not = \emptyset$, and $M \not \ni \logof$ -- this is then an immediate consequence of the inequality $\tilde{\varsigma}(M) \le - |T(M)^{0}_{\se}|_{\s}$ and Definition~\ref{def: cumulant strong subcriticality}.

Now we turn to case (ii) where it is convenient to prove a stronger bound than \eqref{goal0}.
Define another function $\hat{\varsigma}:2^{N^\ast} \rightarrow \R$ by using the same definition for $\tilde{\varsigma}$ as in \eqref{eq: subdiv free big graph homogeneity} but replacing the entire last line of \eqref{eq: subdiv free big graph homogeneity} with $-|\mft(M^{\cu} \sqcup M^{\exte})|_{\s}$. 
In case $(ii)$ it suffices to show $\hat{\varsigma}(M) < 
0$. 
Let $M \in \mcQ$ with $M \ni \logof$, writing $\{M_{j}\}_{j=1}^{n}$ for the $K(\sT)$ connected components of $M \setminus \{ \logof \}$ it is straightforward to see that
\[
\hat{\varsigma}(M) 
= 
\sum_{j = 1}^{n} \hat{\varsigma}(M_{j} \sqcup \{ \logof\})
\]
so it suffices to prove $\hat{\varsigma}(\widehat{M} \sqcup \{\logof\}) < 0$ for all $\widehat{M} \in \widehat{\mcQ}$ where we have defined
\[ 
\widehat{\mcQ}
\eqdef
\left\{ 
\widehat{M} \subset N(\sT):\ 
\begin{array}{c}
\widehat{M} \textnormal{ is } K(\sT)-\textnormal{connected and }
\exists\ M \in \mcQ,\ u \in \CX_{0}\\
\textnormal{ such that }
M \ni \logof
\textnormal{ and }
\widehat{M} = \sT_{\ge}(u) \cap M
\end{array} 
\right\}\;.
\]
Observe that for any $\widehat{M} \in \widehat{\mcQ}$ and $S \in \mcB$ one has $N(S) \subset \widehat{M}$ or $N(S) \cap \widehat{M} = \emptyset$.
A second important observation is that for any $\widehat{M} \in \widehat{\mcQ}$ one has 
\begin{equ}\label{tree homogeneity 2}
\sum_{ e \in K^{\downarrow}(\widehat{M})}
|\clearrootn[\sT_{\ge}(e)^{\sn}_{\se}]|_{\s}
\le\ 
\sum_{e \in \cD \cap K^{\downarrow}(\widehat{M})}
\gamma(e)
\end{equ}
because (i) the conditions $e \in K^{\downarrow}(\widehat{M})$ and $|\clearrootn[\sT_{\ge}(e)^{\sn}_{\se}]|_{\s} > 0$ together force $e \in \cD$ and (ii) the definition of $\gamma(e)$.

The claim we will prove, which will finish the proof of this lemma, is that for every $\widehat{M} \in \widehat{\mcQ}$
\begin{equation}\label{eq: inductive goal}
\hat{\varsigma}(\widehat{M} \sqcup \{\logof\})
\le
\begin{cases}
- |\s| 
- 
|\sT_{\ge}(\rho_{\widehat{M}})^{\sn}_{\se}|_{\s} 
\quad &\textnormal{ if } 
\rho_{\widehat{M}} \not \in e_{\ch}(\enS)\\
-
{\rm frac}(
|\clearrootn[\sT_{\ge}(\tilde{e})_{\se}^{\sn}]|_{\s})
\quad
&
\textnormal{ if }
\rho_{\widehat{M}} = \tilde{e}_{\ch} \textnormal{ for }\tilde{e} \in \enS,
\end{cases}
\end{equation}
where for $t \ge 0$ we define ${\rm frac}(t) = t - \lfloor t \rfloor$. 
Note that in the second case of the RHS such an edge $\tilde{e}$ is necessarily unique, this makes the formula well-defined.
We prove \eqref{eq: inductive goal} inductively with respect to the cardinality of $\widehat{M} \cap \CX_{0}$.

We start with the base case where $|\widehat{M} \cap \CX_{0}| = 1$ -- here there are no edges $e \in \enS$ with $e \subset \widehat{M}$ and the LHS of \eqref{eq: inductive goal} is equal to
\begin{equation}\label{eq: zero integrability base case}
\begin{cases}
-|\s| - |T(\widehat{M})^{\sn}_{\se}|_{\s}
-
\displaystyle\sum_{
e \in \cD \cap K^{\downarrow}(\widehat{M})
}
\gamma(e)
& 
\textnormal{if }
\rho_{\widehat{M}} \not \in e_{\ch}^{-1}(\enS), \\
-
|T(\widehat{M})^{\sn}_{\se}|_{\s}
+
(\mfh(\tilde{e}) - |\s|)
+ 
\gamma(\tilde{e})-1
-
\displaystyle\sum_{
e \in \cD \cap K^{\downarrow}(\widehat{M})
}
\gamma(e)
&
\textnormal{if }
\tilde{e} \in \enS,\ 
\rho_{\widehat{M}} = \tilde{e}_{\ch}\;. 
\end{cases}
\end{equation}
Now let us verify that \eqref{eq: zero integrability base case} implies \eqref{eq: inductive goal}. When $\rho_{\widehat{M}} \not \in e_{\ch}^{-1}(\enS)$ this follows from \eqref{tree homogeneity 2}, when $\rho_{\widehat{M}} = \tilde{e}_{\ch}$ ones uses \eqref{tree homogeneity} and \eqref{tree homogeneity 2} to obtain
\begin{equs}\label{eq: bad renormalization}
\gamma(\tilde{e}) - 1
&=
|\clearrootn[\sT_{\ge}(\tilde{e})_{\se}^{\sn}]|_{\s}
-
{\rm frac}
(
|\clearrootn[\sT_{\ge}(\tilde{e})_{\se}^{\sn}]|_{\s}
)\\
&= 
(|\s| - \mfh(\tilde{e}))
+
|\sT_{\ge}(\rho_{\widehat{M}})_{\se}^{\sn}|_{\s}
-
{\rm frac}
(
|\clearrootn[\sT_{\ge}(\tilde{e})_{\se}^{\sn}]|_{\s}
)\\
&= 
(|\s| - \mfh(\tilde{e}))
+
|T(\widehat{M})^{\sn}_{\se}|_{\s}
+
\sum_{ e \in K^{\downarrow}(\widehat{M})}
|\clearrootn[\sT_{\ge}(e)_{\se}^{\sn}]|_{\s}
-
{\rm frac}
(
|\clearrootn[\bar{T}_{\ge}(\tilde{e})^{\bar \mfn}_{\bar{\mfe}}]|
) \\
&\le
(|\s| - \mfh(\tilde{e}))
+
|T(\widehat{M})^{\sn}_{\se}|_{\s}
+
\sum_{ e \in \cD \cap K^{\downarrow}(\widehat{M})}
\gamma(e)
-
{\rm frac}
(
|\clearrootn[\bar{T}_{\ge}(\tilde{e})^{\bar \mfn}_{\bar{\mfe}}]|
)\;.
\end{equs}
We now turn to the inductive step. 
We fix $q \ge 2$ and assume that \eqref{eq: inductive goal} holds for all elements of $\widehat{\mcQ}$ whose intersection with $\CX_{0}$ is of cardinality strictly less than $q$. 
Suppose that we are given $\widehat{M} \in \widehat{\mcQ}$ with $|\widehat{M} \cap \CX_{0}| = q$. 

We first define
\begin{equs}
\overline{M} &\eqdef 
\begin{cases}
N(S) & \textnormal{ if } \rho_{\widehat{M}} = \rho_{S}, 
\textnormal{ for some }S \in \overline{\mcB}\\
\{ \rho_{\widehat{M}} \} & \textnormal{ otherwise}, 
\end{cases}\\[1.5ex]
K^{\downarrow}_{\rho,\inte}(\widehat{M}) 
&\eqdef K^{\downarrow}(\overline{M}) \cap K(\widehat{M}),
\enskip 
\textnormal{and}
\enskip
K^{\downarrow}_{\rho,\exte}(\widehat{M}) 
\eqdef 
K^{\downarrow}(\overline{M}) \setminus K(\widehat{M})\;.
\end{equs}
Next, we decompose $\widehat{M} \setminus \overline{M}$ into $K(\sT)$-connected components $\{\widehat{M}_{j}\}_{j=1}^{n}$, then the LHS of \eqref{eq: inductive goal} is then equal to
\begin{equs}\label{eq: inductive step 1}
{}
&
- |T(\overline{M})^{\sn}_{\se}|_{\s}
- |\s|
+
\sum_{j=1}^{n} 
\Big[
\hat{\varsigma}(\widehat{M}_{j} \sqcup \{\logof\}) - |\mft(\widehat{M}^{\exte}_{j})|_{\s}
\Big]\\
&
+
\sum_{e \in K^{\downarrow}_{\rho,\inte}(\widehat{M})}
\Big[
\mathbbm{1}\{ e \not \in \enS \}
\mfh(e)
 - 
\mathbbm{1}\{ e \in \enS \}
(\gamma(e) - 1)
\Big]\\
&\quad
-
\sum_{e \in K^{\downarrow}_{\rho,\exte}(\widehat{M})}
\mathbbm{1}\{ e \in \cD \}
\gamma(e)
+
\mathbbm{1}
\left\{ 
\begin{array}{c}
\rho_{\widehat{M}} = \tilde{e}_{\ch}\\
\textnormal{for }\tilde{e} \in \enS 
\end{array}
\right\}
\left[
\mfh(\tilde{e})
+
(\gamma(\tilde{e}) - 1) 
\right]
\end{equs}
Note that the two summation sets appearing in the first two sets are in bijection where one assigns each $\widehat{M}_{j}$ to the unique $e$ with $e_{\ch} = \rho_{\widehat{M}_{j}}$. 

For each $j$, $\widehat{M}_{j} \in \widehat{\mcQ}$ and $|\widehat{M}_{j} \cap \CX_{0}| < q$ -- thus using our inductive hypothesis combined with \eqref{eq: bad renormalization} gives
\begin{equs}\label{eq: induction work 2}
{}
&
\sum_{j=1}^{n} 
\hat{\varsigma}(\widehat{M}_{j} \sqcup \{\logof\}) 
+
\sum_{e \in K^{\downarrow}_{\rho,\inte}(\widehat{M})}
\Big[
\mathbbm{1}\{ e \not \in \enS \}
\mfh(e)
 - 
\mathbbm{1}\{ e \in \enS \}
(\gamma(e) - 1)
\Big]\\
\le
& 
-
\sum_{e \in K^{\downarrow}_{\rho,\inte}(\widehat{M})}
\mathbbm{1}\{ e \not \in \enS\}
\left(
|\bar{T}_{\ge} (e_{\ch})^{\sn}_{\se}|_{\s} + |\s| - \mfh(e) 
\right)
+
\mathbbm{1}\{ e \in \enS \}
|\clearrootn[\sT_{\ge}(e)^{\sn}_{\se}]|_{\s}\\
=&
\sum_{e \in K^{\downarrow}_{\rho,\inte}(\widehat{M})}
|\clearrootn[\bar{T}_{\ge} (e)^{\sn}_{\se}]|_{\s}\;.
\end{equs}
On the other hand,
\begin{equ}\label{induction work 3}
-
\sum_{e \in K^{\downarrow}_{\rho,\exte}(\widehat{M})}
\mathbbm{1}\{ e \in \cD \}
\gamma(e)
\le
-
\sum_{e \in K^{\downarrow}_{\rho,\exte}(\widehat{M})}
|\clearrootn[\bar{T}_{\ge} (e)^{\sn}_{\se}]|_{\s}\;.
\end{equ}
By using \eqref{eq: induction work 2} and \eqref{induction work 3} in \eqref{eq: inductive step 1} we see $\hat{\varsigma}(\widehat{M} \sqcup \{ \logof \} )$ is bounded above by
\begin{equs}
{} 
&
- |T(\overline{M})^{\sn}_{\se}|_{\s}
-|\s|
+
\mathbbm{1}
\left\{ 
\begin{array}{c}
\rho_{\widehat{M}} = \tilde{e}_{\ch}\\
\textnormal{for }\tilde{e} \in \enS 
\end{array}
\right\}
\big[
\mfh(\tilde{e})
+
\gamma(\tilde{e}) - 1
\big]
-
\sum_{e \in K^{\downarrow}(\overline{M})}
|\clearrootn[\bar{T}_{\ge} (e)^{\sn}_{\se}]|_{\s}
\\
&
=
-|\sT_{\ge}(\rho_{\widehat{M}})^{\sn}_{\se}|_{\s}
-|\s|
+
\mathbbm{1}
\left\{ 
\begin{array}{c}
\rho_{\widehat{M}} = \tilde{e}_{\ch}\\
\textnormal{for }\tilde{e} \in \enS 
\end{array}
\right\}
\big[
\mfh(\tilde{e})
+
\gamma(\tilde{e}) - 1 
\big].
\end{equs}
This is the desired bound when $\rho_{\widehat{M}} \not \in e_{\ch}(\enS)$. It is also the desired bound when the indicator function is non-vanishing since in this case
\begin{equs}
{}
&
-|\sT_{\ge}(\rho_{\widehat{M}})^{\sn}_{\se}|_{\s}
-|\s| + \mfh(\tilde{e})
=
|\clearrootn[\sT_{\ge}(\tilde{e})^{\sn}_{\se}]|_{\s}\\
&
\textnormal{and}
\enskip \gamma(\tilde{e}) - 1
=
|\clearrootn[\sT_{\ge}(\tilde{e})_{\se}^{\sn}]|_{\s}
-
{\rm frac}\{
|\clearrootn[\sT_{\ge}(\tilde{e})_{\se}^{\sn}]|_{\s}
\}\;.
\end{equs}
\end{proof}
\begin{lemma}\label{large scale bound for big graph}
The total homogeneity $\varsigma$, as defined by \eqref{single wick homogeneity}, satisifies condition \eqref{decay at large scales} of Theorem~\ref{thm: half graph assumptions}.
\end{lemma}
\begin{proof}
We start by defining a map $\tilde{\varsigma}: 2^{\allnodes} \rightarrow \R$ as follows, for $M \subset \allnodes$ we set
\begin{equs}\label{large scale homogen def}
\tilde{\varsigma}(M)
\eqdef& 
\sum_{
\substack{
e \in K(D_{2p}) \setminus \enS\\
e \cap M \not = \emptyset}
}
\mfh(e)
-
|\mft(M^{\cu})|_{\s}
-
|\mft(M^{\exte})|_{\s}
+
\sum_{
e \in \cD}
\mathbbm{1}
\left\{
\begin{array}{c}
e_{\ch} \in M\\
e_{\p} \not \in M
\end{array}
\right\}\gamma(e)\\
&
+
\Big[
\sum_{
e \in \enS
}
\mfh(e)
\mathbbm{1}
\left\{ 
e_{\ch} \in M
\right\} 
- 
(\gamma(e) - 1)  
\mathbbm{1}
\left\{ 
\begin{array}{c}
e_{\p} \in M\\
e_{\ch} \not \in M
\end{array}
\right\} 
\Big]\\
&
-
|\sn(M)|_{\s}
-
|M| \, |\s|.
\end{equs}
Fix for the remainder of the proof $\bo{T} \in \CU_{\CX}$.

We claim that for $a \in \mathring{\bo{T}}$ with $a \le \{\logof,\mainroot\}^{\uparrow}$ one has, writing $M \eqdef \allnodes \setminus \go{N}(a)$ (where $\go{N}(a)$ is defined as in \eqref{def: rexpanded set2}), 
\[
\sum_{
b \in \go{T}_{\not \ge a}
}
\varsigma_{\go{T}}(b) 
-
|\CX \setminus \bo{L}_{a}| \, |\s|
-
|\mft(M^{\exte})|_{\s}
\ge
\tilde{\varsigma}(M)
\]
The claim is easily justified upon remembering that $e \in \fullcuts$, $e_{\ch} \in \allnodes \setminus \go{N}(a)$ and $e_{\p} \not \in \allnodes \setminus \go{N}(a)$ together imply $e \in \cD$ and that one has, using item 3 of Definition~\ref{def: cumulant homogeneity}, the inequality
\[
\sum_{
b \in \go{T}_{\not \ge a}
}
\sum_{B \in \pi}
\c^{B}(b)
\ge
-
|\mft(M^{\cu})|_{\s}\;.
\]
We now define a collection of node sets $\mcZ \subset 2^{\tilde{N}(\sT)}$ by setting
\[
\mcZ
\eqdef
\left\{ 
M \subset \tilde{N}(\sT):\ 
\begin{array}{c}
\exists a \in \mathring{\bo{T}} \textnormal{ with } a \le \logof \textnormal{ and } a \le \mainroot
\textnormal{ such that }\\
M\textnormal{ is an }K(\sT)\textnormal{-connected }
\textnormal{component of } N^{\ast} \setminus \go{N}(a)
\end{array}
\right\}\;.
\]
The lemma will be proved if we show that, for every $M \in \mcZ$, $\tilde{\varsigma}(M) > 0$.
Here we used that if $M \subset \tilde{N}(\sT)$ decomposes into $K(\sT)$-components $\{M_{j}\}_{j=1}^{n}$ then $\tilde{\varsigma}(M) = \sum_{j=1}^{n} \tilde{\varsigma}(M_{j})$. 

We now fix $M \in \mcZ$ and write $\tilde{e}$ for the unique edge in $e \in K(\sT)$ with $\tilde{e}_{\ch} = \rho_{M}$. 
We start by observing that $\tilde{e} \in \cut \Rightarrow \tilde{e} \in \cD$, and $\tilde{e}$ is the only edge that could possibly contribute to the sum occurring in the last term of the first line of \eqref{large scale homogen def}. 

The proof of the lemma is finished upon proving the claim
\begin{equ}\label{desired inequality}
\tilde{\varsigma}(M) 
\ge
- |
\clearrootn[\bar{T}_{\ge}(\tilde{e})]^{\sn}_{\se}|_{\s}
+
\gamma(\tilde{e}) 
\cdot 
\mathbbm{1} 
\left\{ 
\tilde{e} \in \cut
\right\} > 0.
\end{equ}
The last inequality of \eqref{desired inequality} is a consequence of the definition of $\cut$ and $\gamma(\cdot)$. 

We now prove the first inequality of \eqref{desired inequality}. First observe that
\begin{equs}\label{large scale bound work: -1}
\tilde{\varsigma}(M) 
=&
-
|T(M)^{\sn}_{\se}|_{\s}
+
\mfh(\tilde{e})
-
|\s|
+
\gamma(\tilde{e})
\cdot
\mathbbm{1}\{ \tilde{e} \in \cut \}
\\
&
+
\sum_{e \in K^{\downarrow}(M) \setminus \enS}
\mfh(e)
-
\sum_{e \in K^{\downarrow}(M) \cap \enS}
(\gamma(e) - 1)
\end{equs}

Conversely, by using \eqref{tree homogeneity} we have
\begin{equs}\label{large scale bound - work 0}
-
|
\clearrootn[\sT_{\ge}(\tilde{e})]^{\sn}_{\se}
|_{\s}
=&\ 
-
|\sT_{\ge}(\rho_M)^{\sn}_{\se}|_{\s}
+
\mfh(\tilde{e})
-
|\s|\\
=&
-
|T(M)^{\sn}_{\se}|_{\s}
+
\mfh(\tilde{e})
-
|\s|
-
\sum_{e \in K^{\downarrow}(M)}
|
\clearrootn[\sT_{\ge}(e)]^{\sn}_{\se}
|_{\s}\;.
\end{equs}
The claim is then proved upon observing that
\begin{equs}\label{large scale bound work: 0.5}
&
\forall e \in K^{\downarrow}(M) \setminus \cS,\ 
|
\clearrootn[\sT_{\ge}(e)]^{\sn}_{\se}
|_{\s}
>
-\mfh(e)\\
\textnormal{and}
\enskip
&\forall e \in K^{\downarrow}(M) \cap \cS,\  
|
\clearrootn[\sT_{\ge}(e)]^{\sn}_{\se}
|_{\s}
>
(\gamma(e) - 1)\;.
\end{equs}
The second inequality of \eqref{large scale bound work: 0.5} follows from the definition of $\gamma(\cdot)$ (and holds for all $e \in \cut$).
The first inequality of \eqref{large scale bound work: 0.5} actually holds for all $e \in K(\sT)$ -- one writes
\[
|
\clearrootn[\sT_{\ge}(e)]^{\sn}_{\se}
|_{\s}
=
\Big(
|\sT_{\ge}(e_{\ch})^{\sn}_{\se}|_{\s}
+
|\s|
\Big)
-\mfh(e)
\]
and then notices that the quantity in parentheses on the RHS must be strictly positive by either by the assumption of super-regularity as in Definition~\ref{def: strong subcriticality} or Assumption~\ref{assump - noise + kernel}. 
Using \eqref{large scale bound work: 0.5} in \eqref{large scale bound - work 0} and then using \eqref{large scale bound work: -1} gives the first inequality of \eqref{desired inequality} as required.
\end{proof}
\section{Proofs of the main theorems}\label{sec: proof of main theorem}
We now describe how everything fits together to prove Theorem~\ref{upgraded thm - main theorem}.
\begin{proof}[Proposition~\ref{prop: main estimate}]
Here we combine the outputs of Lemma~\ref{lem: bound on full integrand} with Proposition~\ref{final prop} as input to Theorem~\ref{thm: half graph assumptions} -- the desired bound is obtained from \eqref{the final bound} and setting $r \eqdef - \lfloor \log_{2}(\lambda) \rfloor$.
\end{proof}
\begin{proof}[Theorem~\ref{upgraded thm - main theorem}]
Note that the sums over $\wickleaves \subset L(\sT)$, $\pi \in \fillingpon{L(\sT) \setminus \wickleaves}$, and $(\mathbb{M},\mathbb{G}) \in \mathfrak{R}$ are all finite and therefore contribute only combinatorial factors -- it follows that the estimate of Proposition~\ref{prop: main estimate} combined with the identifications of Lemma~\ref{lem: BPHZ model and tree} and Proposition~\ref{explicit formula} gives one \eqref{eq: uniform bound - main theorem} of Theorem~\ref{upgraded thm - main theorem}.

We now describe how a simple modification of our analysis gives \eqref{eq: convergence bound - main theorem}. By stationarity, it suffices to prove the bound for $z =0$. 
Observe that $X^{\xi} \eqdef \hat{\Upsilon}^{\xi}_{0}[\sT^{\sn}_{\se}](\psi^{\lambda}_{0})$, seen as a function of $\xi \in \mcM(\Omega_{\infty})$ is in fact a $|L(\sT)|$-linear function of $\xi$.

Now we write, for $\xi, \bar{\xi} \in \mcM(\Omega_{\infty})$, 
\begin{equs}
\E
\big[
\big( 
X^{\xi}
-
X^{\bar{\xi}}
\big)^{2p}
\big]
=\ & 
\sum_{j=0}^{2p}
(-1)^{2p - j}
\binom{2p}{j}
\E 
\big[ 
(X^{\xi})^{j}
(X^{\bar{\xi}})^{2p -j }
\big]\\
=\ & 
\sum_{j=0}^{2p}
(-1)^{2p - j}
\binom{2p}{j}
\E 
\big[ 
(X^{\xi})^{j}
(X^{\bar{\xi}})^{2p -j }
-
(X^{\xi})^{2p}
\big]\;.
\end{equs}
We will estimate the above sum term by term. 
Fixing $j$, note the we are estimating can be written as the difference of a single $2p|L(\sT)|$-linear functional on $\mcM(\Omega_{\infty})$ evaluated at two different sets of arguments.
Denoting this functional by $\mcb{X}$ and choosing a convenient ordering for its arguments we then write 
\begin{equs}
{}
&
(X^{\xi})^{j}
(X^{\bar{\xi}})^{2p -j }
-
(X^{\xi})^{2p}\\
&
=\ 
\mcb{X}\big(\xi^{\otimes j|L(\sT)|}
\otimes
\bar{\xi}^{\otimes(2p - j)|L(\sT)|}\big)
-
\mcb{X}\big(\xi^{\otimes 2p|L(\sT)|}\big)\\
&
=
\sum_{k = 0}^{(2p - j)|L(\sT)| - 1}
\mcb{X}\big(\xi^{\otimes j|L(\sT)|+ k} \otimes (\bar{\xi} - \xi) \otimes
\bar{\xi}^{\otimes (2p - j)|L(\sT)| - k - 1}\big)\;.
\end{equs}
Finally, by performing the same analysis as used for the uniform bound, one obtains, for each fixed $k$ in the above sum,  
\begin{equs}
{}
&
\Big|
\E
\big[
\mcb{X}\big(\xi^{\otimes j|L(\sT)|+ k} \otimes (\xi - \bar{\xi}) \otimes
\bar{\xi}^{\otimes (2p - j)|L(\sT)| - k - 1}\big)
\big]
\Big|\\
\lesssim
&\ 
\Big( 
\prod_{e \in K(T)} \|K_{\mft(e)}\|_{|\mft(e)|_{\s},m}
\Big)^{2p}
(\|\xi\|_{N,\c} \vee \|\bar \xi\|_{N,\c})
\| \xi; \bar\xi\|_{j,\c}
\lambda^{2p|\tau|_{\s}}\;.
\end{equs}
The factor of $\| \xi; \bar\xi\|_{N,\c}$ is obtained because at least one of the cumulants appearing in our estimates will involve the difference $\xi - \bar{\xi}$.
\end{proof}
Next we prove Theorem~\ref{main thm: continuity}.
\begin{proof}[Theorem~\ref{main thm: continuity}]
The main ingredients of this proof are Theorem~\ref{upgraded thm - main theorem}, Proposition~\ref{convergence of mollifications}, and a simple diagonalization argument. 

First we construct the mentioned extension of $Z_{\BPHZ}^{\bullet}$. 
We define a homogeneity assignment $|\cdot|_{\s}$ on $\mfL$ by setting, for each $\mft \in \Le$, $|\mft|_{\s} \eqdef \wwnorm{\mft}_{\s} \cdot (1 + \frac{\kappa}{2} \mathbbm{1}\{ \mft \in \Le\})$.
Note that $|\cdot|_{\s}$ is $R$-consistent with $\wnorm{\cdot}_{\s}$. 
We claim that for any $\xi \in \mcM(\Omega_{0},\mfL_{\CCum},\c)$ the map $\eta \mapsto Z^{\xi \ast \eta}_{\BPHZ}$ from $(\Moll, \| \cdot \|_{\kappa/2,0})$ into the $L^{2}$ space of random models on $\mathscr{M}_{\infty}(\mathscr{T})$ has some uniform continuity properties.

To show this we first define $\tilde{\c}$ to be the $\kappa/2$-penalization of $\c$. 
Applying Theorem~\ref{upgraded thm - main theorem} followed by Proposition~\ref{convergence of mollifications}, we 
conclude that for every compact set $\mathfrak{K} \subset \R^{d}$ and $\alpha \in A$ one has the bound
\begin{equs}\label{L2 bound on mollifications}
\E \$ Z^{\xi \ast \eta_{1} }_{\BPHZ}; Z^{\xi \ast \eta_{2}}_{\BPHZ} \$_{\alpha, \mathfrak{K}}^{2}
&\lesssim
(\|\xi \ast \eta_{1}\|_{N,\tilde{\c}} \vee \| \xi \ast \eta_{2} \|_{N,\tilde{\c}})
\| \xi \ast \eta_{1} ; \xi \ast \eta_{2}\|_{N,\tilde{\c}}\\
&\lesssim
\|\xi\|_{N,\c,1}^{3}
\cdot
\| \eta_{1} - \eta_{2}\|_{\kappa/2,0}\;,
\end{equs}
uniformly over $\eta_{1}, \eta_{2} \in \Moll$ and $\xi \in \mcM(\Omega_{0},\mfL_{\CCum},\c)$.
In what follows we denote by $(\eta_{n})_{n \in \N} \subset \Moll$ an arbitrary sequence satisfying $ \lim_{n \rightarrow \infty} \|\eta_{n} - \delta\|_{\frac{\kappa}{2},0} = 0$. 
We see that the map $\eta \mapsto Z^{\xi \ast \eta }_{\BPHZ}$ extends uniquely to $(\overline{\Moll}, \| \cdot \|_{\kappa/2})$ by setting
\[
Z^{\xi \ast \delta }_{\BPHZ} \eqdef \lim_{n \rightarrow \infty} Z^{\xi  \ast \eta_{n}}_{\BPHZ}\;,
\]
where the limit is in $L^{2}$. 

We then define the extension by setting, for each $\xi \in \mcM(\Omega_{0},\mfL_{\CCum},N,\c) \setminus \mcM(\Omega_{\infty},\mfL_{\CCum})$, $Z^{\xi }_{\BPHZ} \eqdef Z^{\xi \ast \delta}_{\BPHZ} $.

Observe that for $\xi \in \mcM(\Omega_{\infty},\mfL_{\CCum})$ one has the equality $ Z^{\xi  \ast \delta }_{\BPHZ} = Z^{\xi  }_{\BPHZ}$ since one has the convergence in probability of $\xi \ast \eta_{n} \rightarrow \xi$ with respect to very strong topologies on $\Omega_{\infty}$ along with a convergence of cumulants (with also sit in regular spaces with uniform bounds), the map $\Z_{\BPHZ}^{\bullet}$ is a continuous function of the data consisting of (i) realizations in $\Omega_{\infty}$, equipped with a suitably strong topology, and (ii) and corresponding cumulants. 
The uniqueness of this extension is immediate since by Proposition~\ref{convergence of mollifications} we see that $\| \xi \|_{N,\c,2} \le \infty$ implies  $\sup_{n \ge 0} \| \xi \ast \eta_{n} \|_{N,\c,1} < \infty$. 

We now prove the desired continuity property. Suppose that we are given random noises $(\xi_{n})_{n \ge 0}$ with $\xi_{n} \rightarrow \xi$ in probability and $\sup_{n \ge 0} \| \xi_{n} \|_{N,\c,1} < \infty$. 
We define, for every $\alpha \in A$, compact $\mathfrak{K} \subset \R^{d}$, and random models $Z,Z'$ defined on the same probability space as $\xi$, 
\[
d_{\alpha,\mathfrak{K}}(Z,Z')
\eqdef
\E \big[ 1 \wedge \$Z;Z'\$_{\alpha,\mathfrak{K}}
\big]\;.
\]
The family of pseudometrics $\{d_{\alpha,\mathfrak{K}}(\cdot,\cdot)\}_{\alpha,\mathfrak{K}}$ 
then generates the topology of convergence in probability on the space of random models. 

Fix $\alpha$ and $\mathfrak{K}$ as above. For any $n,m \in \N$ one has
\begin{equs}
d_{\alpha,\mathfrak{K}}(Z^{\xi}_{\BPHZ}, Z^{\xi_{n}}_{\BPHZ}) 
&\le
d_{\alpha,\mathfrak{K}}(Z^{\xi}_{\BPHZ}, Z^{\xi \ast \eta_{m}}_{\BPHZ})
+
d_{\alpha,\mathfrak{K}}(Z^{\xi \ast \eta_{m}}_{\BPHZ}, Z^{\xi_{n} \ast \eta_{m}}_{\BPHZ}) \\
&\qquad +
d_{\alpha,\mathfrak{K}}(Z^{\xi_{n} \ast \eta_{m}}_{\BPHZ}, Z^{\xi_{n}}_{\BPHZ})\;.\label{final triangle inequality}
\end{equs}
Now fix $\eps > 0$. By applying the bound \eqref{L2 bound on mollifications} it follows that there exists $M > 0$ such that for any $m \ge M$ and any $n \in \N$ both the first and third terms on the RHS of the above inequality are each smaller than $\frac{\eps}{3}$. 

We now claim that for any $m \in \N$ one can find $n_m$ such that for any $n > n_m$ one has 
\[
d_{\alpha,\mathfrak{K}}(Z^{\xi \ast \eta_{m}}_{\BPHZ}, Z^{\xi_{n} \ast \eta_{m}}_{\BPHZ}) < \frac{\eps}{3}\;.
\]
This follows by observing that the map $Z^{\bullet \ast \eta_{m}}_{\BPHZ}$, viewed as a function of a frozen realization in $\bigoplus_{\mft \in \Le} \mcD'(\R^{d})$ and cumulants living in a sufficiently regular space, is again a continuous map and so one can push through the convergence of $\xi_{n} \rightarrow \xi$ in probability along with the convergence of cumulants with uniform bounds. 

By fixing $m \ge M$ we see that for any $n \ge n_m$ one has $d_{\alpha,\mathfrak{K}}(Z^{\xi}_{\BPHZ}, Z^{\xi_{n}}_{\BPHZ}) < \eps$ as required, thus concluding the proof.
\end{proof}
\appendix
\section{Multiscale Clustering}\label{Sec: Multiclustering}
\subsection{Multiscale clustering for a single tree}
Suppose we are given an undirected connected multigraph $\go{G}$ on a vertex set $\mvert \eqdef \{v_0,\dots,v_p\}$ with $p \ge 1$. Here, when we say that $\go{G}$ is a multigraph one should view  $\go{G}$ as some set equipped with some map $\iota_{\go{G}} : \go{G} \rightarrow \CV^{(2)}$. 

The map $\iota_{\go{G}}$ may not be an injection which is why we call $\go{G}$ a multigraph, however the ``duplicated'' edges are seen as distinguishable amongst themselves. 
In practice we will often just present the vertex set $\CV$ and the multigraph $\go{G}$ and the map $\iota_{\go{G}}$ will be obvious. 

We consider $v_0$ to be the ``root'' of $\go{G}$. We also write $\mvert_{0} = \mvert \setminus \{v_0\}$.
\begin{definition}\label{def: coalescence tree}
A \emph{coalescence tree} $\go{T}$ on a vertex set $\CV$ is a rooted tree with at least three nodes with the following structures and properties:
\begin{itemize}
\item The set of leaves of $\go{T}$ is identified with the set $\CV$.
\item Writing $\mathring{\go{T}}$ for the set of internal nodes (i.e. nodes that are not leaves) and $\rho_{\go{T}}$ for the root of $\go{T}$,  we require that every $u \in \mathring{\go{T}} \setminus \{\rho_{\go{T}}\}$ has degree at least $3$ and that $\rho_{\go{T}}$ has degree at least $2$. 
\end{itemize}
We then also equip the set of nodes of $\go{T}$ with the poset structure corresponding to viewing 
$\go{T}$ has a Hasse diagram with its root $\rho_{\go{T}}$ as the unique minimal element.
\end{definition}
Note that unlike in \cite{KPZJeremy} these coalescence trees are not necessarily binary, an internal node is allowed to have more than $2$ ``children''. 
We denote by $\widehat{\CU}_{\mvert}$ the (finite) set of \emph{all} coalescence trees $\go{T}$ on the vertex set $\mvert$. 

The trees $\widehat{\CU}_{\mvert}$ will be used to organize our multiscale expansions, for each fixed $x_{v_{0}}$ they slice the domain $(\R^{d})^{\CV_{0}}$ into regions characterized by the relative distances of the variables $x_{\CV} = (x_{v})_{v \in \CV}$. 
These domains can be visualized in the following way. 
For each fixed $x_{\CV}$, define a family of equivalence relations $\{\sim_{n}\}_{n \in \N}$ on the set $\CV$ 
by taking $\sim_{n}$ to be the transitive closure of the relation
$R_n$ given by
\[
v R_{n} v' \quad \Leftrightarrow\quad  \exists e \in \go{G} \textnormal{ with } \iota_{\go{G}}(e) = \{v,v'\}
\textnormal{ and } |x_{v} - x_{v'}| \le 2^{-n}\;.
\]
The tree $\go{T}$ should then be thought of as pictorially representing one particular way the equivalence classes of the equivalence relations $\sim_{n}$ could merge as one takes $n$ from $-\infty$ to $0$ -- if one ignores issues of coinciding position variables then at some value of $n$ all the elements of $\CV$ are singletons and by $n=0$ there is only one equivalence class. 
The domain associated to a tree $\go{T}$ would then correspond to all values of $x_{\CV_{0}}$ for which the ``coalescence history'' prescribed by $\go{T}$ holds. 

The description above is only for intuition, concrete definitions are given below. 
However one can already see that it is natural to label the inner nodes of $\go{T}$ with numbers that indicate the specific scale at which a coalescence event takes place.
\begin{definition}\label{def: labels for coalescence trees}
Given a vertex set $\CV$ and $\go{T} \in \widehat{\CU}_{\CV}$, we define $\go{\mathrm{Lab}}_{\go{T}}$ to be the collection of all maps $\go{s}: \mathring{T} \rightarrow \N$ with the property that $u < v \Rightarrow \go{s}(u) < \go{s}(v)$.
The pair $(\go{T},\go{s})$ is then called a \emph{labeled coalescence tree}.
\end{definition}
We define the set of all possible \emph{scale assignments} to be given by $\mbn = (n_{e})_{e \in \go{G}}  \in \N^{\go{G}}$. 
We usually fix a possibly smaller set of scales $\mcN_{\go{G}} \subset \N^{\go{G}}$ relevant to the given context when deploying the machinery of this section. 

For any fixed $\mbn \in \N^{\go{G}}$ and $r \in \N$ we define the sub-multigraph $\go{G}^{\mbn}_{r} = \{e \in \go{G}:\  n_{e} \ge r\}$ and also define $\mvert^{\mbn}_{r} \subset 2^{\mvert}$ to be the collection of vertex sets of the connected components of $\go{G}^{\mbn}_{r}$. We consider singletons as connected components so that, for every 
$r$, $\CV^{\mbn}_{r}$ is a partition of $\CV$.

The sequence $(\mvert^{\mbn}_{r})_{r \in \N}$ determines a labelled coalescence tree $(\go{T},\go{s})$ via the following procedure. 
The set of nodes for $\go{T}$ is given by 
${\go{T}} = \bigcup_{r=0}^{\infty} \mvert^{\mbn}_{r}$. Since elements
of ${\go{T}}$ are themselves subsets of $\CV$, they are partially ordered by inclusion.
Given two distinct nodes $a,b \in \mathring{\go{T}}$ we then connect $a$ and $b$ if $a \subset b$ 
maximally in ${\go{T}}$. In this way, the set of leaves is indeed given by $\CV \subset \go{T}$ since,
 for $r$ sufficiently large, $\mvert^{\mbn}_{r}$ consists purely of singletons.
The root is always given by $\rho_{\go{T}} = \CV$, by considering $r$ sufficiently small. 
It is easy to verify that the required properties hold for $\go{T}$ as a consequence of the fact that 
the children of any node, viewed as subsets of $\CV$, form a non-trivial partition of that node.
The labeling $\go{s}(\cdot)$ on internal nodes is defined as follows. For each $a \in \mathring{\go{T}}$, 
we set
\[
\go{s}(a) 
\eqdef
\max \{ r \in \N:\ a \in \mvert^{\mbn}_{r} \}\;.
\]
This is always finite since elements of $\mathring{\go{T}}$ are not singletons, while there always 
exists some $r$ such that $\mvert^{\mbn}_{r} = \{\{v\}\,:\, v \in \CV\}$.
This completes our construction of the labeled coalescence tree $(\go{T},\go{s})$ with the caveat that for purely aesthetic reasons we identify the ``singleton'' leaves of $\go{T}$ with their lone constituent element.
We will henceforth always treat the elements of $\mathring{\go{T}}$ as ``abstract'' nodes, once we are done constructing the tree we forget how they correspond to non-singleton subsets of $\CV$.  
 
The above procedure gives us a map $\hat{\mcT}: \N^{\go{G}} \rightarrow \widehat{\CU}_{\mvert} \ltimes \go{\mathrm{Lab}}_{\bullet}$ taking scale assignments to labeled coalescence trees, we write it $\mbn \mapsto (\mcT(\mbn),\go{s}(\mbn))$. 
We then define $\mtree_{\mvert} \eqdef \mcT(\mcN_{\go{G}})$. 
\begin{remark}
The usage of the notation $\mtree_{\mvert}$ instead of $\widehat{\CU}_{\mvert}$ indicates that we are looking at a sub-collection of coalesence trees which have been determined by fixing a multigraph $\go{G}$ on $\CV$ and a corresponding set of scales $\mcN_{\go{G}}$ even though this is suppressed from the notation $\mtree_{\mvert}$.
\end{remark}
Given $\go{T} \in \widehat{\CU}_{\CV}$ and $f \subset \mvert$ we write $f^{\uparrow}$ for the maximal internal node which is a \emph{proper} ancestor of all the elements of $f$. 
When $f = \{a\}$ we may write $a^{\uparrow}$ instead of $\{a\}^{\uparrow}$. We define $f^{\Uparrow}$ to be the maximal internal node which is a proper ancestor of $f^{\uparrow}$ if $f^{\uparrow} \not = \rho_{\go{T}}$, otherwise we set $f^{\Uparrow} = f^{\uparrow} = \rho_{\go{T}}$.
For $a \in \mathring{\go{T}}$ we write $\go{L}_{a}$ for the set of leaves of $\go{T}$ which are descendants of $a$.

Let us give a concrete example to clarify these definitions. 
For our vertex set we fix $\CV = \{v_{0},\dots,v_{6}\}$.
Below we fix a graph $\go{G}$ by presenting it pictorially and also specify a scale assignment $\mbn \in \N^{\go{G}}$ by labeling each of the edges.
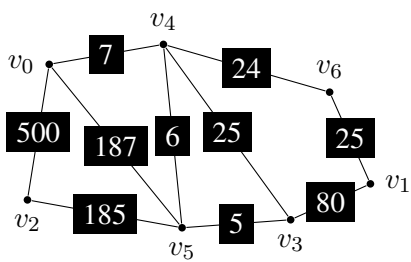
\begin{center}
\begin{tikzpicture}[rotate=-25,yscale=0.7]
    \node [style=dot, label=left:$v_{0}$] (v0) at (-4.5, 0.5) {};
    \node [style=dot, label=below:$v_{2}$] (v2) at (-4, -2) {};
    \node [style=dot, label=above:$v_{4}$] (v4) at (-3.25, 1.75) {};
    \node [style=dot, label=below:$v_{3}$] (v3) at (-0.75, -0.25) {};
    \node [style=dot, label=right:$v_{1}$] (v1) at (0, 1) {};
    \node [style=dot, label=above:$v_{6}$] (v6) at (-1, 2.25) {};
    \node [style=dot, label=below:$v_{5}$] (v5) at (-2, -1.25) {};

    \draw (v0) -- ++ (v2) 
      node [midway,fill=white,color=black,text=white] {500};

    \draw (v0) -- ++ (v5) 
      node [midway,fill=white,color=black,text=white] {187};

    \draw (v2) -- ++ (v5) 
      node [midway,fill=white,color=black,text=white] {185};
    
    \draw (v1) -- ++ (v3) 
      node [midway,fill=white,color=black,text=white] {80};

    \draw (v4) -- ++ (v0) 
      node [midway,fill=white,color=black,text=white] {7};

    \draw (v4) -- ++ (v5) 
      node [midway,fill=white,color=black,text=white] {6};

    \draw (v6) -- ++ (v1) 
      node [midway,fill=white,color=black,text=white] {25};

    \draw (v3) -- ++ (v4) 
      node [midway,fill=white,color=black,text=white] {25};

    \draw (v4) -- ++ (v6) 
      node [midway,fill=white,color=black,text=white] {24};

    \draw (v3) -- ++ (v5) 
      node [midway,fill=white,color=black,text=white] {5};

\end{tikzpicture}
\end{center}
Below we draw the corresponding coalescence tree $\go{T} = \mcT(\mbn)$, labeling the leaves using the set $\CV$ and introducing new labels for the internal nodes of $\go{T}$. 
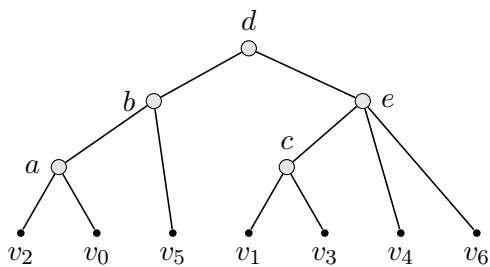
\begin{center}
\begin{tikzpicture}[yscale=0.7]
    \node [style=dot, label=below:$v_{2}$] (v2) at (-3, -2) {};
    \node [style=dot, label=below:$v_{0}$] (v0) at (-2, -2) {};
    \node [style=dot, label=below:$v_{5}$] (v5) at (-1, -2) {};
    \node [style=dot, label=below:$v_{1}$] (v1) at (0, -2) {};
    \node [style=dot, label=below:$v_{3}$] (v3) at (1, -2) {};
    \node [style=dot, label=below:$v_{4}$] (v4) at (2, -2) {};
    \node [style=dot, label=below:$v_{6}$] (v6) at (3, -2) {};
    \node [style=coalnode, label=left:$a$] (a) at (-2.5, -0.75) {};
    \node [style=coalnode, label=left:$b$] (b) at (-1.25, 0.5) {};
    \node [style=coalnode, label=above:$c$] (c) at (0.5, -0.75) {};
    \node [style=coalnode, label=above:$d$] (d) at (0, 1.5) {};
    \node [style=coalnode, label=right:$e$] (e) at (1.5, 0.5) {};

    \draw[coalline] (v2) -- (a); 
    \draw[coalline] (v0) -- (a);

    \draw[coalline] (v5) -- (b);

    \draw[coalline] (a) -- (b);

    \draw[coalline] (v1) -- (c);
    \draw[coalline] (v3) -- (c);

    \draw[coalline] (v4) -- (e);
    \draw[coalline] (v6) -- (e);

    \draw[coalline] (c) -- (e);

    \draw[coalline] (b) -- (d);
    \draw[coalline] (e) -- (d);
\end{tikzpicture}
\end{center}
The labeling $\go{s} = \go{s}(\mbn) \in \go{\mathrm{Lab}}_{\go{T}}$ is then given by
\begin{center}
\smallskip
\def\bar{\,\,\,\smash{\rule[-.6em]{.3pt}{1.9em}}\,\,\,}
\begin{tabular}{c|ccccc}
\toprule
  $u \in \mathring{\go{T}}$ & $a$ & $b$ & $c$ & $d$ & $e$ \\ 
\midrule
  $\go{s}(u)$ & $500$ & $187$ & $80$ & $7$ & $25$\\
\bottomrule
\end{tabular}
\smallskip
\end{center}
In this example, we have
\begin{equ}
\{v_2,v_5\}^\uparrow = \{v_5\}^\uparrow = b\;,\quad 
\{v_2,v_5\}^\Uparrow = d\;,\quad 
\{v_1,v_3,v_5\}^\Uparrow = \{v_1,v_3,v_5\}^\uparrow = d\;.
\end{equ}
We also define, for any $(\go{T},\go{s}) \in \widehat{\CU}_{\CV} \ltimes \go{\mathrm{Lab}}_{\bullet}$, $\mcN_{\mathrm{tri}}(\go{T},\go{s}) \subset \mcN_{\go{G}}$ to be the set of all those 
scale assignments $\mbn$ with $\hat{\mcT}(\mbn) = (\go{T},\go{s})$ and the property that for every $e \in \go{G}$
\begin{equ}\label{triangle condition for scales}
| n_{e} - \go{s}(e^{\uparrow}) | < 2 C \cdot|\mvert|,
\end{equ}
where $C > 0$ is chosen to be the same as \eqref{equ: domain for tree}. Here we are abusing notation and writing $e^{\uparrow}$ instead of $\iota_{\go{G}}(e)^{\uparrow}$, we commit this abuse of notation frequently.
Clearly $|\mcN_{\mathrm{tri}}(\go{T},\go{s})|$ is finite and bounded uniform in $(\go{T},\go{s}) \in \widehat{\CU}_{\CV} \ltimes \go{\mathrm{Lab}}_{\bullet}$.

For each $(\go{T},\go{s}) \in \widehat{\CU}_{\CV} \ltimes \go{S}_{\bullet}$ we define
\begin{equ}\label{equ: domain for tree}
\DDom(\go{T},\go{s},x_{v_{0}}) \eqdef 
\{ x
\in 
(\R^{d})^{\mvert_{0}}:\ \forall e = \{v_i,v_j\} \in \mvert^{(2)},\link{x_{v_i}}{  x_{v_j}}{\go{s}(e^{\uparrow})}
\},
\end{equ} 
where we've used the notation \eqref{defining annular region}.
\begin{definition}\label{def:homogeneity}
For $\go{T} \in \widehat{\CU}_{\CV}$ we call a map $\varsigma_{\go{T}}: \mathring{\go{T}} \rightarrow \R$ a 
$\go{T}$-homogeneity. 
A collection of such maps $\varsigma = (\varsigma_{\go{T}})_{\go{T} \in \widehat{\CU}_{\CV}}$ is called 
a \textit{total homogeneity}. Addition of total homogeneities is defined pointwise. 
\end{definition}
We introduce two special families of total homogeneities which will play the role of ``Kronecker deltas'' out of which we will build other total homogenieties. 
Given any subset $\tilde{\CV} \subset \CV$, the total homogeneities $\delta^{\uparrow}[\tilde{\CV}]$ and $\delta^{\Uparrow}[\tilde{\CV}]$ are given by setting, for every $\go{T} \in \widehat{\CU}_{\CV}$ and $a \in \mathring{\go{T}}$, 
\begin{equ}\label{def: delta total homogeneities}
\delta_{\go{T}}^{\uparrow}[\tilde{\CV}](a)
\eqdef
\mathbbm{1}\left\{ a = \tilde{\CV}^{\uparrow,\go{T}} \right\}
\quad
\textnormal{and}
\quad
\delta_{\go{T}}^{\Uparrow}[\tilde{\CV}](a)
\eqdef
\mathbbm{1}\left\{ a = \tilde{\CV}^{\Uparrow,\go{T}} \right\},
\end{equ} 
where the superscript $\go{T}$ is used to remind readers that these operations are $\go{T}$-dependent. 
For $u \in \CV$ we will write $\delta^{\uparrow}[u]$ or $\delta^{\Uparrow}[u]$ instead of $\delta^{\uparrow}[\{u\}]$ or $\delta^{\Uparrow}[\{u\}]$.

Given $\go{T} \in \widehat{\CU}_{\CV}$, a $\go{T}$-homogeneity $\varsigma_{\go{T}}$, and $\go{s} \in \go{\mathrm{Lab}}_{\go{T}}$ we often use the shorthand
\[
\langle \varsigma_{\go{T}}, \go{s} \rangle = \sum_{a \in \mathring{\go{T}}} \varsigma_{\go{T}}(a) \go{s}(a)\;.
\]\martin{We should really give a slightly better name to these ``families of functions''.} 
\begin{definition}\label{def:boundedBy}
Given a set of scale assignments $\mcN_{\go{G}}$ and a total homogeneity $\varsigma$ we say that a family of continuous compactly supported functions $F = \left(F^{\mbn}\right)_{\mbn \in \mcN_{\go{G}}}$ on $(\R^{d})^{\CV_{0}}$ is bounded by $\varsigma$ if the following conditions hold.
\begin{enumerate}
\item There exists $x_{v_{0}} \in \R^{d}$ such that for each $(\go{T},\go{s}) \in \mtree_{\mvert} \ltimes \go{\mathrm{Lab}}_{\bullet}$, and $\mbn \in \mcN_{\mathrm{tri}}(\go{T},\go{s})$, one has
\begin{equation}\label{domain condition}
\supp
\left( 
F^{\mbn}(\cdot)
\right)
\subset
\DDom(\go{T},\go{s},x_{v_{0}})\;.
\end{equation} 
\item One has the bound
\begin{equation}\label{supremum bound} 
\| 
F
\|_{\varsigma, \mcN_{\go{G}}}
\eqdef
\sup_{
\substack{\go{T} \in \mtree_{\mvert}\\
\go{s} \in \go{\mathrm{Lab}}_{\go{T}}\\
\mbn \in \mcN_{\mathrm{tri}}(\go{T},\go{s})
}
}
2^{-\langle \varsigma_{\go{T}},\go{s}\rangle}
\sup_{x \in (\R^{d})^{\mvert_{0}}}
\big| F^{\mbn}(x) \big|
< \infty.
\end{equation}
(In the particular case $\mcN_{\go{G}} = \N^{\go{G}}$ we will also just write $\|F\|_{\varsigma}$.)
\end{enumerate}
\end{definition}
\begin{remark} Because of the domain constraint \eqref{domain condition}, it is clear that $F^{\mbn}$ must vanish unless $\mbn \in \mcN_{\mathrm{tri},\go{G}}$, where
\[
\mcN_{\mathrm{tri},\go{G}} \eqdef 
\bigsqcup_{(\go{T},\go{s}) \in \mtree_{\mvert} \ltimes \mathrm{Lab}_{\bullet}}
\mcN_{\mathrm{tri}}(\go{T},\go{s})\;.
\]
\end{remark}
\begin{remark} 
The notion of being ``bounded'' by a total homogeneity $\varsigma$ depends on a invisible ``combinatorial'' constant $C$ hidden in \eqref{equ: domain for tree} -- this affects both the domain constraint \eqref{domain condition} and the definition of $\| \cdot \|_{\varsigma, \mcN_{\go{G}}}$. 
In practice we want to be able to formulate that this constant $C$ can be chosen independently of certain parameters. 
Thus, if we have a collections of families of functions $F_{\theta} = (F^{\mbn}_{\theta})_{\mbn \in \mcN_{\go{G}}}$ where $\theta$ varies as a parameter in some set $\Theta$ we say that a the collection of families  $F_{\theta}$ are bounded uniform in $\theta \in \Theta$ by a total homogeneity $\varsigma$ if one can use the same constant $C$ in \eqref{equ: domain for tree} for all values of $\theta \in \Theta$.
\end{remark}
\begin{definition}
Given a set of scale assignments $\mcN_{\go{G}}$ and a total homogeneity $\varsigma$, we say that $\varsigma$ is subdivergence free in $\bar{\mvert} \subset \mvert$ for the set of scales $\mcN_{\go{G}}$ if for every $\go{T} \in \mcU_{\mvert}$ and every $a \in \mathring{\go{T}} \setminus \{\rho_{\go{T}}\}$ with $\go{L}_{a} \subset \bar{\mvert}$ one has
\begin{equ}\label{eq: divergence free condition}
\sum_{
b \in \mathring{\go{T}}_{\ge a}}
\varsigma_{\go{T}}(b)
< (|\go{L}_{a}| -1 ) |\s|.
\end{equ}
\end{definition}
\begin{definition}\label{order of total homogeneity}
We say a total homogeneity $\varsigma$ is of order $\alpha \in \R$ if for every $\go{T} \in \widehat{\mtree}_{\mvert}$ one has $\sum_{a \in \mathring{\go{T}}} \varsigma_{\go{T}}(a) - (|\mvert| - 1)|\s| = \alpha$.
\end{definition}
The following definition will be useful in getting good bounds on various integrals.
\begin{definition}
Given a family of continuous compactly supported functions $F = (F^{\mbn})_{\mbn \in \mcN_{\go{G}}}$ on $(\R^d)^{\CV_{0}}$ and a subset $A \subset \CV_{0}$ we denote by $\mathrm{Mod}_{A}(F)$ the collection of all family of continuous compactly supported functions $\tilde{F} = (\tilde{F}^{\mbn})_{\mbn \in \mcN_{\go{G}}}$ such that there exists $x_{v_{0}} \in \R^{d}$ with the property that for every $\mbn \in \mcN_{\go{G}}$ one has
\begin{equ}\label{domain condition for modification}
\supp
\left( 
\tilde{F}^{\mbn}(\cdot)
\right)
\subset
\DDom(\go{T},\go{s},x_{v_{0}})
\end{equ}
and for any $x \in (\R^{d})^{A}$
\[
\int_{\CV_{0} \setminus A} dy\ 
F^{\mbn}(x \sqcup y)
=
\int_{\CV_{0} \setminus A} dy\ 
\tilde{F}^{\mbn}(x \sqcup y)\;.
\]
We also use the notation $\mathrm{Mod}(F) = \bigcup_{A \subset \CV_{0}} \mathrm{Mod}_{A}(F)$.
\end{definition} 
\begin{theorem}\label{multicluster 1}
Suppose we are given a set of scales $\mcN_{\go{G}}$ and a family of smooth compactly supported functions $F = (F^{\mbn})_{\mbn \in \mcN_{\go{G}}}$ on $(\R^{d})^{\CV_{0}}$ and a total homogeneity $\varsigma$ on the trees of $\widehat{\CU}_{\CV}$ which is of order $\alpha$ and subdivergence free on $\mvert$ for the set of scales $\mcN_{\go{G}}$. 

For $r \in \N$ we define
\begin{equ}[e:defN>]
\mcN_{\go{G}, > r} \eqdef \{ \mbn \in \mcN_{\go{G}}:\ \min_{e \in \go{G}} n_{e} > r \}
\enskip
\textnormal{and}
\enskip
\mcN_{\go{G}, \le r} \eqdef \{ \mbn \in \mcN_{\go{G}}:\ \min_{e \in \go{G}} n_{e} \le r \}\;.
\end{equ}
Then, for $\alpha > 0$, one has
\[
\sum_{\mbn \in \mcN_{\go{G}, \le r}}
\int_{\mvert_{0}} \back dy\ 
\left|F^{\mbn}(y)
\right|
\le 
\mathrm{const}(|\CV|)
2^{\alpha r}
\inf_{\tilde{F} \in \mathrm{Mod}(F)}
\|\tilde{F}\|_{\varsigma,\mcN_{\go{G}, \le r}}
\]
while for $\alpha < 0$ one has
\[
\sum_{\mbn \in \mcN_{\go{G}, > r}}
\int_{\mvert_{0}} \back dy\ 
\left|F^{\mbn}(y)
\right|
\le 
\mathrm{const}(|\CV|)
2^{\alpha r}
\inf_{\tilde{F} \in \mathrm{Mod}(F)}
\|\tilde{F}\|_{\varsigma,\mcN_{\go{G},> r}}\;.
\]
Here, $\mathrm{const}(|\CV|)$ is a combinatorial factor depending only 
on $|\CV|$ and not on $r$. 
\end{theorem}
\begin{proof}
This is essentially a special case of \cite[Lem.~A.10]{KPZJeremy} with $\nu_\star$
equal to the root of $\go{T}$. The only difference is that our ``subdivergence-free condition''
does not include the root itself. In the case $\alpha < 0$, Definition~\ref{order of total homogeneity}
implies that \eqref{eq: divergence free condition} also holds for the root and we
can apply \cite[Lem.~A.10]{KPZJeremy}. In the case $\alpha > 0$,  this is not the case, but 
the proof of \cite[Lem.~A.10]{KPZJeremy} still applies, the only difference being that the 
sum appearing in the base case $|\mathring{\go{T}}| = 1$ runs over large scales instead of
small scales.
\end{proof}
For the next theorem and what follows, for any $\go{T} \in \widehat{\CU}_{\CV}$ and $a \in \mathring{\go{T}}$ we define 
\[
\go{T}_{\not\ge a} \eqdef \{ b \in \mathring{\go{T}}:\ b \not \ge a\}\;.
\]
\begin{theorem}\label{thm: both bounds}
Let $\mcN_{\go{G}}$ be fixed and let $\go{G}_{\ast} \subset \go{G}$ be non-empty subset of edges which connects the collection of vertices of $\CV$ it is incident with, we denote this collection of vertices by $\CV_{\ast}$. 

Suppose that we are given a family of functions $F = (F^{\mbn})_{\mbn \in \mcN_{\go{G}}}$ and a total homogeneity $\varsigma$ which is sub-divergence free on $\CV$ for the scales $\mcN_{\go{G}}$, of order $\alpha < 0$, and satisfies the following large scale integrability condition: for every $\go{T} \in \mtree_{\mvert}$ and $u \in \mathring{\go{T}}$ with both $u \le (\mvert_{\ast})^{\uparrow}$ and $\go{T}_{\not\ge u} \not = \emptyset$ one has 
\begin{equ}\label{second cond of HQ}
\sum_{ 
w \in \go{T}_{\not\ge u}
}
 \varsigma_{\go{T}}(w) 
> 
|\s| \, |\CV \setminus \bo{L}_{u} |\;.
\end{equ}
Then if we set, for any $r \in \N$,  
\[
\mcN_{\go{G}, > r, \go{G}_{\ast}}
\eqdef
\Big
\{ \mbn \in \mcN_{\go{G}}:\ \min_{
e \in \go{G}_{\ast}
}
n_{e}
\ge
r
\Big\}
\]
we have the bound, uniform in $r$, 
\[
\sum_{\mbn \in \mcN_{\go{G}, > r, \go{G}_{\ast}}}
\int dy_{\mvert_{0}}\ 
\left|F^{\mbn}(y)
\right|
\lesssim
2^{\alpha r}
\inf_{\tilde{F} \in \mathrm{Mod}(F)}
\|\tilde{F}\|_{\varsigma, \mcN_{\go{G}}}\;.
\]
\end{theorem}
\begin{proof}
This is precisely the content of \cite[Lem.~A.10]{KPZJeremy}.
\end{proof}
We close the section by introducing more notation. 
Given any total homogeneity $\varsigma$ on the set of coalesence trees $\widehat{\CU}_{\CV}$, any $A \subset \CV$, and any multi-index $(k_{v})_{v \in A} \in (\N^{d})^{A}$ we define another total homogeneity $D^{k}\varsigma$ by setting
\begin{equ}\label{derivative of homogeneity}
D^{k}\varsigma \eqdef 
\varsigma + \sum_{v \in A}|k_{v}|_{\s} \delta^{\uparrow}[v]\;.
\end{equ}
\begin{definition}\label{def: a derivative norm}
Given a family of compactly supported functions $(F^{\mbn})_{\mbn \in \mcN_{\go{G}}}$ on $(\R^{d})^{\CV_{0}}$ and a total homogeneity $\varsigma$ on the trees of $\widehat{\CU}_{\CV}$ and a subset $\tilde{\CV} \subset \CV_{0}$, we define 
\[
\| 
F
\|_{\varsigma,\mcN_{\go{G}},\tilde{\CV}}
\eqdef
\sup_{k \in A}
\| 
F_{k}
\|_{D^{k}\varsigma, \mcN_{\go{G}}}
\]
where 
\[
A \eqdef
\Big\{
k \in (\N^{d})^{\allnodes}:\ 
k \textnormal{ supported on }\tilde{\CV}\textnormal{ and }\sup_{u \in \allnodes} |k(u)|_{\s} \le |\s|
\Big\},
\]
and for $k \in A$ we defined $F_{k} = \big(F_{k}^{\mbn}\big)_{\mbn \in \mcN_{\go{G}}}$ via $F_{k}^{\mbn} \eqdef D^{k}F^{\mbn}$. 
\end{definition}
\subsection{Multiscale clustering for cumulants}\label{cumulant homogeneities}
The types of bounds on cumulants we ask for in Definitions~\ref{def: cumulantbound} and 
\ref{def: metric on mollified measures} place major limitations on the types of driving noises 
we can accommodate -- we are basically restricted to space-time random fields that are either
Gaussian, Gaussian-like (the kind of situation one encounters when considering a central-limit type
convergence result as in \cite{HS15}), or non-Gaussian 
with rapidly decreasing cumulants. 

In order to work in a more general setting and which as few restrictions as possible we will require additional data: in addition to a set of types $\mfL = \mfL_{+} \sqcup \mfL_{-}$ and a homogeneity assignment $|\cdot|_{\s}$ for $\mfL$,
we prescribe (i) a set of non-vanishing cumulants and (ii) a corresponding \textit{cumulant homogeneity}.  

To describe the latter, we use the multiclustering notation introduced earlier to describe our cumulant bounds. For $N \in \N$ with $N \ge 2$ we view $[N]$ as a vertex set, we take $\go{G} \eqdef [N]^{(2)}$ to be the complete graph and we set $1 \in [N]$ to be the pinned vertex.
We set $\mcN_{\go{G}} \eqdef \N^{\go{G}}$ for our
set of scales. 
One then has $\CU_{[N]} = \widehat{\CU}_{[N]}$, the set of all coalescence trees on $[N]$. 
\begin{definition}\label{def: cumulant homogeneity}
A cumulant homogeneity $\c$ for $\mfL_{\CCum}$ is a collection 
\[
\c 
=
\big\{\c^{(\mft,[M])} \,:\, (\mft,[M]) \in \mfL_{\CCum}\big\}
\]
where each $\c^{(\mft,[M])}$ is a total homogeneity on the trees of $\CU_{[M]}$, so $\c^{(\mft,[M])} =
(\c^{(\mft,[M])}_{\go{T}})_{\go{T} \in \CU_{[M]}}$.

We also impose that $\c$ is imposed to be invariant under
permutations in the following sense.
For each $(\mft,[M]) \in \mfL_{\CCum}$ a permutation $\sigma: [M] \rightarrow [M]$ induces bijections $\sigma: 2^{[M]} \rightarrow 2^{[M]}$ and  $\sigma: \widehat{\CU}_{[M]} \rightarrow \widehat{\CU}_{[M]}$ since there is a canonical association $ \go{T} \in \widehat{\CU}_{[M]} \leftrightarrow N_{\go{T}} \subset 2^{[M]}$. 
We then require that for any $A \subset [M]$ with at least $2$ elements and for any $\go{T} \in \widehat{\CU}_{[M]}$ one has 
\[
\c^{(\mft,[M])}_{\go{T}}(A^{\uparrow,\go{T}})
=
\c^{(\mft \circ \sigma,[M])}_{\sigma(\go{T})}(\sigma(A)^{\uparrow,\go{T}})\;.
\]
Thus, given a finite set $B$ equipped with a type map $\mft:B \rightarrow \mfL_{-}$ we can write $\c^{(\mft,B)}$ without ambiguity, and this notation is compatible with the analogous notation introduced in Assumption~\ref{assump: set of cumulants}. 
\end{definition}
\begin{remark}
Note that for any cumulant homogeneity on a set of cumulants $\mfL_{\CCum}$ one has that for any $(\mft,[2]) \in \mfL_{\CCum}$ the total homogeneity $\c^{(\mft,[2])}$ is encoded by just a single scalar value. We sometime abuse notation and treat $\c^{(\mft,[2])}$ as a scalar. 
\end{remark}
We will need the following notion of consistency. 
\begin{definition}
Given a homogeneity assignment $|\cdot|_{\s}$, we say $\c$ is \textit{consistent} with $|\cdot|_{\s}$ if, 
for every finite set $B$,
every $(\mft,B) \in \mfL_{\CCum}$ and $\go{T} \in \widehat{\CU}_{B}$ the following conditions hold. 
\begin{enumerate}
\item Total homogeneities are `correct': 
$
\sum_{
u \in \mathring{\go{T}}
}
\c^{(\mft,B)}_{\go{T}}(u)
=
-|\mft(B)|_{\s}\;.
$
\item For every $A \subset B$,
\begin{equation}\label{cumulant homogeneity upper bound}
\sum_{
\substack{
u \in \mathring{\go{T}}\\
u \le v \textnormal{ for some }v \in A
}
}
\c^{(\mft,B)}_{\go{T}}(u)
\ge
-
|\mft(A)|_{\s}\;.
\end{equation}
\item For every $a \in \mathring{\go{T}}$,\martin{Doesn't sound right: one probably wants to 
take only $A$'s such that $\go{L}_{A^{\uparrow}} = A$...}\ajay{Agree with problem, change made.}
\begin{equation}\label{cumulant homogeneity lower bound}
\sum_{
u \in \go{T}_{\ge a}
}
\c^{(\mft,B)}_{\go{T}}(u)
\le
-
|\mft(\go{L}_{a})|_{\s}\;.
\end{equation}
\item If $M \ge 3$ then for every $a \in \mathring{\go{T}}$ with $|\go{L}_{a}| \le 3$ 
one has 
\begin{equ}\label{eq: for cumulant homogeneities all cumulants are good}
\sum_{
u \in \go{T}_{\ge a}
}
\c^{(\mft,[M])}_{\go{T}}(u) < |\s| \cdot (|\go{L}_{a}| - 1)\;.
\end{equ}
\end{enumerate}
\end{definition}
\begin{remark} 
It is possible to see that if item $1$ of Definition~\ref{def: cumulant homogeneity} is satisfied, 
then items $2$ and $3$ are equivalent. One way of interpreting these conditions is that $\c^{(\mft,B)}_{\go{T}}$
is always obtained by distributing the homogeneity 
$-|\mft(a)|_\s$ of each leaf $a$ among the nodes of $\mathring{\go{T}}$ connecting $a$ to the root.
\end{remark}
With these definitions at hand, we are now ready to formulate our ``distance'' on the space
$\mathcal{M}(\Omega_{0},\mfL_{\CCum})$ of stationary noise processes.

\begin{definition}\label{def: cumulant bound2} 
Given a set of cumulants $\mfL_{\CCum}$ and a cumulant homogeneity $\c$ for $\mfL_{\CCum}$, $N \ge 2$, $r \in \N$, and $\xi \in \mathcal{M}(\Omega_{0},\mfL_{\CCum})$ we set
\[
\|\xi\|_{N,\c,r}
\eqdef
\mathrm{Diag}_{\c}(\xi)
\vee
\max_{
\substack{
M \le N\\
(\mft,[M]) \in \mfL_{\CCum}
}
}
\max_{
k \in D_{M,r}
}
\|F_{k,(\mft,[M])}
\|_{D^{k}\c^{(\mft,[M])}}
\]
Here we have set 
\begin{equ}\label{set of derivatives in cumulant norm}
D_{M,r}
\eqdef
\{ k = (k_{i})_{i=1}^{M} \in (\N^{d})^{M}:\ 
\forall i \in [M],\ |k_{i}|_{\s} \le 3r|\s|
\}\;,
\end{equ}
We have also defined the family $F_{k,(\mft,[M])} = (F^{\mbn}_{k,(\mft,[M])}:\ \mbn \in \N^{[M]^{2}})$ to be given by setting $F^{\mbn}_{k,(\mft,[M])}(x_{[M] \setminus \{1\}})$ equal to
\[
\Big(\prod_{1 \le j < m \le M}
\Psi^{(n_{\{j,m\}})} (x_{j} - x_{m})
\Big)
D^{k} \Cum[\{\xi_{\mft(i)}(x_{i})\}_{i = 1}^{M}]
\Big|_{x_{1} = 0}\;.
\]
and we are using the notation of \eqref{supremum bound}. We also used the set of derivatives 

The term $\mathrm{Diag}_{\c}(\xi)$ is defined as follows. 
If there exists any $(\mft,[2]) \in \mfL_{\CCum}$ such that $\reduce[\Cum[\{\xi_{\mft(1)},\xi_{\mft(2)}\}]]$ is not of order $(\c_{(\mft,[2])} - |\s|) \vee 0$ at $0$ then we set  $\mathrm{Diag}(\xi) \eqdef \infty$.
Otherwise we set\martin{Probably $\bullet^{k}$ should be truncated in some way, otherwise this might not exist...}\ajay{Replaced $\bullet^{k}$ with $p_{k}$ which is truncated.}
\begin{equ}
\label{def: control over the diagonal2}
\mathrm{Diag}_{\c}(\xi)
\eqdef
\max_{(\mft,[2]) \in \mfL_{\CCum} }
\max_{
\substack{
k \in \N^{d},\\
|k|_{\s} < \c^{(\mft,[2])} - |\s| 
}
}
|\reduce[\Cum[\{\xi_{\mft(1)},\xi_{\mft(2)}\}]](p_{k})|\;.
\end{equ}
where $\reduce[\cdot]$ is as in \eqref{transinvar second cumulant}.
\end{definition}
\begin{definition}\label{def: metric on mollified measures2}
Given a set of cumulants $\mfL_{\CCum}$, a cumulant homogeneity $\c$ for $\mfL_{\CCum}$, $N \ge 2$, $r \in \N$ and $\xi, \bar \xi \in \mathcal{M}(\Omega_{0},\mfL_{\CCum})$ defined on the same probability space and jointly admitting pointwise cumulants, we define the quantity 
$\|\xi; \bar{\xi}\|_{N,\c,r}$ by 
\[
\|\xi; \bar\xi\|_{N,\c,r}
\eqdef
\mathrm{Diag}_{\c}(\xi ; \bar{\xi})
\vee
\max_{
\substack{ 
\hat M \le M \le N\\
(\mft,[M]) \in \mfL_{\CCum}
}
}
\max_{
k \in A_{M,r}
}
\|
\hat{F}_{k,(\mft,[M]),\hat{M}}
\|_{D^{k}\c^{(\mft,M)},\N^{[M]^{(2)}}}
\;.
\]
Here $D_{M,r}$ is defined as in \eqref{set of derivatives in cumulant norm}.
We have set the family $\hat{F}_{k,(\mft,[M]),\hat{M}} = (\hat{F}^{\mbn}_{(\mft,[M]),\hat{M}}:\ \mbn \in \N^{[M]^{2}})$ to be given by setting $\hat{F}^{\mbn}_{k,(\mft,[M]),\hat{M}}(x_{[M] \setminus \{1\}})$ equal to
\[
\Big(
\prod_{1 \le j < m \le M}
\Psi^{(n_{\{j,m\}})} (x_{j} - x_{m})
\Big)
D^{k}
\Cum
[
\{\hat{\xi}_{\mft(i),i}(x_{i})\}_{i =1}^{M}
]
\Big|_{x_{1} = 0}\;,
\] and the random fields $(\hat{\xi}_{\mft,i})_{i=1}^{M}$ above are given by
\[
\hat{\xi}_{\mft,i}(x)
\eqdef
\xi_{\mft}(x) \one\{i \le \hat M\} - \bar\xi_{\mft}(x) \one\{i \ge \hat M\}\;.
\]
We set 
The quantity $\mathrm{Diag}_{\c}(\xi;\bar{\xi})$ is defined as follows. 
If there exists any $(\mft,[2]) \in \mfL_{\CCum}$ and $\hat{M} \in \{1,2\}$ such that $\reduce[\Cum[\{\hat{\xi}_{\mft(1)},\hat{\xi}_{\mft(2)}\}]]$ is not of order $\c_{(\mft,[2])} - |\s|$ at $0$ then we set  $\mathrm{Diag}_{\c}(\xi;\bar{\xi}) \eqdef \infty$.
Otherwise we set
\begin{equ}\label{eq: control over diagonal 2}
\mathrm{Diag}_{\c}(\xi ; \bar{\xi})
\eqdef
\max_{
\substack{
(\mft,[2]) \in \mfL_{\CCum}\\
1 \le \hat{M} \le 2 }
}
\max_{
\substack{
k \in \N^{d},\\
|k|_{\s} < \c^{(\mft,[2])} - |\s| 
}
}
|\reduce[\Cum[\{\hat{\xi}_{\mft(1)},\hat{\xi}_{\mft(2)}\}]](p_{k})|\;.
\end{equ}
If $\xi$ and $\bar{\xi}$ don't admit joint pointwise cumulants we set $\|\xi; \bar\xi\|_{N,\c} = \infty$ for every $N$. 
\end{definition}
\begin{remark}
We often write $\|\xi|_{N,\c}$ or $\|\xi; \bar\xi\|_{N,\c}$ instead of $\|\xi\|_{N,\c,0}$ or $\|\xi; \bar\xi\|_{N,\c,0}$.
\end{remark}
\begin{remark} Going back to an earlier mentioned example, suppose $\psi$ is a random, stationary, centered Gaussian, element of the scaled Holder-Besov space $C^{\alpha}_{\s}(\R^{d})$ where $\alpha \in (-\frac{|\s|}{4},0)$. Then for any sequence of smooth compactly supported approximate identities $(\eta_{n})_{n \in \N}$ appropriately converging to a Dirac delta as $n \rightarrow \infty$ one expects that $(\psi \ast \eta_{n})^{2}(\cdot) - \E[(\psi \ast \eta_{n})^{2}(0)]$ should converge in probability on $C^{2\alpha}_{\s}(\R^{d})$ to a limit we denote $\psi^{\diamond 2}$.

The random distribution $\psi^{\diamond 2}$ is sometimes referred to as the ``second Wick power'' of $\psi$. 
Suppose we are looking at a system of equations driven by only $\psi^{\diamond 2}$. Writing $\mft_{-}$ for the type-label for this noise one could set $|\mft_{-}|_{\s} = 2 \alpha$ and we now describe a cumulant homogeneity $\c$ one could use in this situation. 

For any $N \ge 2$, $\go{T} \in \widehat{\CU}_{[N]}$, and $a \in \mathring{\go{T}}$ one sets, 
\[
\c^{(\mft_{-},[N])}_{\go{T}}(a)
\eqdef
2 \alpha \cdot ( d(a) - 1) + 2 \alpha \mathbbm{1}\{ a = \rho_{\go{T}}\}
\]
where for $a \in N_{\go{T}}$ we write $d(a) \eqdef |\mathrm{Min}(\{ b \in N_{\go{T}}:\ b > a\})|$.
\end{remark}
\begin{definition}\label{def: external homogeneity}
Assume we are given a cumulant homogeneity $\c$ on $\mfL_{\CCum}$ consistent with a homogeneity assignment $|\cdot|_{\s}$, and finite set $A$ and $D$ equipped with a type map $\mft:A \sqcup D \rightarrow \Le$. 
We then define the quantity $| \mft(A) |_{\s,\c,D} \le 0$ as follows: if $A \not \in \mfL_{\CCum}$ we set $| \mft(A) |_{\s,\c,D} \eqdef 0$ and otherwise we set $| \mft(A) |_{\s,\c,D}$ equal to
\begin{equ}\label{def: extended homogeneity}
\inf
\left\{ 
-
\sum_{
u \in \go{T}_{\ge a}
}
\c^{(\mft',A \sqcup [N])}
(u)
:\ 
\begin{aligned}
&N \ge 1,\ (\mft', A \sqcup [N]) \in \mfL_{\CCum}\\
& \mft' \restr {A} = \mft,\ \mft'([N]) \subset \mft(D)\\
&\go{T} \in \CU_{A \sqcup [N]} 
\textnormal{ with }
 a \in \mathring{\go{T}} \textnormal{ s.t.\ } \go{L}_{a} = A.
\end{aligned} 
\right\}\;.
\end{equ}
Above, we are enforcing that $\mft'([N]) \subset \mft(D)$ in the sense of sets, not multi-sets.
\end{definition}
\begin{remark}\label{remark on extended homogeneity}
Observe that for any leaf typed sets $A$ and $D$ we have $|A| \le 3 \Rightarrow | \mft(A) |_{\s,\c,D} \ge [(|A| - 1) \vee 0] \cdot |\s|$ thanks to \eqref{eq: for cumulant homogeneities all cumulants are good} and by \eqref{cumulant homogeneity lower bound} one has 
\begin{equ}\label{eq: useful fact for cumulant bound}
|\mft(A) |_{\s,\c,D} \ge | \mft(A) |_{\s}\;.
\end{equ}
The above inequality is used many times throughout the paper, often implicitly.
\end{remark} 
In order to get the desired stochastic estimates when we use a cumulant homogeneity $\c$ to control 
cumulants we also need a stronger notion of super-regularity.
\begin{definition}\label{def: homogeneity gain}
Given a homogeneity assignment $|\cdot|_{\s}$, a set of cumulants $\mfL_{\CCum}$, and a cumulant homogeneity $\c$ on $\mfL_{\CCum}$ consistent with $|\cdot|_{\s}$ and a leaf-typed finite set $D$ we define a map $h_{\c,D}$ defined on leaf typed finite sets $(\mft,A)$ as follows.
If $A = \emptyset$ we set $\mcb{h}_{\c,D}(\emptyset) \eqdef 0$ and for $A \neq \emptyset$ we set
\begin{equ}\label{homogeneity gain from missed noise}
\mcb{h}_{\c,D}(A)
\eqdef
\min_{\emptyset \not = B \subset A}
( |\mft(B) |_{\s,\c,D} - | \mft(B) |_{\s})\;.
\end{equ}
\end{definition}
\begin{remark}
$\mcb{h}_{\c,D}(A)$ gives the minimum homogeneity gain from having at least one element of $A$ participate in a cumulant external to $A$ with noises drawn from $D$.
\end{remark}
For the renormalization of second cumulants it is convenient to define, for any leaf typed set $B$,
\begin{equ}\label{def: fict}
\fict(B)
\eqdef 
\mathbf{1}\{ |B| = 2\}
\big(
\lceil 
-|\mft(B)|_{\s} - |\s| 
\rceil
\vee 0
\big)
\;.
\end{equ} 
The quantity $\fict(B)$ gives the power-counting gain from the renomalization of a second cumulant corresponding to $B$. The next lemma states that we don't need to renormalize any other cumulants.
\begin{lemma}\label{lemma: higher cumulants are fine}
Let $M$ and $D$ be a leaf typed sets and $(\mft,M) \in \mfL_{\CCum}$. Then 
\begin{equ}\label{power counting cumulant bound}
|\mft(M)|_{\s,\c,D} \wedge \big(\fict(M) +  |\mft(M)|_{\s}\big) + (|M| - 1) \cdot |\s|
>
0 
\end{equ}
\end{lemma}
\begin{proof}
If $|M| = 2$ then this follows from either \eqref{eq: for cumulant homogeneities all cumulants are good} or the definition of $\fict(M)$. 
If $|M| \ge 3$ this follows from either Remark~\ref{remark on extended homogeneity} or Assumption~\ref{assump - noise + kernel}. 
\end{proof}
\begin{definition}\label{def: cumulant strong subcriticality}
A semi-decorated $T^{\mfn}_{\mfe}$ is said to be $(\c,|\cdot|_{\s},\mfL_{\CCum})$-\emph{super-regular} if, for every subtree $S$ of $T$ with $|N(S)| > 1$, one has
\begin{equ}\label{super-reg 1}
|S^{0}_{\mfe}|_{\s} + \Bigl(\frac{|\s|}{2} \wedge \mcb{h}_{\c,L(T)}(L(S)) \wedge \mcb{j}_{L(T)}(L(S))\Bigr) > 0\;.
\end{equ}
\end{definition}
\begin{remark}\label{rem:backtrack}
Given a set of non-vanishing cumulants $\mfL_{\CCum}$ and a homogeneity assignment $|\cdot|_{\s}$ one choice of cumulant homogeneity $\c$ for $\mfL_{\CCum}$ consistent with $|\cdot|_{\s}$ is given by setting, for each $(\mft,[M]) \in \mfL_{\CCum}$, the total homogeneity $(\c^{(\mft,[M])}_{\go{T}})_{\go{T} \in \CU_{[M]}}$ to be given by $\c^{(\mft,[M])}_{\go{T}}(a) \eqdef |\mft([M])|_{\s} \cdot \mathbbm{1}\{a = \rho_{\go{T}}\}$ for every $\go{T} \in \CU_{[M]}$ and $a \in \mathring{\go{T}}$. 
This is precisely the choice of $\c$ being referred to in Remark~\ref{remark: simplified cumulant homogeneity}. 
\end{remark}
It is useful to formalize how the cumulant homogeneity will be used with regards to larger integrands, which motivates the following definitions.
\begin{definition}\label{tree restriction}
Let $A$ be a finite set and let $B \subset A$ be of cardinality at least $2$. 
For each $\go{T} \in \widehat{\CU}_{A}$ there is a natural choice of $\go{S} \in \widehat{\CU}_{B}$ corresponding to the ``the restriction of $\go{T}$ to $B$''. 
First recall that $N_{\go{T}}$ can be identified with a subset of $2^{A}$. 
Using this identification, we then define $N_{\go{S}} = \{ D \cap B:\ D \in N_{\go{T}},\ D \cap B \not = \emptyset \}$. 
We then equip $N_{\go{S}}$ with the reverse inclusion partial order as in Section~\ref{Sec: Multiclustering} -- this finishes the specification of $\go{S} \in \CU_{B}$.
For what follows, recall that, except for the leaves, we usually treat $N_{\go{S}}$ as an abstract set (except for the leaves).

Note that there is a natural injection $\iota: \mathring{\go{S}} \rightarrow \mathring{\go{T}}$ given by the map $a \mapsto (\go{L}_{a})^{\uparrow,\go{T}}$ where $a \in \mathring{\go{S}}$ and the set of descendents of $a$ given by $L_{\go{a}}$ is determined by $\go{S}$.  

If a coalesence tree $\go{S}$ on a vertex set $B$ is obtained in such a way from a a coalesence tree $\go{T}$ on a vertex set $A \supset B$ then we say $\go{S}$ is a \emph{restriction} of $\go{T}$ to $B$ and sometimes write $\go{S} \eqdef \go{T} \restr_{B}$.  

We draw a picture to make this clearer. Here $A = \{a_{i}\}_{i=1}^{4} \sqcup \{b_{i}\}_{i=1}^{5}$ and $B = \{b_{i}\}_{i=1}^{5}$. 
To the left we have drawn a choice of $\go{T} \in \widehat{\CU}_{A}$ and on the right the corresponding $\go{T}\restr_{B} \in \CU_{B}$. 
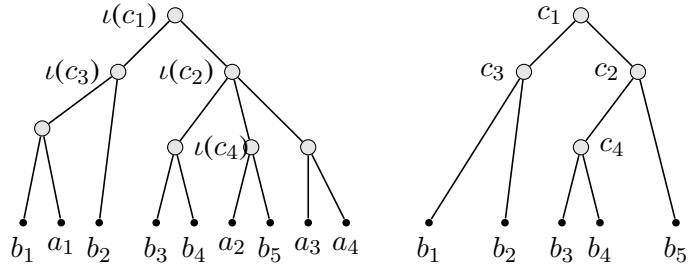
\begin{center}
\begin{tikzpicture}

    \node [style=dot, label=below:$b_{1}$] (1) at (-1.75, 0) {};
    \node [style=dot, label=below:$a_{1}$] (2) at (-1.25, 0) {};
    \node [style=dot, label=below:$b_{2}$] (3) at (-.75, 0) {};
    \node [style=dot, label=below:$b_{3}$] (4) at (0, 0) {};
    \node [style=dot, label=below:$b_{4}$] (5) at (0.5, 0) {};
    \node [style=dot, label=below:$a_{2}$] (6) at (1, 0) {};
    \node [style=dot, label=below:$b_{5}$] (7) at (1.5, 0) {};
    \node [style=dot, label=below:$a_{3}$] (8) at (2, 0) {};
    \node [style=dot, label=below:$a_{4}$] (9) at (2.5, 0) {};

    \node [style=coalnode, label=left:$\iota(c_1)$] (c1) at (.25, 2.75) {};
    \node [style=coalnode] (c2) at (-1.5, 1.25) {};
    \node [style=coalnode, label=left:$\iota(c_2)$] (c3) at (1, 2) {};

    \node [style=coalnode] (c4) at (1.25, 1) {};

    \node [style=coalnode, label=left:$\iota(c_3)$] (c5) at (-0.5, 2) {};

    \node [style=coalnode] (c6) at (2, 1) {};

    \node [style=coalnode, label=right:$\iota(c_4)$] (c7) at (0.25, 1) {};

    \draw[coalline] (1) -- (c2); 
    \draw[coalline] (2) -- (c2);
    \draw[coalline] (c2) -- (c5);
    \draw[coalline] (3) -- (c5); 

    \draw[coalline] (4) -- (c7); 
    \draw[coalline] (5) -- (c7);

    \draw[coalline] (6) -- (c4);
    \draw[coalline] (7) -- (c4);

    \draw[coalline] (8) -- (c6);
    \draw[coalline] (9) -- (c6);

    \draw[coalline] (c7) -- (c3); 
    \draw[coalline] (c4) -- (c3);
    \draw[coalline] (c6) -- (c3);

    \draw[coalline] (c3) -- (c1);
    \draw[coalline] (c5) -- (c1);

\end{tikzpicture}
\quad
\begin{tikzpicture}

    \node [style=dot, label=below:$b_{1}$] (1) at (-1.75, 0) {};

    \node [style=dot, label=below:$b_{2}$] (3) at (-.75, 0) {};
    \node [style=dot, label=below:$b_{3}$] (4) at (0, 0) {};
    \node [style=dot, label=below:$b_{4}$] (5) at (0.5, 0) {};

    \node [style=dot, label=below:$b_{5}$] (7) at (1.5, 0) {};

    \node [style=coalnode, label=left:$c_1$] (c1) at (.25, 2.75) {};

    \node [style=coalnode, label=left:$c_2$] (c3) at (1, 2) {};

    \node [style=coalnode, label=left:$c_3$] (c5) at (-0.5, 2) {};

    \node [style=coalnode, label=right:$c_4$] (c7) at (0.25, 1) {};

    \draw[coalline] (1) -- (c5); 
    \draw[coalline] (3) -- (c5); 

    \draw[coalline] (4) -- (c7); 
    \draw[coalline] (5) -- (c7);

    \draw[coalline] (7) -- (c3);

    \draw[coalline] (c7) -- (c3); 

    \draw[coalline] (c3) -- (c1);
    \draw[coalline] (c5) -- (c1);
\end{tikzpicture}
\end{center}
\end{definition}
\begin{definition}\label{def: useful cumulant bound}
Let $\c$ be a cumulant homogeneity, let $A$ be a finite set, and let $B \subset A$ be of cardinality at least $2$ and suppose we also have a type map $\mft:B \rightarrow \Le$.\\
We define a total homogeneity $\c^{(\mft,B)} = (\c^{(\mft,B)}_{\go{T}})_{\go{T} \in \widehat{\CU}_{A}}$ as follows.
Fix $\go{T} \in \widehat{\CU}_{A}$, let $\go{S} \eqdef \go{T}\restr_{B}$, and let $\iota$ be the corresponding injection $\iota: \mathring{\go{S}} \rightarrow \mathring{\go{T}}$. 
We then set
\[
\c^{(\mft,B)}_{\go{T}}(u)
\eqdef
\begin{cases}
\c^{(\mft,B)}_{\go{S}}(a)
&
\textnormal{ if }
u = \iota(a)
\textnormal{ for }a \in \mathring{\go{S}}\\
0
&\ 
\textnormal{ otherwise.}
\end{cases}
\]
\end{definition} 

\subsection{Multiscale clustering for Wick contractions}

\begin{definition} Let $A$ be a finite set and $p$ be a positive integer. We denote by $\{ A^{(j)} \}_{j=1}^{p}$ $p$ distinct copies of the set $A$. A $p$-fold Wick contraction on $A$ is a partition $\pi$ of $\bigsqcup_{j=1}^{p} A^{(j)}$ which satisfies the condition that for any $B \in \pi$ we have $B \not \subset A^{(j)}$ for any $1 \le j \le p$. In other words, each element of the partition is required to straddle at least two copies of $A$. 
\end{definition}
\begin{theorem}\label{thm: half graph assumptions}
Let $\CX$ be a finite set of cardinality at least two with a distinguished element $\logof \in \CX$ and a subset $\genwickleaves \subset \CX_{0} \eqdef \CX \setminus \{\logof\}$ equipped with a type map $\mft: \genwickleaves \rightarrow \Le$.

Let $\bo{H} \subset \CX^{(2)}$ be a connected multigraph on $\CX$ and $\mcN_{\bo{H}} \subset \N^{\go{H}}$ be a set of scales. 
Let $\bo{H}_{\ast} \subset \bo{H}$ be a non-empty connected subset of edges such that $\CX_{\ast}$, the set 
of vertices incident to $\bo{H}_*$, is disjoint from $\genwickleaves$.

Let $\varsigma$ be a total homogeneity on $\widehat{\CU}_{\CX}$ with the following properties:
\begin{enumerate}
\item $\varsigma$ is of order $\alpha < |\mft(\genwickleaves)|_{\s}$. 
\item Using the notation $\mcb{h}_{\c,\cdot}(\cdot)$ of \eqref{homogeneity gain from missed noise}, one has, for every $\bo{T} \in \mcU_{\mcX}$ and $a \in \mathring{\bo{T}}$,
\begin{equs}\label{integrability condition - big graph}
\sum_{
b \in \bo{T}_{\ge a}
}
\varsigma_{\bo{T}}(b)
<&
|\s| \,
(|\bo{L}_{a}|-1)
+
|
\mft\left( 
\bo{L}_{a}^{\exte}
\right)|_{\s}\\
&
+
\mathbbm{1} 
\left\{ 
\bo{L}_{a} \not \ni \logof
\right\}
\Big[
\mcb{h}_{\c,\ell}(\bo{L}_{a}^{\exte})
\wedge 
\mcb{j}_{\ell}(\go{L}_{a}^{\exte})
\wedge
\frac{|\s|}{2}
\Big]\;,
\end{equs}
where we used the notation $\go{L}_{a}^{\exte} \eqdef \{u \in \genwickleaves :\ u \in \go{L}_{a}\}$.
\item For any $\bo{T} \in \mcU_{\mcX}$ and $a \in \mathring{\bo{T}}$ with both $a \le \CX_{\ast}^{\uparrow}$ and $\go{T}_{\not\ge a} \not = \emptyset$ one has
\begin{equ}\label{decay at large scales}
\sum_{
b \in \go{T}_{\not\ge a}
}
\varsigma_{\go{T}}(b)
>
|\s| \, |\mcX \setminus \bo{L}_{a}|
+
|
\mft(L
\setminus
\bo{L}_{a}^{\exte}
)|_{\s}\;.
\end{equ}
\end{enumerate}
Then, for any $p \in \N$, there exists $C > 0$ such that for any family of functions $G = (G^{\mbn})_{\mbn \in \mcN_{\bo{H}}}$ on $(\R^{d})^{\CX_{0}}$ bounded by $\varsigma$ (for some choice of $x_{\logof} \in \R^{d}$ where $\logof$ serves as the pinned vertex of $\CX$), one has the estimate
\begin{equs}\label{the final bound}
{}
&
\E
\Big(
\int_{\CX_{0}} dy \,
\sum_{\mbn \in \mcN_{\bo{H}, > r, \go{H}_{\ast}}}
G^{\mbn}(y)
\cdot
\wwick{ \{\xi_{\mft(u)}(y_{u}) \}}_{u \in \genwickleaves}
\Big)^{2p}\\
&
\qquad \qquad \qquad
\le
C
\|\xi\|_{2p|\genwickleaves|,\c}
\cdot
2^{2p\beta r}
\cdot
\inf_{\tilde{G} \in \mathrm{Mod}_{\genwickleaves}(G)}
\|\tilde G\|^{2p}_{
\varsigma,
\mcN_{\go{H}},
\genwickleaves
},
\end{equs}
where the norm on the RHS was defined in Definition~\ref{def: a derivative norm} and $\beta \eqdef \alpha - |\mft(\genwickleaves)|_{\s}$.
\end{theorem}
\begin{proof}
Observe that the LHS of \eqref{the final bound} does not change if we replace $G$ with $\tilde{G}$ in $\mathrm{Mod}_{\genwickleaves}(G)$ -- for the rest of the proof we fix such a $\tilde{G}$. 
Next, recall that 
\[
\sum_{\mbn \in
\mcN_{\bo{H}, > r, \go{H}_{\ast}}}
\tilde{G}^{\mbn}
=
\sum_{\go{T} \in \CU_{\CX}}
\Big( 
\sum_{
\substack{
\mbn \in \mcN_{\bo{H}, > r, \go{H}_{\ast}}\\
\mcT(\mbn) = \go{T}
}
}
\tilde{G}^{\mbn}
\Big)\;,
\]
where the number of terms in outer sum on the RHS is bounded by some finite combinatorial constant depending on $|\CX|$.
Fix some choice of $\go{T} \in \CU_{\CX}$. By the triangle inequality it suffices to prove \eqref{the final bound} where on the LHS we replace 
\[
\sum_{\mbn \in \mcN_{\bo{H}, > r, \go{H}_{\ast}}}
G^{\mbn}
\enskip
\textnormal{with}
\enskip
\hat{G}
\eqdef
\sum_{
\substack{
\mbn \in \mcN_{\bo{H}, > r, \go{H}_{\ast}}\\
\mcT(\mbn) = \go{T}
}
}
\tilde{G}^{\mbn}
\]
and on the RHS of \eqref{the final bound} we drop the infinum over $\tilde{G}$ (since it has been fixed). 

We write $\CX_{0}^{(j)}$ for $j=1,\ldots,2p$ be $2p$ disjoint copies of $\CX_{0}$. 
We write $\CV_{0} \eqdef \bigsqcup_{j=1}^{2p} \CX_{0}^{(j)}$, $\CV \eqdef \CV_{0} \sqcup \{\logof\}$, and $\CX^{(j)} \eqdef \CX^{(j)}_{0} \sqcup \{\logof\}$.
We also set $\allwickleaves \eqdef \bigsqcup_{j=1}^{2p} \genwickleaves^{(j)} \subset \CV_0$,
so that $2p$-fold Wick contractions are naturally identified with partitions of $\allwickleaves$. 
Note that the type map $\mft$ naturally yields a type map on $\allwickleaves$ which we again call $\mft$.

Using standard facts about the expansion of moments of Wick powers into sums over Wick 
contractions (see \cite[Lem.~4.5]{HS15}) it suffices to show that for any  
any fixed $2p$-fold Wick 
contraction $\pi$ on $\genwickleaves$ one has
\begin{equs}\label{eq: the final bound}
\Big|
&
\int_{\CV_{0}}
dy \,
\Big(
\prod_{j=1}^{2p}
\hat{G}(y^{(j)})
\Big)
\Big( 
\prod_{B \in \pi}
\Cum
\left[ 
\{\xi_{\mft(u)}(y_{u})\}_{u \in B}
\right]
\Big)
\Big|\\
&
\qquad \qquad
\lesssim
\|\tilde{G}\|_{
\varsigma,
\mcN_{\go{H},\lambda},
\genwickleaves
}^{2p}
\|\xi\|_{2p|\genwickleaves|,\c}
\lambda^{2p\beta}\;.
\end{equs}
where we use the shorthand $y^{(j)} \eqdef y_{\CX^{(j)}_{0}}$.

For $j \in [2p]$ and $\bo{T} \in \CU_{\CX}$, we write $\go{H}^{(j)}$ and $\bo{T}^{(j)}$ for 
the corresponding copies of $\go{H}$ and $\bo{T}$ on the vertex set $\CX^{(j)}$. 
To prove \eqref{eq: the final bound} we will apply Theorem~\ref{thm: both bounds} with $\CV$ as the vertex set, and $\logof$ as the root vertex -- we know that $\hat{G}$ satisfies \eqref{domain condition} for some choice of $x_{\logof} \in \R^{d}$ -- this will be fixed throughout our whole application of Theorem~\ref{thm: both bounds} but our bounds don't depend on this variable.

The underlying multigraph is given by $\go{G} 
\eqdef
\tilde{\go{G}}
\sqcup
\go{C}$ where
\begin{equ}
\tilde{\go{G}}
\eqdef
\bigsqcup_{j=1}^{2p} \go{H}^{(j)} 
\enskip
\textnormal{and}
\enskip
\go{C}
\eqdef
\Big\{ 
\{v,v'\} \in \CV^{(2)}:\ 
\exists B \in \pi \textnormal{ with } 
v,v' \in B
\Big\}\;.
\end{equ}
We define a distinguished set of edges $\go{G}_{\ast} = \bigsqcup_{j = 1}^{2p} \go{H}^{(j)}_{\ast}$ where $\go{H}^{(j)}_{\ast}$ denote the copies of $\go{H}_{\ast}$. 
There is a natural identification of $\bigtimes_{j=1}^{2p}\N^{\go{H}^{(j)}}$ with $\N^{\tilde{\go{G}}}$ and using this identification we define our set of scales $\mcN_{\go{G}} \subset \N^{\tilde{\go{G}}} \times \N^{\go{C}} = \N^{\go{G}}$ by 
\[
\mcN_{\go{G}} 
\eqdef 
\left(
\bigtimes_{j=1}^{2p}
\left[
\mcN_{\bo{H}, > r, \go{H}_{\ast}}
\cap \mcT^{-1}(\go{T}^{(j)})
\right]
\right)
\times
\N^{\go{C}}
.
\] 
Note that then $\mcN_{\bo{G}, > r, \go{G}_{\ast}} =
\mcN_{\bo{G}}$. 
We write $\CU_{\CV}$ for the corresponding set of coalescence trees on $\CV$.
Elements $\mbm \in \mcN_{\go{G}}$ are often denoted as pairs $\mbm = (\mbn, \mbj) \in \N^{\tilde{\go{G}}} \times \N^{\go{C}}$.

Let the family of functions $
F 
=
(F^{\mbm})_{\mbm \in \mcN_{\go{G}}}$ on $(\R^{d})^{\CV_{0}}$ for which we apply Theorem~\ref{thm: both bounds} be given by 
\begin{equs}\label{the almost final integrand}
F^{(\mbn,\mbj)}_{\lambda}(y)
\eqdef
\Big(
\prod_{j=1}^{2p}
\tilde{G}^{\mbn_{j}}(y^{(j)})
\Big)
\mathrm{WickCum}^{\mbj}(y)
\end{equs}
where we used the natural correspondance $(\mbn_{j})_{j =1}^{2p} \leftrightarrow \mbn$ and set
\[
\mathrm{WickCum}^{\mbj}(y_{\allwickleaves})
\eqdef
\prod_{B \in \pi}
\Big[
\Cum
\left[ 
\{\xi_{\mft(u)}(y_{u})\}_{u \in B}
\right]
\prod_{\{u,v\} \in B^{(2)}}
\Psi^{(j_{\{u,v\}})}(y_{u} - y_{v})
\Big].
\]
We will define a total homogeneity $\tilde{\varsigma} = \{\tilde{\varsigma}_{\bo{S}}\}_{\bo{S} \in \widehat{\CU}_{\CV}}$ and $\tilde{F} \in \mathrm{Mod}(F)$ such that $\tilde{F}$ is bounded by $\tilde{\varsigma}$ (the root vertex position for \eqref{domain condition} will be given by $x_{\logof}$) and with the property that
\begin{equ}\label{eq: final supremum bound}
\|\tilde{F}\|_{\tilde{\varsigma},\mcN_{\go{G}}} \lesssim 
\|\xi\|_{2p|\genwickleaves|,\c} \cdot \|\tilde{G}\|_{
\varsigma,
\mcN_{\go{H}},
\genwickleaves
}^{2p}\;.
\end{equ}
We will then check that $\tilde{\varsigma}$ satisfies the conditions of Theorem~\ref{thm: both bounds}.

We start by defining $\tilde{\varsigma}$.
For each $j \in [2p]$ and $\go{S} \in \widehat{\CU}_{\CV}$ we define a map $\theta_{j,\go{S}}: 
\mathring{\bo{T}}^{(j)}
\rightarrow
\mathring{\bo{S}}
$
by setting, for any $a \in \mathring{\bo{T}}^{(j)}$, 
\[
\theta_{j,\go{S}}(a)
\eqdef
(\bo{L}_{a})^{\uparrow,\go{S}}\;.
\]
For the remainder of the proof, whenever we write $\uparrow$ or $\Uparrow$ without specifying a tree then this operation is taken in $\go{S}$.

For $j \in [2p]$ we define total homogeneities $\varsigma^{j} = \{\varsigma_{\bo{S}}^{j}\}_{\go{S} \in \CU_{\CV}}$ by setting, for each $\go{S} \in \CU_{\CV}$ and $a \in \mathring{\go{S}}$,  
\[
\varsigma_{\bo{S}}^{j}(a)
\eqdef
\sum_{b \in \theta_{j,\go{S}}^{-1}(a)}
\varsigma_{\bo{T}^{(j)}}(b) 
\]
where the total homogeneity $\varsigma$ was fixed in the assumption, $\go{T}$ is the previously fixed tree in $\CU_{\CX}$, and the $\varsigma_{\bo{T}^{(j)}}$ are the corresponding copies of $\varsigma_{\bo{T}}$. 

We introduce more total homogeneities on the trees of $\widehat{\CU}_{\CV}$, setting 
$
\bar{\varsigma}
\eqdef
\sum_{j=1}^{2p} 
\varsigma^{j}$, $\varsigma^{C} \eqdef \sum_{B \in \pi} \c^{B}$ where we use Definition~\ref{def: useful cumulant bound}, as well as $\varsigma^{R}$ by setting, for each $\go{S} \in \CU_{\CV}$, 
\begin{equs}
\varsigma^{R}_{\go{S}}
\eqdef&
\sum_{B \in R(\bo{S})}
\fict(B)
\left(
\delta^{\Uparrow}_{\go{S}}[B]
-
\delta^{\uparrow}_{\go{S}}[ B]
\right), \textnormal{ where}\\
R(\bo{S}) 
\eqdef& 
\left\{ B \in \pi:\ 
\fict(B) > 0
\textnormal{ and }
\go{L}_{B^{\uparrow}}
= 
B
\right\}\;,
\end{equs}
and $\fict(B)$ was defined in \eqref{def: fict}.
Finally, we define $\tilde{\varsigma} \eqdef \bar{\varsigma} + \varsigma^{C} + \varsigma^{R}$.

The family $\tilde{F}$ will be given by setting
\[
\tilde{F}^{\mbk}
\eqdef
\begin{cases}
F^{\mbk} & \textnormal{ if } R(\mcT(\mbk)) = \emptyset\;,\\
\mathring{F}^{\mbk} & \textnormal{ if } R(\mcT(\mbk)) \not = \emptyset\;.\\
\end{cases}
\]
We will define $\mathring{F}^{\mbk}$ on a tree by tree basis later. It suffices to check the domain condition and desired supremum bound in each of the two cases separately.

We first treat the former case, that is when $\go{S} \in \CU_{\CV}$ satisfies $R(\go{S}) = \emptyset$. 
The domain property \eqref{domain condition} is straightforward to check and next we show the supremum bound. 
Uniformly over $j \in [2p]$, $\go{s} \in \go{\mathrm{Lab}}_{\go{S}}$, and $(\mbn,\mbj) \in \mcN_{\mathrm{tri}}(\go{S},\go{s})$, one has 
\[
\sup_{y \in (\R^{d})^{\CV_{0}}}
\big| \tilde{G}^{\mbn^{j}}(y^{(j)}) \big|
\lesssim
2^{\langle \varsigma_{\bo{S}}^{j},\go{s}\rangle}
\|\tilde{G}\|_{\varsigma,\mcN_{\go{H}}}\;,
\]
as well as
\[
\sup_{y \in (\R^{d})^{\CV_{0}}}
\big| \mathrm{WickCum}^{\mbj}(y_{\allwickleaves}) \big|
\lesssim
2^{\langle
\varsigma^{C}_{\go{S}},
\go{s}
\rangle
}
\|\xi\|_{2p|\genwickleaves|,\c}\;.
\]
Combining these two estimates and setting $\mbm  = (\mbn,\mbj)$ as above yields 
\[
\sup_{y \in (\R^{d})^{\CV_{0}}}
\left| F^{\mbm}(y) \right|
\lesssim 
2^{\langle \varsigma_{\go{S}},\go{s} \rangle}
\|\tilde{G}\|_{\varsigma,\mcN_{\go{H}}}^{2p} 
\|\xi\|_{2p|\genwickleaves|,\c}\;.
\]
Now we treat the other case for $\tilde{F}$, we fix $\go{S} \in \CU_{\CX}$ and suppose that $R(\bo{S}) \not = \emptyset$. 
We work similarly to Lemma~\ref{genvertbd} here.
We again label, for each $B \in R(\bo{S})$, each of the elements of $B$ with a $+$ or $-$, define sets of multi-indices $\mathrm{Der}(B)$, and operators $\mathscr{Y}_{B}$. 
Then for any $\mbm  = (\mbn,\mbj) \in \mcN_{G}$ with $\mcT(\mbm) = \bo{S}$ we define $\mathring{F}^{\mbm}_{\lambda}$ via
\begin{equs}
\mathring{F}^{(\mbn,\mbj)}(y)
\eqdef
\Big(
\prod_{B \in R(\bo{S})}
(1 - 
\mathscr{Y}_{B}
)
\Big)
\Big[
\bigotimes_{j=1}^{2p}\tilde{G}^{\mbn^{j}}
\Big](y)
\cdot
\mathrm{WickCum}^{\mbj}(y_{\allwickleaves})\;.
\end{equs}
To see that both $\mathring{F}^{(\mbn,\mbj)}_{\lambda}$ and $F^{(\mbn,\mbj)}_{\lambda}$ integrate to the same value, one proceeds as in Lemma~\ref{genvertbd}. The desired domain constraint for $\mathring{F}^{(\mbn,\mbj)}_{\lambda}$ is also straightforward to check -- all that is left is the supremum bound.

Using \cite[Prop.~A.1]{Regularity} we get, uniform in $\bo{s} \in \go{\mathrm{Lab}}_{\go{S}}$ and $(\mbn,\mbj) \in \mcN_{\mathrm{tri}}(\bo{S},\bo{s})$, the supremum bound
\begin{equs}\label{noise renormalization}
\sup_{y \in (\R^{d})^{\CV_{0}}}
\left|
\mathring{F}^{(\mbn,\mbj)}_{\lambda}(y)
\right|
&\lesssim
\max_{
k \in 
\mathrm{Der}(R(\bo{S})) 
}
\Big( 
\prod_{B \in R(\bo{S})}
2^{-|k_{(+,B)}|_{\s} j_{B} }
\Big)\\
&
\quad
\quad
\cdot
\sup_{y \in (\R^{d})^{\CV_{0}}}
\Big|
\mathrm{WickCum}^{\mbj}(y_{\allwickleaves})
\prod_{j = 1}^{2p}
D^{k_{j}}
\tilde{G}^{\mbn^{j}}(y^{(j)})
\Big|
\end{equs}
where $
\mathrm{Der}(R(\bo{S})) \eqdef \sum_{B \in R(\bo{S})} \partial_{(+,B)} \mathrm{Der}(B)$  
and, for any $k \in (\N^{d})^{\CV}$ and $j \in [2p]$, we set $
k_{j}
\eqdef
\sum_{l \in \genwickleaves^{(j)}}
k_{l}\;$. 
Now for any $k$ appearing in \eqref{noise renormalization} we have, uniform in $\go{s} \in \go{\mathrm{Lab}}_{\go{S}}$ and $(\mbn,\mbj) \in \mcN_{\mathrm{tri}}(\go{S},\go{s})$, 
\begin{equation*}
\begin{split}
&
\sup_{y \in (\R^{d})^{\CV_{0}}}
\big|D^{k_{j}}\hat{G}^{\mbn^{j}}(y^{(j)}) \big|\\
\le
&
\exp_{2}
\Big[
\Big\langle
\varsigma_{\bo{S}}^{j}
+
\sum_{
\substack{
B \in R(\go{S})\\
(+,B) \in \genwickleaves^{(j)}}
}
|k_{(+,B)}|_{\s}
\mathbbm{1}
\left\{
\bullet =
\theta_{j,\go{S}}
\big[
(+,B)^{\uparrow,\bo{T}^{(j)}}
\big]
\right\},
\go{s}
\Big\rangle
\Big]
\|\tilde{G}\|_{
\varsigma,
\mcN_{\go{H}},
\genwickleaves
}^{2p}\\
&
\le
\exp_{2}
\Big[
\Big\langle
\varsigma_{\bo{S}}^{j}
+
\sum_{
\substack{
B \in R(\go{S})\\
(+,B) \in \genwickleaves^{(j)}
}
}
|k_{(+,B)}|_{\s}
\delta^{\Uparrow}_{\go{S}}[B],
\go{s}
\Big\rangle
\Big]
\|\tilde{G}\|_{
\varsigma,
\mcN_{\go{H}},
\genwickleaves
}^{2p}\;.
\end{split}
\end{equation*}
In going to the third line above we used the fact that in the partial order on the nodes of $\bo{S}$ one has $B^{\Uparrow}
\ge
\theta_{j,\go{S}}[(+,B)^{\uparrow,\bo{T}^{(j)}}]$ which can be justified as follows: the inner node on the left corresponds to the first coalescence event in $\go{S}$ where the component containing $(+,B)$ contains a distinct element from the same copy of $\CX_{0}$ that $(+,B)$ is from, however this must happen strictly after $(+,B)$ combines with $(-,B)$ to form a component of cardinality $2$. 

It follows that uniform in $\go{s} \in \go{\mathrm{Lab}}_{\go{S}}$ and $(\mbn,\mbj) \in \mcN_{\mathrm{tri}}(\go{S},\go{s})$ we have the bound
\begin{equation*}
\begin{split}
&
\sup_{y \in (\R^{d})^{\CV_{0}}}
\left|
\mathring{F}^{(\mbn,\mbj)}_{\lambda}(y)
\right|
\lesssim
2^{
\langle
\bar{\varsigma}_{\bo{S}},
\go{s}
\rangle
}
\|\xi\|_{2p|\genwickleaves|,\c} 
\|\tilde{G}\|_{
\varsigma,
\mcN_{\go{H}},
\genwickleaves
}^{2p}
\\
&
\qquad
\cdot
\max_{
k \in \mathrm{Der}(\bo{S})
}
\exp_{2}
\big[  
\langle
\sum_{
B \in R(\go{S})
}
|k_{(+,B)}|_{\s}
\big(
\delta^{\Uparrow}_{\go{S}}[B]
-
\delta_{\go{S}}^{\uparrow}[B]
\big),
\go{s}
\rangle
\big]\;.
\end{split}
\end{equation*}
To obtain the desired result observe that the maximum over $k$ on the RHS is acheived when $|k_{(+,B)}|_{\s} =
\fict(B)$ for each $B \in R(\go{S})$. 
This finishes the proof of \eqref{eq: final supremum bound}.

We now check that $\tilde{\varsigma}$ satisfies the assumptions of Theorem~\ref{thm: both bounds}. 
It is easy to see that the total homogeneity $\tilde{\varsigma}$ is of order $\beta \eqdef 2p(\alpha - |\mft(\genwickleaves)|_{\s}) < 0$. 

Next we prove that $\tilde{\varsigma}$ is subdivergence free on $\CV$ for the set of scales $\mcN_{\go{G}}$. 
In what follows we make frequent and implicit use of Assumption~\ref{assump - noise + kernel} and the following trivial bound -- for any $L' \subset \genwickleaves^{\textnormal{all}}$
\begin{equ}\label{eq: simple noise bound}
-
\sum_{
\substack{
B \in \pi\\
B \subset L'
}
}
|\mft(B)|_{\s}
-
\sum_{
\substack{
B \in \pi\\
B \not \subset L'
}
}
|\mft(B)|_{\s,\c,\ell}
\le
-
|\mft(L')|_{\s}\;.
\end{equ}
Fix $\bo{S} \in \CU_{\CV}$ and $a \in \mathring{\bo{S}}$.
For each $j \in [2p]$ and $A \subset [2p]$ we write  
\begin{equ}
\bo{L}_{a}^{j} \eqdef \bo{L}_{a} \cap \CX^{(j)},\enskip
\bo{L}_{a}^{j,\exte} \eqdef \go{L}_{a}^{j} \cap \genwickleaves^{j},\enskip 
\bo{L}_{a}^{A} \eqdef \bigsqcup_{i \in A} \bo{L}_{a}^{i},\enskip
\textnormal{and}\enskip
\bo{L}_{a}^{A,\exte} = \bigsqcup_{i \in A} \bo{L}_{a}^{i,\exte}\;.
\end{equ}
We also define three sets of indices 
\begin{equ}\label{indices}
J_{1}
\eqdef
\{ j \in [2p]:\ 
|\bo{L}^{(j)}_{a}| = 1
\},\; 
J_{2}
\eqdef
\{ j \in [2p]:\ 
|\bo{L}^{(j)}_{a}| \ge 2
\},\textnormal{ and }
J \eqdef J_{1} \sqcup J_{2}\;.
\end{equ}
Given $a \in \mathring{\go{S}}$ and $j \in [2p]$, we also write $a^{j}$ as a shorthand for
\begin{equ}
a^{j} \eqdef (\bo{L}_{a}^{j})^{\uparrow, \bo{T}^{(j)}}\;.
\end{equ}
We first treat the case when $\bo{L}_{a} \not \ni \logof$; this is done by checking a variety of subcases. 
Restricting to this case, for every $j \in J_{1}$ we have, by $\tilde{\go{G}}$-connectivity, $\bo{L}^{(j)}_{a} \subset \genwickleaves^{(j)}$.
In the subcase where $J_{2} = \emptyset$ we then have $\go{L}^{J,\exte}_{a} = \bo{L}_{a}$ and
\begin{equs}
\sum_{
b \in \bo{S}_{\ge a}
}
\tilde{\varsigma}_{\bo{S}}(b)
\le
-
\sum_{
\substack{
B \in \pi\\
B \subset \bo{L}_{a}
}
}
\big(
\fict(B)
\mathbbm{1}
\{\bo{L}_{a} = B\}
+
|\mft(B)|_{\s}
\big)
-
\sum_{
\substack{
B \in \pi\\
B \not \subset \go{L}_{a}
}
}
|\mft(B \cap \go{L}_{a})|_{\s,\c,\ell}.
\end{equs}
Again, by $\tilde{\go{G}}$-connectivity, there is a unique $\bar{B} \in \pi$ with $\go{L}_{a} \subset \bar{B}$ -- it follows that the RHS above is bounded by the LHS of \eqref{power counting cumulant bound} with $M = \go{L}_{a}$. 
The desired bound then follows from Lemma~\ref{lemma: higher cumulants are fine} and this finishes the subcase $J_{2} = \emptyset$.

For the subcase $J_{1} = \emptyset$ and $J_{2} = \{j\}$, we write 
$B^{a,j} \eqdef B \cap \go{L}^{j,\exte}_{a}$ and set $\pi^{a,j} \eqdef \{B \in \pi\,:\, B^{a,j} \neq \emptyset\}$.
With this notation, we have
\begin{equ}\label{single tree}
\sum_{
b \in \bo{S}_{\ge a}
}
\tilde{\varsigma}_{\bo{S}}(b)
\le
\sum_{
c \in {\bo{T}}^{(j)}_{\ge a^{j}}
}
\varsigma_{\go{T}^{(j)}}(c)
-
\sum_{
B \in \pi^{a,j}}
|\mft(B^{a,j})|_{\s,\c,\ell}\;.
\end{equ}
To see that the RHS is bounded by $(|\go{L}_{a}| -1) |\s| = (|\go{L}^{j}_{a}| -1) |\s|$ first note that it is immediate by \eqref{integrability condition - big graph} when the second sum on the RHS of \eqref{single tree} is empty. 
When this sum is not empty we then use $|\mft(\bullet)|_{\s} \le |\mft(\bullet)|_{\s,\c,\ell}$ for all but one block $B$ contributing to the sum to get the inequality
\begin{equ}\label{single tree 2}
-
\sum_{
B \in \pi^{a,j}}
|\mft(B^{a,j})|_{\s,\c,\ell}
\le
-|\mft(\go{L}^{j,\exte}_{a})|_{\s}
+
\min_{
B \in \pi^{a,j}
}
\big(
|\mft(B^{a,j})|_{\s,\c,\ell} - |\mft(B^{a,j})|_{\s}
\big)\;.
\end{equ}
The last term on the RHS of \eqref{single tree 2} is bounded above by $\mcb{h}_{\c,\ell}(\go{L}^{j,\exte}_{a})$ and therefore by \eqref{integrability condition - big graph} the RHS of \eqref{single tree} must be strictly bounded above by $(|\go{L}^{j}_{a}| -1 ) |\s|$. 

Now we treat the subcase where $J_{2} = \{j\}$ and $J_{1} \not = \emptyset$. 
Since $\go{L}_{a}$ is connected by $\go{G}$ and $\go{L}_{a} \not \ni  \logof$ it follows that there exist $B_{1},\dots,B_{n} \in \pi$ such that $\cup_{m=1}^{n} B_{m} \supset \go{L}_{a}^{J,\exte}$ and, for every $m \in [n]$, one has $B_{m} \cap \go{L}_{a}^{j,\exte} \not = \emptyset$.
Now if it is the case that for every $m$ one has $B_{m} \cap \go{L}_{a}^{J,\exte} \in \mfL_{\CCum}$ then,
\[
-
|\mft(\go{L}_{a}^{J_{1}})|_{\s}
-
|\s|\cdot|\go{L}_{a}^{J_{1}}|
\le 
- \mcb{j}_{\ell}(\go{L}_{a}^{j,\exte}),
\] 
and also using $\sum_{
b \in {\bo{S}}_{\ge a}
} \varsigma^{C}_{\bo{S}}(b) \le -
|\mft(\go{L}_{a}^{j,\exte})|_{\s} -
|\mft(\go{L}_{a}^{J_{1}})|_{\s}$ we have 
\begin{equs}\label{wick contraction work}
\sum_{
b \in {\bo{S}}_{\ge a}
}
\tilde{\varsigma}_{\bo{S}}(b)
\le&
\sum_{
c \in {\bo{T}}^{(j)}_{\ge a^{(j)}}
}
\varsigma_{\go{T}^{(j)}}(c)
-
|\mft(\go{L}_{a}^{j,\exte})|_{\s} -
|\mft(\go{L}_{a}^{J_{1}})|_{\s}\\
\le&
\sum_{
c \in {\bo{T}}^{(j)}_{\ge a^{(j)}}
}
\varsigma_{\go{T}^{(j)}}(c)
-
|\mft(\go{L}_{a}^{j,\exte})|_{\s}
-
\mcb{j}_{\ell}(\go{L}_{a}^{j,\exte})
+
|\s| \cdot |\go{L}_{a}^{J_{1}}|
\end{equs}
and we then get the desired bound by applying \eqref{integrability condition - big graph}. 
On the other hand, if there exists $m \in [n]$ such that $B_{m} \cap \go{L}_{a}^{J,\exte} \not \in \mfL_{\CCum}$ then it follows that $|B_{m} \cap \go{L}_{a}^{j,\exte}| = 1$ and $|\mft(B_{m} \cap \go{L}_{a}^{\exte})|_{\s,\c,\ell} = 0$, so writing $B_{m} \cap \go{L}_{a}^{j,\exte} = \{u\}$ we have
we have
\begin{equs}
\sum_{
b \in {\bo{S}}_{\ge a}
}
\tilde{\varsigma}_{\bo{S}}(b)
\le&
\sum_{
c \in {\bo{T}}^{(j)}_{\ge a^{(j)}}
}
\varsigma_{\go{T}^{(j)}}(c)
-
|\mft(\go{L}_{a}^{j,\exte} \setminus \{u\})|_{\s}
-
|\mft(\go{L}_{a}^{J_{1}})|_{\s}\\
\le
&
\sum_{
c \in {\bo{T}}^{(j)}_{\ge a^{(j)}}
}
\varsigma_{\go{T}^{(j)}}(c)
-
|\mft(\go{L}_{a}^{j,\exte})|_{\s}
+
\big[
|\mft(\{u\})|_{\s} - |\mft(\{u\})|_{\s,\c,\ell} 
\big]
+|\s| \cdot |J_{1}|
\end{equs}
The bracketed quantity on the last line is bounded below by $- \mcb{h}_{\c,\ell}(\go{L}^{j,\exte})$ and so by \eqref{integrability condition - big graph} the last line without the last term is bounded above by $(|\go{L}_{a}^{j}| - 1) \cdot |\s|$ which finishes this subcase.

Finally, suppose $|J_{2}| \ge 2$. We then have
\begin{equs}
\sum_{
b \in \bo{S}_{\ge a}
}
\tilde{\varsigma}_{\bo{S}}(b)
&\le
\Big[
\sum_{j \in J_{2}}
\Big(
\sum_{
c \in {\bo{T}}^{(j)}_{\ge a^{j}}
}
\varsigma_{\go{T}^{(j)}}(c)
-
|\mft(\go{L}^{j,\exte}_{a})|_{\s}
\Big)
\Big]
-
\sum_{j \in J_{1}}
|
\mft\left( \go{L}^{j,\exte}_{a}
\right)|_{\s}
\\
&<
\sum_{j \in J_{2}}
|\s| 
\left[ 
|\bo{L}^{j}_{a}|
- 
\frac{1}{2}
\right]
-
\sum_{j \in J_{1}}
|
\mft\left( \go{L}^{j,\exte}_{a}
\right)|_{\s}\;.\label{upper bound, generic case}
\end{equs} 
In going to the second line of \eqref{upper bound, generic case} we used, for each $j \in J_{2}$, \eqref{integrability condition - big graph}.
Now by using $\sum_{j \in J_{2}} |\bo{L}_{a}^{j}| = |\bo{L}_{a}| - |J_{1}|$ and Assumption~\ref{assump - noise + kernel} we see that the RHS of \eqref{upper bound, generic case} is strictly bounded above by $|\s| ( 
|\bo{L}_{a}| - \frac{1}{2}|J_{2}| 
)$ so we are done. 

We now treat the case $\bo{L}_{a} \ni \logof$. Here $\bo{L}_{a}^{j} \ni \logof$ for each $j \in [2p]$ and \eqref{integrability condition - big graph} gives
\begin{equation}\label{upper bound, basepoint case}
\begin{split}
\sum_{
b \in \bo{S}_{\ge a}
}
\tilde{\varsigma}_{\bo{S}}(b)
\le&\  
\sum_{j \in J_{2}}
\Big(
\sum_{
c \in {\bo{T}}^{(j)}_{\ge a^{j}}
}
\varsigma_{\go{T}^{(j)}}(c)
-
|
\mft
(\go{L}^{j,\exte}_{a})
|_{\s}
\Big)
\\
<&\ 
|\s| 
\sum_{j \in J_{2}}
\big(
|\bo{L}_{a}^{j}|
- 1
\big)
=
|\s|
(|\bo{L}_{a}| - 1)\;.
\end{split}
\end{equation}
This completes the proof that the total homogeneity $\tilde{\varsigma}$ is subdivergence free on $\CV$ for the set of scales $\mcN_{\bo{G}}$.

We now turn to showing \eqref{second cond of HQ} in our context. 
Suppose we are given $u \in \mathring{\bo{S}}$ satisfying both $u \le (\mvert_{\ast})^{\uparrow}$ and $
\go{S}_{\not\ge u} \not = \emptyset$.
We define, for each $j \in [2p]$,  
\[
\bo{Q}^{j}_{u} \eqdef \CX^{(j)} \setminus \bo{L}_{u} 
\enskip
\textnormal{and}
\enskip
D \eqdef \{ j \in [2p]:\ \bo{Q}^{j}_{u} \not = \emptyset\}\;.
\] 
Clearly $D$ is non-empty.
We then have 
\begin{equation*}
\begin{split}
\sum_{ 
w \in \go{S}_{\not\ge u}
}
&\tilde{\varsigma}_{\go{S}}(w) - |\CV \setminus \bo{L}_{u}|\, |\s|\\
&\ge
-
\sum_{w \in \go{S}_{\not\ge u}}
\sum_{B \in \pi}
\c^{(\mft,B)}(w)
+
\sum_{j \in D}
\Big[
\sum_{
w \in \go{T}^{(j)}_{\not\ge u^{j}}
}
\varsigma_{\bo{T}^{(j)}}(w)
-
|\CX^{(j)} \setminus \bo{L}_{u^{j}}| \, |\s|
\Big]\\
&\ge
\sum_{j \in D}
\Big[
\sum_{
w \in \go{T}^{(j)}_{\not\ge u^{j}}
}
\varsigma_{\bo{T}^{(j)}}(w)
-
|\mft( \genwickleaves^{(j)} \setminus \bo{L}_{u}^{j})|_{\s}
-
|\CX^{(j)} \setminus \bo{L}_{u^{j}}| \, |\s|
\Big]
>\
0\;.
\end{split}
\end{equation*}
In the second inequality we used \eqref{decay at large scales} and in the first we used
\begin{equs}
\sum_{w \in \go{S}_{\not\ge u}}
\sum_{B \in \pi}
\c^{(\mft,B)}(w)
=&
-|\mft(\allwickleaves)|_{\s}
-
\sum_{
w \in \go{S}_{\ge u}}
\c^{(\mft,B)}(w)\\
\le&
-|\mft(\allwickleaves)|_{\s}
+
|\mft(\go{L}_{u}^{\exte})|_{\s} = -|\mft(\allwickleaves \setminus \go{L}_{u}^{\exte})|_{\s}\;.
\end{equs}
\end{proof}
\section{Convergence of mollified approximations}\label{Sec: Convergence of measures}
Throughout this section we take the dimension $d$ and space-time scaling $\s$ as fixed along with a set of noise types $\mfL_{-}$ and a set of non-vanishing cumulants $\mfL_{\CCum}$.

The primary question of this section is as follows: given $\xi \in \CM(\Omega_{0})$ and a sequence of smooth, compactly supported approximate identities $\{\eta_{\eps}\}_{\eps \in (0,1]}$ on $\R^{d}$ which converge to the Dirac delta function $\delta$, does $\|\xi;\xi_{\eps}\|_{\c,N} \rightarrow 0$ as $\eps \down 0$ where we use the notation of Definition~\ref{def: metric on mollified measures2} and we set $\xi_{\eps} = \{ \xi_{\mft} \ast \eta_{\eps}\}_{\mft \in \Le}$.  

We denote by $\Moll$ the collection of all smooth functions $\eta \in \mcb{C}(\R^{d})$ which are supported on the closed ball $\{x \in \R^{d}: |x| \le 1\}$ and satisfy $\int dx\ \eta(x) = 1$. 
The following lemma is immediate. 
\begin{lemma}\label{lem: mollification is good}
For any $\eta \in \Moll$ and $\zeta \in \mathcal{M}(\Omega_{0},\mfL_{\CCum})$ the noise $\zeta \ast \eta = (\zeta_{\mft} \ast \eta)_{\mft \in \mfL_{-}}$ belongs to $\mathcal{M}(\Omega_{\infty},\mfL_{\CCum})$.
\end{lemma}
Throughout this section $\delta$ will always denote the Dirac delta distribution on $\R^{d}$.
We also write $\overline{\Moll} \eqdef \Moll \sqcup \{\delta\}$, and also quantitative notions of the convergence in $\overline{\Moll}$ by setting, for each $\kappa \in [0,1)$, $r \in \N$, and any distribution $\eta \in \mcD'(\R^{d})$ determined by a smooth function away from the origin, 
\[
\|\eta \|_{\kappa,r}
\eqdef
\max_{
\substack{
k \in \N^{d}\\
|k|_{\s} < 4r|\s|
}
}
\Big[
|
\eta (p_{k})
|
\vee
\big(
\sup_{y \in \R^{d} \setminus \{0\}}
|D^{k}\eta(y)
|
\cdot
|y|^{|\s| + |k|_{\s} + \kappa}
\big)
\Big].
\]
As an example, observe that if one fixes $\eta \in \Moll$ and defines, for each $\eps > 0$, 
\begin{equ}\label{scaled mollifier}
\eta^{\eps}(x_{1},\dots,x_{d}) \eqdef \eps^{-|\s|}\eta
\left(
\eps^{-\s_{1}}x_{1},\dots,\eps^{-\s_{d}}x_{d}
\right)\;,
\end{equ}
then for any $\kappa \in [0,1)$ one has $\|\eta^{\eps} - \delta\|_{\kappa,r} \sim \eps^{\kappa}$ as $\eps \downarrow 0$ for fixed $r$.
In what follows we use the notation 
\[
\widetilde{\Moll} \eqdef \Moll \sqcup \{ \eta_{1} - \eta_{2}:\ \eta_{1},\eta_{2} \in \Moll \}
\]

Before stating the main proposition of this section we need additional notation. 
\begin{definition}\label{def: cumulant homogeneity penalization}
Given a cumulant homogeneity $\c$ on $\mfL_{\CCum}$ we define the the $\kappa$-penalization of $\c$ to be the cumulant homogeneity $\tilde{\c}$ on $\mfL_{\CCum}$ defined via setting, for each $(\mft,[N]) \in \mfL_{\CCum}$, 
\[
\tilde{\c}^{(\mft,[N])} \eqdef
\c^{(\mft,[N])}
+
\kappa
\sum_{v \in [N]}
\delta^{\uparrow}[v],
\]
where we are using the notation of \eqref{def: delta total homogeneities}.
\end{definition}
\begin{proposition}\label{convergence of mollifications}
For any $\xi \in \mcM(\Omega_{0},\mfL_{\CCum})$ and $\eta \in \Moll$ one has the random noise type valued variable $\xi \ast \eta \eqdef (\xi_{\mft} \ast \eta)_{\mft \in \Le}$ is an element of $\mcM(\Omega_{\infty},\mfL_{\CCum})$. 

Additionally, for any $N \ge 2$, and cumulant homogeneity $\c$ on $\mfL_{\CCum}$ and $r \in \N$, one has, uniform in $\xi \in \mcM(\Omega_{0},\mfL_{\CCum})$, and $\kappa$ satisfying 
\[
0 < \kappa < 
\frac{1}{2}
\min_{(\mft,[2]) \in \mfL_{\CCum}}
\Big( \lceil \c^{(\mft,[2])} \rceil - \c^{(\mft,[2])} \Big),
\] 
and $\eta_{1},\eta_{2} \in \Moll_{r} \sqcup \{0\}$, the bounds
\begin{equ}\label{difference between mollified measures}
\|\xi \ast \eta_{1} ; \xi \ast \eta_{2}\|_{N,\c',r}
\lesssim
\|\eta_{1} - \eta_{2}\|_{\kappa}
\cdot
(\|\eta_{1}\|_{\kappa} \vee \|\eta_{2}\|_{\kappa})^{N-1}
\cdot
\|\xi\|_{N,\c,r+1}\;,
\end{equ}
where $\c'$ is the $\kappa$-penalization of $\c$. 
\end{proposition}
\begin{remark}
Note that by setting $\eta_{2} \eqdef 0$ in the above proposition we see that \eqref{difference between mollified measures} actually implies
\[
\|\xi \ast \eta_{1}\|_{N,\c',r}
\lesssim
\|\eta_{1}\|_{\kappa}^{N}
\cdot
\|\xi\|_{N,\c,r+1}\;.
\]
\end{remark}
\begin{proof}
It is immediate that $\xi \ast \eta_{1}$ and $\xi \ast \eta_{2}$ jointly admit pointwise cumulants. 
We first establish control over second cumulants.

Fix $(\mft,[2]) \in \mfL_{\CCum}$ and the shorthand $\alpha \eqdef \c^{(\mft,[2])}  - |\s|$.
For any $\beta \in \R$, $q \in \N$, and $H \in \mcb{S'}(\R^{d})$ with singular support contained in $\{0\}$, we define
\[
\wnorm{H}_{\beta,q}
\eqdef
\Big(
\max_{
\substack{k \in \N^{d}\\
|k|_{\s} < \alpha
}
}
H(p_{k})
\Big)
+
\max_{
k \in A_{q}
}
\sup_{x \in \R^{d}}
|D^{k}H(x)| \cdot |x|^{\beta + |k|_{\s}}\;. 
\]
Here $A_{q} \eqdef \{ k \in \N^{d}:\ |k|_{\s} \le 3q|\s| \}$
Clearly $\wnorm{\reduce{\Cum[\{\xi_{\mft(1)},\xi_{\mft(2)}\}]}}_{\alpha,2r+2} \le \|\xi\|_{\c,2,r+1}$ and conversely $\wnorm{\reduce{\Cum[\{\xi_{\mft(1)},\xi_{\mft(2)}\}]}}_{\alpha,2r}$ controls the contribution of second cumulants to $\|\xi\|_{\c,2,r}$. 

For any $H$ as above, $q \le 3$, and $\rho \in \widetilde{\Moll}$ we claim that for $\theta \in \{\alpha , \alpha + \kappa\}$, 
\begin{equ}\label{claim for second cumulant bound}
\wnorm{H \ast \rho}_{\theta + \kappa,q}
\lesssim
\|\rho \|_{\kappa}
\cdot
\wnorm{ H }_{ \theta ,q+1}\;.
\end{equ}
The desired bound on second cumulants follows from the claim since by applying it twice we get
\begin{equs}
{}
&
\wnorm{\reduce\big(\Cum[\{\xi_{\mft(1)} \ast \eta_{1}, \xi_{\mft(2)} \ast (\eta_{1} - \eta_{2}) \}]
\big)}_{\alpha + 2 \kappa,2r}\\
=
&
\wnorm{\reduce\big(\Cum[\{\xi_{\mft(1)}, \xi_{\mft(2)} \}]
\big)
\ast
\eta_{1} \ast \mathrm{Ref}[\eta_{1} - \eta_{2}]}_{ \alpha + 2\kappa,0}\\
\lesssim 
&
\wnorm{ \reduce \big( \Cum[\{\xi_{\mft(1)},\xi_{\mft(2)}\}] \big) \ast \eta_{1}}_{\alpha + \kappa,2r+1}
\cdot \|\eta_{1}-\eta_{2}\|_{\kappa}\\
\lesssim
&
\wnorm{ \reduce \big( \Cum[\{\xi_{\mft(1)},\xi_{\mft(2)}\}] \big)}_{\alpha,2r+2}
\cdot
\|\eta_{1}\|_{\kappa}
\cdot \|\eta_{1}-\eta_{2}\|_{\kappa},
\end{equs}
where $\mathrm{Ref}(g)(x) = g(-x)$ for any distribution $g$ on $\R^{d}$.

We now prove \eqref{claim for second cumulant bound}, assuming without loss of generality the RHS side of the bound is finite. 
Observe that for $m \in A_{q}$, $D^{m}(H \ast \rho)$ has empty singular support and is given everywhere by integration against the smooth function 
\begin{equs}
D^{m}(H \ast \rho)(x)
\eqdef&
\sum_{|k|_{\s}
<
\alpha
}
\frac{
H(p_{k})}
{k!}
(D^{k + m}\rho)(x)
+
D^{m}I(x)\\
\textnormal{where}\;
I(x)
\eqdef&
\int_{\R^{d}} dz\
\Big(
\rho(x-z) -
\sum_{|k|_{\s}
<
\alpha
}
\frac{(-z)^{k}}{k!}
(D^{k}\rho)(x)
\Big)
H(z)\;.
\end{equs}
Since $m \in A_{q}$ we get the bound, uniform in $x \in \R^{d}$, 
\begin{equ}\label{ easier second cumulant bound}
\Big|
\sum_{|k|_{\s}
<
\alpha
}
\frac{
H(p_{k})}
{k!}
(D^{k + m}\rho)(x)
\Big|
\lesssim
\| \rho \|_{\kappa} \cdot 
\wnorm{H}_{\theta ,0}
\cdot |x|^{-|m|_{\s} - |\s| - \kappa}
\;.
\end{equ}  
For the second term we perform a multiscale expansion where we again use coalesence trees to organize our bounds. 
Here we will be leaving more variables frozen instead of integrating all but one of them so we will perform more of the analysis by hand instead of just referencing the earlier multiclustering lemmas.

As in \cite[Lem.~A.4]{KPZJeremy}, we can construct a families of smooth functions $\{\rho_{j}\}_{j=0}^{\infty}$ and $\{H_{j}\}_{j=0}^{\infty}$ such that, for $F = H$ or $\rho$,
\begin{enumerate}
\item The identity $\sum_{j=0}^{\infty}
F_{j}
=
F$
holds in the sense of distributions on $\R^d$.
\item $F_{j}(x)$ is supported on $x$ with $2^{-j-2} \le |x| \le 2^{-j}$.
\item For any $k \in \N^{d}$ with $|k|_{\s} < 10|\s|$ one has, uniform in $j \ge 0$ and $x \in \R^{d}$,
\[
|D^{k}\rho_{j}(x)|
\lesssim
\| \rho \|_{\kappa}
\cdot
|x|^{- |\s| - |k|_{\s} - \kappa}\;.
\]
\item For each $k \in \N^{d}$ one has, uniform in $j \ge 0$, and $x \in \R^{d}$ and with $k \in A_{2}$ one has
\[
|D^{k} H_{j}(x)|
\lesssim
\wnorm{H}_{\theta,||k||}
\cdot
|x|^{-\theta - |k|_{\s}}
\]
\item For any $j > 0$, $\int dx\, \rho_{j}(x)
=
0$.
\item For any $j > 0$ and any $k \in \N^{d}$ with $|k|_{\s} \le \alpha$, $\int dx\, x^{k} \cdot H_{j}(x)
=
0$
\end{enumerate}
The integrand determining $I$ as a function of $x$, $z$ which we view as the positions of the vertices $a$ and $b$, respectively. 
We also introduce a root vertex $c$ pinned at $0$.
Thus we set $\CV \eqdef \{a,b,c\}$, $\CV_{0} \eqdef \{a,b\}$, and we fix the multigraph $\go{G}$ to be the complete graph on $\CV$. 
We also set $\mcN_{\go{G}} \eqdef \N^{\go{G}}$.

Fix $\bar{n} \ge 0$.
Then for any $x \in \R^{d}$ with $2^{-\bar{n} - 2} \le |x| < 2^{-\bar{n}}$ one has, by exploiting triangle inequality constraints on scale indices,
\begin{equs}\label{thing to bound for convergence of second cumulants}
D^{m}I(x)
=
\sum_{
\substack{
\go{T} \in \widehat{\CU}_{\CV}\\
\go{s} \in \mathrm{Lab}_{\go{T}}(\bar{n})\\
\mbn \in \mcN_{\mathrm{tri}}(\go{T},\go{s})
}
}
\int
dz\ \mathring{I}_{m}^{\mbn}(x,z)
+
\int_{\R^{d}}dz\
\sum_{|k|_{\s}
<
\alpha
}
(D^{k+m}\rho)(x)
\frac{(-z)^{k}}{k!}
H_{0}(z)\;.
\end{equs}
where $\mathring{I}_{m}^{\mbn}(x,z) \eqdef D^{m}\rho_{n_{\{a,b\}}}(x-z)  H_{n_{\{b,c\}}}(z) 
\Psi^{(n_{\{a,c\}})}(x)$, for $\go{T} \in \widehat{\CU}_{\CV}$, 
\[
\mathrm{Lab}_{\go{T}}(\bar{n}) \eqdef \{ \go{s} \in \mathrm{Lab}_{\go{T}}:\ \go{s}(\{a,c\}^{\uparrow}) - \bar{n}| \le 12 \}\;,
\]
and we choose some constant $C \ge 4$ for the constraint defining $\mcN_{\mathrm{tri}}(\go{T},\go{s})$.
The second integral on the RHS of \eqref{thing to bound for convergence of second cumulants} satisfies a bound like \eqref{ easier second cumulant bound} so we are left with estimating the contribution from the first integral.

Since there are only four coalesence trees in $\widehat{\CU}_{\CV}$ we present $\varsigma$ by showing these four trees below (which we will denote, going from left to right, $\go{T}_{1},\go{T}_{2},\go{T}_{3},\go{T}_{4}$) and then labeling the internal nodes with the weight that $\varsigma$ gives them: 
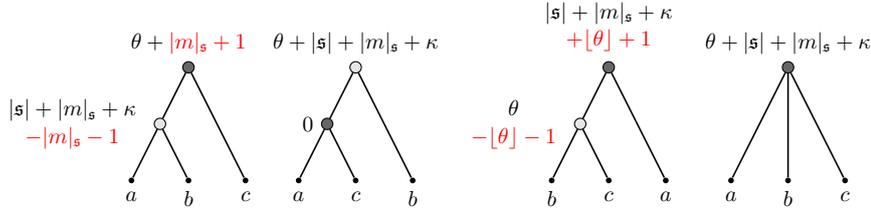
\begin{equ}
\begin{tikzpicture}[scale = .75, every node/.style={scale=0.75}]
    \node [style=dot, label=below:$a$] (one) at (-1.5, -1) {};
    \node [style=dot, label=below:$b$] (two) at (-0.5, -1) {};

    \node [style=dot, label=below:$c$] (three) at (0.5, -1) {};

    \node [style=coalnode, label=left:$
    \begin{array}{c}
    |\s| + |m|_{\s} + \kappa\\ 
     {\color{red} - |m|_{\s} - 1}
    \end{array}$] (coal1) at (-1, 0) {};
    
    \node [style=coalnode2, label=above:$\theta  + {\color{red}|m|_{\s} + 1}$] (root) at (-0.5, 1) {};

    \draw[coalline] (one) -- (coal1); 
    \draw[coalline] (two) -- (coal1); 
    \draw[coalline] (coal1) -- (root);
    \draw[coalline] (three) -- (root);
\end{tikzpicture}
\;
\begin{tikzpicture}[scale = .75, every node/.style={scale=0.75}]
    \node [style=dot, label=below:$a$] (one) at (-1.5, -1) {};
    \node [style=dot, label=below:$c$] (two) at (-0.5, -1) {};

    \node [style=dot, label=below:$b$] (three) at (0.5, -1) {};

    \node [style=coalnode2, label=left:$0$] (coal1) at (-1, 0) {};
    
    \node [style=coalnode, label=above:$ \theta + |\s| + |m|_{\s} + \kappa$] (root) at (-0.5, 1) {};

    \draw[coalline] (one) -- (coal1); 
    \draw[coalline] (two) -- (coal1); 
    \draw[coalline] (coal1) -- (root);
    \draw[coalline] (three) -- (root);
\end{tikzpicture}
\;
\begin{tikzpicture}[scale = .75, every node/.style={scale=0.75}]
    \node [style=dot, label=below:$b$] (one) at (-1.5, -1) {};
    \node [style=dot, label=below:$c$] (two) at (-0.5, -1) {};

    \node [style=dot, label=below:$a$] (three) at (0.5, -1) {};

    \node [style=coalnode, label=left:$
    \begin{array}{c}
    \theta  \\
    {\color{red} - \lfloor \theta \rfloor - 1}
    \end{array}$] (coal1) at (-1, 0) {};
    
    \node [style=coalnode2, label=above:$
    \begin{array}{c}
    |\s| + |m|_{\s} + \kappa \\
     {\color{red}+ \lfloor \theta \rfloor + 1}
    \end{array}$] (root) at (-0.5, 1) {};

    \draw[coalline] (one) -- (coal1); 
    \draw[coalline] (two) -- (coal1); 
    \draw[coalline] (coal1) -- (root);
    \draw[coalline] (three) -- (root);
\end{tikzpicture}
\;
\begin{tikzpicture}[scale = .75, every node/.style={scale=0.75}]
    \node [style=dot, label=below:$a$] (one) at (-1.5, -1) {};
    \node [style=dot, label=below:$b$] (two) at (-0.5, -1) {};

    \node [style=dot, label=below:$c$] (three) at (0.5, -1) {};
    
    \node [style=coalnode2, label=above:$\theta + |\s| + |m|_{\s} + \kappa$] (root) at (-0.5, 1) {};

    \draw[coalline] (one) -- (root); 
    \draw[coalline] (two) -- (root); 
    \draw[coalline] (three) -- (root);
\end{tikzpicture}
\end{equ}
Writing $\mathring{I}_{m} \eqdef \{\mathring{I}_{m}^{\mbn}\}_{\mbn \in \mcN_{\go{G}}}$, we claim that there is a modification $\tilde{I}_{m} \in \mathrm{Mod}_{\{b\}}(\mathring{I}_{m})$ which is bounded by a total homogeneity $\varsigma$ with
\[
\|
\tilde{I}_{m}
\|_{\varsigma,\mcN_{\go{G}}}
\lesssim
\wnorm{H}_{\theta, j+1 } \cdot
\|\rho\|_{\kappa}.
\]
We have darkened the internal node $\{a,c\}^{\uparrow}$ whose scale $\go{s}(\{a,c\}^{\uparrow})$ will be frozen to be close to $\bar{n}$ and we have colored in red the contributions coming from renormalization. 

For any $\mbn \in \mcN_{\go{G}}$ with $\mcT(\mbn) \in \{\go{T}_{2}, \go{T}_{4}\}$ we set $\tilde{I}_{m}^{\mbn} \eqdef \mathring{I}_{m}^{\mbn}$ and the desired estimate comes from the bound
\[
\big|
\mathring{I}_{m}^{\mbn}
\big|
\lesssim 
\big(
2^{n_{\{a,b\}}(|m|_{\s} + |\s| + \theta)}
\|\rho\|_{\theta}
\big)
\big(
2^{n_{\{b,c\}} \alpha}
\wnorm{H}_{\alpha,0}
\big)\;.
\]
For any $\mbn \in \mcN_{\go{G}}$ with $\mcT(\mbn) = \go{T}_{1}$ we set
\begin{equs}
\tilde{I}_{m}^{\mbn}(y_{a},y_{b})
\eqdef&
D^{m}\rho_{n_{\{a,b\}}}(y_{a}-y_{b}) \cdot \Psi^{(n_{\{a,c\}})}(y_{a})\\
&
\enskip
\cdot
\big(
H_{n_{\{b,c\}}}(y_{b})
-
\sum_{
\substack{
k \in \N^{d}\\ 
|k|_{\s} \le |m|_{\s}
}}
\frac{(y_{b} - y_{a})^{k}}{k!}
H_{n_{\{b,c\}}}(y_{a}) 
\big)\;.
\end{equs}
To see that this is compatible with $\tilde{I}^{m} \in \mathrm{Mod}_{\{b\}}(\mathring{I}^{m})$ observe that for $\mbn$ as above we have $n_{\{a,b\}} \not = 0$ and so $\int dy_{b}\ 
\tilde{I}_{m}^{\mbn}(y_{a},y_{b})
=
\int dy_{b}\  
\tilde{I}_{m}^{\mbn}(y_{a},y_{b})$ - the support constraint is also easy to check. 
Writing $\go{s} = \go{s}(\mbn)$ and $\mathring{\go{T}}_{1} = \{ \rho_{\go{T}_{1}},u\}$, a Taylor remainder estimate then gives
\begin{equs}
|
\tilde{I}_{m}^{\mbn}(y_{a},y_{b})|
\lesssim&
2^{\go{s}(u)(|m|_{\s} + |\s| + \kappa)}
\cdot
\|\rho\|_{\kappa}
\cdot
2^{
\go{s}(\rho_{\go{T}_{1}})\theta
+
(\go{s}(\rho_{\go{T}_{1}}) - \go{s}(u))
(|m|_{\s} +1)}
\wnorm{H}_{\theta,j+1}\;.
\end{equs}
Finally, for any $\mbn \in \mcN_{\go{G}}$ with $\mcT(\mbn) = \go{T}_{3}$ we set
\begin{equs}
\tilde{I}_{m}^{\mbn}(y_{a},y_{b})
\eqdef&
H_{n_{\{b,c\}}}(y_{b})
\cdot
\Psi^{(n_{\{a,c\}})}(y_{a})\\
{}&
\Big(
D^{m}\rho_{n_{\{a,b\}}}(y_{a} - y_{b})
-
\sum_{|k|_{\s}
<
\alpha
}
D^{k+m}\rho_{n_{\{a,b\}}}(y_{a})
\frac{(-y_{b})^{k}}{k!}
\Big)
\end{equs}
As before, writing $\go{s} \eqdef \go{s}(\mbn)$ and $\mathring{\go{T}}_{3} = \{ \rho_{\go{T}_{3}},v\}$, a Taylor remainder estimate gives
\begin{equs}
\|
\tilde{I}_{m}^{\mbn}(y_{a},y_{b})
\|
\lesssim&
\exp_{2}\Big[\go{s}(\rho_{\go{T}_{3}})(|m|_{\s} + |\s| + \kappa)
+
(\go{s}(\rho_{\go{T}_{3}}) - \go{s}(v))
(\lfloor \theta \rfloor + 1)
\Big]
\cdot
\|\rho\|_{\kappa}\\
{}
&
\cdot
\exp_{2}
\Big[
\go{s}(v)\theta 
\Big]
\wnorm{H}_{\theta,j}
\end{equs}
The proof for our bounds for second cumulants is finished upon observing that for $1 \le i \le 4$, 
\begin{equs}
\sum_{
\go{s} \in \mathrm{Lab}_{\go{T}_{i}}(\bar{n})
}
\sum_{
\substack{
\mbn \in \mcN_{\mathrm{tri}}(\go{T},\go{s})
}
}
\int
dy_{b}\ 
|\tilde{I}_{m}^{\mbn}(y_{a},y_{b})|
&
\lesssim
\sum_{
\go{s} \in \mathrm{Lab}_{\go{T}_{i}}(\bar{n})
}
2^{\langle \varsigma_{\go{T}_{i}},\go{s} \rangle}
2^{-|\s|\go{s}(b^{\uparrow})}
\|
\tilde{I}_{m}
\|_{\varsigma,\mcN_{\go{G}}}\\
&
\lesssim
2^{\bar{n} (\theta + |m|_{\s} + \kappa)}
\|
\tilde{I}_{m}
\|_{\varsigma,\mcN_{\go{G}}}.
\end{equs}
The factor $2^{-|\s|\go{s}(b^{\uparrow})}$ in going to the RHS of the first line comes from the integration of $y_{b}$.
In going to the last line we summed over the labelings $\go{s} \in \mathrm{Lab}_{\go{T}_{i}}$ - they key point is that $\go{s}(\{a,c\}^{\uparrow})$ can be treated as fixed to the value $\bar{n}$. 
For $\go{T}_{4}$ the bound is then immediate since there is no other scale labeling to sum over. 
For $\go{T}_{1},\go{T}_{2},\go{T}_{3}$ one has an infinite sum over the scale labeling for the other internal node - this scale will be constrainted to be greater than $\bar{n}$ for $\go{T}_{1},\go{T}_{3}$ and less than $\bar{n}$ for $\go{T}_{2}$.\\

We now obtain the desired bounds for higher cumulants so fix $M \ge 3$. 
Let $\go{G}$ be the complete graph on $[M]$. 
We will also commit an abuse of notation - for any multi-index $k \in \N^{[M]}$ and function $H$ on $(\R^{d})^{[M]}$ we will also denote by $H_{k}$ the family of functions $H_{k} \eqdef (H_{k}^{\mbn})_{\mbn \in \N^{\go{G}}}$ given by
\[
H^{\mbn}_{k}(x_{1},\dots,x_{M}) 
\eqdef
\Big(
D^{k}
H(x_{1},\dots,x_{M})
\Big)
\cdot
\prod_{\{i,j\} \in \go{G}}\Psi^{(n_{\{i,j\}})}(x_{i} - x_{j}).
\]
Now given a function $H$ smooth on $(\R^{d})^{[M]} \setminus \mathrm{Diag}_{M}$, invariant under $M$-fold simultaneous $\R^{d}$ translations \footnote{that is, $H(x_{1},\dots,x_{M}) = H(x_{1} + h, \dots, x_{M} + h)$ for any $(x_{i})_{i=1}^{M} \in (\R^{d})^{[M]}$ and $h \in \R^{d}$}, a total homogeneity $\varsigma$ on the trees of $\widehat{\CU}_{[M]}$, a subset $J \subset [M]$, and $p \in \N$ we define
\begin{equs}
\wnorm{H}_{\varsigma,J,p}
\eqdef&
\max_{k \in A_{J,p}}
\|H_{k}\|_{D^{k}\varsigma,\N^{\go{G}}}\\
\textnormal{where }
A_{J,p} 
&\eqdef 
\Big\{ 
(k_{i})_{i=1}^{M} \in (\N^{d})^{[M]}:\
\forall 1 \le j \le M,\ ||k_{j}|| \le (p+\mathbbm{1}\{ j \in J \})|\s|
\Big\}\;.
\end{equs}
On the right hand side of the first line we are using the notation of \eqref{supremum bound}, choosing arbitrarily a particular element of $[M]$ as a root and a value to fix it as (due to translation invariance these choices do not matter). 
For any $j \in [J]$, $ \rho \in \widetilde{\Moll}$, and $H$ as above we define
\[
(H \ast_{j} \rho)(x_{1},\dots,x_{M})
\eqdef
\int_{\R^{d}}
dz_{j}\ 
H(x_{1},\dots,x_{j-1},y_{j},x_{j+1},\dots,x_{M})
\rho(x_{j} - z_{j})\;.
\]
The claim we wish to prove for higher cumulants is that for any  $j \in J \subset [M]$, $p \le 4$, and $\rho \in \widetilde{\Moll}$ along with a total homogeneity $\varsigma$ and function $H$ as above,
\begin{equ}\label{penultimate claim}
\wnorm{H \ast_{j} \rho}_{\varsigma + \kappa \delta^{\uparrow}[j], J \setminus \{j\},p}
\lesssim
\wnorm{H}_{\varsigma,J,p} \cdot \| \rho \|_{\kappa}.
\end{equ}
Applying claim $M$ times will give us the desired bound. To see this first note that for any $(\mft,[M]) \in \mfL_{\CCum}$ one has $\wnorm{\Cum[\{\xi_{\mft(i)}\}_{i=1}^{M}]}_{\c,[M],r} \le \|\xi\|_{M,\c,r+1}$ and that $\wnorm{\Cum[\{\xi_{\mft(i)}\}_{i=1}^{M}]}_{\c,\emptyset,r}$ controls the contribution of $\Cum[\{\xi_{\mft(i)}\}_{i=1}^{M}]$ to  $\|\xi\|_{M,\c,r}$. 

We now turn to proving, by using a multiscale expansion, the claim \eqref{penultimate claim}. Fix appropriate $j$, $J$, $\theta$, $p$, and $\rho$.  
The underlying vertex set is given by $\CV \eqdef [M] \sqcup \{o\}$ where the additional vertex $o$ corresponds to the integration variable $z_{j}$ in our formula for $H \ast_{j} \rho$. 
We choose $j \in \CV$ as a root vertex (so $\CV_{0} \eqdef \CV \setminus \{j\}$) and arbitrarily pin its position $y_{j}$ to some value in $\R^{d}$.
We also fix $\go{G}' \eqdef \go{G} \sqcup \{\{o,j\}\}$ and set $\mcN_{\go{G}'} \eqdef \N^{\go{G}'}$.

For any fixed $\bar{\mbn} \in \N^{\go{G}}$ and $n_{\{o,j\}} \in \N$ observe that if one sets $\go{S} \eqdef \mcT(\bar{\mbn}) \in \widehat{\CU}_{[M]}$ and $\go{T} \eqdef \mcT(\bar{\mbn} \sqcup n_{\{o,j\}}) \in \CU_{\CV}$ then $\go{S} = \go{T}\restr_{[M]}$ (see Definition~\ref{tree restriction}).

Now fix $k \in A_{J\setminus \{j\},p}$. 
For any $\mbn \in \mcN_{\go{G}}$ we define a family of functions $\mathring{I}_{k} = \{\mathring{I}_{k}^{\mbn}\}_{\mbn \in \mcN_{\go{G}'}}$ via setting
\begin{equs}
\mathring{I}_{k}^{\mbn}(y_{\CV_{0}})
\eqdef&
\rho_{n_{\{o,j\}}}(y_{j} - y_{o})\tilde{H}^{\mbn}_{k}(y_{\CV_{0}})\;,\textnormal{ where}\\
\tilde{H}^{\mbn}_{k}(y_{\CV_{0}})
\eqdef&
\Big(
D^{k}H(y_{1},\dots,y_{j-1},y_{o},y_{j+1},y_{M})
\Big)
\prod_{
\substack{
\{i,i'\} \in \go{G}'\\
\{i,i'\} \not = \{o,j\}
}
}
\Psi^{(n_{\{i,j\}})}(y_{i} - y_{j})\;.
\end{equs}
With these definitions we have
\[
D^{k}(H \ast_{j} \rho)(y_{\CV_{0} \setminus \{o\}})
=
\sum_{\go{T} \in \widehat{\CU}_{\CV}}
\sum_{\go{s} \in \mathrm{Lab}_{\go{T}}}
\sum_{\mbn \in \mcN_{\tr}(\go{T},\go{s})}
\int dy_{o}\  
\mathring{I}_{k}^{\mbn}(y_{\CV_{0}})\;.
\]
To prove \eqref{penultimate claim} it suffices to find a modification $\tilde{I}_{k} \eqdef \{ \tilde{I}_{k}^{\mbn}\}_{\mbn \in \mcN_{\go{G}'}} \in \mathrm{Mod}_{\{o\}}(\mathring{I}_{k})$ such that, uniform in $(\go{S},\bar{\go{s}}) \in \widehat{\CU}_{[M]} \ltimes \mathrm{Lab}_{\bullet}$ and $\bar{\mbn} \in \mcN_{\tr}(\go{S},\bar{\go{s}})$, one has the bound
\begin{equ}\label{final claim}
\sum_{
\substack{
\go{T} \in \CU_{\CV}\\
\go{T}\restr_{[M]} = \go{S}
}}
\sum_{
\substack{
\go{s} \in \mathrm{Lab}_{\go{T}}\\
\go{s}\restr_{\mathring{\go{S}}}
=
\bar{\go{s}}}
}
\sum_{
\substack{
\mbn \in \mcN_{\tr}(\go{T},\go{s})\\
\mbn \restr_{\go{G}} = \bar{\mbn}
}
}
\Big|
\int_{\R^{d}}dy_{o}\ 
\tilde{I}^{\mbn}(y_{\CV_{0}})
\Big|
\lesssim
\|\rho\|_{\kappa}
\wnorm{H}_{\varsigma,J,p}
\cdot
2^{\langle D^{k}\varsigma_{\go{S}} + \kappa \delta_{\go{S}}^{\uparrow}[j],\bar{\go{s}}\rangle}\;.
\end{equ}
Above we are committing an abuse of notation - for each $\go{T} \in \widehat{\CU}_{\CV}$ with $\go{T} \restr_{[M]} = \go{S}$ we use the inclusion map\footnote{See Definition~\ref{tree restriction}.} $\iota: \mathring{\go{S}} \rightarrow \mathring{\go{T}}$ to identify $\mathring{\go{S}}$ as a subset of $\mathring{\go{T}}$. 
This conventiion is used when we write the condition $\go{s}\restr_{\mathring{\go{S}}}
=
\bar{\go{s}}$ and is used throughout the rest of the proof.
Any use of the notation $\go{L}_{\bullet}$ in what follows refer to using the relevant tree of $\widehat{\CU}_{\CV}$ to determine descendent leaf vertices. 

We now obtain \eqref{final claim}.
We start with defining a total homogeneity $\hat{\varsigma}$ on the trees of $\widehat{\CU}_{\CV}$ by setting, for $\go{T} \in \widehat{\CU}_{\CV}$ and $a \in \mathring{\go{T}}$, 
\begin{equs}
\hat{\varsigma}_{\go{T}}(a)
\eqdef& 
(|\s| + \kappa)
\cdot
\delta_{\go{T}}^{\uparrow}[\{o,j\}]
+
\mathbbm{1}
\Big\{ 
\go{L}_{\{o,j\}^{\uparrow,\go{T}}} = \{o,j\}
\Big\}
\Big(\delta_{\go{T}}^{\Uparrow}[\{o,j\}] 
-
\delta_{\go{T}}^{\uparrow}[\{o,j\}]
\Big)\\
{}
&
\hspace{1cm}
+
\begin{cases}
\varsigma_{\go{S}}(a) &\ 
\textnormal{if }a \in \mathring{\go{S}} \textnormal{ for } \go{S} \eqdef \go{T}\restr_{[M]}\\
0 &\ \textnormal{otherwise.}
\end{cases}
\end{equs}
We now specify $\tilde{I}_{k} \in \mathrm{Mod}_{\{o\}}(\mathring{I}_{k})$ with
\begin{equ}\label{convergence of measures: final norm bound}
\|\tilde{I}_{k}\|_{\hat{\varsigma},\mcN_{\go{G}'}}
\lesssim 
\| \rho \|_{\kappa} \cdot \wnorm{H}_{\varsigma,J,p}\;.
\end{equ} 
For each $\go{T} \in \CU_{\CV}$ and $\mbn \in \mcN_{\go{G}'}$ with $\mcT(\mbn) = \go{T}$ we set $\tilde{I}_{k}^{\mbn} \eqdef \mathring{I}_{k}^{\mbn} $ if $\go{L}_{\{o,j\}^{\uparrow},\go{T}} \supsetneq \{o,j\}$ and if $\go{L}_{\{o,j\}^{\uparrow},\go{T}} = \{o,j\}$ we set
\begin{equ}
\tilde{I}_{k}^{\mbn}(y_{\CV_{0}})
\eqdef
\Big(
\tilde{H}^{\mbn}_{k}(y_{\CV_{0} \setminus \{o\}},y_{o}) -
\tilde{H}^{\mbn}_{k}(y_{\CV_{0} \setminus \{o\}},y_{j}) 
\Big)
\cdot
\rho_{n_{\{o,j\}}}(y_{j} - y_{o})\;.
\end{equ}
The necessary integration property, support property, and the estimate \eqref{convergence of measures: final norm bound} are all straightforward to check. 
Now fix $(\go{S},\bar{\go{s}}) \in \widehat{\CU}_{[M]} \ltimes \mathrm{Lab}_{\bullet}$, $\bar{\mbn} \in \mcN_{\tr}(\go{S},\bar{\go{s}})$, and $\go{T} \in \widehat{\CU}_{\CV}$ with $\go{T}\restr_{[M]} = \go{S}$. 
By using the bound \eqref{convergence of measures: final norm bound} we get
\begin{equs}
\sum_{
\substack{
\go{s} \in \mathrm{Lab}_{\go{T}}\\
\go{s}\restr_{\mathring{\go{S}}}
=
\bar{\go{s}}}
}
\sum_{
\substack{
\mbn \in \mcN_{\tr}(\go{T},\go{s})\\
\mbn \restr_{\go{G}} = \bar{\mbn}
}
}
\Big|
\int_{\R^{d}}dy_{o}\ 
\tilde{I}_{k}^{\mbn}(y_{\CV_{0}})
\Big|
\lesssim
\|\rho\|_{\kappa}
\cdot
\wnorm{H}_{\varsigma,J,p}
\cdot
\sum_{
\substack{
\go{s} \in \mathrm{Lab}_{\go{T}}\\
\go{s}\restr_{\mathring{\go{S}}}
=
\bar{\go{s}}}
}
2^{\langle \hat{\varsigma}_{\go{T}},\go{s} \rangle - \go{s}(o^{\uparrow})|\s|}\;.
\end{equs}
Again, the factor $2^{-\go{s}(o^{\uparrow})|\s|}$ comes from integration. 
What remains is showing that 
\begin{equ}\label{final inequality}
\sum_{
\substack{
\go{s} \in \mathrm{Lab}_{\go{T}}\\
\go{s}\restr_{\mathring{\go{S}}}
=
\bar{\go{s}}}
}
2^{\langle \hat{\varsigma}_{\go{T}},\go{s} \rangle - \go{s}(o^{\uparrow})|\s| }
\lesssim 
2^{\langle D^{k}\varsigma_{\go{S}} + \theta \delta_{\go{S}}^{\uparrow}[j],\bar{\go{s}}\rangle}\;.
\end{equ}
To verify the above inequality one can split into three different cases. 
Remember that it must always be the case that $\go{L}_{o^{\uparrow,\go{T}}} \ni j$ because of our choice of $\go{G}'$. 
We draw an example with $M=3$, the first tree represents $\go{S}$ and the the three other trees show examples of our three cases.
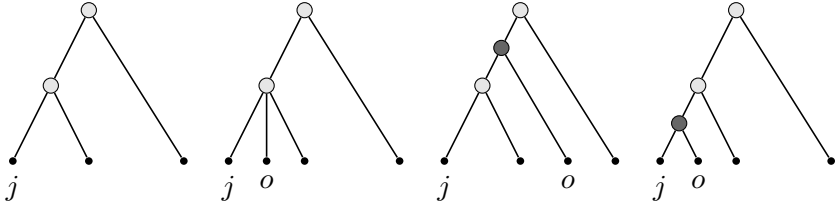
\begin{equ}
\begin{tikzpicture}[scale = .75]
    \node [style=coalnode] (root) at (0, 2) {};
    \node [style=dot, label=below:$j$] (j) at (-1, 0) {};
    \node [style=dot] (1) at (0, 0) {};
    \node [style=coalnode] (inner) at (-0.5, 1) {};
    \node [style=dot] (2) at (1.25, 0) {};

    \draw[coalline] (j) -- (inner);
    \draw[coalline] (1) -- (inner);

    \draw[coalline] (inner) -- (root);

    \draw[coalline] (2) -- (root);
\end{tikzpicture}
\;\;\;
\begin{tikzpicture}[scale = .75]
    \node [style=coalnode] (root) at (0, 2) {};
    \node [style=dot, label=below:$j$] (j) at (-1, 0) {};
    \node [style=dot, label=below:$o$] (o) at (-.5, 0) {};
    \node [style=dot] (1) at (0, 0) {};
    \node [style=coalnode] (inner) at (-0.5, 1) {};
    \node [style=dot] (2) at (1.25, 0) {};

    \draw[coalline] (j) -- (inner);
    \draw[coalline] (1) -- (inner);
    \draw[coalline] (o) -- (inner);

    \draw[coalline] (inner) -- (root);

    \draw[coalline] (2) -- (root);
\end{tikzpicture}
\;\;\;
\begin{tikzpicture}[scale = .75]
    \node [style=coalnode] (root) at (0, 2) {};
    \node [style=dot, label=below:$j$] (j) at (-1, 0) {};
    \node [style=dot, label=below:$o$] (o) at (.625, 0) {};
    \node [style=dot] (1) at (0, 0) {};
    \node [style=coalnode] (inner) at (-0.5, 1) {};
    \node [style=dot] (2) at (1.25, 0) {};
    \node [style=coalnode2] (inner2) at (-0.25,1.5) {};

    \draw[coalline] (j) -- (inner);
    \draw[coalline] (1) -- (inner);
    \draw[coalline] (o) -- (inner2);

    \draw[coalline] (inner2) -- (inner);

    \draw[coalline] (inner2) -- (root);

    \draw[coalline] (2) -- (root);
\end{tikzpicture}
\;\;\;
\begin{tikzpicture}[scale = .75]
    \node [style=coalnode] (root) at (0, 2) {};
    \node [style=dot, label=below:$j$] (j) at (-1, 0) {};
    \node [style=dot, label=below:$o$] (o) at (-.5, 0) {};
    \node [style=dot] (1) at (0, 0) {};
    \node [style=coalnode] (inner) at (-0.5, 1) {};
    \node [style=dot] (2) at (1.25, 0) {};
    \node [style=coalnode2] (inner2) at (-.75,.5) {};

    \draw[coalline] (j) -- (inner2);
    \draw[coalline] (1) -- (inner);
    \draw[coalline] (o) -- (inner2);

    \draw[coalline] (inner2) -- (inner);

    \draw[coalline] (inner) -- (root);

    \draw[coalline] (2) -- (root);
\end{tikzpicture}
\end{equ}
The first case occurs when $\mathring{\go{T}} =\mathring{\go{S}}$, here there is nothing to prove since the sum of the LHS only has one term (where $\go{s} = \bar{\go{s}}$) so both sides of the inequality are equal. 

In the second and third cases one has $\mathring{\go{T}} = \mathring{\go{S}} \sqcup \{ b \}$ where $b = \{o,j\}^{\uparrow}$ is represented with a darker node in our pictures.  
The sum over $\go{s}$ with $\go{s}\restr_{\mathring{\go{S}}} = \bar{\go{s}}$ is just a sum over the value of $\go{s}(b)$ constrained to satisfy an upper bound in the second case or lower bound in the third case. 

For the second case we assume that $\go{L}_{\{o,j\}^{\uparrow,\go{T}}} \supsetneq \{o,j\}$. 
It follows that $b$ is not maximal in the partial order of $\mathring{\go{T}}$ and in fact $j^{\uparrow,\go{T}} = j^{\uparrow,\go{S}} > b$. 
We then have
\begin{equs}
\sum_{
\substack{
\go{s} \in \mathrm{Lab}_{\go{T}}\\
\go{s}\restr_{\mathring{\go{S}}}
=
\bar{\go{s}}}
}
2^{\langle \hat{\varsigma}_{\go{T}},\go{s} \rangle - \go{s}(o^{\uparrow})|\s| }
\lesssim&
\sum_{\go{s}(b) < \bar{\go{s}}(j^{\uparrow})}
2^{(|\s| + \kappa - |\s|) \go{s}(b) }
\prod_{a \in \mathring{\go{S}}} 2^{\hat{\varsigma}_{\go{T}}(a)\bar{\go{s}}(a)}\\
\lesssim&
2^{\kappa \bar{\go{s}}(j^{\uparrow})}
\prod_{a \in \mathring{\go{S}}} 2^{\hat{\varsigma}_{\go{T}}(a)\bar{\go{s}}(a)}
=
2^{\langle D^{k}\varsigma_{\go{S}} + \kappa \delta_{\go{S}}^{\uparrow}[j],\bar{\go{s}}\rangle}\;.
\end{equs}
In the third case one has $\go{L}_{\{o,j\}^{\uparrow,\go{T}}} = \{o,j\}$.
It follows that $b$ is maximal in $\mathring{\go{T}}$ and $b^{\uparrow} = j^{\uparrow,\go{S}}$, thus we have
\begin{equs}
\sum_{
\substack{
\go{s} \in \mathrm{Lab}_{\go{T}}\\
\go{s}\restr_{\mathring{\go{S}}}
=
\bar{\go{s}}}
}
2^{\langle \hat{\varsigma}_{\go{T}},\go{s} \rangle - \go{s}(o^{\uparrow})|\s| }
\lesssim&
\sum_{\go{s}(b) > \bar{\go{s}}(j^{\uparrow,\go{S}})}
2^{(|\s| + \kappa - |\s| - 1) \go{s}(b) + \bar{\go{s}}(j^{\uparrow,\go{S}})}
\prod_{a \in \mathring{\go{S}}} 2^{\hat{\varsigma}_{\go{T}}(a)\bar{\go{s}}(a)}\\
\lesssim&
2^{\kappa \bar{\go{s}}(j^{\uparrow,\go{S}})}
\prod_{a \in \mathring{\go{S}}} 2^{\hat{\varsigma}_{\go{T}}(a)\bar{\go{s}}(a)}
=
2^{\langle D^{k}\varsigma_{\go{S}} + \kappa \delta_{\go{S}}^{\uparrow}[j],\bar{\go{s}}\rangle}\;.
\end{equs}
\end{proof}
\newpage
\section{Symbolic index}

In this appendix, we collect the most used symbols of the article, together
with their meaning and the page where they were first introduced.

 \begin{center}

 \renewcommand{\arraystretch}{1.1}
 \begin{tabular}{lll}\toprule
 Symbol & Meaning & Page\\
 \midrule
 $\bullet$ & The trivial tree consisting only of a root. & \pageref{subsec: trees}\\
 $\mathbf{1}$ & The empty forest. Not considered a tree. & \pageref{subsec: forests}\\
 $\Lab$ & Finite set of ``types'', used to distinguish each& {} \\
 {}     & of the kernels and noises appearing in the system SPDE. & \pageref{def: set of types}\\
 $\Ke$  & Set of $\mft \in \Lab$ which are kernel types, satisfy $|\mft|_{\s} > 0$. & \pageref{def: kernel and leaf types}\\
 $\Le$  & Set of $\mft \in \Lab$ which are leaf types, satisfy $|\mft|_{\s} < 0$. & \pageref{def: kernel and leaf types}\\
 $\mcb{B}_{{\s}}$ & Set of test functions used to & {}\\
 {} & define seminorms on space of models. & \pageref{space of functions}\\
 $\mfL^{\all}_{\CCum}$  & Prescribed set of non-vanishing cumulants. & \pageref{assump: set of cumulants}\\
 $p_{k}$ & For $k \in \N^{d}$, test function given by spatial truncation of $x^{k}$ & \pageref{def: cumulantbound}\\
 $\mcb{j}_{A}(B)$ & Minimum homogeneity gain from noises & {}\\
 {} & in $B$ contracting with noises with types from $A$. & \pageref{external cumulant homogeneity jump}\\
 $\overline{\CA}$ & Set of maximal elements of the poset $\CA$. & \pageref{posets}\\
 $\Max{\CA}$ & Set of maximal elements of the poset $\CA$. & \pageref{posets}\\
 $\underline{\CA}$ & Set of minimal elements of the poset $\CA$. & \pageref{posets}\\
 $\Min{\CA}$ & Set of minimal elements of the poset $\CA$. & \pageref{posets}\\
 $\Div$ & All superficially divergent subtrees & \pageref{Div}\\
 $\cut$ & All potentially useful cuts (edges).& \pageref{def: set of cuts}\\
 $\cC_{\mathbb{M}}$ & Collection of edges in $\cut$ which are not in any element of $\mathbb{M}$. & \pageref{def: cuts away from interval}\\
 $\cut^P(\mathbb{M})$ & Collection of cut sets $\cC$ generating $\mathbb{M}$: $P_{\cC}^{-1}[s(\mathbb{M})] = \mathbb{M}$. & \pageref{equ: cuts giving interval}\\
$\mfM^{P}(\cC)$ & All forest intervals $\mathbb{M}$ generated by $\cC$. & \pageref{equ: intervals arising from cut set}\\
$P^{\mbn}$ & Discards ``dangerous'' trees with respect to $\mbn$. & \pageref{def: projection onto safe forests}\\
$\cG^\mbn(\mcF)$ & Kernel edges benefiting from cancellation when $\mcF$ is contracted. & \pageref{def: cuts to harvest}\\
 $\CF, \CG$ & Generic forests of subtrees. & \pageref{forest of subtrees} \\
 $\mathbb{F}$ & All forests of superficially divergent subtrees. & \pageref{bbF} \\ 
 $\mathbb{F}^{\le k}$ & Forests in $\mathbb{F}$ of depth at most $k$. & \pageref{e:maxForest} \\ 
 $\mathbb{F}[\mathcal{F}]$ & All forests $\CG$ with $\Max \CG = \CF$. & \pageref{e:maxForest} \\
 $\mathbb{F}_<[\mathcal{F}]$ & All $\CG \in \mathbb{F}^{\le 1}$ with trees strictly contained
 in trees of $\CF$. & \pageref{e:maxForest} \\
 $\mathbb{F}_{\le}[\mathcal{F}]$ & All $\CG \in \mathbb{F}^{\le 1}$ with trees contained
 in trees of $\CF$. & \pageref{e:maxForest} \\
 $N_F$ & All nodes of the forest $F$. & \pageref{subsec: forests}\\
 $E_F$ & All edges of the forest $F$. & \pageref{subsec: forests}\\
 $\mathbb{L}(F)$ & Edges of noise type. & \pageref{e:leaf and kernel edges}\\
 \bottomrule
 \end{tabular}

 \newpage
 \begin{tabular}{lll}\toprule
 Symbol & Meaning & Page\\
 \midrule
  $K(F)$ & Edges of kernel type. & \pageref{e:leaf and kernel edges}\\
 $L(F)$ & True nodes corresponding to $\mathbb{L}(F)$, inherits types. & \pageref{def: leaf nodes}\\
 $\sT_{\not \ge}[\cC]$ & Intersection of all subtrees $\{ \sT_{\not\ge}(e): e \in \Min(\cC)\}$ & \pageref{def: subtrees from cut sets}\\
 $N(F)$ & True nodes: $N(F) = N_F \setminus e_\ch(L(F))$. & \pageref{e:leaf and kernel edges}\\
 $\allnodes$ & True nodes of $\bar T$, plus $\logof$. & \pageref{subsec: mapping to analytic objects}\\
 $\logof$ & Node representing basepoint of the model. & \pageref{def: basepoint vertex}\\ 
 $\tilde N(F)$ & Defined as $\tilde N(F) = N(F) \setminus \rho(F)$. & \pageref{e:leaf and kernel edges}\\
 $T_\ge(e)$ & Subtree of $T$ above the edge $e$ (including $e$). & \pageref{subsubsec: positive renorm}\\
 $T_{\not\ge}(e)$ & Subtree of $T$ formed by edges below or not & {}\\
 {} & comparable to the edge $e$. & \pageref{subsubsec: positive renorm} \\
 $\sT_{\not \ge}[\cC]$ & Intersection of all subtrees $\{ \sT_{\not\ge}(e): e \in \Min(\cC)\}$ & \pageref{def: subtrees from cut sets}\\
 $\Tr_{ \ge}[\cC]$ & Collection of maximal subtrees of $\sT$ whose roots are of & {}\\
 {} & degree $1$ and are edge disjoint $\sT_{\not \ge}[\cC]$. & \pageref{def: subtrees from cut sets}\\
 $K^\downarrow(S)$ & Kernel edges incoming to $S$ (i.e. $e_{\p} \in N_{S}$, $e_{\ch} \not \in N_{S}$) in $\bar T$. & \pageref{def: some kernel edge sets}\\
 $\bar K^\downarrow(S)$ & Defined as $K^\downarrow(S) \cup K(S)$. & \pageref{def: some kernel edge sets}\\
 $K_T^\d(S)$ & Defined for $T \le S$ as $K^\downarrow(T) \cap K(S)$. & \pageref{def: some modulo forest kernel sets} \\
 $N^\downarrow(S)$ & Defined as $e_\ch(K^\downarrow(S))$. & \pageref{def: some kernel edge sets} \\
 $C_\CF(S)$ & Maximal elements of $\CF$ restricted to subtrees of $S$. & \pageref{def: immediate children in forest}\\
 $\tilde N_\CF(S)$ & Nodes in $\tilde N(S)$, but not in $\tilde N(T)$
 for $T \in C_\CF(S)$. & \pageref{def: some modulo forest node sets}\\
 $N_\CF(S)$ & Defined as $\tilde N_\CF(S) \cup \{\rho_S\}$. & \pageref{def: some modulo forest node sets}\\
 $\mathring K_\CF(S)$ & Kernel edges in $S$ that neither belong to & \\
 {} & nor are adjacent to 
 any subtree of $C_\CF(S)$. & \pageref{def: some modulo forest kernel sets} \\
 $K_\CF^\d(S)$ & Union of $K_T^\d(S)$ over $T \in C_\CF(S)$. & \pageref{def: some modulo forest kernel sets}\\
 $\fict(B)$ & Gives power-counting gain for renormalization of $\Cum_{\mft(B)}$. & \pageref{def: fict}\\
 $\mfh_{\c,D}(B)$ & Worst-case homogeneity gain from noises in $B$& {}\\
 {} & participating in ``external'' versus ``internal'' cumulants & {}\\
 {} & with external noises drawn from $D$ & \pageref{homogeneity gain from missed noise}\\
 $|\mft(B)|_{\s}$ & Total homogeneity of the typed set $B$ & \pageref{def: homogeneity of a set}\\
 $|\mft(B)|_{\s,\c,D}$ & Worst-case homogeneity given to noises $B$ & {}\\
 {} &  when part of an external cumulant with noises from $D$. & \pageref{def: external homogeneity}\\ 
 $\ident[\mfF_{2}]$ & Collection of all decorated i-forests. & \pageref{def: set of identified forests}\\
 $\ident[\mfT_{2}]$ & Collection of all decorated i-trees. & \pageref{def: set of identified forests}\\
 $\ident[\mfF_{i}],\ \ident[\mfT_{i}]$ & For $i \in \{0,1\}$, set of decorated i-forests & {}\\
 {} & or i-trees with color at most $i$. & \pageref{def: set of identified forests}\\
 $\ident[\widehat{\Tr}_{2}]$ & Collection of all decorated i-trees with root colored $2$. & \pageref{def: i-trees with 2 colored root}\\
 $L_{\mcF}(S)$ & $e \in L(S)$ but not in $L(T)$ for every $T \in C_{\mcF}(S)$. & \pageref{def: leaves mod forest}\\
 $\mcb{C}(\R^{d})$ & Collection of smooth real valued functions on $\R^{d}$. & \pageref{def: smooth functions 1}\\
 $\mcb{C}_{A}$ & All smooth real valued functions on $(\R^{d})^{A}$. & \pageref{def: smooth functions 2}\\ 
 $\allf$ & Collection of all smooth real valued functions on $(\R^{d})^{\allnodes}$. & \pageref{def: all smooth functions}\\
 $\powrootquot{N}{\mfn}{v}{v'}$, etc. & Various notations for products of monomials & {}\\
 {} & or binomials in $\allf$. & \pageref{def: various power functions}\\
 $\genvert{\pi,\mcF,S}$ & Renormalized distribution for divergent $S$. & \pageref{def: renormalized kernel 1}\\
 $C_{\mathscr{C}}(e)$ & For $e \in \cut$, $\mathscr{C} \subset \cuts$, given by $\mathrm{Min}
 \left(
 \{ \bar{e} \in \cC:\ \bar{e} > e \}
 \right)$.& \pageref{def: immediate children - cuts}\\
 \bottomrule
 \end{tabular}
 \end{center}

\endappendix 

\bibliographystyle{./Martin}
\bibliography{./refs}

\end{document}